\documentclass[11pt,a4paper]{article}
\usepackage{amsmath}
\usepackage{amsfonts}
\usepackage{amssymb}
\usepackage{amsthm}
\usepackage{mathrsfs}
\usepackage{dsfont}
\usepackage{paralist}
\usepackage{natbib}
\usepackage{bm}
\usepackage{color}
\usepackage[colorlinks=true,linkcolor=blue,citecolor=blue,pdfborder={0 0 0}]{hyperref}
\usepackage{graphicx}
\usepackage{tabularx}
\usepackage[left=2.35cm,right=2.35cm,top=1.85cm,bottom=1.85cm,includeheadfoot]{geometry}
\usepackage{comment}

\newcount\Comments  
\newcommand{\kibitz}[2]{\ifnum\Comments=1\textcolor{#1}{#2}\fi}

\newtheoremstyle{normal}
{2ex}               
{3ex}               
{}                  
{}                  
{\bfseries} 
{}                  
{2pt}   
{\thmname{#1}\thmnumber{ #2.} \thmnote{(#3)}}

\newtheoremstyle{italic}
{2ex}
{3ex}
{\itshape}
{}
{\bfseries} 
{}
{2pt}
{\thmname{#1}\thmnumber{ #2.} \thmnote{(#3)}}

\theoremstyle{normal}
\newtheorem{definition}{Definition}[section]
\newtheorem{remark}[definition]{Remark}
\newtheorem{example}[definition]{Example}

\newtheorem{assumption}[definition]{Assumption}

\theoremstyle{italic}
\newtheorem{theorem}[definition]{Theorem}
\newtheorem{lemma}[definition]{Lemma}
\newtheorem{proposition}[definition]{Proposition}
\newtheorem{corollary}[definition]{Corollary}

\newcommand\N{\mathbb{N}}
\newcommand\Q{\mathbb{Q}}
\newcommand\R{\mathbb{R}}

\newcommand\eps{\varepsilon}
\newcommand\Prob{\mathbb{P}}    
\newcommand\ind{\mathds{1}}     

\newcommand\Bc{\mathcal{B}}

\newcommand\Dc{\mathcal{D}}

\newcommand\Fc{\mathcal{F}}
\newcommand\Gc{\mathcal{G}}

\newcommand\Lc{\mathcal{L}}
\newcommand\Mc{\mathcal{M}}
\newcommand\Oc{\mathcal{O}}
\newcommand\Pc{\mathcal{P}}
\newcommand\Rc{\mathcal{R}}

\newcommand\Bb{\mathbb{B}}
\newcommand\Cb{\mathbb{C}}
\newcommand\Db{\mathbb{D}}
\newcommand\Eb{\mathbb{E}}

\newcommand\Gb{\mathbb{G}}
\newcommand\Hb{\mathbb{H}}
\newcommand\Kb{\mathbb{K}}

\newcommand\Tb{\mathbb{T}}

\newcommand\Vb{\mathbb{V}}

\newcommand\Yb{\mathbb{Y}}

\newcommand\pf{\mathfrak p}
\newcommand\qf{\mathfrak q}

\newcommand\ovw{\overline{w}}
\newcommand\ovr{\overline{r}}
\newcommand\ovv{\overline{v}}

\newcommand\ovp{\overline{p}}

\newcommand\netir{[0,1]\times\R}

\newcommand\kseqi{\zeta}
\newcommand\gseqi{\theta}
\newcommand\taili{z}
\newcommand\dfi{t}

\newcommand\mingi{\theta_1 \wedge \theta_2}

\newcommand\mindfi{t_1 \wedge t_2}
\newcommand\maxdfi{t_1 \vee t_2}

\newcommand\linne{\ell^{\infty}([0,1])}

\newcommand\pto{\stackrel{\scriptscriptstyle \mathbb P^*}{\rightarrow}}
\newcommand\probto{\stackrel{\scriptscriptstyle \mathbb P}{\rightarrow}}
\newcommand\defeq{:=}
\newcommand\weak{\rightsquigarrow}
\newcommand\weakP{{\, {\weak_\xi}\ }}
\newcommand{\ip}[1]{\lfloor #1 \rfloor}
\newcommand\Deli{\Delta_i^n}
\newcommand\Delj{\Delta_j^n}
\newcommand\Delk{\Delta_k^n}
\newcommand\Indit{\ind_{(- \infty, t ]}}

\newcommand\Truniv{\ind_{\lbrace | \Delta_i^n X^{\scriptscriptstyle (n)} | > v_n \rbrace}}

\newcommand\Trunx{\ind_{\lbrace |\taili| > v_n \rbrace}}

\newcommand\linner{\ell^{\infty}([0,1]\times\R)}
\newcommand\ctir{C \times \R}
\newcommand\linctr{\ell^{\infty}(C \times \R)}

\newcommand\tibigj{\tilde X^{\prime \prime n}}
\newcommand\tibigjal{\tilde X^{\prime \prime}(\alpha)^n}
\newcommand\tibigjeial{\tilde X^{\prime \prime}(8 \alpha)^n}
\newcommand\hatbigjal{\hat X^{\prime \prime}(\alpha)^n}

\newcommand\smooi{\varpi}
\newcommand\thrle{\varkappa}

\newcommand\predfest{\hat \theta_n^*}

\newcommand\entcp{\theta_0}

\newcommand\haFthrleei{\hat \thrle_{n,B}^{(\alpha,\rho)}(1)}
\newcommand\haFthrle{\hat \thrle_{n,B}^{(\alpha,\rho)}(r)}

\newcommand\haFcothrle{\hat \thrle_{\scriptscriptstyle n,B_n}^{(\alpha_n,\rho)}(r)}
\newcommand\haFcothrleor{\hat \thrle_{\scriptscriptstyle n,B_n}^{(\alpha_n,\rho)}(r)}

\newcommand\empFrbodif{F_{n,B}^{(\rho,r)}}
\newcommand\empFrboindif{F_{n,B}^{(\rho,r)-}}

\newcommand\Rep{\text{Re}}

\allowdisplaybreaks[1]

\begin{document}
\title{On detecting changes in the jumps of arbitrary size of a time-continuous stochastic process}

\author{Michael Hoffmann\footnotemark[1], ~ Holger Dette \footnotemark[1]\  \bigskip \\
	{ Ruhr-Universit\"at Bochum }
}

\footnotetext[1]{Ruhr-Universit\"at Bochum,
	Fakult\"at f\"ur Mathematik, 44780 Bochum, Germany.
	{E-mail:} michael.hoffmann@rub.de}

\maketitle

\begin{abstract}
	
	This paper introduces test and estimation procedures for abrupt and gradual changes in the entire jump behaviour of a discretely observed It\=o semimartingale. In contrast to existing work we analyse jumps of arbitrary size which are not restricted to a minimum height. Our methods are based on weak convergence of a truncated sequential empirical distribution function of the jump characteristic of the underlying It\=o semimartingale. Critical values for the new tests are obtained by a multiplier bootstrap approach and we investigate the performance of the tests also under local alternatives. An extensive simulation study shows the finite-sample properties of the new procedures.
	
\end{abstract}


\smallskip

	
\section{Introduction}
\def\theequation{1.\arabic{equation}}
\setcounter{equation}{0}

Stochastic processes are widely used in science nowadays, as they allow for a flexible modelling of time-dependent phenomena. For example, in physics stochastic processes are used to explain the behaviour of quantum systems (see \citealp{vKa07}), but stochastic processes are also suitable for financial modelling. The seminal paper by \cite{DelSch94} suggests to use the special class of It\=o semimartingales in continuous time. Financial models based on It\=o semimartingales satisfy a certain condition on the absence of arbitrage and moreover they are still rich enough to accommodate stylized facts such as volatility clustering, leverage effects and jumps. As a consequence, in recent years a lot of research was focused on the development of statistical procedures for characteristics of It\=o semimartingales based on discrete observations. In particular, the importance of the jump component has been enforced by recent research (see \citealp{AitJac09b} and \citealp{AitJac09a}) and common methods in this field are gathered in the recent monographs by \cite{JacPro12} and \cite{AitJac14}. 

A fundamental topic in statistics for stochastic processes is the analysis of structural breaks. Corresponding test procedures, commonly referred to as change point tests, have their origin in quality control (see \citealp{page1954, Pag55}) and nowadays, these techniques are widely used in many fields of science such as economics (\citealp{Per06}), finance (\citealp{AndGhy09}), climatology (\citealp{Ree07}) and engineering (\citealp{Sto00}). The contributions of the present paper to this field of research are new statistical procedures for the detection of changes in the jump behaviour of an It\=o semimartingale.
In contrast to the existing works \cite{BueHofVetDet15} and \cite{HofVetDet17} this paper introduces methods of inference on the jump behaviour of the underlying process in general, while in the previously mentioned references the authors restrict the analysis to jumps which exceed a minimum size $\eps >0$.

Throughout this work we assume that we have high-frequency data $X_{i\Delta_n}$ $(i=0,1,\ldots,n)$ with $\Delta_n \to 0$, where the process $(X_t)_{t\in \R_+}$ is an It\=o semimartingale with the following decomposition
\begin{multline*} 
X_t = X_0 +  \int_0^t b_s \, ds + \int_0^t \sigma_s\, dW_s + \int_0^t \int_{\R} u 1_{\{|u| \leq 1\}} (\mu-\bar \mu)(ds,du) +  \int_0^t \int_{\R} u 1_{\{|u|>1\}} \mu(du,dz).
\end{multline*}
Here $W$ is a standard Brownian motion and $\mu$ is a Poisson random measure on $\R^+ \times \R$ with predictable compensator $\bar \mu$ satisfying $\bar \mu(ds,du) = ds \: \nu_s(du)$. Our approach is completely non-parametric, that is we only impose structural assumptions on the characteristic triplet $(b_s,\sigma_s,\nu_s)$ of $(X_t)_{t\in \R_+}$. The crucial quantity here is the transition kernel $\nu_s$ which controls the number and the size of the jumps around time $s \in \R_+$. Our aim is to test the null hypothesis 
\begin{align*}
{\bf H}_0: \nu_s(d\taili) = \nu_0(d\taili)
\end{align*}
against various alternatives involving the non-constancy of $\nu_s$. In particular, the detection of abrupt changes in a stochastic feature has been discussed extensively in the literature (see \citealp{auehor2013}  and \citealp{jandhyala:2013} for an overview in a time series context). The first part of this paper belongs to this area of research and introduces tests for $\textbf{H}_0$ versus alternatives of an abrupt change of the form
\begin{align*}
{\bf H}_1^{(ab)}: \nu_s^{(n)}(d\taili) = \ind_{\{s < \ip{n\gseqi_0} \Delta_n\}}\nu_1(d\taili) + \ind_{\{s \geq \ip{n\gseqi_0} \Delta_n\}}\nu_2(d\taili),
\end{align*}
for some unknown $\gseqi_0 \in (0,1)$ and two distinct L\'evy measures $\nu_1 \neq \nu_2$. Similar to the classical setup of detecting changes in the mean of a time series it is only possible to define the change point relative to the length of the data set which in our case is the time horizon $n\Delta_n$. However, for inference on the jump behaviour the time horizon has to tend to infinity ($n\Delta_n \to \infty$) since there are only finitely many jumps of a certain size on every compact interval. Furthermore, we also discuss how to estimate the unknown change point $\gseqi_0$, if the alternative ${\bf H}_1^{\scriptscriptstyle (ab)}$ is true.

A more difficult problem is the detection of gradual (smooth, continuous) changes in a stochastic feature. As a consequence, the setup in most papers on this topic is restricted to non-parametric  location or parametric models  with independently distributed  observations (see e.g. \citealp{bissel1984a}, \citealp{gan1991},  \citealp{siezha1994}, \citealp{huskova1999}, \citealp{husste2002}  and \citealp{Mallik2013}). Gradual changes in a time series context are for instance discussed in \cite{aueste2002} and \cite{VogDet15}. In the second part of this paper we contribute to this development by introducing new procedures for gradual changes in the kernel $\nu_s$, where we basically test ${\bf H}_0$ against the general alternative
\begin{align*}
{\bf H}_1^{(gra)}: \nu_s(d\taili) \text{ is not Lebesgue-almost everywhere constant in } s \in [0,n\Delta_n].
\end{align*}
Moreover, we introduce an estimator for the first time point where the jump behaviour deviates from the null hypothesis.

The remaining paper is organized as follows: In Section \ref{sec:Asss} we give the basic assumptions on the characteristics of the underlying process and the observation scheme. Section \ref{sec:Infabchdf} introduces test and estimation procedures for abrupt changes in the jump behaviour in general by using CUSUM processes. In Section \ref{sec:gradchadf} we discuss how to detect and estimate gradual changes in the entire jump behaviour. Section \ref{sec:fisaper} contains an extensive simulation study investigating the finite-sample performance of the new procedures. Finally, all proofs are relegated to Section \ref{sec:weConv} and the technical appendices \ref{ssecProTh61}  - \ref{appF}.

\section{The basic assumptions}
\label{sec:Asss}
\def\theequation{2.\arabic{equation}}
\setcounter{equation}{0}

In order to accommodate both abrupt and gradual changes in our approach we follow \cite{HofVetDet17} and assume that there is a driving law behind the evolution of the jump behaviour in time which is common for all $n\in\N$. That is we assume that at step $n\in \N$ we observe an It\=o semimartingale $X^{\scriptscriptstyle (n)}$ with characteristics $(b_s^{\scriptscriptstyle (n)},\sigma_s^{\scriptscriptstyle (n)},\nu_s^{\scriptscriptstyle(n)})$ at the equidistant time points $i\Delta_n$ with $i=0,1,\ldots,n$ which satisfies the following rescaling assumption
\begin{align}
\label{ComResAss}
\nu^{(n)}_{s}(dz) = g\Big(\frac{s}{n\Delta_n}, dz\Big)
\end{align}
for a transition kernel $g(y,d\taili)$ from $([0,1],\Bb([0,1]))$ into $(\R,\Bb)$, where here and below $\Bb(A)$ denotes the trace $\sigma$-algebra on $A\subset\R$ of the Borel $\sigma$-algebra $\Bb$ of $\R$. In order to detect changes in the jump behaviour of the underlying It\=o semimartingale in general, we have to draw inference on the kernel $g(y,B)$ for sets $B \in \Bb$ containing the origin. However, $g$ has locally the properties of a L\'evy measure. Thus, if we deviate from the (simple) case of finite activity jumps the total mass of $g$ on every neighbourhood of the origin is infinite and we cannot estimate $g(y,\cdot)$ on sets containing $0$ directly. We address this problem by weighting the kernel $g$ according to an auxiliary function, precisely for change point detection we consider
\begin{equation}
\label{NrhoDef}
N_{\rho}(g;\gseqi,\dfi) := \int_0^\gseqi \int_{-\infty}^\dfi \rho(\taili) g(y,d\taili) dy,
\end{equation}
for $(\gseqi,\dfi) \in \netir$, where $\rho$ is chosen appropriately such that the integral is always defined. Under weak conditions on $\rho$, this so-called L\'evy distribution function $N_\rho$ determines the entire kernel $g$ and therefore the evolution of the jump behaviour in time. The natural approach to draw inference on $N_\rho$ is the following sequential generalization of an estimator in \cite{Nic15}
\[
\tilde N_\rho^{(n)}(\gseqi,t) = \frac{1}{n \Delta_n} \sum \limits_{i = 1}^{\ip{n\gseqi}} \rho(\Deli X^{(n)}) \Indit(\Deli X^{(n)}),
\]
for $(\gseqi,\dfi) \in \netir$, where $\Deli X^{\scriptscriptstyle (n)} = X^{\scriptscriptstyle (n)}_{i\Delta_n} - X^{\scriptscriptstyle (n)}_{(i-1)\Delta_n}$. Using a spectral approach similar to \cite{NicRei12} these authors prove weak convergence of $\sqrt{n \Delta_n} \big(\tilde N_\rho^{\scriptscriptstyle (n)}(1,t) - N_\rho(g;1,t)\big)$ in $\ell^\infty(\R)$ to a tight Gaussian process, but only for L\'evy processes without a diffusion component, i.e. in particular for constant $g(y,\cdot) \equiv \nu(\cdot)$. The main difficulty in generalizing this result is the superposition of small jumps with the roughly fluctuating Brownian component of the process. We solve this problem by using a truncation approach which has originally been used by \cite{mancini} to cut off jumps in order to draw inference on integrated volatility. More precisely, we follow \cite{HofVet15} and identify jumps by inverting the truncation technique of \cite{mancini}, i.e. all test statistics and estimators investigated below are functionals of the sequential truncated empirical L\'evy distribution function
\begin{equation}
\label{NrhonDef}
N_{\rho}^{(n)}(\gseqi,\dfi) = \frac{1}{n \Delta_n} \sum \limits_{i=1}^{\ip{n\gseqi}} \rho(\Deli X^{(n)}) \Indit(\Deli X^{(n)}) \Truniv, \quad  (\gseqi,\dfi) \in \netir,
\end{equation}
for some suitable null sequence $v_n \to 0$. 

As a further improvement to previous studies we analyse the asymptotic behaviour of our tests under local alternatives. That is, in the rescaling assumption \eqref{ComResAss} we let $g = g^{\scriptscriptstyle (n)}$ depend on $n \in \N$, where there exist transition kernels $g_0,g_1,g_2$ satisfying some additional regularity assumptions such that for each $y \in [0,1]$
\begin{equation}
\label{gon012Ass2}
g^{(n)}(y,d\taili) = g_0(y,d\taili) + \frac 1{\sqrt{n\Delta_n}} g_1(y,d\taili) + \Rc_n(y,d\taili)
\end{equation}
and for each $y \in [0,1]$, $B\in\Bb$ and $n\in\N$ the remainder kernel $\Rc_n$ satisfies
\[
\Rc_n(y,B) \leq a_n g_2(y,B)
\]
for a sequence $a_n = o((n\Delta_n)^{-1/2})$ of non-negative real numbers. For constant $g_0(y,\cdot) \equiv \nu_0(\cdot)$ assumption \eqref{gon012Ass2} is exactly the local alternative where the jump behaviour converges to the null hypothesis $g_0(y,\cdot) \equiv \nu_0(\cdot)$ from the direction defined by $g_1$ at rate $(n\Delta_n)^{-1/2}$. In this sense, Theorem \ref{ConvThm}, in which we prove weak convergence of the stochastic process
\[
G_{\rho}^{(n)}(\gseqi,\dfi) = \sqrt{n \Delta_n}\big(N_{\rho}^{(n)}(\gseqi,\dfi) - N_{\rho}(g^{(n)};\gseqi,\dfi)\big), \quad (\gseqi,\dfi) \in \netir
\]
to a tight Gaussian process in $\linner$, is a generalization of the results in \cite{HofVet15} to sequential processes for time dependent variable jump behaviour as in \eqref{gon012Ass2}. 

Critical values for the test procedures introduced below and the optimal choice of a regularization parameter of the new estimator for gradual change points are obtained by a multiplier bootstrap approach. Precisely, Theorem \ref{CondConvThm}, in which we prove conditional weak convergence in a suitable sense  of the bootstrapped version 
\begin{align*}
\hat G_\rho^{(n)}(\gseqi,\dfi) = \frac 1{\sqrt{n\Delta_n}} \sum\limits_{i=1}^{\ip{n\gseqi}} \xi_i \rho\big(\Deli X^{(n)}\big) \ind_{(-\infty,\dfi]}\big(\Deli X^{(n)}\big) \ind_{\{|\Deli X^{(n)} | > v_n\}}, \quad (\gseqi, \dfi) \in \netir
\end{align*}
of $G_\rho^{\scriptscriptstyle (n)}$ to a Gaussian process, where $(\xi_i)_{i\in\N}$ is a sequence of i.i.d.\ multipliers with mean $0$ and variance $1$, complements the paper \cite{HofVet15}.

For the rescaling assumptions \eqref{ComResAss} and \eqref{gon012Ass2} we consider transition kernels $g_i(y,d\taili)$ of the set $\Gc(\beta,p)$ depending on parameters $\beta \in (0,2), p>0$. In order to define this set we denote by $\lambda$ the one-dimensional Lebesgue measure defined on the Lebesgue $\sigma$-algebra $\Lc_1$ of $\R$ and we denote by $\lambda_1$ the restriction of $\lambda$ to the trace $\sigma$-algebra $[0,1] \cap \Lc_1$.

\begin{definition}
	\label{rhoandgass2}
	For $\beta \in (0,2)$ and $p > 0$ the set $\Gc(\beta,p)$ consists of all transition kernels $g(y,d\taili)$ from $([0,1],\Bb([0,1]))$ into $(\R,\Bb)$, such that for each $y \in [0,1]$ the measure $g(y,d\taili)$ has a Lebesgue density $h_y(\taili)$ and there exist $\eta,M >0$ as well as a Lebesgue null set $L \in [0,1] \cap \Lc_1$ such that the following items are satisfied:
	\begin{enumerate}[(1)]
		\item $h_y(\taili) \leq K|\taili|^{-(1+ \beta)}$ holds for all $\taili \in (-\eta,\eta)$, $y \in [0,1] \setminus L$ and for some $K >0$. \label{BGn0Ass}
		\item For $n \in \N$ let $C_n \defeq \lbrace \taili \in \R \mid \frac{1}{n} \leq |\taili| \leq n \rbrace$. Then for each $n \in \N$ there exists a $K_n >0$ with $h_y(\taili) \leq K_n$ for each $\taili \in C_n$ and all $y \in [0,1] \setminus L$.
		\label{DiCondmiddle}
		\item $h_y(\taili) \leq K |\taili|^{-(2p\vee 2) - \epsilon}$ whenever $|\taili| \geq M$ and $y \in [0,1] \setminus L$, for some $K > 0$ and some $\epsilon >0$.
		\label{DiCondinfty}
	\end{enumerate}
\end{definition}
The items above basically say that the densities $h_y$ are bounded by a continuous L\'evy density of a L\'evy measure which behaves near zero like the one of a $\beta$-stable process, whereas this density has to decay sufficiently fast at infinity. Such conditions are well-known in the literature and often used in similar works on high-frequency statistics; see e.g.\ \cite{AitJac09b} or \cite{AitJac10}. From Assumption \ref{Cond1} and Proposition \ref{easier} in Section \ref{sec:weConv} it can be seen that it is even possible to work with a wider class of transition kernels $g(y,d\taili)$ which does not require Lebesgue densities. Nevertheless, we stick to the set $\Gc(\beta,p)$ defined above which is much simpler to interpret. The following example shows that alternatives of abrupt changes in the jump behaviour can be described by transition kernels in the set $\Gc(\beta,p)$.

\begin{example} ({\it abrupt changes})
	\label{Ex:Sitabcha}
	 In Section \ref{sec:Infabchdf} we introduce statistical procedures for inference of abrupt changes in the jump behaviour. In this case the kernel $g_0$ is typically of the form as discussed below. For $\beta \in (0,2)$ and $p > 0$ let $\Mc(\beta,p)$ be the set of all L\'evy measures $\nu$ such that the constant transition kernel $g(y,d\taili) = \nu(d\taili)$ belongs to $\Gc(\beta,p)$.
	
	Let $\gseqi_0 \in (0,1]$ and let $\nu_1, \nu_2 \in \Mc(\beta,p)$ be two L\'evy measures. Then the transition kernel $g_0$ given by
	\begin{align}
	\label{gabDef}
	g_0(y,d\taili) = \begin{cases}
	\nu_1(d \taili), \quad &\text{ for } y \in [0,\gseqi_0] \\
	\nu_2(d \taili), \quad &\text{ for } y \in (\gseqi_0,1].
	\end{cases}
	\end{align}
	is an element of $\Gc(\beta,p)$. In the context of change-point tests $\gseqi_0 = 1$ corresponds to the null hypothesis of no change in the jump behaviour, whereas \eqref{gabDef} describes an abrupt change for $\gseqi_0 \in(0,1)$ and $\nu_1\neq\nu_2$.
\end{example}

The variance gamma process is a common model for the log stock price in finance (see for instance \cite{Mad98}). Moreover, the L\'evy measure of a variance gamma process has the form $\nu(d\taili) = (a_1 \taili^{-1} e^{-b_1\taili} - a_2 \taili^{-1} e^{-b_2\taili})~d\taili$ for $a_1,a_2,b_1,b_2 >0$. Thus, the transition kernel $g_0(y,d\taili)$ belongs to $\Gc(\beta,p)$ for all $\beta\in(0,2)$ and $p>0$, if similar as in \eqref{gabDef} $g_0$ is piecewise constant in $y\in[0,1]$ and on the domains of constancy it is equal to the L\'evy measure of a variance gamma process.

For the asymptotic statements in this paper  we require the  following assumptions.
Our results are also  correct under   less restrictive but more technical conditions. For the sake of a transparent presentation 
these are not presented  here but deferred to Section \ref{altass}.

\begin{assumption}
	\label{EasierCond}
	Let $0< \beta < 2$ and $0<  \tau < (1/5 \wedge \frac{2-\beta}{2+5\beta})$. Furthermore, let $p > \beta+((\frac 12 + \frac 32 \beta) \vee  \frac{2}{1+5\tau})$. At step $n\in\N$ we observe an It\=o semimartingale $X^{\scriptscriptstyle (n)}$ adapted to the filtration of some filtered probability space $(\Omega,\Fc,(\Fc_t)_{t\in\R_+},\Prob)$ with characteristics  $(b_s^{\scriptscriptstyle (n)},\sigma_s^{\scriptscriptstyle (n)},\nu_s^{\scriptscriptstyle(n)})$ at the equidistant time points $\{i\Delta_n \mid i=0,1,\ldots,n\}$ such that the following items are satisfied:
	\begin{compactenum}[(a)]
		\item \textit{Assumptions on the jump characteristic and the function $\rho$:} 
		\label{rhoandgass}
		\begin{enumerate}[(1)]
			\item For each $n\in\N$ and $s \in [0,n\Delta_n]$ we have
			\begin{align}
			\label{RescAss2}
			\nu^{(n)}_{s}(dz) = g^{(n)}\Big(\frac{s}{n\Delta_n}, dz\Big),
			\end{align}
			where there exist transition kernels $g_0,g_1,g_2 \in \Gc(\beta,p)$ such that for each $y \in [0,1]$
			\begin{equation}
			\label{gon012Ass}
			g^{(n)}(y,d\taili) = g_0(y,d\taili) + \frac 1{\sqrt{n\Delta_n}} g_1(y,d\taili) + \Rc_n(y,d\taili)
			\end{equation}
			and for each $y \in [0,1]$, $B\in\Bb$ and $n\in\N$ the kernel $\Rc_n$ satisfies
			$
			\Rc_n(y,B) \leq a_n g_2(y,B)
			$
			for a sequence $a_n = o((n\Delta_n)^{-1/2})$ of non-negative real numbers.
			\item \label{EasrhoCond} $\rho \colon \R \rightarrow \R$ is a bounded $\mathcal C^1$-function with $\rho(0)=0$ and its derivative satisfies $|\rho^{\prime}(\taili)| \leq K |\taili|^{p-1}$ for all $\taili \in \R$ and some constant $K>0$.
			\item \label{rhoneq0} $\rho(\taili) \neq 0$ for each $\taili \neq 0$.
			\item \label{jbidenass} For every $\dfi \in \R$ there exists a finite set $M_{(\dfi)} \subset [0,1]$, such that the function 
			\[ y \mapsto \int_{-\infty}^\dfi \rho(\taili) g_0(y,d\taili) \]  is continuous on $[0,1] \setminus M_{(\dfi)}$. 
		\end{enumerate}
		\item \label{speed}\textit{Assumptions on the truncation sequence $v_n$ and the observation scheme:} \\
		The truncation sequence $v_n$ satisfies
		$
		v_n \defeq \gamma \Delta_n^{\ovw},
		$
		with $\ovw = (1+5\tau)/4$ and some $\gamma >0$. Define further $ t_1 \defeq  ( 1+ \tau )^{-1} $   and  $ t_2 \defeq  ( (7\tau +1)/2 )^{-1} \wedge 1$
		(note that $0 < t_1 < t_2 \leq 1$) and we suppose that the observation scheme satisfies for some $\delta >0$
		\[
		\Delta_n = o(n^{-t_1}) \quad \text{ and } \quad n^{-t_2+\delta} = o(\Delta_n).
		\]
		\item \label{DriDiffMomCond} \textit{Assumptions on the drift and the diffusion coefficient:} \\
		For
		$ m_b = \frac{6+10\tau}{3-5\tau} \leq 4  $ and  $ m_\sigma = \frac{6+10 \tau}{1-5\tau}
		$
		we have
		\[
		\sup\limits_{n\in\N}\sup \limits_{s \in \R_+} \Big\{ \Eb \big|b^{(n)}_s\big|^{m_b} \vee \Eb \big|\sigma^{(n)}_s\big|^{m_\sigma} \Big\} < \infty.
		\]
	\end{compactenum}
\end{assumption}

\begin{remark} Suppose we have complete knowledge of the distribution function $N_{\rho}(g_0;\gseqi,\dfi)$. Obviously, the measure with density $M(dy,d\taili) \defeq \rho(\taili) g_0(y,d\taili)dy$ is completely determined from knowledge of the entire function $N_{\rho}(g_0;\cdot,\cdot)$ and does not charge $[0,1]\times \{0\}$. Therefore, due to Assumption \ref{EasierCond}\eqref{rhoneq0} $1/\rho(\taili) M(dy,d\taili) = g_0(y,d\taili)dy$ and consequently the  jump behaviour corresponding to $g_0$ is known as well. Furthermore, Assumption \ref{EasierCond}\eqref{jbidenass} ensures that a characteristic quantity for a gradual change, which we introduce in Section \ref{sec:gradchadf} is zero if and only if the jump behaviour corresponding to $g_0$ is constant in time. All convergence results in this paper also hold without Assumption \ref{EasierCond}\eqref{rhoneq0} and \eqref{jbidenass}. 
	Moreover, the function 
	\[
	\tilde \rho(x) = \begin{cases}
	0, \quad &\text{if } x=0, \\
	e^{-1/|x|}, \quad &\text{if } |x| >0, \\
	\end{cases}
	\]
	is suitable for any choice of the constants $\beta$ and $\tau$. In practice, however, one would like to work with a polynomial decay at zero, in which case the condition on $p$ comes into play. Here, the smaller the parameter $\beta$, the smaller $p$ can be chosen. For example, for $\beta < 3/5$ and $\tau > 3/35$ even a choice $p<2$ is possible.
	
	Furthermore, it is also important to choose the observation scheme suitably. Obviously, we have  $\Delta_n \rightarrow 0$ and $n \Delta_n \rightarrow \infty$ because of $0< t_1 < t_2 \leq 1$, and a typical choice is $\Delta_n = O(n^{-y})$ and $n^{-y} = O(\Delta_n)$ for some $0< t_1 < y < t_2 \leq 1$.
	Finally, Assumption \ref{EasierCond}\eqref{DriDiffMomCond} requires only a bound on the moments of the remaining characteristics and is therefore extremely mild.
\end{remark}

In the remaining part of this section we illustrate an example of a kernel $g_0 \in \Gc(\beta,p)$ for some suitable $\beta,p$ and a function $\rho$ satisfying Assumption \ref{EasierCond}\eqref{EasrhoCond} and \eqref{rhoneq0}.

\begin{example}({\it gradual changes}) \label{Ex:SitgraCh} In Section \ref{sec:gradchadf}, which is dedicated to inference of gradual changes, we basically test against the general alternative that the jump behaviour is non-constant. In the following we introduce an example of a kernel $g_0$ which can be used to describe a gradual change in the jump behaviour and a corresponding function $\rho$ satisfying Assumption \ref{EasierCond}\eqref{EasrhoCond} and \eqref{rhoneq0}.
	To this end, for $L >0$, $p >1$ let 
	\begin{equation}
	\label{Eq:rhoLpdef}
	\rho_{L,p}(\taili) := L \times \begin{cases}
	2|\taili|^p, \quad &\text{ for } |\taili| \leq 1 \\
	4p|\taili| - p \taili^2 + 2 - 3p, \quad &\text{ for } 1 \leq |\taili| \leq 2 \\
	2+p, \quad &\text{ for } |\taili| \geq 2
	\end{cases}
	\end{equation}
	and for $0 < \beta < 2$, $p> 1$ consider the L\'evy density 
	\[
	h_{\beta,p}(\taili) := |\taili|^{-(1+\beta)} \ind_{\{ 0 < |\taili| <1 \}} + \ind_{\{1 \leq |\taili| \leq 2\}} + |\taili|^{-p} \ind_{\{|\taili| > 2\}}.
	\]
	Furthermore, for $0< \hat \beta < 2$ and $\hat p > 1 \vee \hat \beta$ let $A: [0,1] \to (0, \infty)$, $\beta: [0,1] \to (0,\hat \beta]$ and $p: [0,1] \to [ 2 \hat p  + \eps,\infty)$ for some $\eps > 0$ be Borel measurable functions such that $A$ is bounded. Then, the kernel
	\begin{equation}
	\label{gformgrach}
	g_0(y,d\taili) = A(y) h_{\beta(y),p(y)}(\taili) d\taili, \quad y \in [0,1]
	\end{equation}
	belongs to $\Gc(\hat \beta,\hat p)$ and for arbitrary $L>0$ the function $\rho_{L,\hat p}$ satisfies Assumption \ref{EasierCond}\eqref{EasrhoCond} and \eqref{rhoneq0}.
\end{example}

\section{Statistical inference for abrupt changes}
\label{sec:Infabchdf}
\def\theequation{3.\arabic{equation}}
\setcounter{equation}{0}

In this section we deduce test and estimation procedures for abrupt changes in the jump behaviour of the underlying process, that is we investigate the situation of Example \ref{Ex:Sitabcha}. To this end, we  test the null hypothesis of no change in the jump behaviour

\begin{enumerate}
	\item[${\bf H}_0$:] Assumption~\ref{EasierCond} is satisfied for $g_1=g_2=0$ and there exists a L\'evy measure $\nu_0$ such that $g_0(y,d\taili) = \nu_0(d\taili)$ for Lebesgue almost every $y \in [0,1]$.
\end{enumerate}

against the alternative that the jump behaviour is constant on two intervals

\begin{enumerate}
	\item[${\bf H}_1$:] Assumption~\ref{EasierCond} is satisfied for $g_1=g_2=0$ and there exists some $\theta_0 \in (0,1)$ and two L\'evy measures $\nu_1 \neq \nu_2$ such that $g_0$ has the form \eqref{gabDef}.
\end{enumerate}

The corresponding alternative for fixed $\dfi_0 \in \R$ is given by:
\begin{enumerate}
	\item[${\bf H}_1^{(\rho,\dfi_0)}$:] We have the situation from ${\bf H}_1$, but with $N_\rho(\nu_1;\dfi_0) \neq N_\rho(\nu_2;\dfi_0)$, where 
	\begin{align}
	\label{NrhonuDef}
	N_\rho(\nu;\dfi) = \int_{-\infty}^{\dfi} \rho(\taili) \nu(d\taili)
	\end{align}
	for a L\'evy measure $\nu$. 
\end{enumerate}

Moreover, we investigate the behaviour of the tests introduced in this section under local alternatives which tend to the null hypothesis as $n\to\infty$:

\begin{enumerate}
	\item[${\bf H}^{(loc)}_1$:] Assumption~\ref{EasierCond} is satisfied with $g_0(y,d\taili) = \nu_0(d\taili)$ for Lebesgue-a.e. $y \in [0,1]$ for some L\'evy measure $\nu_0$ and with some transition kernels $g_1,g_2 \in \Gc(\beta,p)$.
\end{enumerate}

\subsection{Weak convergence of test statistics}

Following \cite{Ino01} a suitable approach to introduce tests for the hypotheses above is to investigate the convergence behaviour of the CUSUM process

\begin{align}
\label{Tbrhondef}
\Tb_\rho^{(n)} (\gseqi,\dfi) = \sqrt{n \Delta_n} \Big( N_\rho^{(n)}(\gseqi,\dfi) - \frac{\ip{n\gseqi}}n N_\rho^{(n)}(1,\dfi) \Big),
\end{align}
with $N_\rho^{(n)}(\gseqi,\dfi)$ defined in \eqref{NrhonDef}. The corresponding test rejects the null hypothesis ${\bf H}_0$ for large values of the Kolmogorov-Smirnov-type statistic
\begin{align*}
T_\rho^{(n)} = \sup\limits_{(\gseqi,\dfi) \in \netir} \big|\Tb_\rho^{(n)} (\gseqi,\dfi) \big|.
\end{align*}
The theorem below establishes functional weak convergence of $\Tb_\rho^{\scriptscriptstyle (n)}$ in the general case of local alternatives.

\begin{theorem}
	\label{CUSUMTrhowc}
	Under ${\bf H}_1^{(loc)}$ the process $\Tb_\rho^{(n)}$ converges weakly in $\linner$ to the process $\Tb_\rho + \Tb_{\rho,g_1}$, where the tight mean zero Gaussian process $\Tb_\rho$ has the covariance structure
	\begin{align}
	\label{TrhoCovFkt}
	\Eb\{\Tb_\rho(\gseqi_1,\dfi_1)\Tb_\rho(\gseqi_2,\dfi_2)\} = \{ (\gseqi_1 \wedge \gseqi_2) - \gseqi_1 \gseqi_2 \}  \int_{-\infty}^{\dfi_1 \wedge \dfi_2} \rho^2(\taili) \nu_0(d\taili)
	\end{align}
	and the deterministic function $\Tb_{\rho,g_1} \in \linner$ is given by
	\begin{equation}
	\label{shiftdef}
	\Tb_{\rho,g_1} (\gseqi,\dfi) = N_\rho(g_1;\gseqi,\dfi) - \gseqi N_\rho(g_1;1,\dfi),
	\end{equation}
	where $N_\rho(g_1;\cdot,\cdot)$ is defined in \eqref{NrhoDef}.
\end{theorem}

As an immediate consequence of the previous result and the continuous mapping theorem we obtain weak convergence of the statistic $T_\rho^{\scriptscriptstyle (n)}$.

\begin{corollary}
	\label{KolSmiConvRes}
	Suppose ${\bf H}_1^{(loc)}$ is true, then we have
	\begin{align}
	\label{Trhsugpdef}
	T_\rho^{(n)} \weak T_{\rho,g_1} := \sup \limits_{(\gseqi,\dfi) \in \netir} \big|\Tb_\rho(\gseqi,\dfi) +  \Tb_{\rho,g_1} (\gseqi,\dfi) \big|,
	\end{align}
	in $(\R,\Bb)$ with $\Tb_\rho + \Tb_{\rho,g_1}$ the limit process in Theorem \ref{CUSUMTrhowc}.
\end{corollary}

In applications the L\'evy measure $\nu_0$ which describes the limiting jump behaviour of the underlying process is usually unknown. If one is only interested in the detection of changes in the distribution function $N_\rho(\nu_0;\dfi_0)$ for a fixed $\dfi_0 \in \R$, the processes
\begin{align*}
\Vb_{\rho,\dfi_0}^{(n)}(\gseqi) := \frac{\Tb_\rho^{(n)}(\gseqi,\dfi_0)}{\sqrt{N^{\scriptscriptstyle (n)}_{\rho^2}(1,\dfi_0)}} \ind_{\{N^{\scriptscriptstyle (n)}_{\rho^2}(1,\dfi_0) >0\}}, \quad \gseqi \in [0,1]
\end{align*}
converge weakly to a shifted version of a pivotal limit process.

\begin{proposition}
	\label{poiwconKSdi}
	Under ${\bf H}_1^{(loc)}$ for each fixed $\dfi_0 \in \R$ with $N_{\rho^2}(\nu_0;\dfi_0) >0$ we have $\Vb_{\rho,\dfi_0}^{\scriptscriptstyle (n)} \weak \Kb + \bar \Vb_{\rho,\dfi_0}^{\scriptscriptstyle (g_1)}$ in $\linne$, where $\Kb$ denotes a standard Brownian bridge and with the deterministic function
	\[
	\bar \Vb_{\rho,\dfi_0}^{(g_1)}(\gseqi) := \frac{\Tb_{\rho,g_1}(\gseqi,\dfi_0)}{\sqrt{N_{\rho^2}(\nu_0;\dfi_0)}} \in \linne,
	\] 
	where $N_{\rho^2}(\nu_0;\cdot)$ is defined in \eqref{NrhonuDef}. In particular, 
	\begin{align}
	\label{Vrhosupc}
	V_{\rho,\dfi_0}^{(n)} := \sup\limits_{\gseqi \in [0,1]} \big|\Vb_{\rho,\dfi_0}^{(n)}(\gseqi) \big| \weak \bar V_{\rho,\dfi_0}^{(g_1)} := \sup\limits_{\gseqi \in [0,1]} \big| \Kb(\gseqi) + \bar \Vb_{\rho,\dfi_0}^{(g_1)}(\gseqi) \big|.
	\end{align}
\end{proposition}

Quantiles of functionals of the limit process $\Tb_\rho + \Tb_{\rho,g_1}$ in Theorem \ref{CUSUMTrhowc} are not easily accessible since the distribution of such functionals usually depends in a complicated way on the unknown quantities $\nu_0$ and $g_1$ in the jump characteristic of the underlying process. In order to obtain reasonable approximations for these quantiles we use a multiplier bootstrap approach. That is, in the following we consider bootstrapped processes, $\hat Y_n = \hat Y_n(X_1, \ldots, X_n, \xi_1, \ldots, \xi_n)$, which depend on random variables $X_1, \ldots, X_n$ defined on a probability space $(\Omega_X, \mathcal F_X, \mathbb P_X)$ and on random weights $\xi_1, \ldots, \xi_n$ which are defined on a distinct probability space $(\Omega_{\xi}, \mathcal F_{\xi}, \mathbb P_{\xi})$. Thus, the processes $\hat Y_n$ live on the product space
$(\Omega, \mathcal A,\mathbb P) \defeq (\Omega_X, \mathcal A_X, \mathbb P_X) \otimes (\Omega_{\xi}, \mathcal A_{\xi}, \mathbb P_{\xi})$.
Below we use the notion of weak convergence conditional on the sequence $(X_i)_{i\in\N}$ in probability. It can be found in \cite{Kos08} on pp.\ 19--20.

\begin{definition}
	\label{ConvcondDataDef}
	Let $\hat Y_n = \hat Y_n (X_1, \ldots ,X_n; \xi_1, \ldots , \xi_n) \colon (\Omega, \mathcal A,\mathbb P) \rightarrow \mathbb D$ be a random element taking values in some metric space $\mathbb D$ depending on some random variables $X_1, \ldots, X_n$ and some random weights $\xi_1, \ldots, \xi_n$. Moreover, let $Y$ be a tight, Borel measurable random variable into $\mathbb D$. Then $\hat Y_n$ converges weakly to $Y$ conditional on the data $X_1, X_2, \ldots$ in probability, if and only if
	\begin{enumerate}[(a)]
		\item \label{ConConvDia} $\sup \limits_{f \in \text{BL}_1(\mathbb D)} |\mathbb E_{\xi} f(\hat Y_n) - \mathbb E f(Y) | \pto 0,$
		\item \label{ConConvDib} $\mathbb E_{\xi} f( \hat Y_n)^{\ast} - \mathbb E_{\xi} f( \hat Y_n)_{\ast} \probto 0$ for all
		$f \in \text{BL}_1(\mathbb D).$
	\end{enumerate}
	Here, $\mathbb E_{\xi}$ denotes the conditional expectation over the weights $\xi$ given the data $X_1, \ldots, X_n$, whereas $\text{BL}_1(\mathbb D)$ is the space of all real-valued Lipschitz continuous functions $f$ on $\mathbb D$ with sup-norm $\| f \|_{\Db} \leq 1$ and Lipschitz constant $1$. Here and below we denote the sup-norm of a real valued function $f$ on a set $M$ by $\|f\|_M$. Furthermore, in item \eqref{ConConvDib} $f( \hat Y_n)^{\ast}$ and $f( \hat Y_n)_{\ast}$ denote a minimal measurable majorant and a maximal measurable minorant with respect to the joint probability space $(\Omega, \mathcal A,\mathbb P)$. The type of convergence defined above is denoted by $\hat Y_n \weakP Y$.
\end{definition}

\begin{remark}
	\label{rem:condweak}
	$~~$
	\begin{compactenum}[(i)]
		\item Throughout this work  all expressions $f(\hat Y_n)$, with a bootstrapped statistic $\hat Y_n$ and a Lip\-schitz continuous function $f$, are measurable functions of the random weights. To this end we do not use a measurable majorant or minorant in item \eqref{ConConvDia} in the definition above.
		\item The implication ``(ii) $\Rightarrow$ (i)'' in the proof of Theorem~2.9.6 in \cite{VanWel96} shows that conditional weak convergence $\weakP$ implies unconditional weak convergence $\weak$ with respect to the product measure $\mathbb P$.
	\end{compactenum}
\end{remark}

For the results on conditional weak convergence of the bootstrapped processes below we require a rather mild additional assumption on the sequence of multipliers, which is satisfied for many common distributions such as for instance the Gaussian, the Poisson or the Binomial distribution.

\begin{assumption}
	\label{MultiplAss}
	The sequence $(\xi_i)_{i \in \N}$ is defined on a distinct probability space than the one generating the data $\{X_{i\Delta_n}^{\scriptscriptstyle (n)} \mid i=0,1,\ldots,n\}$ as described above, is i.i.d.\ with mean zero, variance one and there exists an $M >0$ such that for each integer $m \geq 2$ we have
	\begin{align*}
	\Eb |\xi_1|^m \leq m! M^m.
	\end{align*}
\end{assumption}

Reasonable bootstrap counterparts $\hat \Tb_\rho^{\scriptscriptstyle (n)}$ of the processes $\Tb_\rho^{\scriptscriptstyle (n)}$ are given by 

\begin{align}
\label{TbrhohaDef}
&\hat \Tb_\rho^{(n)} (\gseqi,\dfi) := \hat \Tb_\rho^{(n)} (X^{(n)}_{\Delta_n}, \ldots, X^{(n)}_{n\Delta_n}; \xi_1,\ldots,\xi_n;\gseqi,\dfi) :=  \nonumber \\
&\hspace{6mm}= \sqrt{n\Delta_n} \frac{\ip{n\gseqi}}n \frac{n-\ip{n\gseqi}}n \Big[ \frac 1{\ip{n\gseqi} \Delta_n} \sum\limits_{j=1}^{\ip{n\gseqi}} \xi_j \rho(\Delj X^{(n)}) \Indit(\Delj X^{(n)}) \ind_{\{|\Delj X^{(n)}| > v_n\}}  \nonumber \\
&\hspace{31mm}- \frac 1{(n-\ip{n\gseqi})\Delta_n} \sum\limits_{j=\ip{n \gseqi}+1}^n \xi_j \rho(\Delj X^{(n)}) \Indit(\Delj X^{(n)}) \ind_{\{|\Delj X^{(n)}| > v_n\}}  \Big].
\end{align}

In the following theorem we establish conditional weak convergence of $\hat \Tb_\rho^{\scriptscriptstyle (n)}$ under the general assumptions of Section \ref{sec:Asss}.

\begin{theorem}
	\label{BootTrhoThm}
	Let Assumption \ref{EasierCond} be valid and let the multipliers $(\xi_j)_{j\in\N}$ satisfy Assumption \ref{MultiplAss}. Then we have
	\begin{align*}
	\hat \Tb_\rho^{(n)} \weakP \Tb_\rho
	\end{align*}
	in $\linner$, where $\Tb_\rho$ is a tight mean zero Gaussian process in $\linner$ with covariance function 
	\begin{align}
	\label{Tbrbcov}
	\Eb\{\Tb_\rho(\gseqi_1,\dfi_1) \Tb_\rho(\gseqi_2,\dfi_2)\} &= \int_0^{\gseqi_1\wedge \gseqi_2} \int_{-\infty}^{\dfi_1 \wedge \dfi_2} \rho^2(\taili) g_0(y,d\taili) dy - \gseqi_1 \int_0^{\gseqi_2} \int_{-\infty}^{\dfi_1 \wedge \dfi_2} \rho^2(\taili) g_0(y,d\taili) dy  \nonumber \\ &\hspace{7mm}- \gseqi_2 \int_0^{\gseqi_1} \int_{-\infty}^{\dfi_1 \wedge \dfi_2} \rho^2(\taili) g_0(y,d\taili) dy  + \gseqi_1 \gseqi_2 \int_0^1 \int_{-\infty}^{\dfi_1 \wedge \dfi_2} \rho^2(\taili) g_0(y,d\taili) dy.
	\end{align}
\end{theorem}

\begin{remark} The aim of our bootstrap procedure is to mimic the convergence behaviour of $\Tb_\rho^{\scriptscriptstyle (n)}$. The covariance function of the limiting process in Theorem \ref{BootTrhoThm} differs from \eqref{TrhoCovFkt}, because Theorem \ref{BootTrhoThm} holds under the general conditions introduced in Assumption \ref{EasierCond}, i.e.\ for an arbitrary kernel $g_0 \in \Gc(\beta,p)$. Under the null hypothesis ${\bf H}_0$, where we have $g_0(\cdot,dz) = \nu_0(dz)$, the covariance function \eqref{Tbrbcov} coincides with \eqref{TrhoCovFkt}.
\end{remark}

The limit distribution of the Kolmogorov-Smirnov-type test statistic $T_\rho^{\scriptscriptstyle (n)}$ in Corollary \ref{KolSmiConvRes} can be approximated under $\textbf{H}_0$ by the bootstrap statistics in the following corollary, which is an immediate consequence of Proposition 10.7 in \cite{Kos08}.

\begin{corollary}
	\label{haTroCons}
	If Assumption \ref{EasierCond} and Assumption \ref{MultiplAss} are satisfied, we have
	\begin{align*}
	\hat T_\rho^{(n)} := \sup \limits_{(\gseqi,\dfi) \in \netir} | \hat \Tb_\rho^{(n)}(\gseqi,\dfi)| \weakP T_\rho := \sup \limits_{(\gseqi,\dfi) \in \netir} \big|\Tb_\rho(\gseqi,\dfi)  \big|,
	\end{align*}
	with $\Tb_\rho$ the limit process in Theorem \ref{BootTrhoThm}.
\end{corollary}

\subsection{Test procedures for abrupt changes}
\label{sec:tfacha}

The weak convergence results of the previous section make it possible to define test procedures for abrupt changes in the jump behaviour of the underlying process based on L\'evy distribution functions of type \eqref{NrhoDef}. In the following let $B \in \N$ be some large number and let $(\xi^{\scriptscriptstyle (b)})_{b = 1, \ldots,B}$ be independent vectors of i.i.d.\ random variables $\xi^{\scriptscriptstyle (b)} = (\xi_j^{\scriptscriptstyle (b)})_{j=1,\ldots,n}$ with mean zero and variance one, which satisfy Assumption \ref{MultiplAss}. With $\hat \Tb_{\scriptscriptstyle \rho,\xi^{\scriptscriptstyle (b)}}^{\scriptscriptstyle (n)}$ and $\hat T_{\scriptscriptstyle \rho,\xi^{\scriptscriptstyle (b)}}^{\scriptscriptstyle (n)}$ we denote the corresponding bootstrapped quantity calculated with respect to the data and the $b$-th multiplier sequence $\xi^{\scriptscriptstyle (b)}$. For a given level $\alpha \in (0,1)$, we propose to reject ${\bf H}_0$ in favor of ${\bf H}_1$, if
\begin{align}
\label{testvfkglob}
T_\rho^{(n)} \geq \hat q_{1-\alpha}^{(B)} \Big( T_\rho^{(n)} \Big),
\end{align}
where $\hat q_{1-\alpha}^{\scriptscriptstyle (B)} ( T_\rho^{\scriptscriptstyle (n)} )$ denotes the $(1-\alpha)$-sample quantile of $\hat T_{\scriptscriptstyle \rho,\xi^{\scriptscriptstyle (1)}}^{\scriptscriptstyle (n)}, \ldots , \hat T_{\scriptscriptstyle \rho,\xi^{\scriptscriptstyle (B)}}^{\scriptscriptstyle (n)}$. Similarly, for $\dfi_0 \in \R$, ${\bf H}_0$ is rejected in favor of ${\bf H}_1^{\scriptscriptstyle (\rho,\dfi_0)}$, if 
\begin{align}
\label{testvfklok}
W_\rho^{(n,\dfi_0)} := \sup \limits_{\gseqi \in [0,1]} |\Tb_\rho^{(n)}(\gseqi,\dfi_0) | \geq \hat q_{1 - \alpha}^{(B)}\Big( W_\rho^{(n,\dfi_0)} \Big),
\end{align}
where $\hat q_{1 - \alpha}^{\scriptscriptstyle (B)}( W_\rho^{\scriptscriptstyle (n,\dfi_0)} )$ denotes the $(1-\alpha)$-sample quantile of $\hat W_{\scriptscriptstyle \rho, \xi^{\scriptscriptstyle (1)}}^{\scriptscriptstyle (n,\dfi_0)}, \ldots, \hat W_{\scriptscriptstyle \rho, \xi^{\scriptscriptstyle (B)}}^{\scriptscriptstyle (n,\dfi_0)}$, and where $\hat W_{\scriptscriptstyle \rho, \xi^{\scriptscriptstyle (b)}}^{\scriptscriptstyle (n,\dfi_0)} $ $:= \sup_{\gseqi \in [0,1]} |\hat \Tb_{\scriptscriptstyle \rho,\xi^{\scriptscriptstyle (b)}}^{\scriptscriptstyle (n)}(\gseqi, \dfi_0)|$ for $b= 1, \ldots, B$. Furthermore, according to Proposition \ref{poiwconKSdi} we define an exact test procedure, that is ${\bf H}_0$ is rejected in favor of the point-wise alternative ${\bf H}_1^{\scriptscriptstyle (\rho,\dfi_0)}$, if
\begin{align}
\label{testvfklokex}
V_{\rho,\dfi_0}^{(n)} \geq q_{1-\alpha}^K,
\end{align}
where $q_{1-\alpha}^{\scriptscriptstyle K}$ is the $(1-\alpha)$-quantile of the Kolmogorov-Smirnov-distribution, that is the distribution of the supremum of a standard Brownian bridge $K= \sup_{\gseqi \in [0,1]} |\Kb(\gseqi)|$. \\
The following results show the behaviour of the previously introduced tests under the null hypothesis, local alternatives and the alternatives of an abrupt change. In particular, these tests are consistent asymptotic level $\alpha$ tests. First, recall the tight centered Gaussian process $\Tb_\rho$ in $\linner$  with covariance function \eqref{TrhoCovFkt}, let $L_\rho: (\R,\Bb) \to (\R,\Bb)$ be the distribution function of the supremum variable $\sup_{(\gseqi,\dfi) \in \netir} |\Tb_\rho(\gseqi,\dfi)|$ and let $L_\rho^{\scriptscriptstyle (\dfi_0)}$ be the distribution function of $\sup_{\gseqi \in [0,1]} |\Tb_\rho(\gseqi,\dfi_0)|$. Furthermore, recall the random variable
\[
T_{\rho,g_1} = \sup \limits_{(\gseqi,\dfi) \in \netir} \big|\Tb_\rho(\gseqi,\dfi) +  \Tb_{\rho,g_1} (\gseqi,\dfi) \big|,
\]
defined in \eqref{Trhsugpdef} with the deterministic function
\begin{equation*}
\Tb_{\rho,g_1} (\gseqi,\dfi) = N_\rho(g_1;\gseqi,\dfi) - \gseqi N_\rho(g_1;1,\dfi),
\end{equation*}
defined in \eqref{shiftdef} and let
\[
T_{\rho,g_1}^{(\dfi_0)} := \sup \limits_{\gseqi \in [0,1]} \big|\Tb_\rho(\gseqi,\dfi_0) + \Tb_{\rho,g_1} (\gseqi,\dfi_0) \big|.
\] 
Then the results on consistency of the tests are as follows.

\begin{proposition}
	\label{ConsuH1loc}
	Under ${\bf H}_1^{(loc)}$ with $\nu_0 \neq 0$ 
	\begin{multline}
	\label{Prop3101}
	\Prob\big(L_\rho(T_{\rho,g_1}) > 1-\alpha\big) \leq \liminf_{B \to \infty} \lim_{n\to\infty} \Prob\big( T_\rho^{(n)} \geq \hat q_{1-\alpha}^{(B)} \big( T_\rho^{(n)}\big) \big)  \\ \leq \limsup_{B \to \infty} \lim_{n\to\infty} \Prob\big( T_\rho^{(n)} \geq \hat q_{1-\alpha}^{(B)} \big( T_\rho^{(n)}\big) \big) \leq \Prob\big(L_\rho(T_{\rho,g_1}) \ge 1-\alpha\big)
	\end{multline}
	holds for each $\alpha \in (0,1)$ and additionally if $N_{\rho^2}(\nu_0,\dfi_0) >0$ then for all $\alpha \in (0,1)$ we have
	\begin{equation}
	\label{Prop3102}
	\Prob\big(\bar V_{\rho,\dfi_0}^{(g_1)} > q_{1-\alpha}^K \big) \le \liminf_{n\to\infty} \Prob\big(V_{\rho,\dfi_0}^{(n)} \geq q_{1-\alpha}^K\big) \leq \limsup_{n\to\infty} \Prob\big(V_{\rho,\dfi_0}^{(n)} \geq q_{1-\alpha}^K\big) \le \Prob\big(\bar V_{\rho,\dfi_0}^{(g_1)} \ge q_{1-\alpha}^K \big),
	\end{equation}
	with $V_{\rho,\dfi_0}^{(n)}$ and $\bar V_{\rho,\dfi_0}^{(g_1)}$ defined in \eqref{Vrhosupc}, as well as
	\begin{multline}
	\label{Prop3103}
	\Prob\big(L_\rho^{(\dfi_0)}\big(T_{\rho,g_1}^{(\dfi_0)}\big) > 1-\alpha\big) \leq \liminf_{B \to \infty} \lim_{n\to\infty} \Prob\big(W_\rho^{(n,\dfi_0)} \geq \hat q_{1 - \alpha}^{(B)}\big( W_\rho^{(n,\dfi_0)} \big) \big)  \\ \leq \limsup_{B \to \infty} \lim_{n\to\infty} \Prob\big(W_\rho^{(n,\dfi_0)} \geq \hat q_{1 - \alpha}^{(B)}\big( W_\rho^{(n,\dfi_0)} \big) \big) \leq \Prob\big(L^{(\dfi_0)}_\rho\big(T^{(\dfi_0)}_{\rho,g_1}\big) \ge 1-\alpha\big).
	\end{multline}
\end{proposition}

\begin{remark}
According to Corollary 1.3 and Remark 4.1 in \cite{GaeMolRos07} the distribution function $L_\rho$ is continuous on $\R$ and strictly increasing on $\R_+$. Thus, \eqref{Prop3101} basically states that under the local alternative for large $B,n \in \N$ the probability that the test \eqref{testvfkglob} rejects the null hypothesis is approximately equal to the probability that the supremum of the shifted version $T_{\rho,g_1}$ exceeds the $(1-\alpha)$-quantile of the non-shifted version $T_{\rho,0}$. An analysis of the latter probability, which is beyond the scope of this paper, then shows in which direction, i.e.\ for which $g_1$, it is harder to distinguish the null hypothesis from the alternative. The assertions \eqref{Prop3102} and \eqref{Prop3103} can be interpreted in the same way.
\end{remark}

\begin{corollary}
	\label{prop:asledf}
	Under ${\bf H}_0$ the tests \eqref{testvfkglob}, \eqref{testvfklok} and \eqref{testvfklokex} have asymptotic level $\alpha$, that is if $\nu_0 \neq 0$ we have for each $\alpha \in (0,1)$
	\begin{align}
	\label{Cor3111}
	\lim\limits_{B\to \infty} \lim \limits_{n \to \infty} \Prob\big( T_\rho^{(n)} \geq \hat q_{1-\alpha}^{(B)} ( T_\rho^{(n)}) \big) = \alpha
	\end{align}
	and furthermore
	\begin{align}
	\label{Cor3112}
	\lim\limits_{n \to \infty} \Prob\big(V_{\rho,\dfi_0}^{(n)} \geq q_{1-\alpha}^K\big) = \alpha, \quad \lim \limits_{B \to \infty} \lim \limits_{n \to \infty} \Prob\big(W_\rho^{(n,\dfi_0)} \geq \hat q_{1 - \alpha}^{(B)}( W_\rho^{(n,\dfi_0)} ) \big) = \alpha,
	\end{align}
	holds for all $\alpha \in (0,1)$, if $N_{\rho^2}(\nu_0;\dfi_0) >0$.
\end{corollary}

\begin{proposition}
	\label{prop:conuH1}
	The tests \eqref{testvfkglob}, \eqref{testvfklok} and \eqref{testvfklokex} are consistent in the following sense: Under ${\bf H}_1$, for all $\alpha \in (0,1)$ and all $B \in \N$, we have
	\begin{align*}
	\lim \limits_{n \to \infty} \Prob\big( T_\rho^{(n)} \geq \hat q_{1-\alpha}^{(B)} ( T_\rho^{(n)}) \big) =1.
	\end{align*}
	Under ${\bf H}_1^{\scriptscriptstyle (\rho,\dfi_0)}$,  we have  for all $\alpha \in (0,1)$ and all $B \in \N$,
	\begin{align*}
	\lim\limits_{n \to \infty} \Prob\big(V_{\rho,\dfi_0}^{(n)} \geq q_{1-\alpha}^K\big) = 1 \quad \text{ and } \quad  \lim \limits_{n \to \infty} \Prob\big(W_\rho^{(n,\dfi_0)} \geq \hat q_{1 - \alpha}^{(B)}( W_\rho^{(n,\dfi_0)} ) \big) = 1.
	\end{align*}
\end{proposition}

\subsection{Argmax-estimators}
\label{argdfmaxeSe}

If one of the aforementioned tests rejects the null hypothesis in favor of an abrupt alternative the natural question arises of how to estimate the unknown break point $\gseqi_0$. A typical approach in change-point analysis to this estimation problem is the so-called argmax-estimator, that is we basically take the argmax of the function $\gseqi \mapsto \sup_{\dfi\in\R}|\Tb_\rho^{\scriptscriptstyle (n)}(\gseqi,\dfi)|$ as an estimate for $\gseqi_0$. Consistency of our estimators follows with the argmax continuous mapping theorem of \cite{KimPol90} using the following auxiliary result.

\begin{proposition}
	\label{prop:tndfh1}
	Under ${\bf H}_1$, the random function $(\theta,\dfi) \mapsto (n\Delta_n)^{\scriptscriptstyle -1/2} \Tb_\rho^{\scriptscriptstyle (n)}(\theta,\dfi)$ converges in $\linner$ to the function
	\begin{align*}
	T_{(1)}^\rho(\theta,\dfi) \defeq \begin{cases}
	\theta (1-\theta_0) \{ N_\rho(\nu_1 ;\dfi) - N_\rho(\nu_2; \dfi) \}, \quad \text{ if } \theta \leq \theta_0 \\
	\theta_0 (1- \theta) \{ N_\rho(\nu_1;\dfi) - N_\rho(\nu_2; \dfi) \}, \quad \text{ if } \theta \geq \theta_0
	\end{cases}
	\end{align*}
	in outer probability, where $N_\rho(\nu; \cdot)$ is defined in \eqref{NrhonuDef}.
\end{proposition}

For the test problem ${\bf H}_0$ versus ${\bf H}_1$ we consider the estimator
\begin{equation}
\label{argmaxsup}
\tilde \theta_\rho^{(n)} \defeq \operatorname{arg\,max}_{\theta \in [0,1]} \sup_{\dfi \in \R} \big|\Tb_\rho^{(n)}(\gseqi, \dfi)\big|
\end{equation}
and in the setup ${\bf H}_0$ versus ${\bf H}_1^{(\rho,\dfi_0)}$ a suitable estimator for the change point is given by
\begin{equation*}
\tilde \theta^{(n)}_{\rho,\dfi_0} \defeq \operatorname{arg\,max}_{\gseqi \in [0,1]} \big|\Tb_\rho^{(n)}(\gseqi,\dfi_0)\big|.
\end{equation*}
The following proposition establishes consistency of these estimators.

\begin{proposition}
	\label{prop:argmax}
	Under ${\bf H}_1$ we have $\tilde \gseqi_\rho^{\scriptscriptstyle (n)} = \gseqi_0 + o_\Prob(1)$ for $n \to \infty$ and if the special case ${\bf H}_1^{\scriptscriptstyle (\rho,\dfi_0)}$ is true we obtain $\tilde \gseqi^{\scriptscriptstyle (n)}_{\rho,\dfi_0} = \gseqi_0 + o_\Prob(1)$.
\end{proposition}

\begin{remark}
	For the sake of convenience we have focused on the case of one single break. The results on the tests in Section \ref{sec:tfacha} also hold for alternatives with finitely many abrupt changes. Moreover,  the estimation methods depicted above can easily be extended to detect multiple change points by a standard binary segmentation algorithm dating back to \cite{Vos81}.
\end{remark}

\section{Statistical inference for gradual changes}
\label{sec:gradchadf}
\def\theequation{4.\arabic{equation}}
\setcounter{equation}{0}

As a generalization of Proposition \ref{prop:tndfh1} one can show that $(n\Delta_n)^{\scriptscriptstyle -1/2} \Tb_\rho^{\scriptscriptstyle (n)}(\theta,\dfi)$ converges in $\linner$ in outer probability  to the function $\Tb_{\rho,g_0}$ defined in \eqref{shiftdef} whenever Assumption \ref{EasierCond} is satisfied.
Thus, under some minor regularity conditions, argmax$_{\gseqi \in [0,1]} |\Tb_\rho^{\scriptscriptstyle (n)}(\gseqi,\dfi)|$ is a consistent estimator of argmax$_{\gseqi \in [0,1]} |\Tb_{\rho,g_0}(\gseqi,\dfi)|$. However, if the jump behaviour changes gradually at $\gseqi_0$, the function $\gseqi \mapsto |\Tb_{\rho,g_0}(\gseqi,\dfi)|$ is usually maximal at a point $\gseqi_1 > \gseqi_0$.
As a consequence the argmax-estimators investigated in Section \ref{argdfmaxeSe} usually overestimate a change point, if the change is not abrupt. Therefore, in this section we introduce test and estimation procedures which are tailored for gradual changes in the entire jump behaviour.


\subsection{A measure of time variation for the entire jump behaviour}

If the jump behaviour is given by \eqref{ComResAss} for some suitable transition kernel $g = g_0$ from $([0,1],$ $ \Bb([0,1]))$ into $(\R,\Bb)$, we follow \cite{VogDet15} and base our analysis of gradual changes on the quantity
\begin{align}
\label{entmeaotivDef}
D_\rho^{(g_0)}(\kseqi,\gseqi,\dfi) \defeq N_\rho(g_0;\kseqi,\dfi) - \frac \kseqi\gseqi N_\rho(g_0;\gseqi,\dfi), \quad (\kseqi,\gseqi,\dfi) \in C \times \R
\end{align}
with 
\begin{equation}
\label{cdef}
C := \{(\kseqi,\gseqi) \in [0,1]^2 \mid \kseqi \leq \gseqi \}
\end{equation} 
and where $N_\rho(g_0; \cdot, \cdot)$ is defined in \eqref{NrhoDef}. Here and throughout this paper we use the convention $\frac 00 \defeq 1$. We will address $D_\rho^{\scriptscriptstyle (g_0)}$ as the measure of time variation (with respect to $\rho$) of the entire jump behaviour of the underlying process, because the following lemma shows that $D_\rho^{\scriptscriptstyle (g_0)}$ indicates whether there is a change in the jump behaviour.
\begin{lemma}
\label{Drhg0suit}
Let $\gseqi \in [0,1]$. Then $D^{\scriptscriptstyle (g_0)}_\rho(\kseqi,\gseqi,\dfi) =0$ for all $0 \leq \kseqi \leq \gseqi$ and $\dfi \in \R$ if and only if the kernel $g_0(\cdot,d\taili)$ is Lebesgue almost everywhere constant on $[0,\gseqi]$.
\end{lemma}

According to the preceding lemma there exists a (gradual) change in the jump behaviour given by $g_0$ if and only if
\begin{align*}
\sup_{\gseqi \in [0,1]} \tilde \Dc^{(g_0)}_\rho(\gseqi) >0,
\end{align*}
where
$
\tilde \Dc^{(g_0)}_\rho(\gseqi) \defeq \sup\limits_{\dfi \in \R} \sup_{0\leq\kseqi\leq\gseqi}\big|D^{(g_0)}_\rho(\kseqi,\gseqi,\dfi)\big|.
$
As a consequence, the first point of a change in the jump behaviour is given by
\begin{align}
\label{entchanpoi}
\entcp  \defeq \inf \left\{ \gseqi \in [0,1] \mid \tilde \Dc^{(g_0)}_\rho(\gseqi) >0 \right\},
\end{align}
where we set $\inf \varnothing \defeq 1$. We call $\entcp$
the  change point of the jump behaviour of the underlying process. Notice that by the discussion after \eqref{cdef} the definition in \eqref{entchanpoi} is independent of $\rho$.
In Section \ref{subsec:cth0E} we construct an estimator for $\gseqi_0$, where we only consider the quantity  
\begin{align}
\label{supenmotv}
\mathcal D^{(g_0)}_\rho(\gseqi) \defeq \sup \limits_{\dfi \in \R} \sup \limits_{0 \leq \kseqi \leq \gseqi' \leq \gseqi} \big|D^{(g_0)}_\rho( \kseqi, \gseqi', \dfi)\big|, 
\end{align}
instead of $\tilde \Dc^{\scriptscriptstyle (g_0)}_\rho$. On the one hand the monotonicity of $\mathcal D^{\scriptscriptstyle (g_0)}_\rho$ simplifies our entire presentation and on the other hand the first time point where $\mathcal D^{\scriptscriptstyle (g_0)}_\rho$ deviates from $0$ is also given by $\gseqi_0$, so it is equivalent to consider $\mathcal D^{\scriptscriptstyle (g_0)}_\rho$ instead.  Our analysis of gradual changes is based on a consistent estimator $\Db_\rho^{\scriptscriptstyle (n)}$ of $D^{\scriptscriptstyle (g_0)}_\rho$ which we construct in Section \ref{subsec:Drhoe}. Before that we illustrate the quantities introduced in \eqref{entchanpoi} and \eqref{supenmotv} in the situations of Example \ref{Ex:Sitabcha} and Example \ref{Ex:SitgraCh}.

\begin{example}
	\label{exdf1}
	Recall the situation of an abrupt change as in Example \ref{Ex:Sitabcha}. Precisely, let $\beta \in (0,2)$, $p > 0$ and $\nu_1, \nu_2 \in \Mc(\beta,p)$ with $\nu_1 \neq \nu_2$ such that for some $\gseqi_0 \in (0,1)$ the transition kernel $g_0$ has the form
	\begin{align}
	\label{exabtrker}
	g_0(y,d\taili) = \begin{cases}
	\nu_1(d \taili), \quad &\text{ for } y \in [0,\gseqi_0], \\
	\nu_2(d \taili), \quad &\text{ for } y \in (\gseqi_0,1].
	\end{cases}
	\end{align}
	Obviously, for some function $\rho: \R \to \R$ such that Assumption \ref{EasierCond}\eqref{EasrhoCond} and \eqref{rhoneq0} are satisfied we have $D^{\scriptscriptstyle (g_0)}_\rho(\kseqi,\gseqi',\dfi) =0$ for each $(\kseqi,\gseqi',\dfi) \in C \times \R$ with $\gseqi' \leq \gseqi_0$ and consequently $\Dc^{\scriptscriptstyle (g_0)}_\rho(\gseqi) =0$ for each $\gseqi \leq \gseqi_0$. On the other hand, if $\gseqi_0 < \gseqi' \leq 1$ and $\kseqi \leq \gseqi_0$ we have
	\[
	D^{(g_0)}_\rho(\kseqi,\gseqi',\dfi) = \kseqi N_\rho(\nu_1;\dfi) - \frac\kseqi{\gseqi'}(\gseqi_0 N_\rho(\nu_1;\dfi) + (\gseqi' - \gseqi_0) N_\rho(\nu_2;\dfi)) = \kseqi (N_\rho(\nu_2;\dfi) - N_\rho(\nu_1;\dfi)) \big( \frac{\gseqi_0}{\gseqi'} - 1 \big)
	\]
	with \( N_\rho(\nu;\dfi) \) defined in \eqref{NrhonuDef} and we obtain
	\[
	\sup\limits_{\dfi\in\R}\sup\limits_{\kseqi \leq \gseqi_0} |D^{(g_0)}_\rho(\kseqi,\gseqi',\dfi)| = V_0^{\rho} \gseqi_0\big( 1- \frac{\gseqi_0}{\gseqi'} \big),
	\]
	where \( V_0^{\rho} = \sup_{\dfi\in\R} |N_\rho(\nu_1;\dfi) - N_\rho(\nu_2;\dfi)| >0\), because of $\nu_1 \neq \nu_2$ and the assumptions on $\rho$. For $\gseqi_0 < \kseqi \leq \gseqi'$ a similar calculation yields 
	\[
	D^{(g_0)}_\rho (\kseqi,\gseqi',\dfi)= \gseqi_0(N_\rho(\nu_2;\dfi) - N_\rho(\nu_1;\dfi)) \big( \frac\kseqi{\gseqi'} - 1 \big)
	\]
	which gives
	\[
	\sup\limits_{\dfi \in \R} \sup\limits_{\gseqi_0 < \kseqi \leq \gseqi'} \big|D^{(g_0)}_\rho(\kseqi,\gseqi',\dfi)\big| = V_0^{\rho} \gseqi_0 \big(1- \frac{\gseqi_0}{\gseqi'} \big).
	\]
	Therefore, it follows that the quantity defined in \eqref{entchanpoi} is given by $\gseqi_0$, because for $\gseqi > \gseqi_0$ we have
	\begin{equation}
	\label{Drhoabentw}
	\Dc^{(g_0)}_\rho(\gseqi) = \sup_{\gseqi_0 < \gseqi' \leq \gseqi} \max \Big\{ \sup_{\dfi \in \R}\sup_{\kseqi \leq \gseqi_0} \big|D^{(g_0)}_\rho(\kseqi,\gseqi',\dfi)\big|,~ \sup_{\dfi \in \R}\sup_{\gseqi_0 < \kseqi \leq \gseqi'} \big|D^{(g_0)}_\rho(\kseqi,\gseqi',\dfi)\big| \Big\} = V_0^{\rho} \gseqi_0 \big( 1- \frac{\gseqi_0}{\gseqi} \big).
	\end{equation}
\end{example}

\begin{example}
	\label{exdf2}
	Recall the situation of Example \ref{Ex:SitgraCh}. Let the transition kernel $g_0$ be of the form \eqref{gformgrach} such that there exist $\gseqi_0 \in (0,1)$, \(A_0 \in (0,\infty)\), \(\beta_0 \in (0,\hat \beta]\) and \(p_0 \in [2 \hat p  + \eps,\infty)\) for some $\eps >0$ with
	\begin{align}
	\label{vorcpkonst}
	A(y)=A_0, \quad \beta(y) = \beta_0 \quad \text{ and } \quad p(y) = p_0
	\end{align}
	for each $y \in [0,\gseqi_0]$. Additionally, let $\gseqi_0$ be contained in an open interval $U$ with a real analytic function $\bar A:U \to (0,\infty)$ and affine linear functions $\bar \beta:U \to (0,\hat \beta]$, $\bar p:U \to [2 \hat p + \eps,\infty)$ such that at least one of the functions $\bar A$, $\bar \beta$ and $\bar p$ is non-constant and
	\begin{align}
	\label{atcpanalytic}
	A(y) = \bar A(y), \quad \beta(y) = \bar \beta(y), \quad \text{ as well as } \quad p(y) = \bar p(y)
	\end{align}
	for all $y \in [\gseqi_0,1) \cap U$. Then the quantity defined in \eqref{entchanpoi} is given by $\gseqi_0$.
\end{example}	

\subsection{The empirical measure of time variation and its convergence behaviour}
\label{subsec:Drhoe}

Suppose we have established that $N_\rho^{\scriptscriptstyle (n)}(\cdot,\cdot)$ is a consistent estimator for $N_\rho(g_0;\cdot,\cdot)$. Then with the set $C$ defined in \eqref{cdef} it is reasonable to consider
\begin{align}\label{enttivaryest}
\Db_\rho^{(n)}(\kseqi, \gseqi, \dfi) \defeq N_\rho^{(n)}(\kseqi, \dfi) - \frac \kseqi\gseqi N_\rho^{(n)}(\gseqi, \dfi),~~ (\kseqi, \gseqi, \dfi) \in C \times \R,
\end{align}
as an estimate for the measure of time variation of the entire jump behaviour $D_\rho^{\scriptscriptstyle (g_0)}$ defined in  \eqref{entmeaotivDef}.
In the following we want to establish consistency of the empirical measure of time variation $\Db_\rho^{\scriptscriptstyle (n)}$. To be precise, the following two theorems show that the process 
\begin{align}
\label{bbhrhonDef}
\Hb_\rho^{(n)}(\kseqi, \gseqi, \dfi) := \sqrt{n\Delta_n}\big(\Db_\rho^{(n)}(\kseqi, \gseqi, \dfi) - D^{(g_0)}_\rho(\kseqi, \gseqi, \dfi)\big).
\end{align}
and its bootstrapped counterpart converge weakly or weakly conditional on the data in probability, respectively, to a suitable tight mean zero
Gaussian process.

\begin{theorem}
	\label{SchwKentmotv}
	If Assumption  \ref{EasierCond} is satisfied, then  the process $\Hb_\rho^{\scriptscriptstyle (n)}$ defined in \eqref{bbhrhonDef} converges weakly, that is $\Hb_\rho^{\scriptscriptstyle (n)} \weak \Hb_\rho + D_\rho^{\scriptscriptstyle (g_1)}$ in $\linctr$, where $\Hb_\rho$ is a tight mean zero Gaussian process with covariance function
	\begin{align}
	\label{HbrhoProCov}
	\operatorname{Cov}\big(\Hb_\rho(&\kseqi_1, \gseqi_1, \dfi_1),\Hb_\rho(\kseqi_2, \gseqi_2, \dfi_2)\big) = \nonumber \\
	&= \int_0^{\kseqi_1 \wedge \kseqi_2} \int_{-\infty}^{\dfi_1 \wedge \dfi_2} \rho^2(\taili) g_0(y,d\taili) dy - \frac{\kseqi_1}{\gseqi_1} \int_0^{\kseqi_2 \wedge \gseqi_1} \int_{-\infty}^{\dfi_1 \wedge \dfi_2}
	\rho^2(\taili) g_0(y,d\taili) dy \nonumber \\
	&\hspace{10mm}- \frac{\kseqi_2}{\gseqi_2} \int_0^{\kseqi_1 \wedge \gseqi_2} \int_{-\infty}^{\dfi_1 \wedge \dfi_2} \rho^2(\taili) g_0(y,d\taili) dy + \frac{\kseqi_1 \kseqi_2}{\gseqi_1 \gseqi_2} \int_0^{\gseqi_1 \wedge \gseqi_2} \int_{-\infty}^{\dfi_1 \wedge \dfi_2} \rho^2(\taili) g_0(y,d\taili) dy.
	\end{align}
\end{theorem}

For the statistical change-point inference proposed  in the following sections we require quantiles of
functionals of  the limiting distribution in Theorem \ref{SchwKentmotv}. \eqref{HbrhoProCov} shows that this  distribution depends in a complicated way on the unknown underlying kernel $g_0$
and therefore corresponding quantiles are difficult to estimate. In order to solve this problem we want to use a multiplier bootstrap approach similar to Section \ref{sec:Infabchdf}. To this end, we define the following bootstrap  counterpart of the process $\Hb_\rho^{\scriptscriptstyle (n)}$
\begin{align}
\label{hatHbrhondeq}
\hat \Hb_\rho^{(n)} (\kseqi, \gseqi, \dfi)
&\defeq
\hat \Hb_\rho^{(n)} (X^{(n)}_{\Delta_n}, \ldots, X^{(n)}_{n \Delta_n}; \xi_1, \ldots, \xi_n; \kseqi, \gseqi, \dfi ) \nonumber \\
&:=
\frac{1}{\sqrt {n \Delta_n}} \bigg[ \sum \limits_{j=1}^{\lfloor n\kseqi \rfloor} \xi_j  \rho(\Delj X^{(n)}) \ind_{(-\infty,\dfi]}(\Delj X^{(n)}) \ind_{\{|\Delj X^{(n)} | > v_n\}} - \nonumber \\
&\hspace{3cm}
- \frac{\kseqi}{\gseqi} \sum \limits_{j = 1}^{\ip{n \gseqi}}\xi_j \rho(\Delj X^{(n)}) \ind_{(-\infty,\dfi]}(\Delj X^{(n)}) \ind_{\{|\Delj X^{(n)} | > v_n\}} \bigg ].
\end{align}
The  result below establishes consistency of $\hat \Hb_\rho^{\scriptscriptstyle (n)}$. 



\begin{theorem}
	\label{BootHnConvThm}
	Let Assumption \ref{EasierCond} be valid and let the multiplier sequence $(\xi_i)_{i \in \N}$ 
	satisfy Assumption \ref{MultiplAss}. Then we have $\hat \Hb_\rho^{\scriptscriptstyle (n)} \weakP \mathbb H_\rho$
	in $\linctr$, where    the  tight mean zero Gaussian process $\mathbb H_\rho$ has the covariance structure \eqref{HbrhoProCov}.
\end{theorem}

\subsection{Estimating the gradual change point}
\label{subsec:cth0E}

For the sake of a unique definition of the (gradual) change point $\gseqi_0$  in \eqref{entchanpoi} we suppose throughout this section that Assumption \ref{EasierCond} holds with $g_1=g_2=0$. Recall the definition 
\[
\Dc^{(g_0)}_\rho(\gseqi) = \sup \limits_{\dfi \in \R} \sup \limits_{0 \leq \kseqi \leq \gseqi' \leq \gseqi} \big|D^{(g_0)}_\rho( \kseqi, \gseqi', \dfi)\big| 
\]
in \eqref{supenmotv}, then by Theorem \ref{SchwKentmotv}  the process $\Db_\rho^{\scriptscriptstyle (n)}(\kseqi,\gseqi,\dfi) $  from \eqref{enttivaryest}
is a consistent estimator of $D^{\scriptscriptstyle (g_0)}_\rho(\kseqi, \gseqi, \dfi) $. Therefore, we set
\begin{align*}
\Db_{\rho,*}^{(n)}(\gseqi) \defeq \sup \limits_{\dfi \in \R} \sup \limits_{0 \leq \kseqi \leq \gseqi' \leq \gseqi} \big|\Db_\rho^{(n)}(\kseqi,\gseqi',\dfi)\big|,
\end{align*}
and an application of the continuous mapping theorem and  Theorem \ref{SchwKentmotv}  yields  the following result.

\begin{corollary} 
	If Assumption  \ref{EasierCond} is satisfied with $g_1=g_2=0$, then $(n\Delta_n)^{1/2} \Db_{\rho,*}^{(n)} \weak \Hb_{\rho,*}$ in $\ell^{\infty}\big([0,\gseqi_0]\big)$, where $\Hb_{\rho,*}$ is
	the tight process in $\linne$ defined by
	\begin{align*}
	\Hb_{\rho,*} (\gseqi) \defeq \sup \limits_{\dfi \in \R} \sup \limits_{0 \leq \kseqi \leq \gseqi' \leq \gseqi} |\Hb_\rho(\kseqi, \gseqi', \dfi)|,
	\end{align*}
	with the centered Gaussian process $\Hb_\rho$ defined in Theorem \ref{SchwKentmotv}.
\end{corollary}

Below we obtain that the rate of convergence of an estimator for $\gseqi_0$ depends on the smoothness of the curve $\theta \mapsto \Dc^{\scriptscriptstyle (g_0)}_\rho(\gseqi)$ at $\gseqi_0$. Thus, we impose a kind of Taylor expansion of the function $\Dc^{\scriptscriptstyle (g_0)}_\rho$. More precisely, we assume throughout this section  
that $\gseqi_0 <1$  and that there exist constants $\iota, \eta, \smooi, c  >0$ such that $\Dc^{\scriptscriptstyle (g_0)}_\rho$ admits an  expansion of  the form
\begin{equation} \label{additass}
\Dc^{(g_0)}_\rho(\gseqi) = c \big( \gseqi - \gseqi_0 \big)^{\smooi} + \aleph(\gseqi)
\end{equation}
for all  $\gseqi \in [\gseqi_0, \gseqi_0 + \iota]$, where the remainder term satisfies
$|\aleph(\gseqi)| \leq K\big(\gseqi - \gseqi_0 \big)^{\smooi + \eta}$ for some $K>0$. 
According to Theorem \ref{SchwKentmotv} we have  $(n\Delta_n)^{\scriptscriptstyle  1/2} \Db_{\rho,*}^{\scriptscriptstyle (n)}(\gseqi) \rightarrow \infty$ in probability  for any $\gseqi \in  (\gseqi_0,  1]$. Consequently, if the deterministic sequence  $\thrle_n \rightarrow \infty$ is chosen appropriately, the statistic
\[
r_\rho^{(n)}(\gseqi) \defeq \ind_{\lbrace (n\Delta_n)^{1/2} \Db_{\rho,*}^{(n)}(\gseqi) \leq \thrle_n \rbrace}, \]
should satisfy
\[r_\rho^{(n)}(\gseqi) \probto \begin{cases}
1, \quad &\text{ if } \gseqi \leq \gseqi_0, \\
0, \quad &\text{ if } \gseqi > \gseqi_0.
\end{cases}\]
Thus, we define the estimator for the change point by
\begin{align}
\label{grestdef}
\hat \theta_\rho^{(n)} = \hat \theta_\rho^{(n)}(\thrle_n) \defeq \int_0^1 r_\rho^{(n)}(\gseqi) d \gseqi.
\end{align}
The theorem below establishes consistency of the estimator $\hat \theta_\rho^{\scriptscriptstyle (n)}$ under  mild additional assumptions on the sequence $(\thrle_n)_{n\in\N}$.

\begin{theorem}
	\label{SchaezerKonvT}
	If  Assumption \ref{EasierCond} is satisfied with $g_1=g_2=0$, $\gseqi_0 <1$, and  \eqref{additass} holds
	for some $\smooi >0$, then
	\[
	\hat \theta_\rho^{(n)} - \gseqi_0 = O_\Prob \Big (  \Big ( \frac{\thrle_n}{\sqrt{n\Delta_n}}\Big)^{1/\smooi} \Big),
	\]
	for any  sequence $\thrle_n \rightarrow \infty$ with $\thrle_n / \sqrt{n\Delta_n} \rightarrow 0$.
\end{theorem}

Theorem \ref{SchaezerKonvT} describes how the curvature of $\Dc^{\scriptscriptstyle (g_0)}_\rho$ at $\gseqi_0$ determines the convergence behaviour of the estimator: A lower degree of smoothness of $\Dc^{\scriptscriptstyle (g_0)}_\rho$ in $\gseqi_0$  yields a  better rate of convergence. However, the estimator depends on the choice of the   threshold level $\thrle_n$ and we explain below how to choose this sequence with bootstrap methods in order to control the probability of over- and underestimation. But before that the following theorem investigates the mean squared error
\begin{align*}
\operatorname{MSE}(\thrle_n)= \Eb \Big[ \big( \hat \theta_\rho^{(n)}(\thrle_n) - \gseqi_0 \big)^2 \Big]
\end{align*}
of the estimator $\hat \theta_\rho^{\scriptscriptstyle (n)}$. Recall the definition of $\Hb_\rho^{\scriptscriptstyle (n)}$ in \eqref{bbhrhonDef} and define
\begin{equation*}
\Hb_{\rho,*}^{(n)}(\gseqi) \defeq \sup \limits_{\dfi \in \R} \sup \limits_{0 \leq \kseqi \leq \gseqi ' \leq \gseqi} |\Hb_\rho^{(n)}(\kseqi, \gseqi ', \dfi)|, \quad \gseqi \in [0,1], 
\end{equation*}
which is an upper bound for the distance between the estimator $\Db_{\rho,*}^{\scriptscriptstyle (n)}(\gseqi)$ and the true value $\Dc^{\scriptscriptstyle (g_0)}_\rho(\gseqi)$. For a sequence $\alpha_n \rightarrow \infty$ with $\alpha_n = o(\thrle_n)$ we decompose the MSE into 
\begin{align*}
\text{MSE}^{(\rho)}_1(\thrle_n,\alpha_n) &\defeq  \Eb \Big[ \big( \hat \theta_\rho^{(n)}(\thrle_n) - \gseqi_0 \big)^2 \ind_{\left \{ \Hb_{\rho,*}^{(n)}(1) \leq \alpha_n \right\}} \Big], \\
\text{MSE}^{(\rho)}_2(\thrle_n,\alpha_n) &\defeq   \Eb \Big[ \big( \hat \theta_\rho^{(n)}(\thrle_n) - \gseqi_0 \big)^2 \ind_{\left \{ \Hb_{\rho,*}^{(n)}(1) > \alpha_n \right\}} \Big] \leq \Prob \big( \Hb_{\rho,*}^{(n)}(1) > \alpha_n \big),
\end{align*}
which can be considered as the MSE due to small and large estimation error. 
\begin{theorem}
	\label{MSEdfzerlThm}
	Suppose that $\gseqi_0 <1$, \eqref{additass} and Assumption \ref{EasierCond} with $g_1=g_2=0$ are satisfied. Then for any
	sequence $\alpha_n \rightarrow \infty$ with $\alpha_n = o(\thrle_n)$ we have
	\begin{align*}
	K_1 \Big ( \frac{\thrle_n}{\sqrt{n\Delta_n}}\Big)^{2/\smooi} \leq &\operatorname{MSE}^{(\rho)}_1(\thrle_n,\alpha_n) \leq K_2 \Big ( \frac{\thrle_n}{\sqrt{n\Delta_n}}\Big)^{2/\smooi}  \\
	\nonumber
	&\operatorname{MSE}^{(\rho)}_2(\thrle_n,\alpha_n) \leq  \Prob \big( \Hb_{\rho,*}^{(n)}(1) > \alpha_n \big),
	\end{align*}
	for $n \in \N$ sufficiently large, where $	
	K_1 = \big( \frac{1-\varphi}{c} \big)^{2/\smooi} $ and  $K_2 = \big( \frac{1+ \varphi}{c} \big)^{2/\smooi}
	$
	for  some $\varphi \in (0,1)$.
\end{theorem}

In the following we discuss the choice of the regularizing sequence $\thrle_n$ for the estimator $\hat \theta_\rho^{\scriptscriptstyle (n)}$ in order to  control the probability of over- and underestimation of the change point $\gseqi_0 \in (0,1)$.
Let $\predfest$ be a preliminary consistent estimate of  $\theta_0$. For example, if \eqref{additass} holds for some $\smooi >0$, one can take $\predfest = \hat \theta_\rho^{\scriptscriptstyle (n)}(\thrle_n)$ for a sequence $\thrle_n \rightarrow \infty$ satisfying the assumptions of Theorem \ref{SchaezerKonvT}. 
In the sequel, let $B\in\N$ be some large number and let $(\xi^{\scriptscriptstyle  (b)})_{b=1, \dots ,B}$ denote independent sequences of random variables, $\xi^{\scriptscriptstyle (b)} \defeq (\xi_j^{\scriptscriptstyle (b)})_{j\in\N}$, satisfying Assumption \ref{MultiplAss}.
We denote by $\hat \Hb_{\rho,*}^{\scriptscriptstyle  (n, b)}$ the particular bootstrap statistics calculated with respect to the data and the bootstrap multipliers $\xi^{\scriptscriptstyle (b)}_1, \ldots, \xi^{\scriptscriptstyle  (b)}_n$ from the $b$-th iteration, where
\begin{align}
\label{hatFnstardef}
\hat \Hb_{\rho,*}^{(n)}(\gseqi) \defeq \sup \limits_{\dfi \in \R} \sup \limits_{0 \leq \kseqi \leq \gseqi ' \leq \gseqi} \big|\hat \Hb_\rho^{(n)}(\kseqi, \gseqi ', \dfi)\big|
\end{align}
for $\gseqi \in [0,1]$. With these notations for $B,n \in \N$ and $0< r \leq 1$ we define the following empirical distribution function
\begin{align*}
\empFrbodif (x) &= \frac 1B \sum \limits_{i=1}^B \ind_{\big\{   \big(\hat \Hb_{\rho,*}^{\scriptscriptstyle (n,i)}(\predfest)\big)^r \leq x  \big\}},
\end{align*}
and we denote by  $
\empFrboindif(y) := \inf \big\{x \in \R ~ \big| ~ \empFrbodif(x) \geq y \big\}$
its pseudo-inverse. Then in the sense of the theorems below the optimal choice of the threshold is given by
\begin{align}
\label{thrledfdefeq}
\haFthrle :=
\empFrboindif(1-\alpha). 
\end{align}
for a confidence level $\alpha \in (0,1)$.

\begin{theorem}
	\label{OptthChodfThm}
	Let $0<\alpha<1$ and assume that  Assumption \ref{EasierCond} is satisfied with $g_1=g_2=0$ and with $0< \gseqi_0 <1$ for $\gseqi_0$ defined in \eqref{entchanpoi}. Suppose further that there exists some $\dfi_0 \in \R$ with $N_{\rho^2}(g_0;\gseqi_0,\dfi_0) > 0$.
	Then the limiting probability for underestimation of the change point $\gseqi_0$ is bounded by $\alpha$. Precisely,
	\begin{align*}
	\limsup \limits_{B \rightarrow \infty} \limsup \limits_{n \rightarrow \infty} \Prob\Big(\hat \theta_\rho^{(n)}\big(\haFthrleei\big) < \gseqi_0 \Big) \leq \alpha.
	\end{align*}
\end{theorem}

\begin{theorem}
	\label{KorThrCdfThm}
	Let Assumption \ref{EasierCond} be satisfied with $g_1=g_2=0$, let $0<r<1$ and for $\gseqi_0$ defined in \eqref{entchanpoi} let $0< \gseqi_0 <1$. Furthermore, suppose that \eqref{additass} holds for some $\smooi, c >0$ and that there exists a $\dfi_0 \in \R$ satisfying $N_{\rho^2}(g_0;\gseqi_0,\dfi_0) > 0$.
	Additionally, let the bootstrap multipliers be either bounded in absolute value or standard normal distributed. Then for each $K > \big(  1/c \big)^{\scriptscriptstyle 1/\smooi}$ and all sequences $(\alpha_n)_{n \in \N} \subset (0,1)$ with $\alpha_n \rightarrow 0$ and $(B_n)_{n \in \N} \subset \N$ with $B_n \rightarrow \infty$ such that
 $\alpha_n^2 B_n \rightarrow \infty,$  $(n \Delta_n)^{\frac{1-r}{2r}} \alpha_n \rightarrow \infty,$
		 $\alpha_n^{-1} n \Delta_n^{1+\tau} \to 0 $ (with $\tau >0$ from Assumption \ref{EasierCond}),
	we have
	\begin{align}
	\label{ovestdfabEq}
	\lim \limits_{n \rightarrow \infty} \Prob\Big(\hat\gseqi_\rho^{(n)}\big(\haFcothrleor\big) > \gseqi_0 + K \varphi^*_n \Big) =0,
	\end{align}
	where $\varphi^*_{n} = \big(\haFcothrleor/\sqrt{n \Delta_n}\big)^{1/\smooi} \probto 0$, while $\haFcothrleor \probto \infty$.
\end{theorem}

Theorem \ref{KorThrCdfThm} is meaningless without the statement  $\varphi^*_{n} \probto 0$. With the additional parameter $r \in (0,1)$ this assertion can be proved by using the assumptions $(n \Delta_n)^{\frac{1-r}{2r}} \alpha_n \rightarrow \infty$ and $\alpha_n^{-1} n \Delta_n^{1+\tau} \to 0$ only. However, it seems that for $r=1$ the statement  $\varphi^*_{n} \probto 0$ can only be verified under very restrictive conditions on the underlying process.


We conclude this section with an example which shows that the expansion \eqref{additass} and the additional assumption $N_{\rho^2}(g_0;\gseqi_0,\dfi_0) > 0$ of the preceding theorems are satisfied  in the situations of Example \ref{Ex:Sitabcha} and Example \ref{Ex:SitgraCh}. A proof for this example can be found in Section \ref{prresec4}.

\begin{example} 
	\label{exdf3} ~ \vspace{-2mm}
	\begin{enumerate}[(1)]
		\item Recall the situation of an abrupt change considered in Example \ref{exdf1}. In this case it follows from \eqref{Drhoabentw} that
		\begin{equation*}
		\Dc^{(g_0)}_\rho(\gseqi) = V_0^\rho \gseqi_0 \big( 1- \frac{\gseqi_0}\gseqi \big) = V_0^\rho (\gseqi - \gseqi_0) - \frac{V_0^\rho}\gseqi(\gseqi - \gseqi_0)^2 >0,
		\end{equation*}
		whenever $\gseqi_0 < \gseqi \leq 1$. Consequently, \eqref{additass} is satisfied with $\smooi =1$ and $\aleph(\gseqi) = - \frac{V_0^\rho}\gseqi (\gseqi - \gseqi_0)^2 = O((\gseqi - \gseqi_0)^2)$ for $\gseqi \to \gseqi_0$. Moreover, if $\nu_1 \neq 0$ and the function $\rho$ meets Assumption \ref{EasierCond}\eqref{rhoneq0}, the transition kernel given by \eqref{exabtrker} satisfies the additional assumption $N_{\rho^2}(g_0;\gseqi_0,\dfi_0) >0$ in Theorem \ref{OptthChodfThm} and Theorem \ref{KorThrCdfThm} for some $\dfi_0 \in \R$.
		\item \label{exdf3Nr2} In the situation discussed in Example \ref{exdf2} let
		\begin{equation*}
		\bar N(y,\dfi) = \bar A(y) \int_{-\infty}^\dfi \rho_{L,\hat p}(\taili) h_{\bar \beta(y), \bar p(y)}(\taili) d\taili
		\end{equation*}
		for $y \in U$ and $\dfi \in \R$. Then we have
	$
		k_0 := \min \{ k \in \N ~ \big| ~ \exists t \in \R \colon N_k(\dfi) \neq 0 \} < \infty,
	$	
		where for $k \in \N_0$ and $\dfi \in \R$
		\begin{equation*}
		N_k(\dfi) := \Big( \frac{\partial^k \bar N}{\partial y^k} \Big) \Big |_{(\gseqi_0,\dfi)}
		\end{equation*}
		denotes the $k$-th partial derivative of $\bar N$ with respect to $y$ at $(\gseqi_0,\dfi)$, which is a bounded function on $\R$. Furthermore, there exists a $\iota >0$ such that
		\begin{equation}
		\label{DcrhoLpexpan}
		\Dc_{\rho_{L,\hat p}}^{(g_0)}(\gseqi) = \Big( \frac 1{(k_0 +1)!} \sup_{t \in \R} |N_{k_0}(\dfi)| \Big) (\gseqi - \gseqi_0)^{k_0 +1} + \aleph(\gseqi)
		\end{equation}
		on $[\gseqi_0,\gseqi_0 + \iota]$ with $|\aleph(\gseqi)| \leq K(\gseqi - \gseqi_0)^{k_0 +2}$ for some $K>0$. Obviously, $N_{\rho_{L,\hat p}^2}(g_0;\gseqi_0,\dfi_0) >0$ holds for some $\dfi_0 \in \R$.
	\end{enumerate}
\end{example}

\subsection{Testing for a gradual change} 

In Section \ref{sec:Infabchdf} we introduced change point tests for the situation of an abrupt change as in Example \ref{Ex:Sitabcha}, where the jump behaviour is assumed to be constant before and after the change point. 
In this section we illustrate a reasonable way to derive test procedures for the existence of a gradual change in the data. In order to formulate suitable hypotheses for a gradual change point   recall
the definition of the measure of time variation for the entire jump behaviour $D_\rho^{\scriptscriptstyle (g_0)}$ in
\eqref{entmeaotivDef} and define for $\dfi_0 \in \R$ and  $\gseqi \in [0,1]$ the quantities
\begin{align*}
\mathcal D_\rho^{(g_0)}(\gseqi) &\defeq \sup \limits_{\dfi \in \R} \sup \limits_{0 \leq \kseqi \leq \gseqi' \leq \gseqi} \big|D^{(g_0)}_\rho( \kseqi, \gseqi', \dfi)\big| \\
\mathcal D_{\rho,\dfi_0}^{(g_0)}(\gseqi) &\defeq \sup \limits_{0 \leq \kseqi \leq  \gseqi' \leq \gseqi} \big|D^{(g_0)}_\rho(\kseqi, \gseqi', \dfi_0)\big|.
\end{align*}
We test the null hypothesis

\begin{enumerate}
	\item[${\bf H}_0$:] Assumption~\ref{EasierCond} is satisfied with $g_1=g_2=0$ and there exists a L\'evy measure $\nu_0$ such that $g_0(y,d\taili) = \nu_0(d\taili)$ holds for Lebesgue almost every $y \in [0,1]$.
\end{enumerate}

versus the general alternative of non-constant jump behaviour

\begin{enumerate}
	\item[${\bf H}_1^*$:] Assumption~\ref{EasierCond} holds with $g_1=g_2=0$ and we have $\Dc_\rho^{(g_0)}(1) >0$.
\end{enumerate}


If one is interested in gradual changes in $N_\rho(\nu_s^{\scriptscriptstyle (n)};\dfi_0)$ for a fixed $\dfi_0 \in \R$, one can consider
the corresponding alternative

\begin{enumerate}
	\item[${\bf H}_1^*(\dfi_0)$:] Assumption~\ref{EasierCond} is satisfied with $g_1=g_2=0$ and we have $\Dc_{\rho,\dfi_0}^{(g_0)}(1) >0$.
\end{enumerate}

Furthermore, we investigate the behaviour of the tests introduced below under local alternatives of the form

\begin{enumerate}
	\item[${\bf H}^{(loc)}_1$:] Assumption~\ref{EasierCond} holds with $g_0(y,d\taili) = \nu_0(d\taili)$ for Lebesgue-a.e. $y \in [0,1]$ for some L\'evy measure $\nu_0$ and some transition kernels $g_1,g_2 \in \Gc(\beta,p)$.
\end{enumerate}


\begin{remark}
	Note that  the function $D^{\scriptscriptstyle (g_0)}_\rho$ in \eqref{entmeaotivDef} is uniformly continuous in $(\kseqi,\gseqi) \in C$ uniformly in $\dfi \in \R$, that is for any $\eta >0$ there exists a $\delta >0$ such that
	\[\big|D^{(g_0)}_\rho(\kseqi_1,\gseqi_1,\dfi) - D^{(g_0)}_\rho(\kseqi_2, \gseqi_2, \dfi)\big| < \eta\]
	holds for each $\dfi \in \R$ and all pairs $(\kseqi_1,\gseqi_1), (\kseqi_2, \gseqi_2) \in C = \{(\kseqi, \gseqi) \in [0,1]^2 \mid \kseqi \leq \gseqi \}$ with maximum distance $\|(\kseqi_1,\gseqi_1)- (\kseqi_2, \gseqi_2)\|_\infty < \delta$.
	Therefore, the function $D^*_\rho(g_0;\kseqi, \gseqi) = \sup_{\dfi \in \R} |D^{\scriptscriptstyle (g_0)}_\rho(\kseqi, \gseqi, \dfi)|$ is uniformly continuous on $C$ and as a consequence $\Dc^{\scriptscriptstyle (g_0)}_\rho$ is continuous on $[0,1]$. Thus, $\Dc^{\scriptscriptstyle (g_0)}_\rho(1) >0$ holds if and only if the point $\theta_0$ defined in \eqref{entchanpoi} satisfies $\theta_0 < 1$.
\end{remark}

The idea of the following tests is to reject the null hypothesis ${\bf H}_0$ for large values  of the corresponding estimators $\Db_{\rho,*}^{\scriptscriptstyle (n)}(1)$ and $\sup_{(\kseqi, \gseqi) \in C} |\Db_\rho^{\scriptscriptstyle (n)}(\kseqi, \gseqi,\dfi_0)|$ for $\mathcal D^{\scriptscriptstyle (g_0)}_\rho(1) $
and  $\mathcal D_{\rho,\dfi_0}^{\scriptscriptstyle (g_0)}(1)$, respectively. In order to obtain critical values we use the multiplier bootstrap approach introduced in Section \ref{subsec:Drhoe}.
For this purpose we denote by  $(\xi^{\scriptscriptstyle (b)})_{b=1,\ldots,B}$ for some large $B\in\N$ independent sequences $\xi^{\scriptscriptstyle (b)}=(\xi^{\scriptscriptstyle (b)}_j)_{j\in\N}$ of multipliers satisfying Assumption \ref{MultiplAss}.
We denote by $\hat \Hb_{\rho}^{\scriptscriptstyle  (n, b)}$
the processes defined in \eqref{hatHbrhondeq} calculated from $\{ X^{\scriptscriptstyle (n)}_{i \Delta_n} \mid i=0,\ldots,n\}$  and the $b$-th bootstrap multipliers $\xi^{\scriptscriptstyle (b)}_1, \ldots, \xi^{\scriptscriptstyle  (b)}_n$.
For a given level $\alpha \in (0,1)$, we
propose to reject  ${\bf H}_0$ in favor of ${\bf H}_1^*$, if
\begin{equation} \label{testdfglobal}
(n\Delta_n)^{1/2} \Db_{\rho,*}^{(n)}(1) \geq \hat q^{(B)}_{1 - \alpha} \Big( {\Hb}^{(n)}_{\rho,*}(1) \Big),
\end{equation}
where $\hat q^{(B)}_{1 - \alpha} \big({\Hb}^{(n)}_{\rho,*}(1) \big)$ denotes the $(1- \alpha)$-quantile of the sample $\hat{\Hb}_{\rho,*}^{\scriptscriptstyle (n,1)}(1), \ldots, \hat{\Hb}_{\rho,*}^{\scriptscriptstyle (n,B)}(1)$ with $\hat \Hb_{\rho,*}^{\scriptscriptstyle (n,b)}$ defined in \eqref{hatFnstardef}.
Similarly, for $\dfi_0 \in \R$, the null hypothesis
${\bf H}_0$ is rejected in favor of ${\bf H}^*_1(\dfi_0)$ if
\begin{equation} \label{testdflokal}
R_{\rho,\dfi_0}^{(n)} \defeq (n \Delta_n)^{1/2} \sup \limits_{(\kseqi, \gseqi) \in C} \big|\Db_\rho^{(n)}(\kseqi, \gseqi,\dfi_0)\big| \geq \hat q^{(B)}_{1 - \alpha}\big(R_{\rho,\dfi_0}^{(n)} \big),
\end{equation}
where $\hat q^{(B)}_{1 - \alpha}\big(R_{\rho,\dfi_0}^{(n)} \big)$ denotes the $(1- \alpha)$-quantile of the sample $\hat R_{\rho,\dfi_0}^{\scriptscriptstyle (n,1)}, \ldots, \hat R_{\rho,\dfi_0}^{\scriptscriptstyle (n,B)}$, and
\[
\hat R_{\rho,\dfi_0}^{(n,b)} \defeq \sup_{(\kseqi, \gseqi) \in C} \big| \hat \Hb^{(n,b)}_{\rho}(\kseqi, \gseqi,\dfi_0) \big|.
\]

In the following we show the behaviour of the aforementioned tests under ${\bf H}_0$, ${\bf H}_1^{(loc)}$ and the alternatves ${\bf H}_1^*$, ${\bf H}_1^*(\dfi_0)$. To this end, recall the limit process $\Hb_{\rho,g_1} \defeq \Hb_{\rho} + D_\rho^{\scriptscriptstyle (g_1)}$ in Theorem \ref{SchwKentmotv}, where $D_\rho^{\scriptscriptstyle (g_1)}$ is defined in \eqref{entmeaotivDef} and where the tight mean zero Gaussian process $\Hb_\rho$ in $\linctr$ has the covariance function \eqref{HbrhoProCov}. Under the general Assumption \ref{EasierCond} let $K_\rho : (\R,\Bb) \to (\R,\Bb)$ be the c.d.f.\ of $\sup_{(\kseqi,\gseqi,\dfi)\in C\times \R} |\Hb_\rho(\kseqi,\gseqi,\dfi)|$ and let $K_\rho^{\scriptscriptstyle (\dfi_0)} : (\R,\Bb) \to (\R,\Bb)$ be the c.d.f.\ of $\sup_{(\kseqi,\gseqi)\in C} |\Hb_\rho(\kseqi,\gseqi,\dfi_0)|$. Furthermore, let 
\begin{align*}
H_{\rho,g_1} &\defeq \sup_{(\kseqi,\gseqi,\dfi)\in C\times \R} |\Hb_\rho(\kseqi,\gseqi,\dfi) + D_\rho^{(g_1)}(\kseqi,\gseqi,\dfi)|, \\
H_{\rho,g_1}^{(\dfi_0)} &\defeq \sup_{(\kseqi,\gseqi)\in C} |\Hb_\rho(\kseqi,\gseqi,\dfi_0) + D_\rho^{(g_1)}(\kseqi,\gseqi,\dfi_0)|.
\end{align*}

The proposition below shows the performance of the new tests under the local alternative ${\bf H}_1^{(loc)}$.

\begin{proposition}
	\label{localtgrapro}
	Under ${\bf H}_1^{(loc)}$ we have for each $\alpha \in (0,1)$
	\begin{multline*}
	\Prob\big(K_\rho(H_{\rho,g_1}) > 1-\alpha \big) \leq \liminf_{B \to \infty} \lim_{n\to \infty} \Prob \Big(  (n\Delta_n)^{1/2} \Db_{\rho,*}^{(n)}(1) \geq \hat q^{(B)}_{1 - \alpha} \big( {\Hb}^{(n)}_{\rho,*}(1) \big)  \Big) \\
	\leq \limsup_{B \to \infty} \lim_{n\to \infty} \Prob \Big(  (n\Delta_n)^{1/2} \Db_{\rho,*}^{(n)}(1) \geq \hat q^{(B)}_{1 - \alpha} \big( {\Hb}^{(n)}_{\rho,*}(1) \big) \Big) \leq \Prob\big(K_\rho(H_{\rho,g_1}) \geq 1-\alpha \big),
	\end{multline*}
	if there exist $\bar \dfi \in \R$, $\bar \kseqi \in (0,1)$ with $N_{\rho^2}(g_0;\bar \kseqi, \bar \dfi) > 0$, and furthermore
	\begin{multline*}
	\Prob\big(K_\rho^{(\dfi_0)}(H^{(\dfi_0)}_{\rho,g_1}) > 1-\alpha \big) \leq \liminf_{B \to \infty} \lim_{n\to \infty} \Prob \Big( R_{\rho,\dfi_0}^{(n)} \geq \hat q^{(B)}_{1 - \alpha}\big(R_{\rho,\dfi_0}^{(n)} \big) \Big) \\
	\leq  \limsup_{B \to \infty} \lim_{n\to \infty} \Prob \Big( R_{\rho,\dfi_0}^{(n)} \geq \hat q^{(B)}_{1 - \alpha}\big(R_{\rho,\dfi_0}^{(n)} \big) \Big) \leq \Prob\big(K_\rho^{(\dfi_0)}(H^{(\dfi_0)}_{\rho,g_1}) \ge 1-\alpha \big)
	\end{multline*}
	holds for each $\alpha \in (0,1)$, if there exists a $\bar \kseqi \in (0,1)$ with $N_{\rho^2}(g_0;\bar \kseqi, \dfi_0) > 0$.
\end{proposition}

With the result above and an inspection of the limiting probability $\Prob\big(K_\rho(H_{\rho,g_1}) \geq 1-\alpha \big)$, which is beyond the scope of this paper, one can show for which direction $g_1$ it is more difficult to distinguish the null hypothesis from the alternative. An immediate consequence of Proposition \ref{localtgrapro} is that the tests \eqref{testdfglobal} and \eqref{testdflokal} hold the level $\alpha$ asymptotically.

\begin{corollary}
	\label{H0graprop}
	The tests \eqref{testdfglobal} and \eqref{testdflokal} are asymptotic level $\alpha$ tests in the following sense: Under ${\bf H}_0$ with $\nu_0\neq 0$ we have for each $\alpha \in(0,1)$
	\[\lim_{B \to \infty} \lim_{n\to \infty} \Prob \Big(  (n\Delta_n)^{1/2} \Db_{\rho,*}^{(n)}(1) \geq \hat q^{(B)}_{1 - \alpha} \big( {\Hb}^{(n)}_{\rho,*}(1) \big)  \Big) = \alpha \]
	and moreover
	\[\lim_{B \to \infty} \lim_{n\to \infty} \Prob \Big( R_{\rho,\dfi_0}^{(n)} \geq \hat q^{(B)}_{1 - \alpha}\big(R_{\rho,\dfi_0}^{(n)} \big) \Big)=\alpha,\]
	holds for all $\alpha\in (0,1)$, if $N_{\rho^2}(\nu_0;\dfi_0) >0$.
\end{corollary}

The tests \eqref{testdfglobal} and \eqref{testdflokal} are also consistent under the fixed alternatives ${\bf H}_1^*$, ${\bf H}_1^*(\dfi_0)$ in the sense of the following proposition.

\begin{proposition} \label{CorConGradf}
	Under ${\bf H}_1^*$, we have for all $B \in \mathbb N$ 
	\begin{equation*}
	\lim \limits_{n \rightarrow \infty} \mathbb P \Big( (n\Delta_n)^{1/2} \Db_{\rho,*}^{(n)}(1) \geq \hat q^{(B)}_{1 - \alpha} \big( {\Hb}^{(n)}_{\rho,*}(1) \big) \Big)  = 1.
	\end{equation*}
	Under ${\bf H}_1^*(\dfi_0)$, we have for all $B \in \mathbb N$
	\[
	\lim \limits_{n \rightarrow \infty} \mathbb P \Big( R_{\rho,\dfi_0}^{(n)} \geq \hat q^{(B)}_{1 - \alpha}\big( R_{\rho,\dfi_0}^{(n)} \big) \Big) =1.
	\]
\end{proposition}

\section{Finite sample properties}
\def\theequation{5.\arabic{equation}}
\setcounter{equation}{0}
\label{sec:fisaper}

In this section we present the results of a simulation study assessing the finite sample properties of the new statistical procedures. We divide this study into two parts: In Section \ref{MCSim} we investigate the performance of the new tests and estimators by means of a simulation study. Finally, we apply the new methods to high-frequency stock exchange prices in Section \ref{RDApp}.

\subsection{Monte Carlo experiments}
\label{MCSim}

This section is dedicated to a Monte Carlo simulation study. The design of this study is as follows:

\smallskip
\noindent
(i)  We apply our estimators and test statistics to $n$ data points $\{X_{\Delta_n}, \ldots, X_{n\Delta_n}\}$ as realizations of an It\=o semimartingale $(X_t)_{t \in \R_+}$ with characteristics $(b,\sigma,\nu_s)$. For the sample size we choose either $n=10000$ or $n=22500$, where for the effective sample size we consider the choices $k_n := n\Delta_n = 50,100,200$ in the case $n=10000$ resulting in frequencies $\Delta_n^{-1} = 200,100,50$ and in the case $n=22500$ we consider $k_n= n\Delta_n = 50,75,100,150,250$ resulting in $\Delta_n^{-1} = 450,300,225,150,90$.

\smallskip
\noindent
(ii)  Corresponding to our basic rescaling assumption \eqref{ComResAss} the jump characteristic satisfies
	\[
	\nu_s(d\taili) = g\Big(\frac{s}{n\Delta_n},d\taili \Big),
	\]
	where the transition kernel $g(y,d\taili)$ is given by 
	\begin{equation}
	\label{SimMod}
	g(y,[z,\infty)) = \begin{cases}
	\Big(\frac{\eta(y)}{\pi \taili}\Big)^{1/2} - \Big(\frac{1}{\pi 10^6}\Big)^{1/2}, \quad &\text{ if } 0 < \taili \leq \eta(y) 10^6, \\
	0, &\text{ otherwise, }
	\end{cases}
	\end{equation}
	and $g(y,(-\infty,z]) = 0$ for all $z <0$.

\smallskip
\noindent
(iii)  In order to simulate data points $\{X_{\Delta_n}, \ldots, X_{n\Delta_n}\}$ including an abrupt change we choose
	\begin{equation}
	\label{etaabch}
	\eta(y) = \begin{cases}
	1, \quad &\text{ if } y \leq \gseqi_0, \\
	\psi, \quad &\text{ if } y > \gseqi_0,
	\end{cases} \quad \quad\quad (y \in [0,1])
	\end{equation}
	for $\gseqi_0 \in (0,1)$, $\psi \ge 1$ and we use a modification of Algorithm 6.13 in \cite{ConTan04} to simulate pure jump It\=o semimartingales under ${\bf H}_0$, i.e. for $\psi=1$. Under the alternative of an abrupt change, i.e. for $\psi >1$, we merge two paths of independent semimartingales together.

\smallskip
\noindent
(iv)  A gradual change in the jump characteristic is realized by choosing
	\begin{equation}
	\label{etagrch}
	\eta(y) = \begin{cases}
	1, \quad & \text{ if } y \le \gseqi_0, \\
	(A(y-\gseqi_0)^w +1)^2, \quad & \text{ if } y \ge \gseqi_0,
	\end{cases} \quad \quad\quad (y \in [0,1])
	\end{equation}
	in \eqref{SimMod} for some $\gseqi_0 \in [0,1]$, $A>0$ and $w>0$. In order to obtain pure jump It\=o semimartingale data according to this model we sample $15$ times more frequently, i.e. for $j \in \{1,\ldots,15n\}$ we use a modification of Algorithm 6.13 in \cite{ConTan04} to simulate an increment $Z_j = \tilde X_{j\Delta_n /15}^{(j)} - \tilde X_{(j-1)\Delta_n /15}^{(j)}$ of a $1/2$-stable pure jump L\'evy subordinator with characteristic exponent
	\[
	\Phi^{(j)}(u) = \int (e^{iuz} - 1) \nu^{(j)}(dz),
	\]
	where $\nu^{(j)}(dz) = g(j/(15n),dz)$. For the resulting data vector $\{X_{\Delta_n}, \ldots, X_{n\Delta_n}\}$ we use
	\[
	X_{k\Delta_n} = \sum_{j=1}^{15k} Z_j, \quad (k=1,\ldots,n).
	\]
	
\smallskip
	\noindent
(v)  In order to investigate the performance of our truncation method we either use the plain pure jump data vector $\{X_{\Delta_n}, \ldots, X_{n\Delta_n}\}$ as described above, resulting in the characteristics $b=\sigma=0$ for the continuous part, or we use $\{X_{\Delta_n} + S_{\Delta_n}, \ldots, X_{n\Delta_n} + S_{n\Delta_n}\}$, where $S_t = W_t +t$ with a Brownian motion $(W_t)_{t \in \R_+}$ resulting in $b=\sigma=1$. In the graphics depicted below the results for pure jump data are presented on the left-hand side, while the results including a continuous component are always placed on the right-hand side.

\smallskip
\noindent
(vi)  For the truncation sequence $v_n= \gamma \Delta_n^{\ovw}$ we choose $\gamma =1$ and $\ovw = 3/4$ in each run resulting in the parameter $\tau = 2/15$ in Assumption \ref{EasierCond}.

\smallskip
\noindent
(vii)  Due to computational reasons we approximate the supremum in $\dfi \in \R$ by taking the maximum either over  the finite grid $T_1 := \{0.1\cdot j \mid j=1, \ldots, 30\}$ or the finite grid $T_2 := \{0.1 + j \cdot 0.3 \mid j=0,1,\ldots,9\}$.

\smallskip
\noindent
(viii)  For the function $\rho$ we use $\rho_{L,p}$ from \eqref{Eq:rhoLpdef} in Example \ref{Ex:SitgraCh} with parameters $L=1$ and $p=2$.

\smallskip
	\noindent
(ix) Each combination of parameters we present below is run $500$ times and if the statistical procedure includes a bootstrap method we always use $B=200$ bootstrap replications. In order to illustrate the power of our test procedures we display simulated rejection probabilities, i.e. the mean of the $500$ test results. Furthermore, we measure the performance of our estimators by mean absolute deviation, i.e. if $\Theta = \{\hat \gseqi_1, \ldots, \hat \gseqi_{500} \}$ is the set of obtained estimation results we depict
	\[
	\ell^1(\Theta,\gseqi_0) = \frac 1{500} \sum_{j=1}^{500} | \hat \gseqi_j - \gseqi_0 |,
	\]
	where $\gseqi_0$ is the location of the change point.

\begin{table}[b!]
	\begin{center}
		\footnotesize{
			\begin{tabular}{ c ||c||c|c|c|c|c|c|c| }
				\hline
				\multicolumn{1}{|c||}{$k_n$} & \multicolumn{1}{c||}{Test \eqref{testvfkglob}} & \multicolumn{1}{c|}{Pointwise Tests} & 
				\multicolumn{1}{c|}{$\dfi_0 = 0.5$} & \multicolumn{1}{c|}{$\dfi_0 = 1$} & \multicolumn{1}{c|}{$\dfi_0 = 1.5$} & 
				\multicolumn{1}{c|}{$\dfi_0 = 2$} & \multicolumn{1}{c|}{$\dfi_0 = 2.5$} & \multicolumn{1}{c|}{$\dfi_0 = 3$} \\
				\hline
				\multicolumn{1}{|c||}{$50$} & \multicolumn{1}{c||}{0.026} & \multicolumn{1}{c|}{\eqref{testvfklok}} & 
				\multicolumn{1}{c|}{$0.062$} & \multicolumn{1}{c|}{$0.036$} & \multicolumn{1}{c|}{$0.024$} & 
				\multicolumn{1}{c|}{$0.036$} & \multicolumn{1}{c|}{$0.026$} & \multicolumn{1}{c|}{$0.036$} \\
				\multicolumn{1}{|c||}{ } & \multicolumn{1}{c||}{ } & \multicolumn{1}{c|}{\eqref{testvfklokex}} & 
				\multicolumn{1}{c|}{$0.060$} & \multicolumn{1}{c|}{$0.042$} & \multicolumn{1}{c|}{$0.030$} & 
				\multicolumn{1}{c|}{$0.030$} & \multicolumn{1}{c|}{$0.016$} & \multicolumn{1}{c|}{$0.020$} \\
				\hline
				\multicolumn{1}{|c||}{75} & \multicolumn{1}{c||}{0.052} & \multicolumn{1}{c|}{\eqref{testvfklok}} & 
				\multicolumn{1}{c|}{$0.058$} & \multicolumn{1}{c|}{$0.048$} & \multicolumn{1}{c|}{$0.046$} & 
				\multicolumn{1}{c|}{$0.040$} & \multicolumn{1}{c|}{$0.046$} & \multicolumn{1}{c|}{$0.050$} \\
				\multicolumn{1}{|c||}{ } & \multicolumn{1}{c||}{ } & \multicolumn{1}{c|}{\eqref{testvfklokex}} & 
				\multicolumn{1}{c|}{$0.040$} & \multicolumn{1}{c|}{$0.046$} & \multicolumn{1}{c|}{$0.032$} & 
				\multicolumn{1}{c|}{$0.036$} & \multicolumn{1}{c|}{$0.028$} & \multicolumn{1}{c|}{$0.030$} \\
				\hline
				\multicolumn{1}{|c||}{100} & \multicolumn{1}{c||}{0.050} & \multicolumn{1}{c|}{\eqref{testvfklok}} & 
				\multicolumn{1}{c|}{$0.046$} & \multicolumn{1}{c|}{$0.054$} & \multicolumn{1}{c|}{$0.042$} & 
				\multicolumn{1}{c|}{$0.046$} & \multicolumn{1}{c|}{$0.038$} & \multicolumn{1}{c|}{$0.042$} \\
				\multicolumn{1}{|c||}{ } & \multicolumn{1}{c||}{ } & \multicolumn{1}{c|}{\eqref{testvfklokex}} & 
				\multicolumn{1}{c|}{$0.038$} & \multicolumn{1}{c|}{$0.038$} & \multicolumn{1}{c|}{$0.036$} & 
				\multicolumn{1}{c|}{$0.040$} & \multicolumn{1}{c|}{$0.028$} & \multicolumn{1}{c|}{$0.032$} \\
				\hline
				\multicolumn{1}{|c||}{150} & \multicolumn{1}{c||}{0.068} & \multicolumn{1}{c|}{\eqref{testvfklok}} & 
				\multicolumn{1}{c|}{$0.038$} & \multicolumn{1}{c|}{$0.054$} & \multicolumn{1}{c|}{$0.054$} & 
				\multicolumn{1}{c|}{$0.054$} & \multicolumn{1}{c|}{$0.058$} & \multicolumn{1}{c|}{$0.066$} \\
				\multicolumn{1}{|c||}{ } & \multicolumn{1}{c||}{ } & \multicolumn{1}{c|}{\eqref{testvfklokex}} & 
				\multicolumn{1}{c|}{$0.036$} & \multicolumn{1}{c|}{$0.036$} & \multicolumn{1}{c|}{$0.050$} & 
				\multicolumn{1}{c|}{$0.042$} & \multicolumn{1}{c|}{$0.052$} & \multicolumn{1}{c|}{$0.044$} \\
				\hline
				\multicolumn{1}{|c||}{250} & \multicolumn{1}{c||}{0.060} & \multicolumn{1}{c|}{\eqref{testvfklok}} & 
				\multicolumn{1}{c|}{$0.068$} & \multicolumn{1}{c|}{$0.056$} & \multicolumn{1}{c|}{$0.056$} & 
				\multicolumn{1}{c|}{$0.058$} & \multicolumn{1}{c|}{$0.064$} & \multicolumn{1}{c|}{$0.060$} \\
				\multicolumn{1}{|c||}{ } & \multicolumn{1}{c||}{ } & \multicolumn{1}{c|}{\eqref{testvfklokex}} & 
				\multicolumn{1}{c|}{$0.046$} & \multicolumn{1}{c|}{$0.034$} & \multicolumn{1}{c|}{$0.034$} & 
				\multicolumn{1}{c|}{$0.032$} & \multicolumn{1}{c|}{$0.044$} & \multicolumn{1}{c|}{$0.052$} \\
				\hline
			\hline
				\hline
				\multicolumn{1}{|c||}{$50$} & \multicolumn{1}{c||}{0.040} & \multicolumn{1}{c|}{\eqref{testvfklok}} & 
				\multicolumn{1}{c|}{$0.038$} & \multicolumn{1}{c|}{$0.042$} & \multicolumn{1}{c|}{$0.036$} & 
				\multicolumn{1}{c|}{$0.054$} & \multicolumn{1}{c|}{$0.034$} & \multicolumn{1}{c|}{$0.036$} \\
				\multicolumn{1}{|c||}{ } & \multicolumn{1}{c||}{ } & \multicolumn{1}{c|}{\eqref{testvfklokex}} & 
				\multicolumn{1}{c|}{$0.036$} & \multicolumn{1}{c|}{$0.030$} & \multicolumn{1}{c|}{$0.028$} & 
				\multicolumn{1}{c|}{$0.042$} & \multicolumn{1}{c|}{$0.026$} & \multicolumn{1}{c|}{$0.028$} \\
				\hline
				\multicolumn{1}{|c||}{75} & \multicolumn{1}{c||}{0.058} & \multicolumn{1}{c|}{\eqref{testvfklok}} & 
				\multicolumn{1}{c|}{$0.024$} & \multicolumn{1}{c|}{$0.050$} & \multicolumn{1}{c|}{$0.030$} & 
				\multicolumn{1}{c|}{$0.048$} & \multicolumn{1}{c|}{$0.058$} & \multicolumn{1}{c|}{$0.050$} \\
				\multicolumn{1}{|c||}{ } & \multicolumn{1}{c||}{ } & \multicolumn{1}{c|}{\eqref{testvfklokex}} & 
				\multicolumn{1}{c|}{$0.030$} & \multicolumn{1}{c|}{$0.032$} & \multicolumn{1}{c|}{$0.020$} & 
				\multicolumn{1}{c|}{$0.042$} & \multicolumn{1}{c|}{$0.046$} & \multicolumn{1}{c|}{$0.036$} \\
				\hline
				\multicolumn{1}{|c||}{100} & \multicolumn{1}{c||}{0.050} & \multicolumn{1}{c|}{\eqref{testvfklok}} & 
				\multicolumn{1}{c|}{$0.044$} & \multicolumn{1}{c|}{$0.050$} & \multicolumn{1}{c|}{$0.040$} & 
				\multicolumn{1}{c|}{$0.046$} & \multicolumn{1}{c|}{$0.048$} & \multicolumn{1}{c|}{$0.052$} \\
				\multicolumn{1}{|c||}{ } & \multicolumn{1}{c||}{ } & \multicolumn{1}{c|}{\eqref{testvfklokex}} & 
				\multicolumn{1}{c|}{$0.034$} & \multicolumn{1}{c|}{$0.040$} & \multicolumn{1}{c|}{$0.026$} & 
				\multicolumn{1}{c|}{$0.046$} & \multicolumn{1}{c|}{$0.040$} & \multicolumn{1}{c|}{$0.048$} \\
				\hline
				\multicolumn{1}{|c||}{150} & \multicolumn{1}{c||}{0.054} & \multicolumn{1}{c|}{\eqref{testvfklok}} & 
				\multicolumn{1}{c|}{$0.040$} & \multicolumn{1}{c|}{$0.050$} & \multicolumn{1}{c|}{$0.048$} & 
				\multicolumn{1}{c|}{$0.056$} & \multicolumn{1}{c|}{$0.048$} & \multicolumn{1}{c|}{$0.060$} \\
				\multicolumn{1}{|c||}{ } & \multicolumn{1}{c||}{ } & \multicolumn{1}{c|}{\eqref{testvfklokex}} & 
				\multicolumn{1}{c|}{$0.040$} & \multicolumn{1}{c|}{$0.032$} & \multicolumn{1}{c|}{$0.038$} & 
				\multicolumn{1}{c|}{$0.038$} & \multicolumn{1}{c|}{$0.030$} & \multicolumn{1}{c|}{$0.038$} \\
				\hline
				\multicolumn{1}{|c||}{250} & \multicolumn{1}{c||}{0.060} & \multicolumn{1}{c|}{\eqref{testvfklok}} & 
				\multicolumn{1}{c|}{$0.046$} & \multicolumn{1}{c|}{$0.058$} & \multicolumn{1}{c|}{$0.036$} & 
				\multicolumn{1}{c|}{$0.056$} & \multicolumn{1}{c|}{$0.062$} & \multicolumn{1}{c|}{$0.058$} \\
				\multicolumn{1}{|c||}{ } & \multicolumn{1}{c||}{ } & \multicolumn{1}{c|}{\eqref{testvfklokex}} & 
				\multicolumn{1}{c|}{$0.036$} & \multicolumn{1}{c|}{$0.050$} & \multicolumn{1}{c|}{$0.030$} & 
				\multicolumn{1}{c|}{$0.044$} & \multicolumn{1}{c|}{$0.054$} & \multicolumn{1}{c|}{$0.046$} \\
				\hline
			\end{tabular}
		}
		\caption{	\it Simulated rejection probabilities of the tests \eqref{testvfkglob}, \eqref{testvfklok} and  \eqref{testvfklokex} under the null hypothesis. Upper part: pure jump subordinator data. Lower part:  jump subordinator data plus  a Brownian motion with drift. }
		\label{tab:h011}
	\end{center}
	\vspace{-.1cm}
\end{table}

\subsubsection{Statistical inference for abrupt changes}
\label{fsperCUSUMpr}

To  illustrate the  finite sample performance of the procedures introduced in Section \ref{sec:Infabchdf} we choose the sample size $n=22500$ and the grid $T_1 = \{0.1\cdot j \mid j=1, \ldots, 30\}$ to approximate the supremum in $\dfi \in \R$. The confidence level of the test procedures is $\alpha = 5\%$ in each run.

In Table \ref{tab:h011}  we display the rejection probabilities of the tests \eqref{testvfkglob}, \eqref{testvfklok} and \eqref{testvfklokex} under the null hypothesis. We observe  a reasonable approximation of the nominal level $\alpha = 0.05$. The test \eqref{testvfklokex} appears to be slightly more conservative than the test \eqref{testvfklok}. Note that the  process investigated  in the lower part of
Table \ref{tab:h011}  includes a continuous component with $b=\sigma=1$.

\begin{figure}[t!]
	\centering
	\includegraphics[width=0.33\textwidth]{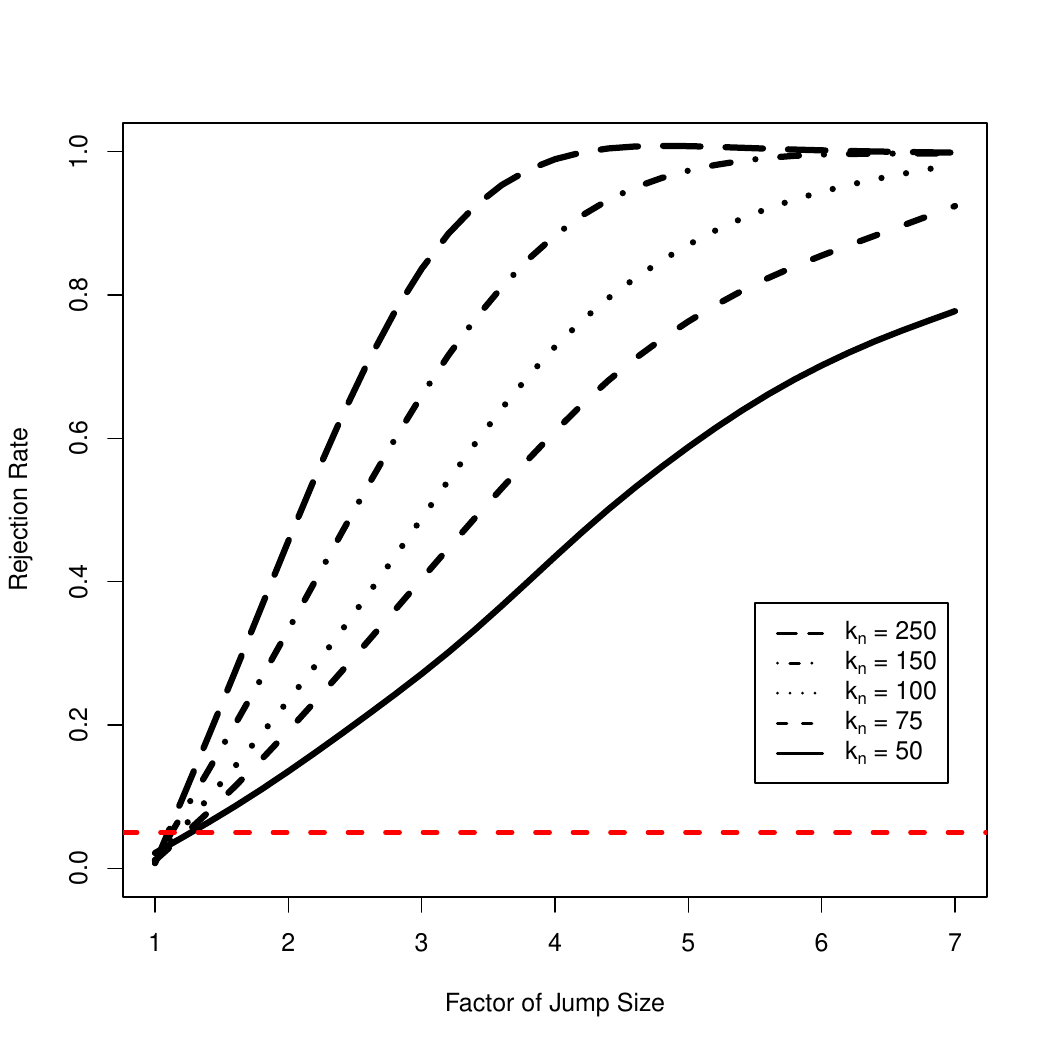}~~~ \includegraphics[width=0.33\textwidth]{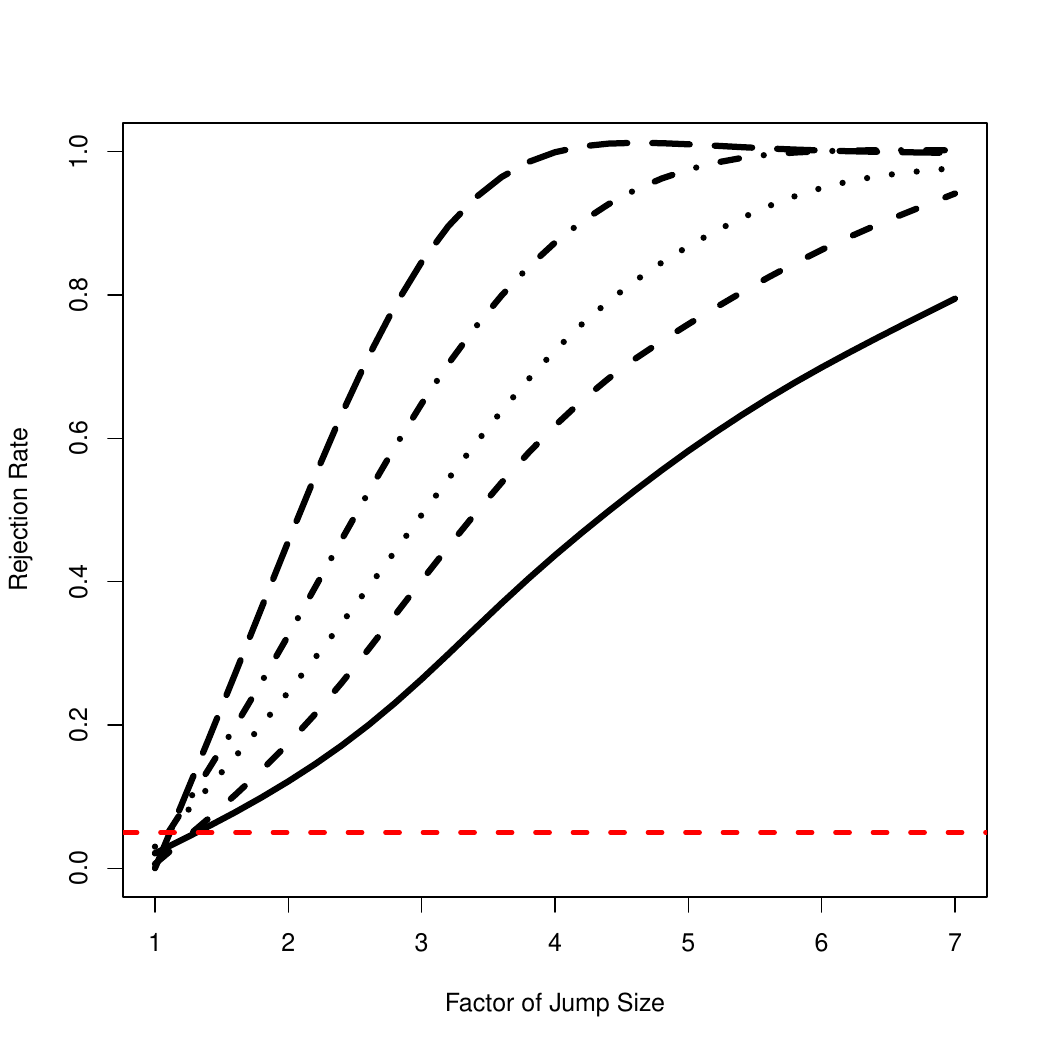}
	\includegraphics[width=0.33\textwidth]{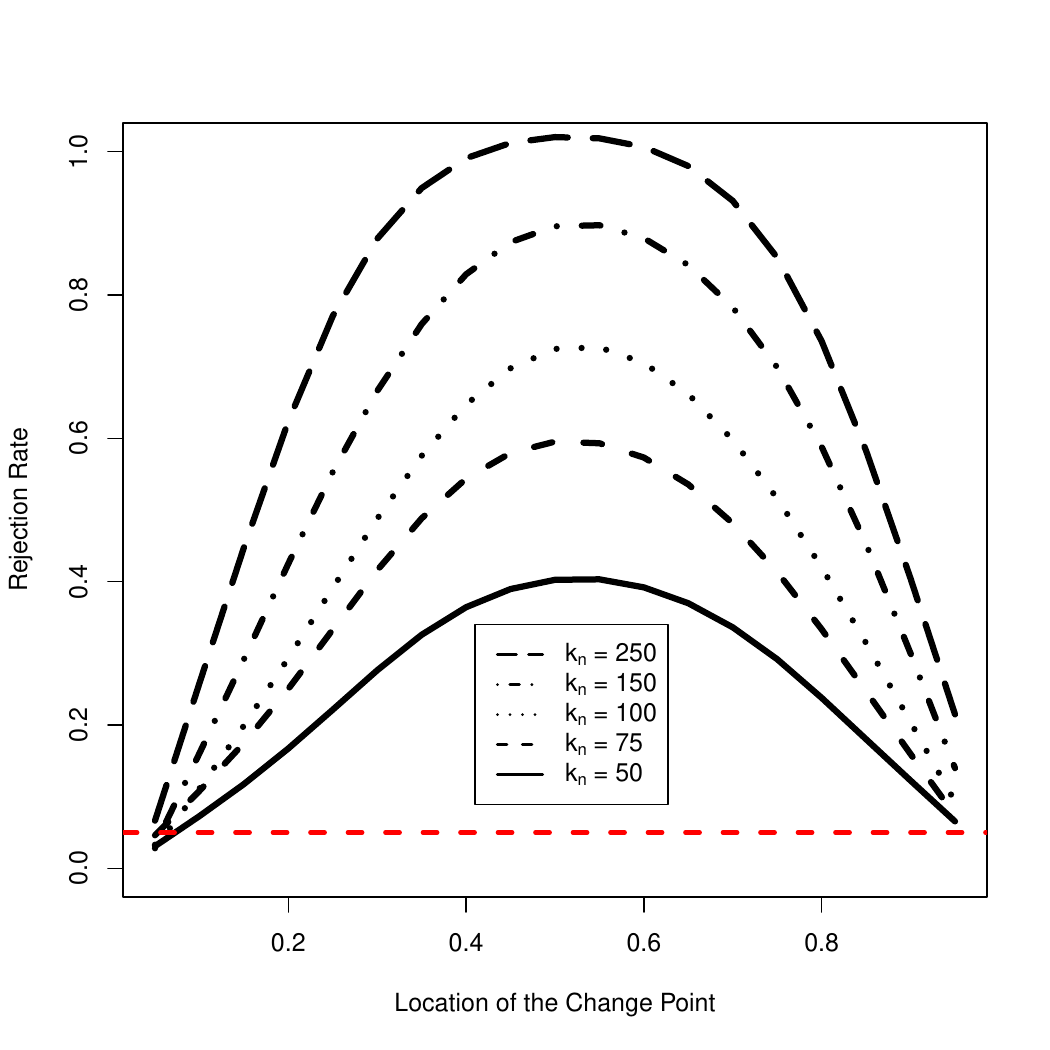}~~~ \includegraphics[width=0.33\textwidth]{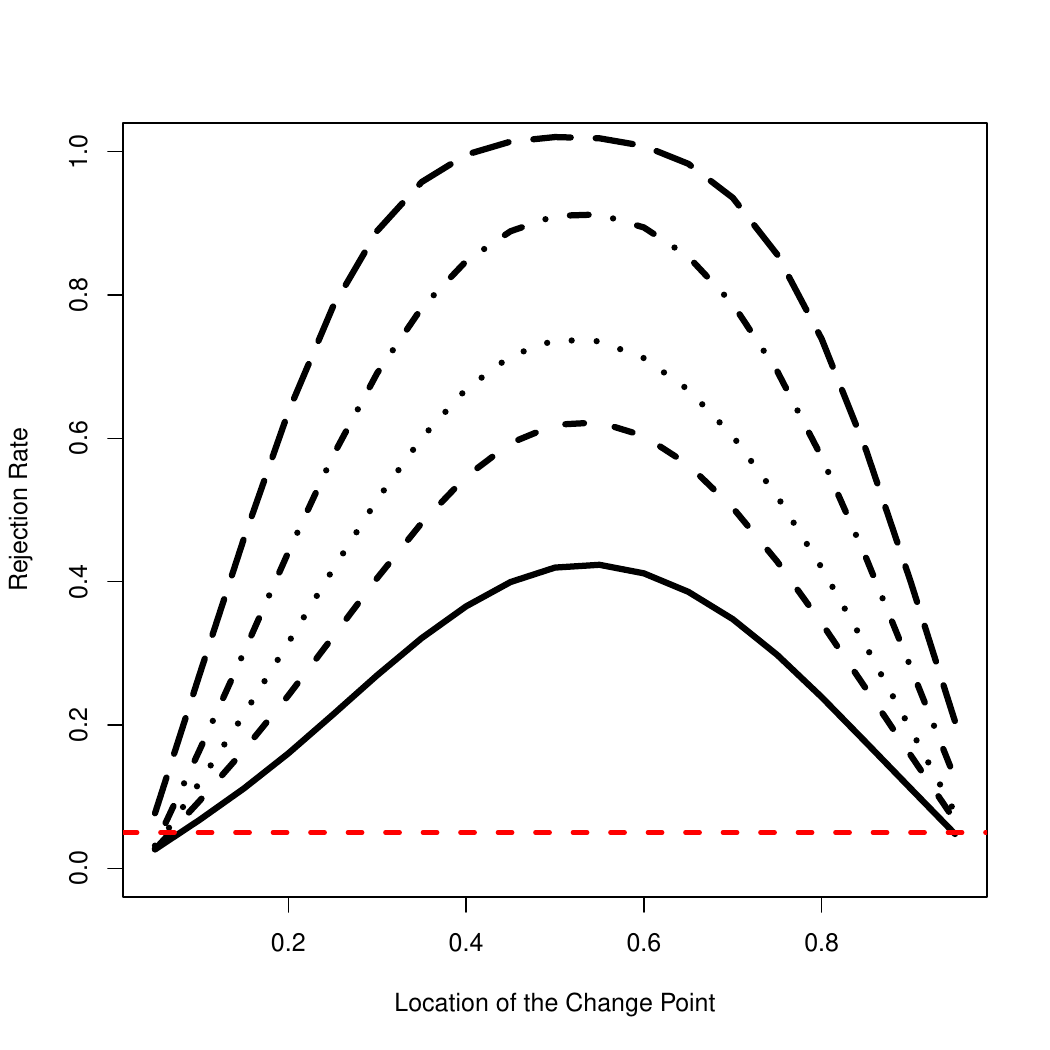}
	\includegraphics[width=0.33\textwidth]{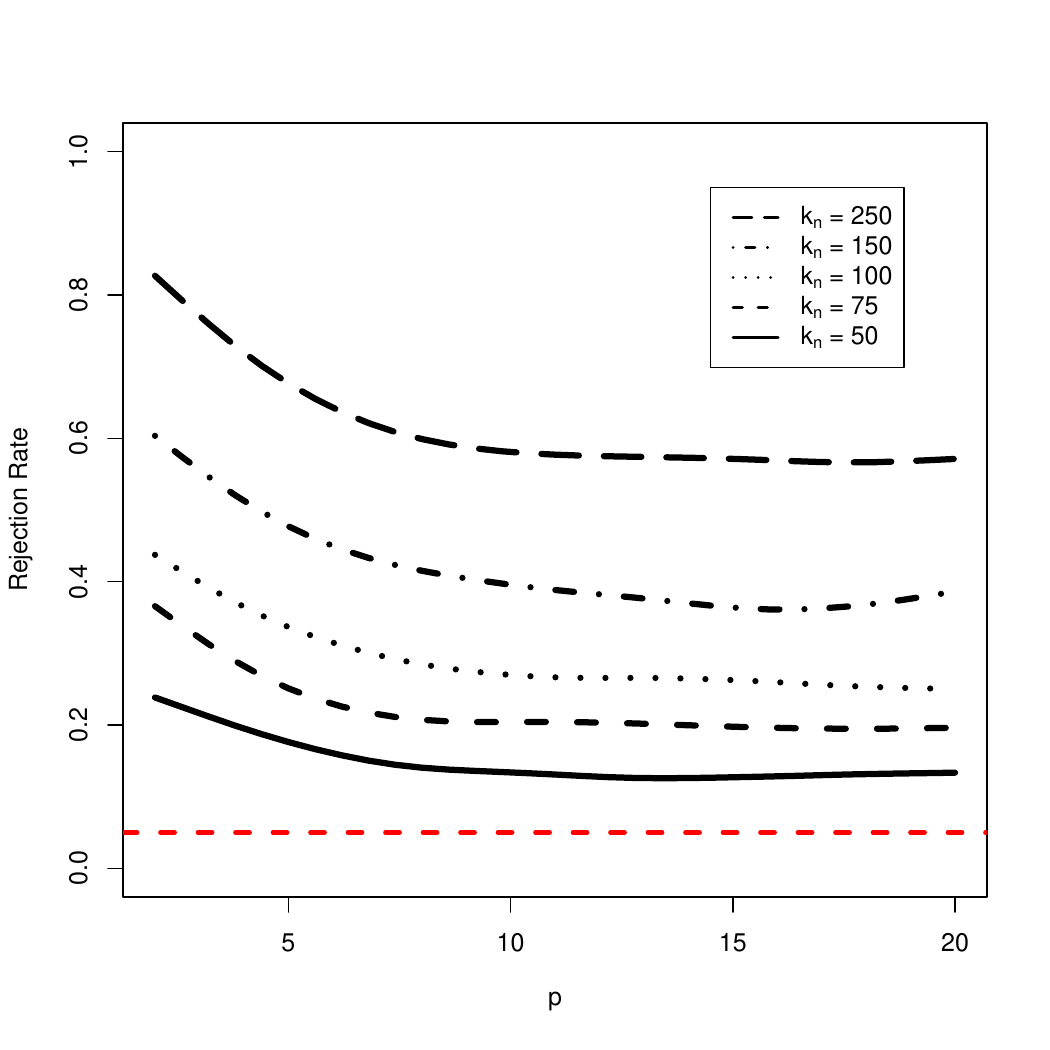}~~~ \includegraphics[width=0.33\textwidth]{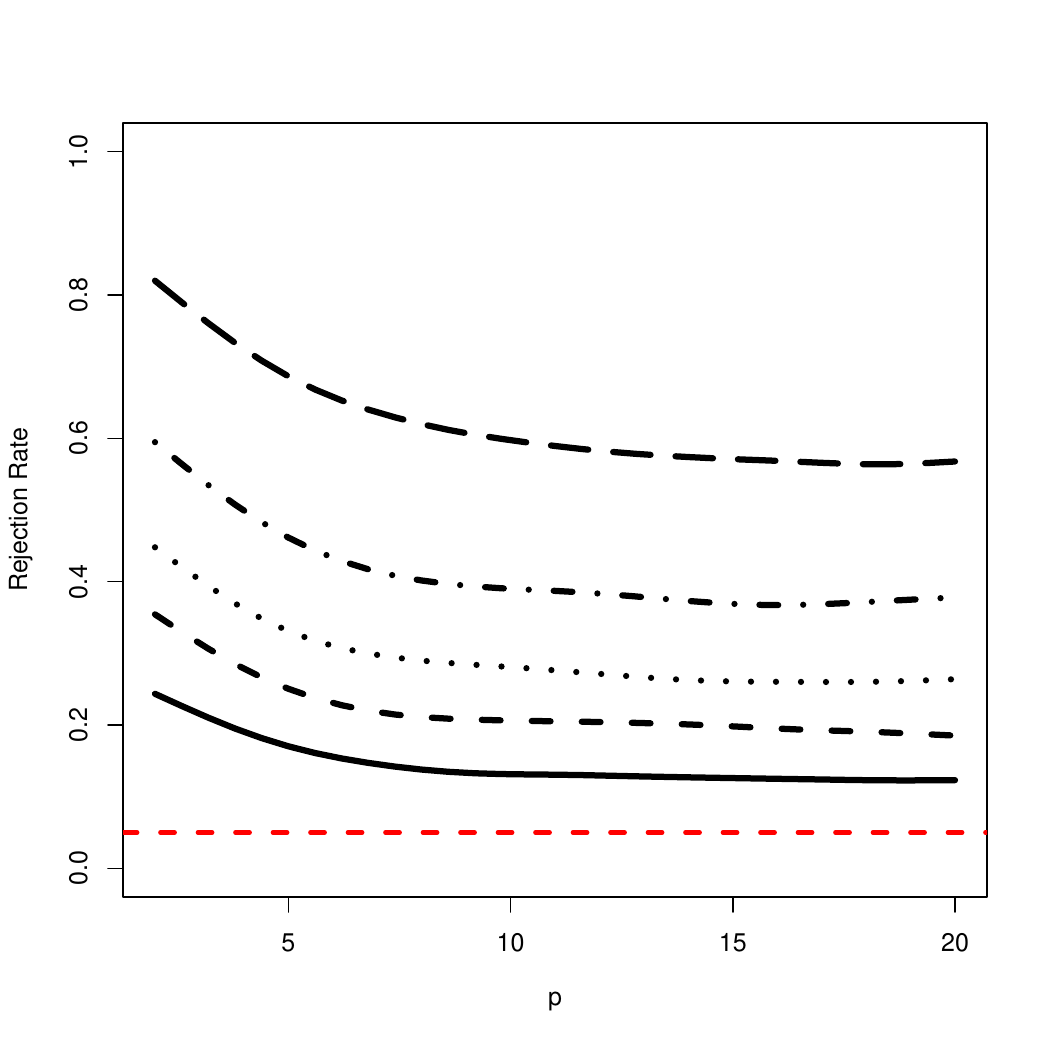}
	\vspace{-.5cm}
	\caption{\label{fig:CUSbeDu} 
		\it Simulated rejection probabilities of the test \eqref{testvfkglob}.
		Upper part:  different factors of jump size $\psi$ in \eqref{etaabch} (location of change point fixed at $\gseqi_0 = 0.5$). 
		Middle part: different locations of the change point $\gseqi_0$ ($\psi =4$ fixed). Lower part: 
		 different values of the parameter $p \ge 2$ in the function $\rho_{1,p}$  ($\gseqi_0 = 0.5$, $\psi = 3$ fixed). Left panels:
		Pure jump data. Right panels: pure jump data plus  a Brownian motion with drift. The dashed red line indicates $\alpha = 5\%$.}
\end{figure}

In the upper part of Figure \ref{fig:CUSbeDu} we depict the  rejection probabilities of the test \eqref{testvfkglob} for different effective sample sizes $k_n = n \Delta_n$. The factor of jump size corresponds to $\psi$ in \eqref{etaabch} and the dashed red line indicates the nominal level $\alpha = 5\%$. The change point is located at $\gseqi_0 = 0.5$.   Large differences of the jump size before and after the change yield higher rejection probabilities. 
Moreover, 
the  rejection probabilities  increase with  $k_n = n \Delta_n$. Notice also that the results for pure jump It\=o semimartingales and for data including a continuous component are very similar. 
This fact indicates a reasonable performance of the proposed truncation technique for an ordinary sample size $n=22500$. The middle part of Figure \ref{fig:CUSbeDu}  shows the rejection probabilities for varying locations of the change point 
$\gseqi_0$, where $\psi =4$ in \eqref{etaabch}. Our results illustrate that an abrupt change can be detected best, if it is located close
to $\gseqi_0 \approx 0.5$. Furthermore, in this case the power of the test is increasing  with  $k_n = n \Delta_n$  and the performance for data including a continuous component is nearly the same. In  the lower part of Figure \ref{fig:CUSbeDu}  we display the  rejection probabilities for different values of the parameter $p \in [2,20]$ of the function $\rho_{1,p}$  in \eqref{Eq:rhoLpdef}, which  is used to calculate the process $\Tb_{\rho_{1,p}}^{\scriptscriptstyle (n)}(\gseqi,\dfi)$. Here the change point is located at $\gseqi_0 = 0.5$ and we choose $\psi = 3$ in \eqref{etaabch}. The results suggest to use the lowest possible value of the parameter $p$ in order to obtain the maximum power of the test. Again, the rejection probabilities of the test are nearly unaffected by the presence of a Brownian component.

\begin{figure}[t!]
	\centering
	\includegraphics[width=0.33\textwidth]{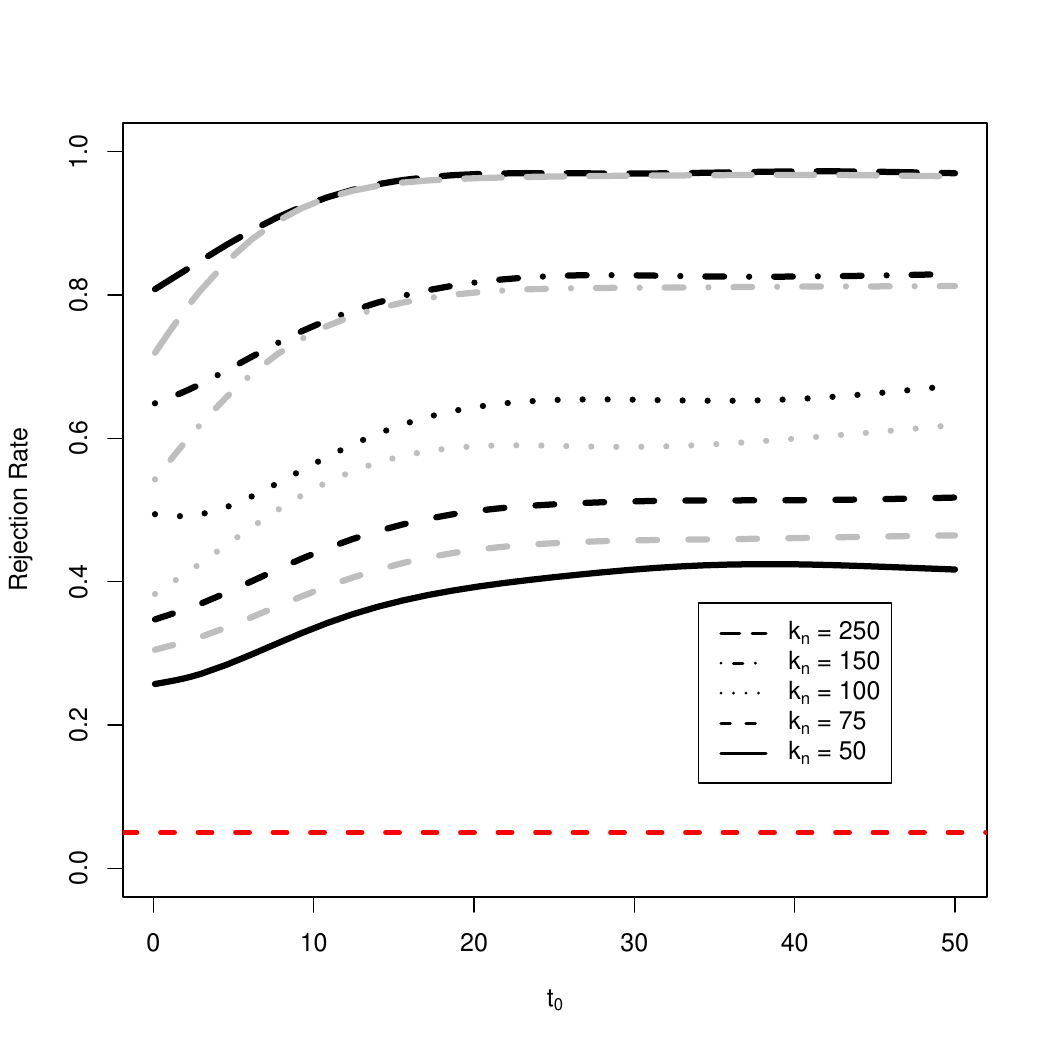}~~~ \includegraphics[width=0.33\textwidth]{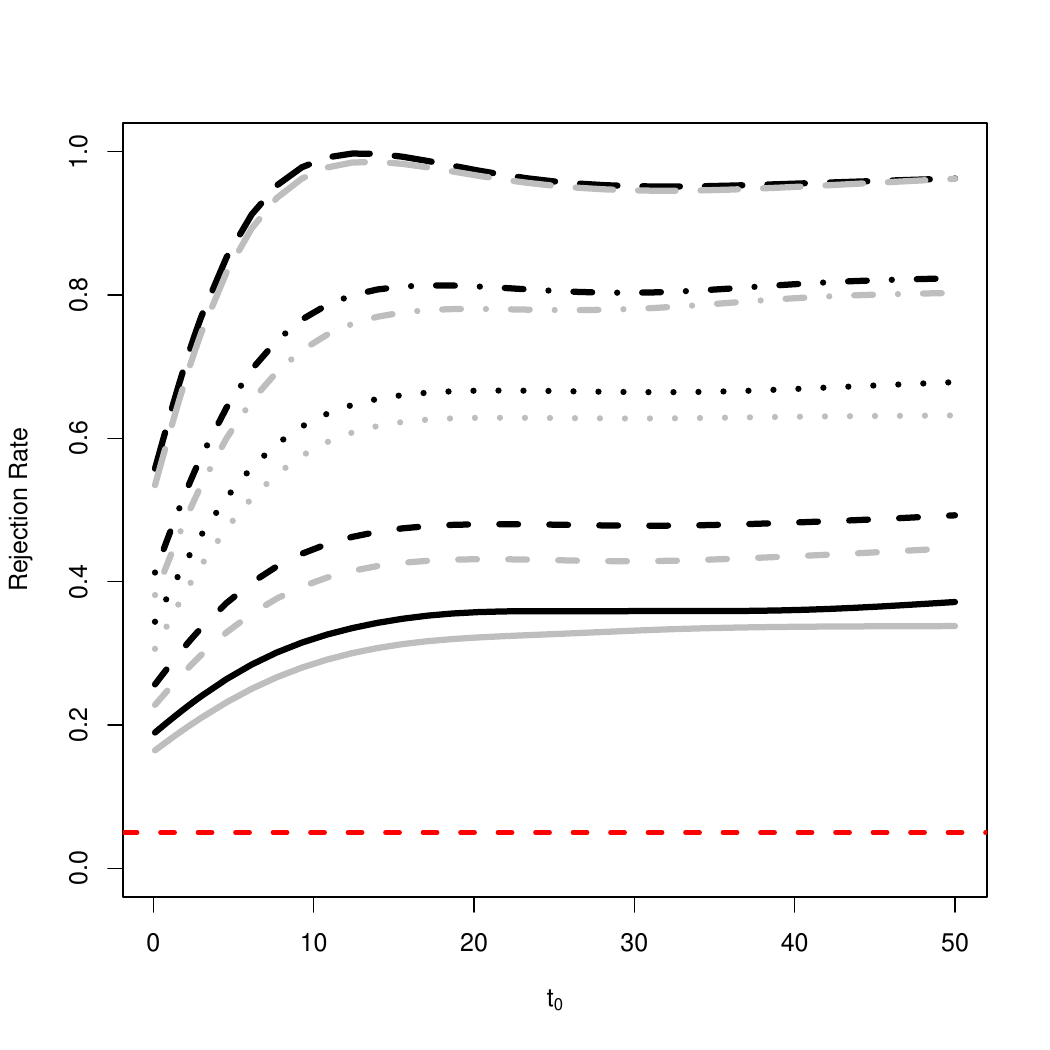}
	\vspace{-.5cm}
	\caption{\label{fig:CUStDu} 
		\it Simulated rejection probabilities of the test \eqref{testvfklok} (black lines) and the test \eqref{testvfklokex} (grey lines) for different values $\dfi_0$ for pure jump data (left-hand side) and with an additional Brownian motion with drift (right-hand side). The dashed red line indicates the nominal level $\alpha = 5\%$.
	}
\end{figure}

In Figure \ref{fig:CUStDu} we depict rejection probabilities of the tests \eqref{testvfklok} and \eqref{testvfklokex} for different values of $\dfi_0 \in [0.1,50]$. In the underlying model \eqref{SimMod} we use $\eta(y)$ defined in \eqref{etaabch} with $\gseqi_0 = 0.5$ and $\psi = 3$. We observe that the  test \eqref{testvfklok} has slightly more power than the  test \eqref{testvfklokex} and the power of both tests is increasing for small values of $\dfi_0$. The latter can be explained by the fact that less increments of the underlying It\=o semimartingale which take values in the interval $(v_n,\dfi_0]$ are used to calculate the test statistics. The effect is even more significant when a Brownian component is present (right panel). In this case it is more difficult to detect a change, because of the superposition of small increments with an i.i.d.\ sequence of random variables following a normal distribution with variance $\Delta_n$ (see also Figure 3 in \cite{BueHofVetDet15}). Furthermore, one can show (see, for instance, Lemma 6.3 in \cite{HofVetDet17}) that in the case of a pure jump It\=o semimartingale the probability of the event that $m$ increments exceed the value $\dfi_0$ is bounded by $K \dfi_0^{\scriptscriptstyle -m/2}$. As a consequence, for large $\dfi_0$ the power of both tests reaches a saturation, because only a negligible proportion of increments exceed $\dfi_0$.

\begin{figure}[t!]
	\centering
	\includegraphics[width=0.33\textwidth]{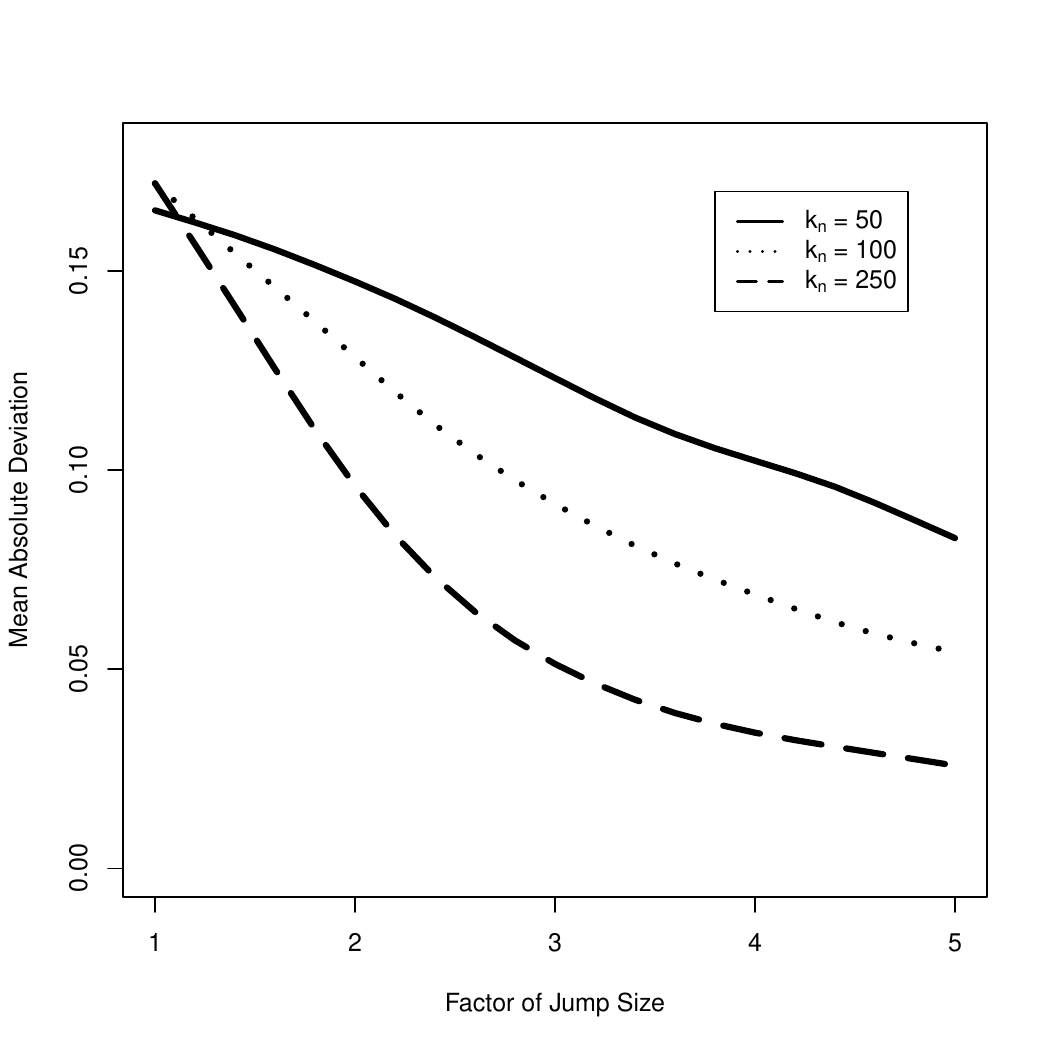}~~~ \includegraphics[width=0.33\textwidth]{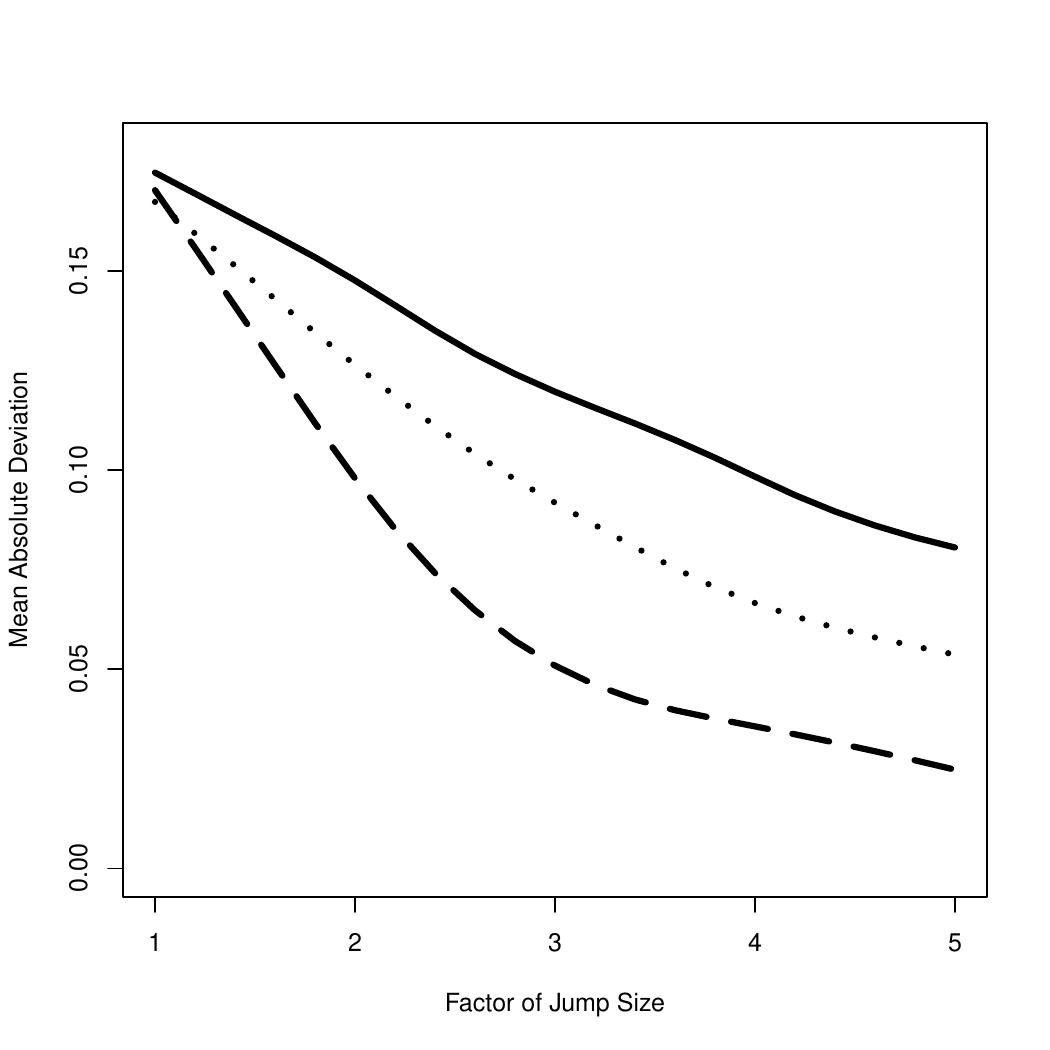}
	\includegraphics[width=0.33\textwidth]{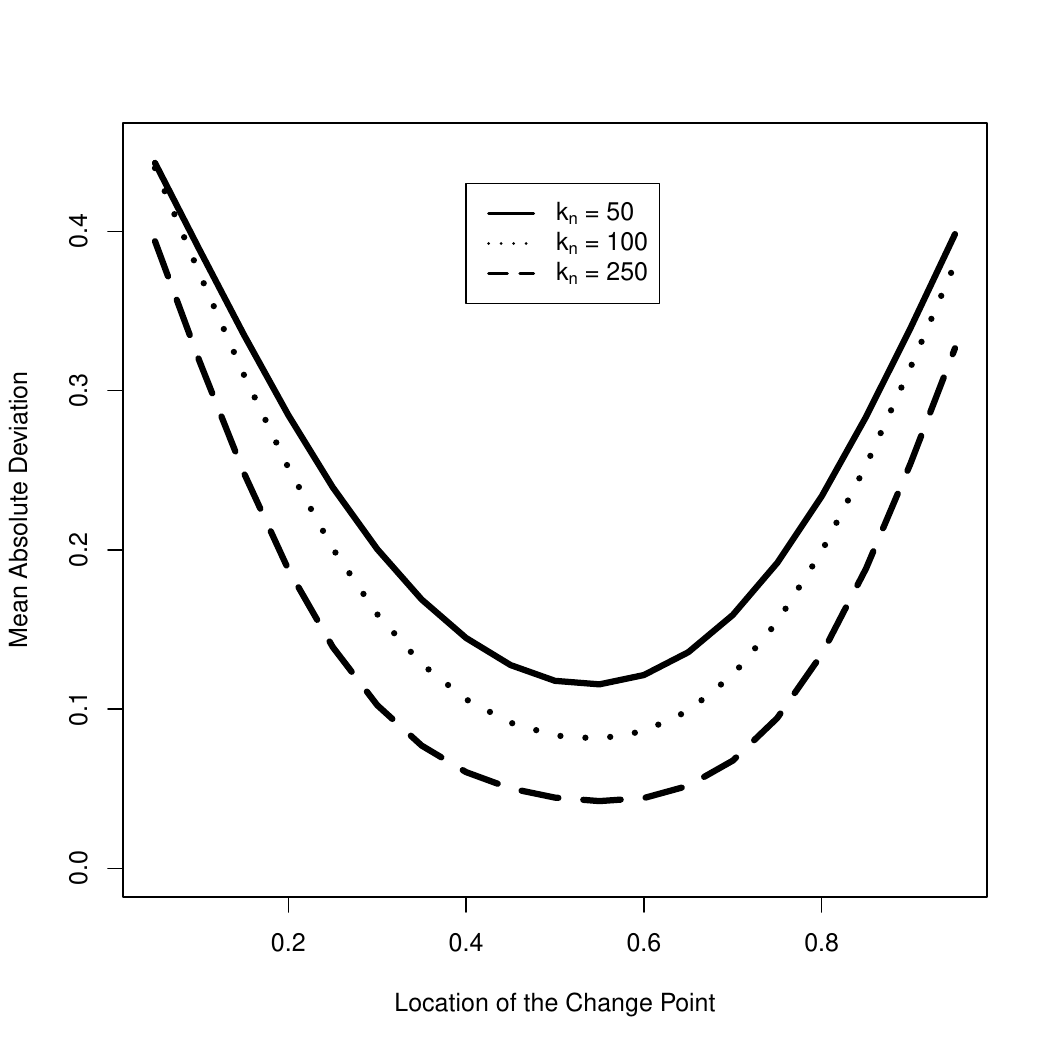}~~~ \includegraphics[width=0.33\textwidth]{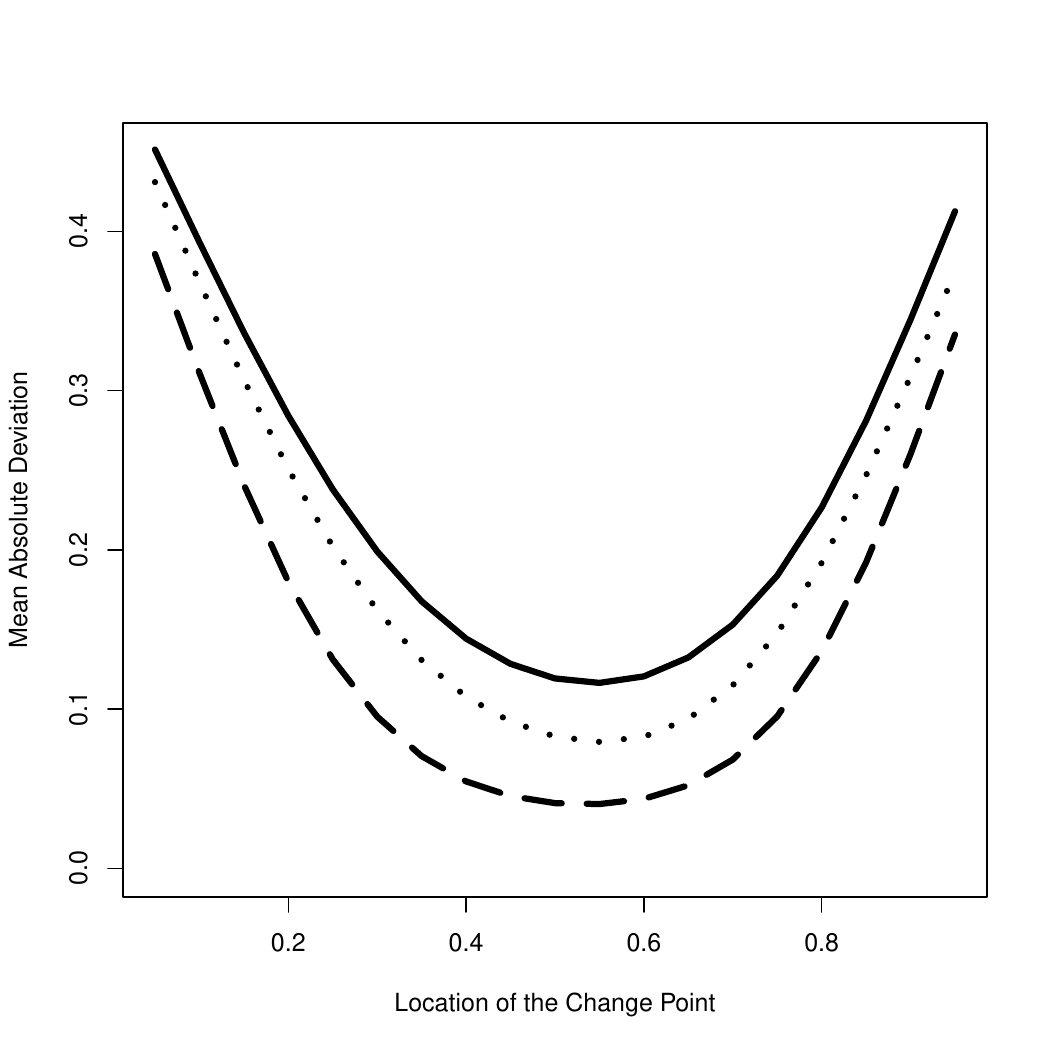}
	\vspace{-.5cm}
	\caption{\label{fig:argmbeDu} 
		\it Mean absolute deviation of the  estimator \eqref{argmaxsup}. 
		Upper part: different values of   $\psi $ in \eqref{etaabch}, $\gseqi_0 = 0.5$  fixed. Lower part:
		different locations of $\gseqi_0 \in (0,1)$, $\psi =3$ fixed. Left panels: pure jump data. Right panels: pure jump data 
		  plus an additional Brownian motion with drift.}
\end{figure}

We conclude this section with a brief investigation of the argmax-estimator \eqref{argmaxsup}.
In the upper part of Figure \ref{fig:argmbeDu}   
we display mean absolute deviations of  the estimator \eqref{argmaxsup} for different values $\psi \in [1,5]$ in \eqref{etaabch}
 ($\gseqi_0 = 0.5$  fixed), and   in the lower part we consider different locations of the change point $\gseqi_0 \in (0,1)$ ($\psi =3$ fixed). 
 The results in the upper part  correspond to Figure \ref{fig:CUSbeDu} in the sense that large values of $\psi$ yield a better performance of the statistical procedure. Additionally, we also observe that due to the truncation approach the mean absolute deviation is nearly unaffected by the presence of a Brownian component. Similar to the middle part of
 Figure \ref{fig:CUSbeDu} the results in the lower part
 suggest that a change point can be detected best if it is located at  $\gseqi_0 \approx 0.5$. Note also that the estimation error is decreasing with
  the effective sample size $k_n$.

\subsubsection{Statistical inference for gradual changes}
\label{sec:simgracha}

In this section we investigate the finite sample performance of the statistical procedures introduced in Section \ref{sec:gradchadf}.

In 
Table \ref{tab:h0gr}  we show the simulated rejection probabilities of  the tests \eqref{testdfglobal} and  \eqref{testdflokal} under the null hypothesis, i.e. for $\psi =1$ in \eqref{etaabch}. The sample size is $n=22500$ and for the test \eqref{testdfglobal} we approximate the supremum in $\dfi \in \R$ by taking the maximum over the finite grid $T_1 = \{0.1\cdot j \mid j=1, \ldots, 30\}$. In both cases for pure jump It\=o semimartingales ($b=\sigma =0$) and for It\=o semimartingales including a Brownian component ($b=\sigma =1$) we observe a reasonable approximation of the nominal level $\alpha = 5\%$.

\begin{table}[t!]
	\begin{center}
		\begin{tabular}{ c ||c||c|c|c|c|c|c| }
			\hline
			\multicolumn{1}{|c||}{ } & \multicolumn{1}{c||}{Test \eqref{testdfglobal}} & \multicolumn{6}{c|}{Test \eqref{testdflokal}} \\
			\hline
			\multicolumn{1}{|c||}{$k_n$} & \multicolumn{1}{c||}{$T_1$} & 
			\multicolumn{1}{c|}{$\dfi_0 = 0.5$} & \multicolumn{1}{c|}{$\dfi_0 = 1$} & \multicolumn{1}{c|}{$\dfi_0 = 1.5$} & 
			\multicolumn{1}{c|}{$\dfi_0 = 2$} & \multicolumn{1}{c|}{$\dfi_0 = 2.5$} & \multicolumn{1}{c|}{$\dfi_0 = 3$} \\
			\hline
			\multicolumn{1}{|c||}{$50$} & \multicolumn{1}{c||}{$0.050$} &  
			\multicolumn{1}{c|}{$0.028$} & \multicolumn{1}{c|}{$0.020$} & \multicolumn{1}{c|}{$0.030$} & 
			\multicolumn{1}{c|}{$0.040$} & \multicolumn{1}{c|}{$0.056$} & \multicolumn{1}{c|}{$0.046$} \\
			\hline
			\multicolumn{1}{|c||}{$75$} & \multicolumn{1}{c||}{$0.048$} &  
			\multicolumn{1}{c|}{$0.058$} & \multicolumn{1}{c|}{$0.058$} & \multicolumn{1}{c|}{$0.048$} & 
			\multicolumn{1}{c|}{$0.048$} & \multicolumn{1}{c|}{$0.048$} & \multicolumn{1}{c|}{$0.044$} \\
			\hline
			\multicolumn{1}{|c||}{$100$} & \multicolumn{1}{c||}{$0.056$} & 
			\multicolumn{1}{c|}{$0.062$} & \multicolumn{1}{c|}{$0.038$} & \multicolumn{1}{c|}{$0.046$} & 
			\multicolumn{1}{c|}{$0.038$} & \multicolumn{1}{c|}{$0.046$} & \multicolumn{1}{c|}{$0.060$} \\
			\hline
			\multicolumn{1}{|c||}{$150$} & \multicolumn{1}{c||}{$0.076$} & 
			\multicolumn{1}{c|}{$0.056$} & \multicolumn{1}{c|}{$0.062$} & \multicolumn{1}{c|}{$0.066$} & 
			\multicolumn{1}{c|}{$0.054$} & \multicolumn{1}{c|}{$0.062$} & \multicolumn{1}{c|}{$0.078$} \\
			\hline
			\multicolumn{1}{|c||}{$250$} & \multicolumn{1}{c||}{$0.062$} &
			\multicolumn{1}{c|}{$0.070$} & \multicolumn{1}{c|}{$0.070$} & \multicolumn{1}{c|}{$0.058$} & 
			\multicolumn{1}{c|}{$0.056$} & \multicolumn{1}{c|}{$0.054$} & \multicolumn{1}{c|}{$0.066$} \\
			\hline
						\hline
			\hline
			\multicolumn{1}{|c||}{$50$} & \multicolumn{1}{c||}{$0.044$} &  
			\multicolumn{1}{c|}{$0.036$} & \multicolumn{1}{c|}{$0.026$} & \multicolumn{1}{c|}{$0.028$} & 
			\multicolumn{1}{c|}{$0.044$} & \multicolumn{1}{c|}{$0.040$} & \multicolumn{1}{c|}{$0.040$} \\
			\hline
			\multicolumn{1}{|c||}{$75$} & \multicolumn{1}{c||}{$0.042$} &  
			\multicolumn{1}{c|}{$0.050$} & \multicolumn{1}{c|}{$0.054$} & \multicolumn{1}{c|}{$0.042$} & 
			\multicolumn{1}{c|}{$0.044$} & \multicolumn{1}{c|}{$0.038$} & \multicolumn{1}{c|}{$0.044$} \\
			\hline
			\multicolumn{1}{|c||}{$100$} & \multicolumn{1}{c||}{$0.074$} & 
			\multicolumn{1}{c|}{$0.040$} & \multicolumn{1}{c|}{$0.038$} & \multicolumn{1}{c|}{$0.036$} & 
			\multicolumn{1}{c|}{$0.046$} & \multicolumn{1}{c|}{$0.062$} & \multicolumn{1}{c|}{$0.068$} \\
			\hline
			\multicolumn{1}{|c||}{$150$} & \multicolumn{1}{c||}{$0.044$} & 
			\multicolumn{1}{c|}{$0.036$} & \multicolumn{1}{c|}{$0.056$} & \multicolumn{1}{c|}{$0.058$} & 
			\multicolumn{1}{c|}{$0.052$} & \multicolumn{1}{c|}{$0.042$} & \multicolumn{1}{c|}{$0.044$} \\
			\hline
			\multicolumn{1}{|c||}{$250$} & \multicolumn{1}{c||}{$0.050$} &
			\multicolumn{1}{c|}{$0.034$} & \multicolumn{1}{c|}{$0.042$} & \multicolumn{1}{c|}{$0.056$} & 
			\multicolumn{1}{c|}{$0.062$} & \multicolumn{1}{c|}{$0.062$} & \multicolumn{1}{c|}{$0.058$} \\
			\hline
		\end{tabular}
		\caption{\it Simulated rejection probabilities of the tests \eqref{testdfglobal} and  \eqref{testdflokal}
		 under the null hypothesis.
		Upper part:  pure jump It\=o semimartingale data. Lower part:
		pure jump It\=o semimartingale data  plus a Brownian motion with drift.}
		\label{tab:h0gr}
	\end{center}
	\vspace{-.1cm}
\end{table}

To save  computational time  the rejection probabilities of the  tests \eqref{testdfglobal} and \eqref{testdflokal} 
under the alternative 
 are obtained for the  sample size $n=10000$ and effective sample size $k_n \in \{50,100,200\}$. 
The upper part of 
Figure \ref{fig:grtwDu} shows the simulated rejection probabilities of the test \eqref{testdfglobal} for different degrees of smoothness  of the change $w$ in \eqref{etagrch}. The change is located at $\gseqi_0 = 0.4$ and $A$ is chosen such that the characteristic quantity for a gradual change satisfies $\Dc_\rho^{\scriptscriptstyle (g)}(1) = 3$ in each scenario. As expected, it is more difficult to distinguish a very smooth change from the null hypothesis and therefore the rejection probability is decreasing in $w$. Similar to the CUSUM test investigated in Section \ref{fsperCUSUMpr} the power of the test is increasing with  $k_n = n \Delta_n$. 
In the lower part of Figure \ref{fig:grtwDu}   we depict the  rejection rates of the test \eqref{testdfglobal} for different locations of the change point $\gseqi_0 \in (0,1)$. We simulate a linear change, i.e. we have $w=1$ in \eqref{etagrch}, and $A$ is chosen such that $\Dc_\rho^{\scriptscriptstyle (g)}(1)= 0.3$ holds in each run. As before, the power of the test is increasing in the effective sample size $k_n = n\Delta_n$ and moreover it is decreasing in $\gseqi_0$. The latter
observation  can be explained by the shape of model \eqref{etagrch}, because  for  larger values of  $\gseqi_0$ the jump characteristic is ``closer'' to the null hypothesis.
\\
Note also that all results  are very similar for pure jump processes and processes including a Brownian component. This indicates that our truncation approach also works in this setup.
\begin{figure}[t!]
	\centering
	\includegraphics[width=0.33\textwidth]{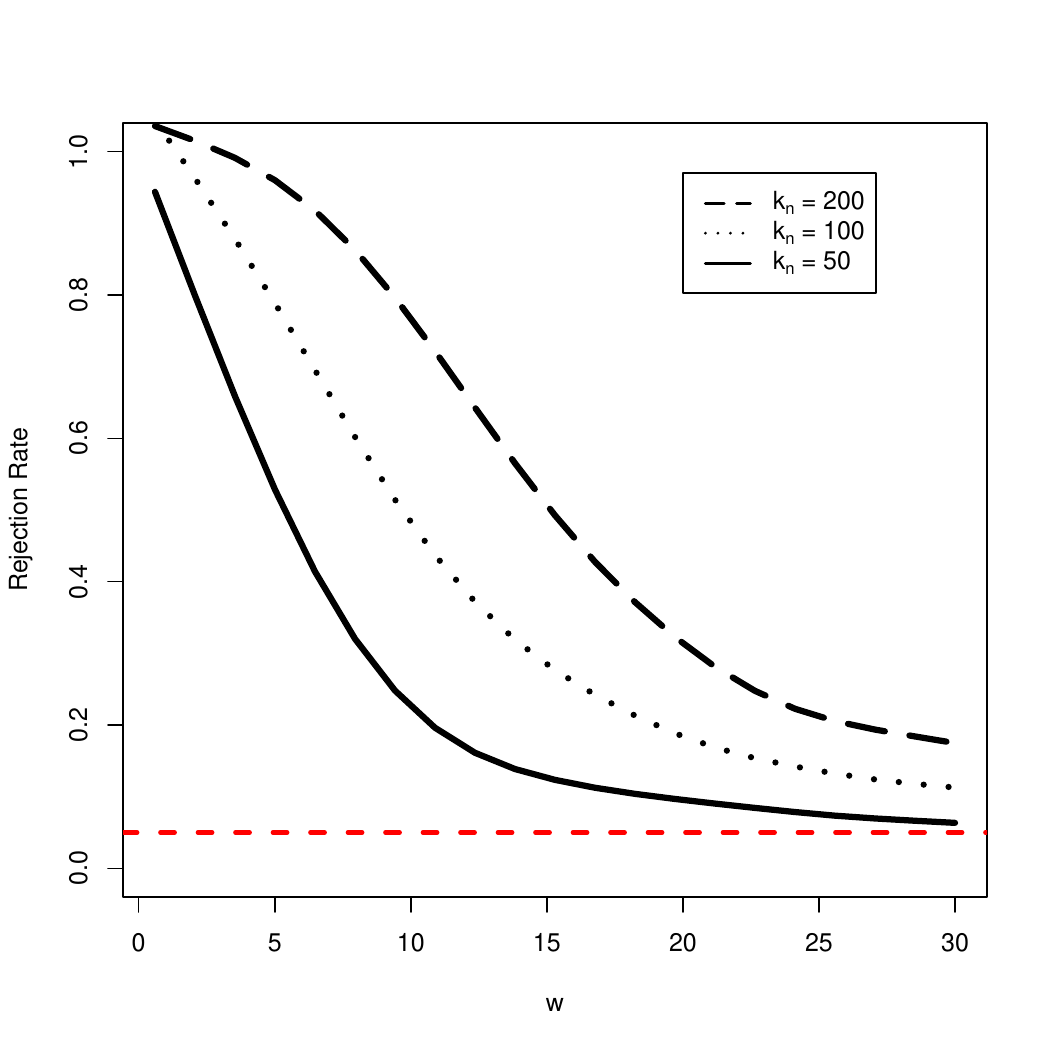}~~~ \includegraphics[width=0.33\textwidth]{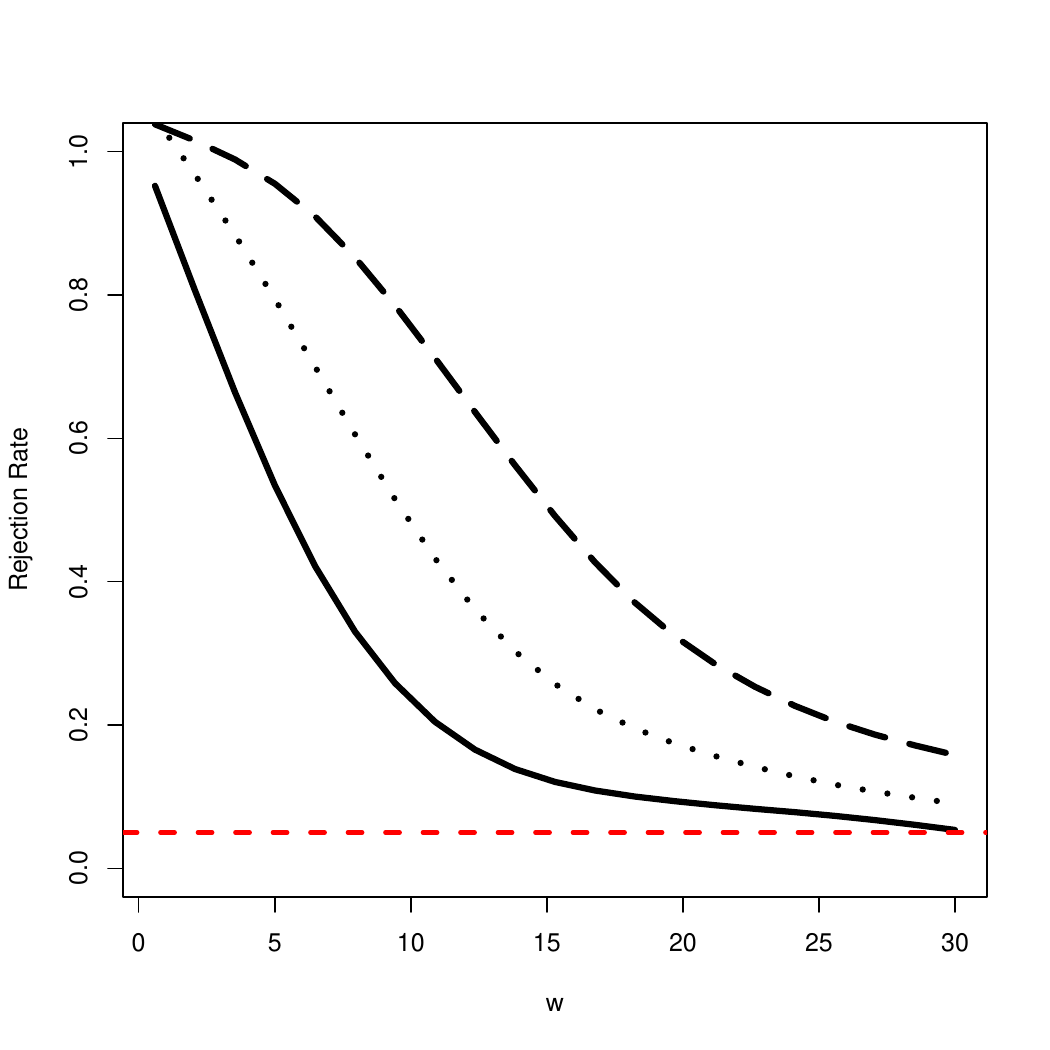}
	\includegraphics[width=0.33\textwidth]{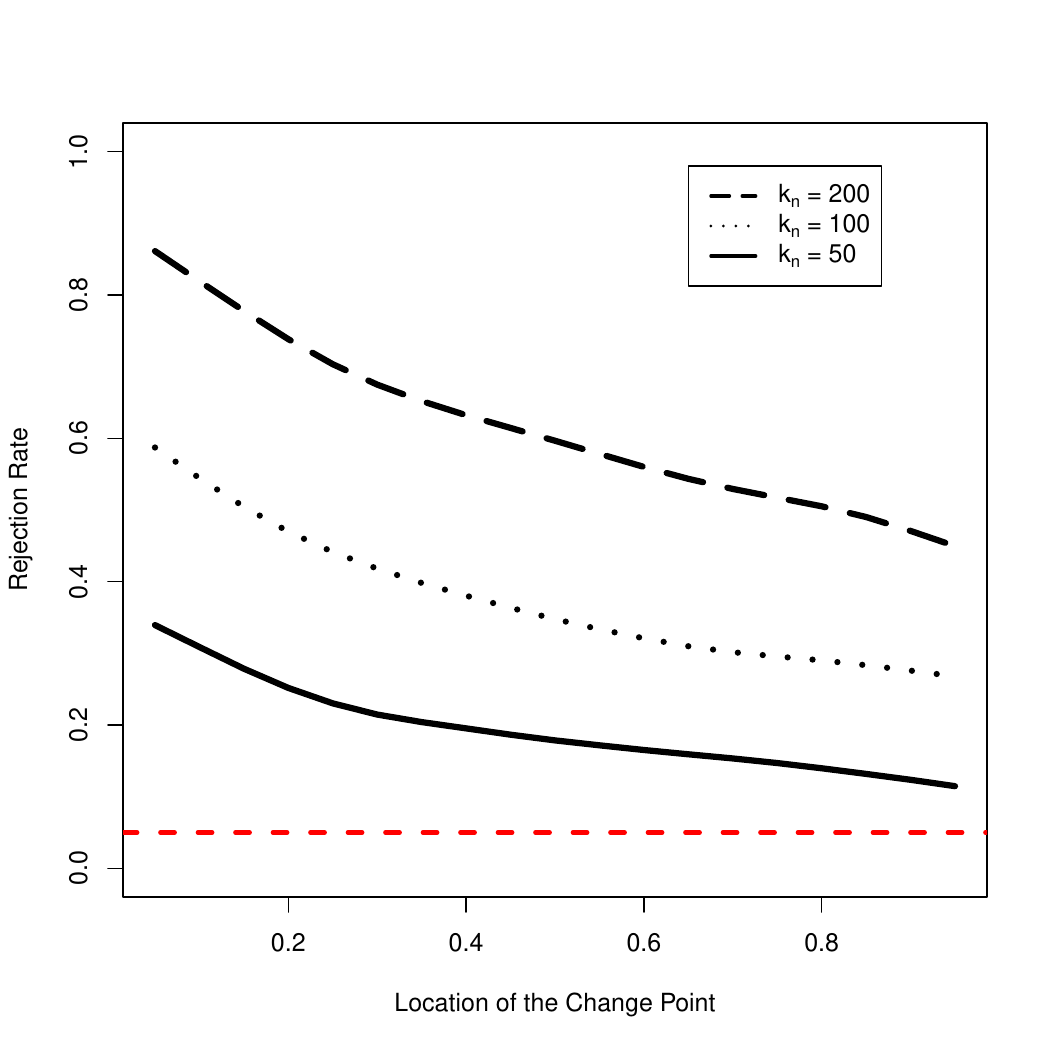}~~~ \includegraphics[width=0.33\textwidth]{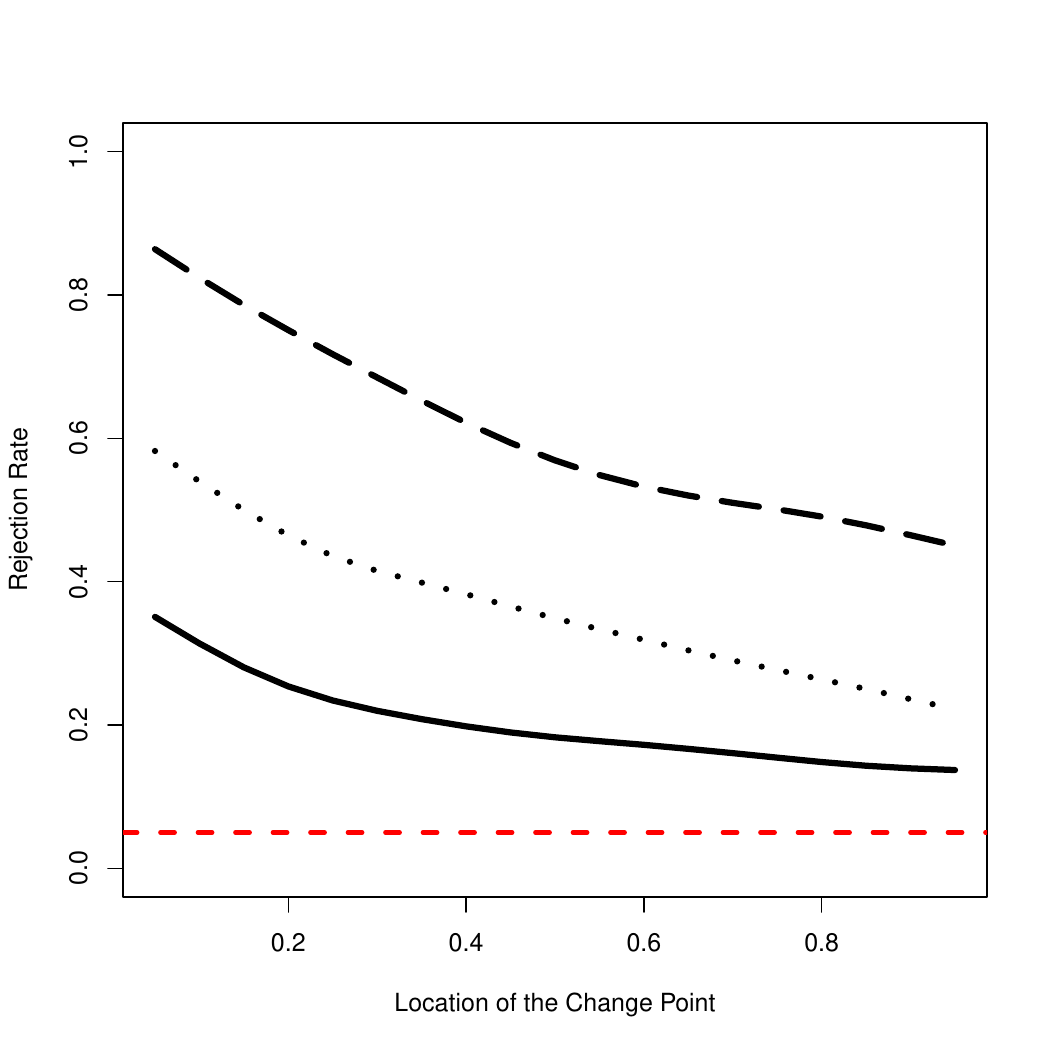}
	\vspace{-.5cm}
	\caption{\label{fig:grtwDu} 
		\it Simulated rejection probabilities of the test \eqref{testdfglobal} under the alternative. Upper part: different values $w \in [0.6,30]$ in \eqref{etagrch},  $\theta_{0}=0.4$ fixed. Lower part:
		different locations of the change point $\gseqi_0 \in (0,1)$, $w=1$ fixed. Left panels:
		pure jump It\=o semimartingales; Right panels:  pure jump It\=o semimartingales plus a Brownian motion with drift. 
		The dashed red line indicates the nominal level $\alpha = 5\%$.
	}
\end{figure}

We conclude this section with a  study of   the change point estimator $\hat \gseqi_\rho^{\scriptscriptstyle (n)}$  in \eqref{grestdef}.
Following \cite{HofVetDet17} we implement the estimator $\hat \gseqi_\rho^{\scriptscriptstyle (n)}$ in five steps as follows:

\smallskip
\smallskip
\noindent
\textit{Step 1.} Choose a preliminary estimate \(\hat \gseqi^{\scriptscriptstyle (pr)} \in (0,1)\), a probability level \(\alpha \in (0,1) \) and a parameter \\
$~~~~~~~~~~$ \(r \in (0,1]\).

\smallskip
\noindent
\textit{Step 2.} Initial choice of the tuning parameter \(\thrle_n\): 
	Evaluate \eqref{thrledfdefeq} for $\hat \gseqi^{\scriptscriptstyle (pr)}, \alpha$ and $r$ (with \( B=200\) \\
$~~~~~~~~~~$ as described above and where the supremum in $\dfi \in \R$ is approximated by the maximum \\
$~~~~~~~~~~$ over $\dfi \in T_2 = \{0.1 + j \cdot 0.3 \mid j=0,1,\ldots,9\}$) and obtain \( \hat\thrle^{\scriptscriptstyle (in)}\).

	\smallskip
\noindent
\textit{Step 3.} Intermediate estimate of the change point. 
	Evaluate \eqref{grestdef} for \(\hat\thrle^{\scriptscriptstyle (in)}\) and obtain \(\hat \gseqi^{\scriptscriptstyle (in)}\).

	\smallskip
\noindent
\textit{Step 4.} Final choice of the tuning parameter \(\thrle_n\): 
	Evaluate \eqref{thrledfdefeq} for $\hat \gseqi^{\scriptscriptstyle (in)}, \alpha, r$ and obtain \( \hat\thrle^{\scriptscriptstyle (fi)}\).

\smallskip
\noindent
\textit{Step 5.} Estimate $\gseqi_0$. 
	Evaluate \eqref{grestdef} for \(\hat\thrle^{\scriptscriptstyle (fi)}\) and obtain the final estimate \(\hat \gseqi\) of the change point.

\medskip

From the theoretical point of view as discussed in Section \ref{subsec:cth0E} we have to ensure that the preliminary estimate \(\hat \gseqi^{\scriptscriptstyle (pr)}\) in Step 1 is consistent in order to guarantee consistency of the final estimate \(\hat \gseqi\). If not mentioned otherwise, we always make the ``arbitrary'' choice \(\hat \gseqi^{\scriptscriptstyle (pr)} = 0.1\)  for two reasons: First, a simulation study which is not included in this paper, where the estimation procedure is started in Step 2 with the choice \( \hat\thrle^{\scriptscriptstyle (in)} = \sqrt[3]{n\Delta_n}\) (which yields consistency according to Theorem \ref{SchaezerKonvT}) has shown similar results as the ones depicted below. Secondly, with the small choice of \(\hat \gseqi^{\scriptscriptstyle (pr)} = 0.1\) in Step 1 we obtain smaller values of the thresholds \( \hat\thrle^{\scriptscriptstyle (in)}\), \( \hat\thrle^{\scriptscriptstyle (fi)}\) and this reduces the calculation time. Furthermore, in the following simulation study we choose 
for the sample size $n=22500$ and vary the effective sample size $k_n = n \Delta_n$ in $ \{50,100,250\}$. For the evaluation of \eqref{thrledfdefeq} we always use $\alpha = 10\%$ and for computational reasons suprema in $\dfi \in \R$ are approximated by maxima over $\dfi \in T_2 = \{0.1 + j \cdot 0.3 \mid j=0,1,\ldots,9\}$. If not mentioned otherwise, we simulate a linear change, i.e. $w =1$ in \eqref{etagrch}, which is located at $\gseqi_0 = 0.4$. $A$ is always chosen such that the characteristic quantity for a gradual change satisfies $\Dc_\rho^{\scriptscriptstyle (g)}(1) = 3$ in all scenarios.

\begin{figure}[t!]
	\centering
	\includegraphics[width=0.33\textwidth]{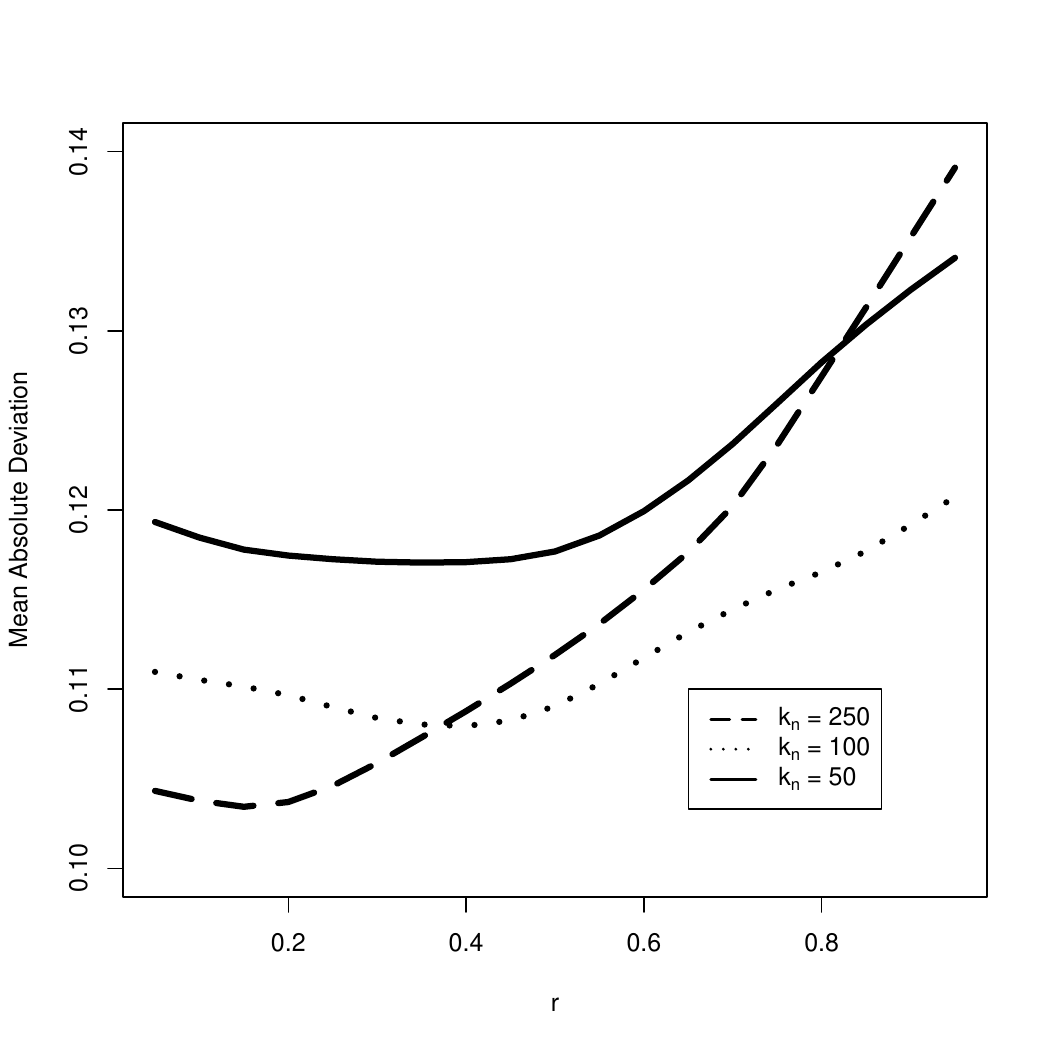}~~~ \includegraphics[width=0.33\textwidth]{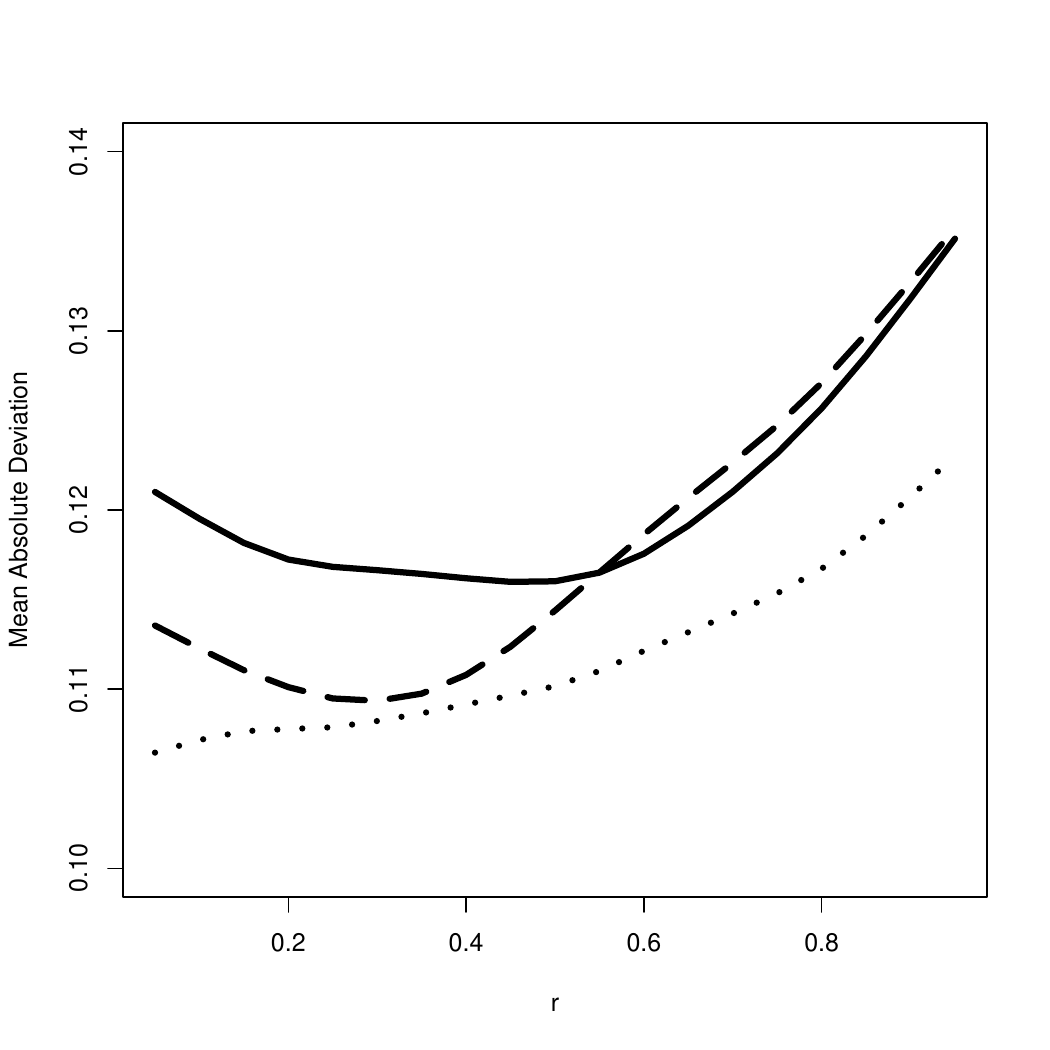}
	\includegraphics[width=0.33\textwidth]{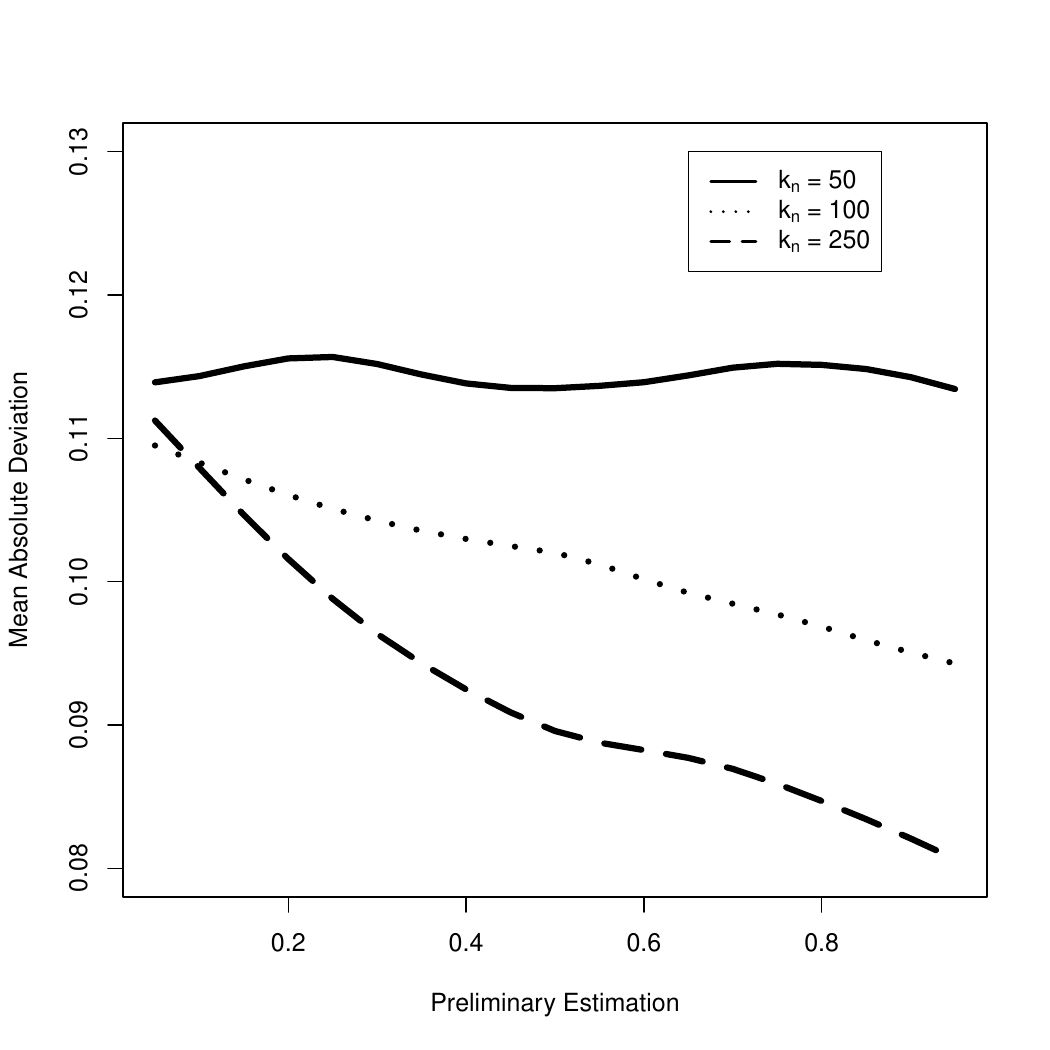}~~~ \includegraphics[width=0.33\textwidth]{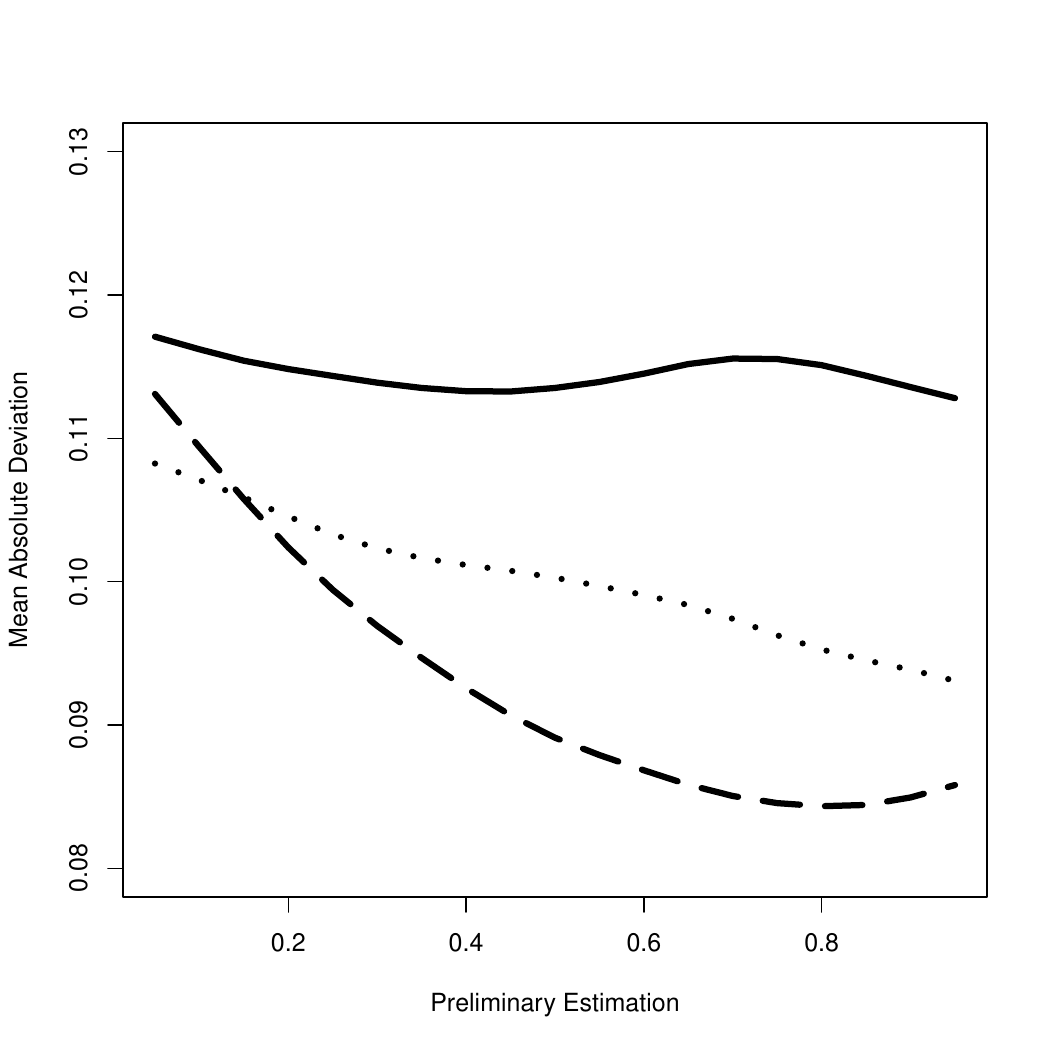}
	\vspace{-.5cm}
	\caption{\label{fig:grschrDu} 
		\it Mean absolute deviation of  the estimator \eqref{grestdef}. Upper part:   different choices of $r \in (0,1]$ in Step 1.
		Lower part:  different choices of the preliminary estimate \(\hat \gseqi^{\scriptscriptstyle (pr)} \in (0,1)\) in Step 1. Left panels:
		pure jump It\=o semimartingales. Right panels: pure jump It\=o semimartingales plus an additional Brownian component.	 
	}
\end{figure}

The upper part of 
Figure \ref{fig:grschrDu} shows  the mean absolute deviation  of the estimator \eqref{grestdef} 
 for different choices of $r \in (0,1]$ in Step 1. 
We observe  that in all cases the mean absolute deviation for $r=0.3$ is close to its overall minimum. Thus, we choose $r=0.3$ in Step 1 in all following investigations. In the lower part  we display the  mean absolute deviation for different choices  of the preliminary estimate \(\hat \gseqi^{\scriptscriptstyle (pr)} \in (0,1)\) in Step 1. The smallest error is obtained, if the preliminary estimate is chosen close to $1$. These findings were confirmed by  a further simulation study which is not
presented here and demonstrates that the  procedure \eqref{grestdef} tends to underestimate the change point. As a consequence, \(\hat \gseqi^{\scriptscriptstyle (pr)}\) close to $1$ induces larger values of the quantities \( \hat\thrle^{\scriptscriptstyle (in)}, \hat \gseqi^{\scriptscriptstyle (in)}, \hat\thrle^{\scriptscriptstyle (fi)}\) in Steps 2-4 and prevents the underestimation error.

\begin{figure}[t!]
	\centering
	\includegraphics[width=0.33\textwidth]{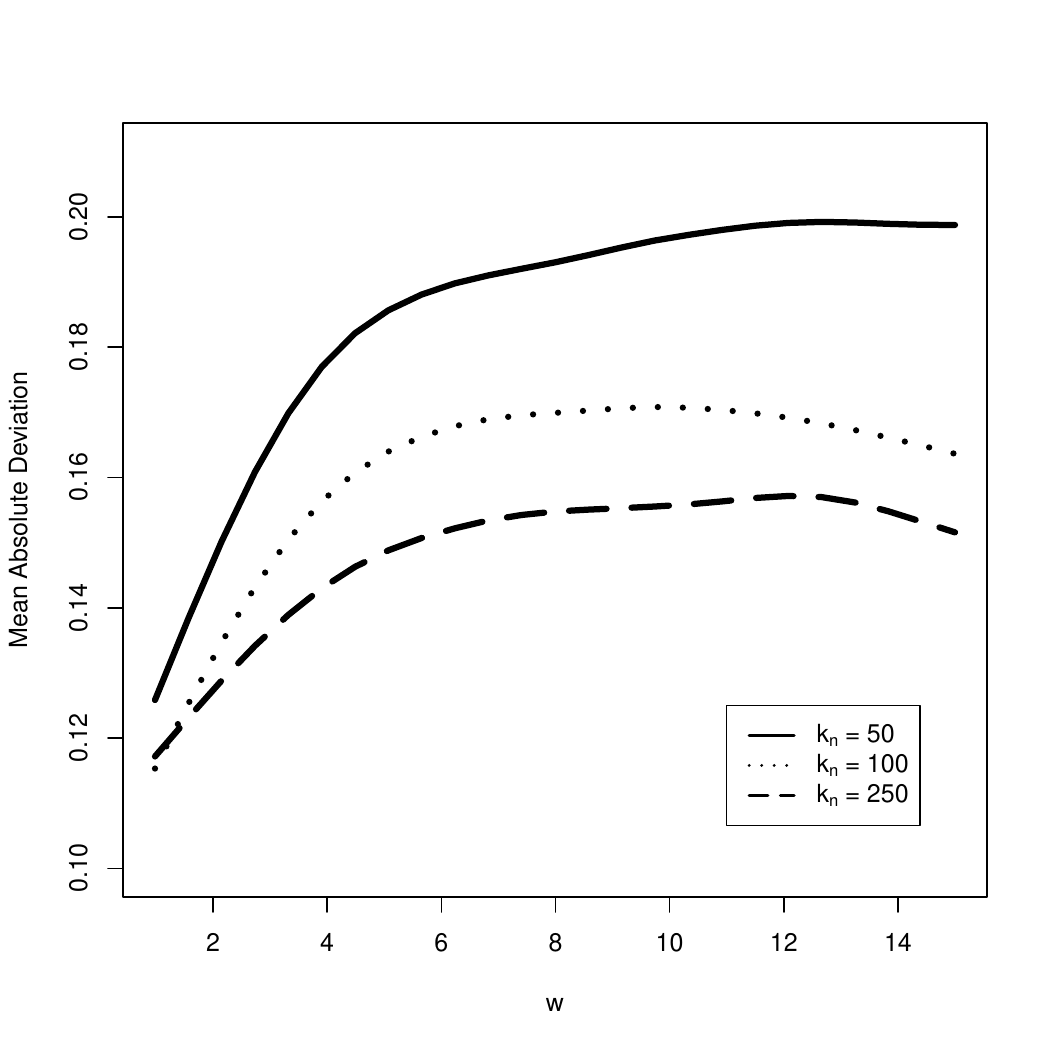}~~~ \includegraphics[width=0.33\textwidth]{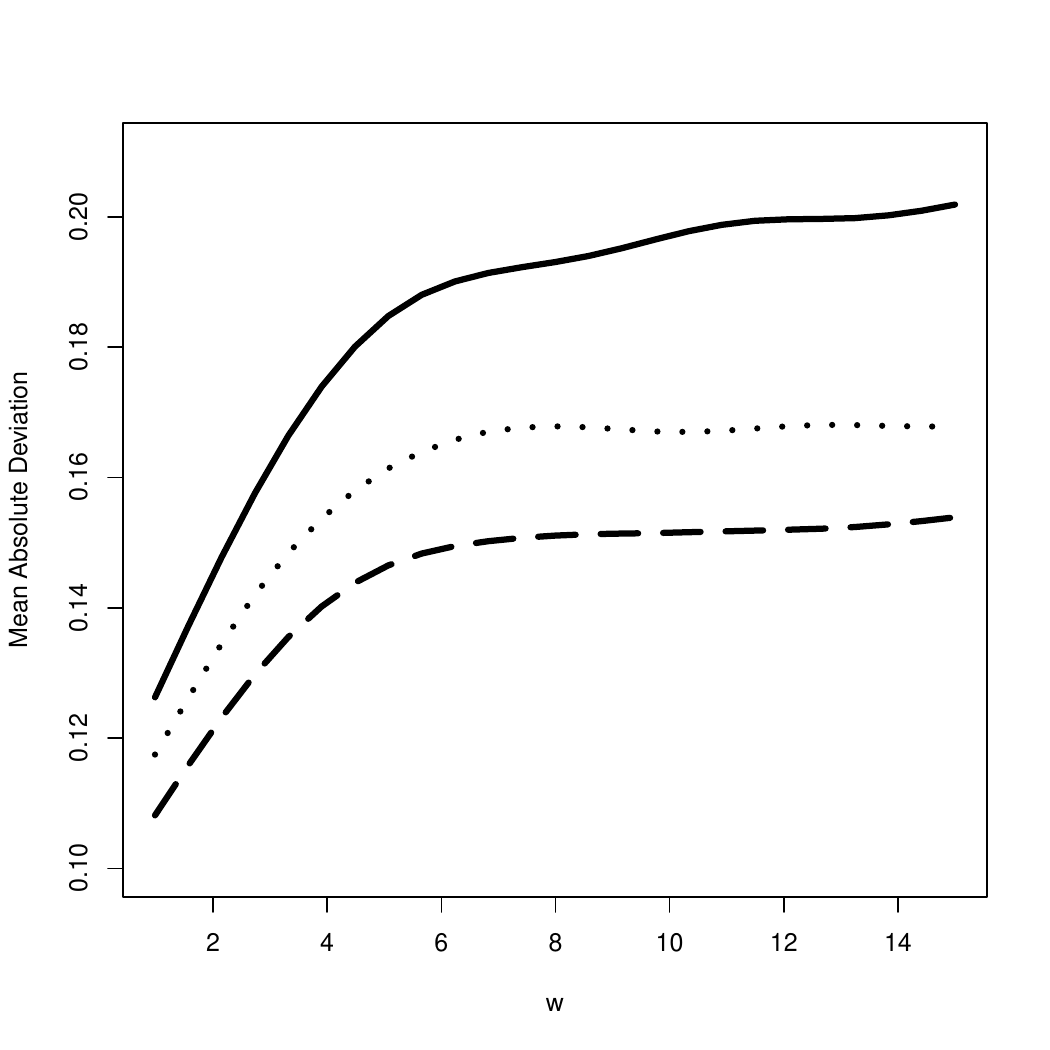}
	\includegraphics[width=0.33\textwidth]{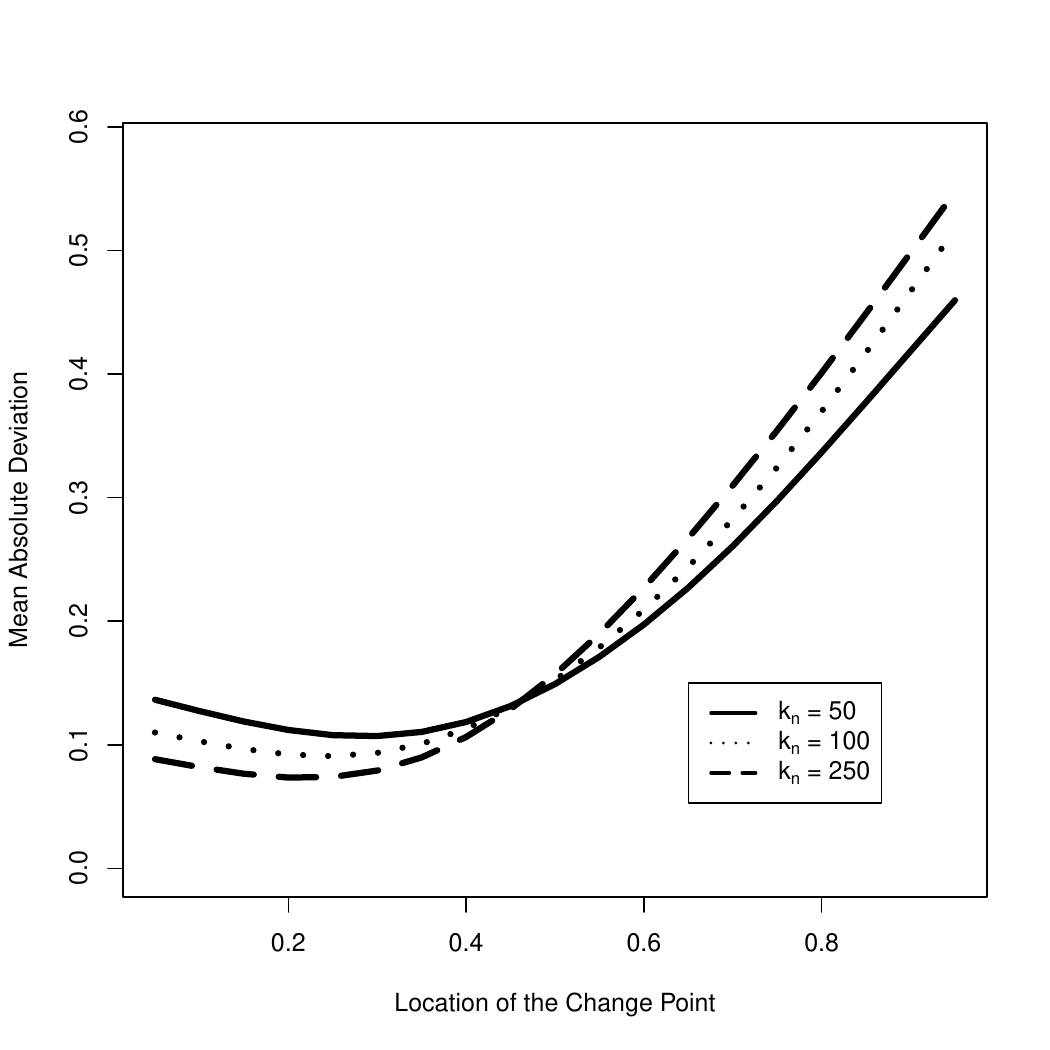}~~~ \includegraphics[width=0.33\textwidth]{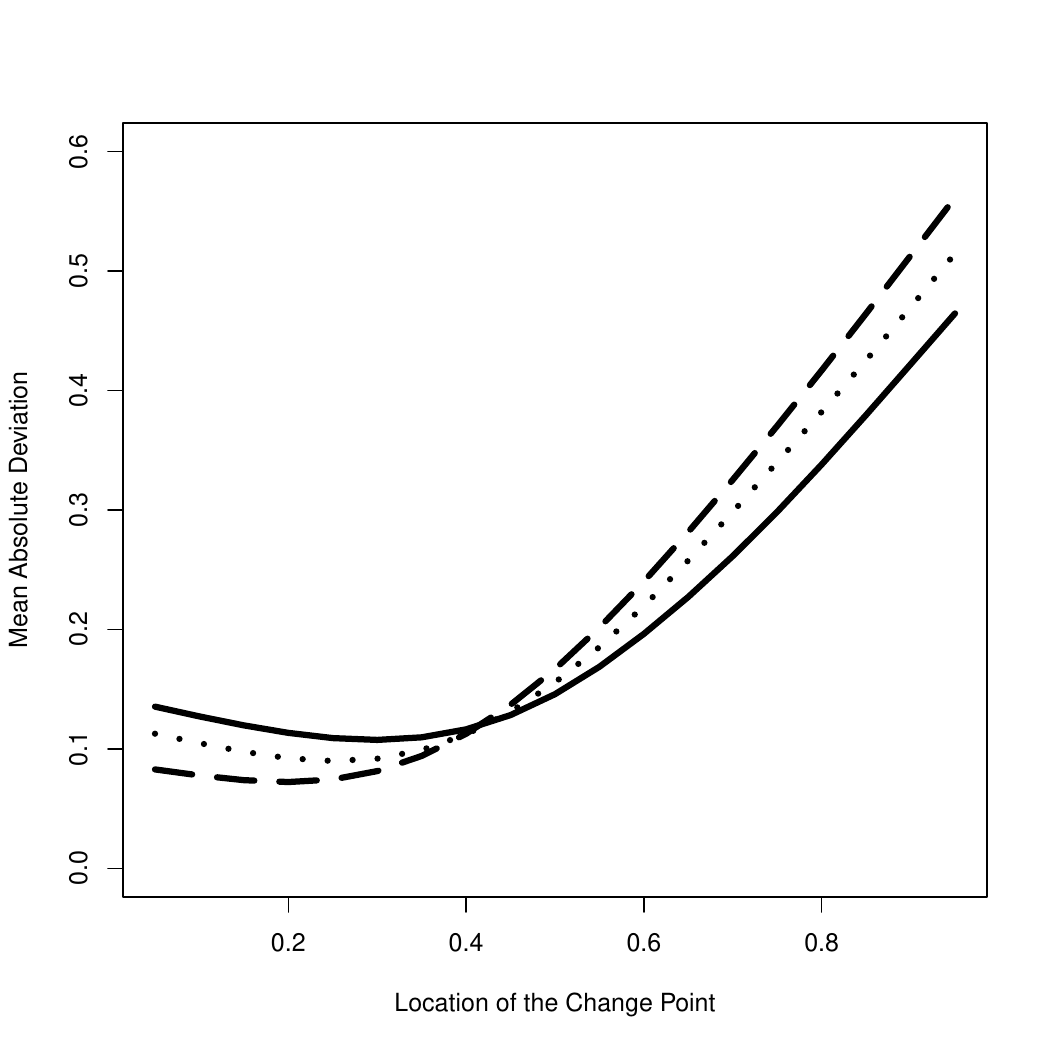}
	\vspace{-.5cm}
	\caption{\label{fig:grschsDu} 
		\it Mean absolute deviation of the estimator \eqref{grestdef}. Upper part:  different degrees of smoothness of the change $w$ in \eqref{etagrch}.
		Lower part: different locations of the change point. Left panels:
		  pure jump processes. Right panels:  pure jump processes plus an additional additional Brownian motion with drift 
	}
\end{figure}

The upper part of 
Figure \ref{fig:grschsDu} shows the simulated mean absolute deviation of the estimator \eqref{grestdef} for different degrees of smoothness of the change $w$ in \eqref{etagrch}. The results correspond to the upper part of  Figure \ref{fig:grtwDu} and confirm the intuitive idea that a smooth change is more difficult to detect. Moreover, larger effective sample sizes $k_n = n \Delta_n$ reduce the estimation error.  In the lower part 
we display the simulated mean absolute deviation of the estimator $\hat \gseqi_\rho^{\scriptscriptstyle (n)}$ for different locations of the change point $\gseqi_0 \in (0,1)$ in \eqref{etagrch}. The results correspond to lower  part of  Figure \ref{fig:grtwDu} and show that for small values of $\gseqi_0$ the change point can be detected best. This is a consequence of model \eqref{etagrch}, where for larger values of   $\gseqi_0 \in (0,1)$ the jump behaviour is nearly constant.

\subsection{Real data application}
\label{RDApp}

\begin{figure}[t!]
	\centering
	\includegraphics[width=0.49\textwidth]{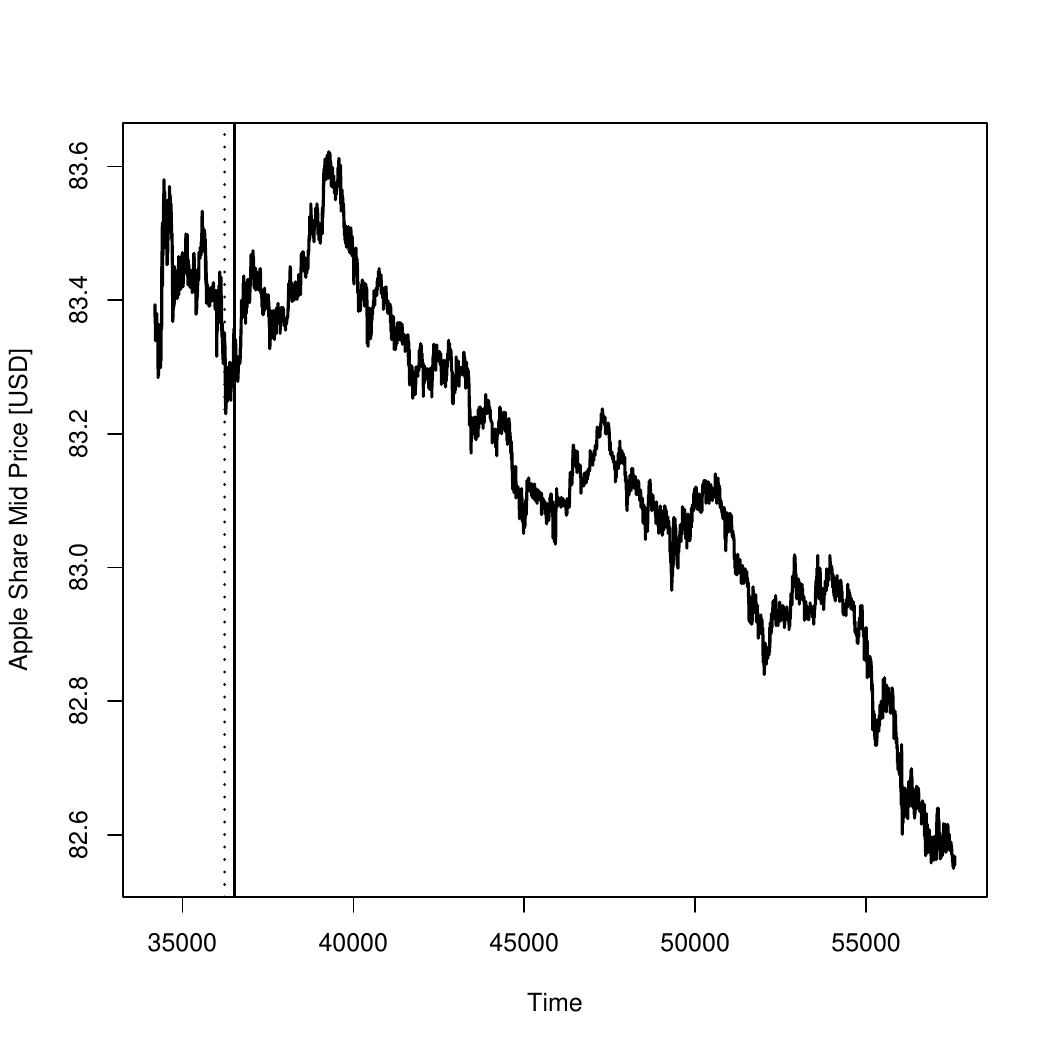}~~~ 
	\vspace{-.5cm}
	\caption{\label{apamappl} 
		\it Mid prices of Apple  shares in US dollar between 09:30 and 16:00 on 21-06-2012. The time is measured on the x-axis by seconds after midnight. The solid vertical line shows the result of the argmax-estimator \eqref{argmaxsup} (abrupt change), while the dashed vertical line indicates the result of an application of estimator \eqref{grestdef} (gradual change).
		}
\end{figure}

In this section we show the results of an application of the new methods to mid price data (in US dollar) of Apple shares between 09:30 and 16:00 on 21-06-2012, which is depicted in  Figure \ref{apamappl} and consists of $n=106626$ data points. 
We choose $k_n = n\Delta_n = 23400$, which corresponds to  the number of seconds between 09:30 and 16:00. Furthermore, we use again the function $\rho_{L,p}$ from \eqref{Eq:rhoLpdef} in Example \ref{Ex:SitgraCh} with parameters $L=1$ and $p=2$. For the truncation sequence we choose $v_n=\gamma (k_n/n)^{3/4}$, where we use $\gamma = 0.005$
 to address the fact that the increments $\Deli X$  in the data are very small, due to the extremely high frequency of sampling.
For the same reason we approximate the supremum in $\dfi \in \R$ in the methods from Section \ref{sec:Infabchdf} by the maximum over the finite grid $\{j\cdot 0.0005 \mid j \in \{1,\ldots,80\}\}$, while the supremum in $\dfi \in \R$ in the methods from Section \ref{sec:gradchadf} is approximated by the maximum over the finite grid $\{j\cdot 0.004 \mid j \in \{1,\ldots,10\}\}$. As in Section \ref{MCSim} we use $200$ bootstrap replications whenever a procedure requires resampling. 

Test \eqref{testvfkglob} and test \eqref{testdfglobal} reject the null hypothesis of no change in the jump behaviour with confidence level $\alpha = 5\%$. In order to locate the abrupt change point an application of the argmax-estimator \eqref{argmaxsup} results in the solid vertical lines in  Figure \ref{apamappl}. The dashed vertical line in Figure \ref{apamappl} is obtained by the same 5-step-procedure for estimator \eqref{grestdef} introduced in Section \ref{sec:simgracha} where we also choose \(\hat \gseqi^{\scriptscriptstyle (pr)} = 0.1\) and $\alpha = 0.1$ in step 1. Due to the huge sample size we use $r=0.8$ in order to reduce the calculation time.

\section{Proofs}
\label{sec:weConv}
\def\theequation{6.\arabic{equation}}
\setcounter{equation}{0}

The proofs of the results in this paper are technically very  demanding  and  we decompose the arguments in several parts. The main steps are given in this section.
We begin stating  general assumptions in Section \ref{altass}
 which are sufficient for all  results presented in  this paper and implied by the more readable  assumptions made in Section \ref{sec:Asss}.  In Section \ref{sec62} we state results regarding the
weak convergence of two empirical processes, which are used in the definition of the statistics considered in 
 Section \ref{sec:Infabchdf} and  \ref{sec:gradchadf}. Proofs for the results in these sections can be found in Section 
  \ref{sec63}    and    \ref{sec64}. All arguments presented  here rely on several technical  auxiliary results, which can be found  in Appendix \ref{ssecProTh61} - \ref{appF}
  of the supplement. 
  
\subsection{Alternative Assumptions}  \label{altass}

All results in this paper also hold under the weaker assumptions given below. Here and throughout
this section $K$ or $K(\delta)$ denote generic constants depending in some cases on a  quantity $\delta$
 and
 may change from place to place.

\begin{assumption}
	\label{Cond1} 
	At step $n\in\N$ we observe an It\=o semimartingale $X^{\scriptscriptstyle (n)}$ adapted to the filtration of some filtered probability space $(\Omega,\Fc,(\Fc_t)_{t\in\R_+},\Prob)$ with characteristics  $(b_s^{\scriptscriptstyle (n)},\sigma_s^{\scriptscriptstyle (n)},\nu_s^{\scriptscriptstyle(n)})$ at the equidistant time points $\{i\Delta_n \mid i=0,1,\ldots,n\}$. Furthermore, the following assumptions are satisfied:
	\begin{compactenum}[(a)]
		\item \textit{Assumptions on the jump characteristic and the function $\rho$:} \\
		For each $n\in\N$ and $s \in [0,n\Delta_n]$ we have
		\begin{align}
		\label{RescAss3}
		\nu^{(n)}_{s}(dz) = g^{(n)}\Big(\frac{s}{n\Delta_n}, dz\Big),
		\end{align}
		where there exist transition kernels $g_0,g_1,g_2$ from $([0,1],\Bb([0,1]))$ into $(\R,\Bb)$ such that for each $y \in [0,1]$
		\begin{equation}
		\label{gon012Ass3}
		g^{(n)}(y,d\taili) = g_0(y,d\taili) + \frac 1{\sqrt{n\Delta_n}} g_1(y,d\taili) + \Rc_n(y,d\taili)
		\end{equation}
		and for each $y \in [0,1]$, $B\in\Bb$ and $n\in\N$ the kernel $\Rc_n$ satisfies
		$
		\Rc_n(y,B) \leq a_n g_2(y,B)
		$
		for a sequence $a_n = o((n\Delta_n)^{-1/2})$ of non-negative real numbers. Furthermore, we have
		\begin{enumerate}[(1)]
			\item \label{BlGetCond} There exists $\beta \in [0,2]$ with 
			\[
			\max_{i=0,1,2}\Big(\lambda_1 - \mathrm{ess~sup}_{y \in [0,1]} \Big( \int \big ( 1 \wedge |\taili|^{(\beta+\delta)\wedge 2} \big ) g_i(y,d\taili)\Big) \Big) \leq K(\delta) < \infty
			\]
			for each $\delta >0$.
			\item \label{RhoCond}
			$\rho \colon \R \rightarrow \R$ is a bounded $\mathcal C^1$-function with $\rho(0)=0$. Furthermore, there exists some $p > \beta+(\beta \vee 1)$
			such that the derivative satisfies $| \rho^{\prime}(\taili) | \leq K |\taili|^{p-1}$ for all $\taili \in \R$ and some $K>0$. 
			\item For $\ovp = (p-1) \vee 1$ with $p$ from \eqref{RhoCond} we have
			\[ \max_{i=0,1,2}\Big( \lambda_1 - \mathrm{ess~sup}_{y \in [0,1]}\Big(\int |\taili|^{\ovp} \ind_{\lbrace |\taili| \geq 1 \rbrace} g_i(y,d\taili)\Big)\Big)  < \infty. \]
			\label{LevMeaMomCond}
			\item \begin{compactenum}[(I)]
				\item There exist $\ovr > \ovv >0$, $\alpha_0 >0$, $q >0$ and $K>0$ such that for every choice $m_1,m_2 \in \{g_0,g_1,g_2\}$
				\label{FiLevyDistCond}
				\begin{multline*}
				\lambda_2 - \mathrm{ess~sup}_{y_1,y_2 \in [0,1]}\Big(\int \int \ind_{\lbrace |x-\taili| \leq \Delta_n^{\ovr} \rbrace} \ind_{\lbrace \Delta_n^{\ovv}/2 < |x| \leq 
					\alpha_0 \rbrace } \times \\ \times \ind_{\lbrace \Delta_n^{\ovv}/2 < |\taili| \leq \alpha_0 \rbrace} m_1(y_1,dx) m_2(y_2,d\taili)\Big)                \leq K \Delta_n^q,
				\end{multline*}
				holds for $n \in \mathbb N$ sufficiently large, where $\lambda_2$ denotes the restriction of the two-dimensional Lebesgue measure to the measure space $([0,1]^2,[0,1]^2 \cap \Lc_2)$ with the two-dimensional Lebesgue $\sigma$-algebra $\Lc_2$ on $\R^2$.
				\item For each $\alpha >0$ there is a $K(\alpha)>0$ such that for every choice $m_1,m_2 \in \{g_0,g_1,g_2\}$ we have
				\begin{multline*}
				\lambda_2 - \mathrm{ess~sup}_{y_1,y_2 \in[0,1]} \Big(\int \int \ind_{\lbrace |x-\taili| \leq \Delta_n^{\ovr} \rbrace} \ind_{\lbrace |x| > \alpha \rbrace} \times \\ \times
				\ind_{\lbrace |\taili| > \alpha \rbrace} m_1(y_1,dx) m_2(y_2,d\taili) \Big) \leq K(\alpha) \Delta_n^q,
				\end{multline*}
				for $n \in \N$ large enough with the constants from \eqref{FiLevyDistCond}.
				\label{SeLevyDistCond}
			\end{compactenum}
		\end{enumerate}
		\item \textit{Assumptions on the truncation sequence $v_n$ and the observation scheme:} 
		\label{ObsSchCondit}
		We have $v_n = \gamma \Delta_n^{\ovw}$ for some $\gamma >0$ and $\ovw$ satisfying
		$ \frac{1}{2(p-\beta)} < \ovw < \frac{1}{2} \wedge \frac{1}{2\beta}. $
		Furthermore, the observation scheme satisfies with the constants from the previous assumptions: 
		\begin{enumerate}[(1)]
			\item $\Delta_n \rightarrow 0,$
			
			\item $n \Delta_n \rightarrow \infty,$
			
			\item \label{ObsSchCond3} $n \Delta_n^{1+ q/2} \rightarrow 0,$
			\item \label{ObsSchCond4} $n \Delta_n^{1+2 \ovw} \rightarrow 0,$
			\item \label{ObsSchCond5b} $n \Delta_n^{1+2\ovv(p-\beta-\delta)} \rightarrow 0$ for some $\delta >0$,
			\item \label{ObsSchCond6} $n \Delta_n^{2(1-\beta \ovw (1+ \epsilon))} \rightarrow 0$ for some $\epsilon >0$,
			\item \label{ObsSchCond7} $n \Delta_n^{((1+ 2(\ovr - \ovw)) \vee 1) + \delta} \rightarrow \infty$ for some $\delta >0$.
		\end{enumerate}
		\item \label{DrianDiffCond} \textit{Assumptions on the drift and the diffusion coefficient:} 
		For
		$ m_b = \frac{1+2\ovw}{1-\ovw} \leq 4  $ and $  m_\sigma = \frac{1+ 2 \ovw}{1/2 - \ovw}, $ 
		we have
		\[\sup\limits_{n\in\N}\sup \limits_{s \in \R_+} \Big\{ \Eb \big|b^{(n)}_s\big|^{m_b} \vee \Eb \big|\sigma^{(n)}_s\big|^{m_\sigma} \Big\} < \infty.\]
	\end{compactenum}
\end{assumption}

Throughout the following proofs we will work with Assumption \ref{Cond1} without further mention. This is due to the following result which proves that Assumption \ref{EasierCond} implies the set of conditions above. 

\begin{proposition} \label{easier}
	Assumption \ref{EasierCond} is sufficient for Assumption \ref{Cond1}.
\end{proposition}

\begin{proof}
	Let $0< \beta < 2$, $0<  \tau < (1/5 \wedge \frac{2-\beta}{2+5\beta})$ and $p > \beta+((\frac 12 + \frac 32 \beta) \vee  \frac{2}{1+5\tau})$ and suppose that Assumption \ref{EasierCond} is satisfied for these constants. In order to verify Assumption \ref{Cond1} define the following quantities:
	\begin{align}
	\label{KonstDefEq}
	\ovr := 3 \tau, \quad \ovv := \frac{\tau}{1+3 \beta}, \quad q := \ovr - (1+ 3 \beta) \ovv = 2 \tau,
	\end{align}
	and recall that $\ovw = (1+5\tau)/4$.
	
	$\rho$ is suitable for Assumption \ref{Cond1}\eqref{RhoCond}, as in particular $p > \beta + (\beta \vee 1)$ is satisfied due to $(1+3\beta)/2 >\beta$ and $2/(1+5\tau) > 1$. Assumption \ref{Cond1}\eqref{ObsSchCondit} is established, since $1/(2(p-\beta)) < \ovw = (1+5\tau)/4$ is equivalent to $p> \beta+(2/(1+5 \tau))$ and $\ovw = (1+5\tau)/4 < 1/2 \wedge 1/(2\beta)$ holds due to $\tau < (1/5 \wedge \frac{2-\beta}{2+5\beta})$. Furthermore, simple calculations show
	\begin{align}
	\label{OrdnVergl}
	(1 + 2 \ovr - 2 \ovw) \vee 1  &= t_2^{-1} < 1+ \tau = t_1^{-1} = (1+ \frac{q}{2}) \nonumber \\
	&= 2(1-  \beta \ovw (1 + \epsilon)) < (1+ 2 \ovv(p-\beta)) \wedge (1+2\ovw)
	\end{align}
	with $\epsilon = \frac{2-2\tau-\beta(1+5\tau)}{\beta(1+5\tau)}>0$, since $\tau <  \frac{2-\beta}{2+5\beta}$ and $(p-\beta)>(1+3\beta)/2$. Therefore, all conditions on the observation scheme are satisfied. 
	
	Additionally, if $\eta,M >0$ and a Lebesgue null set $L \in [0,1] \cap \Lc_1$ are chosen such that the requirements of Definition \ref{rhoandgass2} hold, we have 
	$ h_y^{(i)}(\taili) |\taili|^{(\beta+\delta)\wedge 2} \leq K|\taili|^{(-1 + \delta) \wedge (1-\beta)} $
	for each $\delta >0$ and all $y \in [0,1]\setminus L$, $\taili \in (-\eta,\eta)$, $i \in \{0,1,2\}$, where $h_y^{\scriptscriptstyle (i)}$ denotes a density for the kernel $g_i$. Therefore, and due to Definition \ref{rhoandgass2}\eqref{DiCondmiddle} and \eqref{DiCondinfty}, we obtain $\lambda_1 - \text{ess sup} \big(\int \big ( 1 \wedge |\taili|^{(\beta+ \delta)\wedge 2} \big ) g_i(y,d\taili) \big) \leq K(\delta) < \infty$ for every $\delta >0$ and all $i \in\{0,1,2\}$. Moreover, due to Definition \ref{rhoandgass2}\eqref{DiCondinfty} we have
	\begin{equation}
	\label{1mipgezEq}
	h_y^{(i)}(\taili) |\taili|^{\ovp} \leq K|\taili|^{-1-\epsilon},
	\end{equation}
	for all $|\taili| \ge M$, $y\in[0,1]\setminus L$, $i \in \{0,1,2\}$ and some $K>0$. So together with Definition \ref{rhoandgass2}\eqref{DiCondmiddle} we obtain $\lambda_1 - \text{ess sup}_{y \in [0,1]} \big( \int |\taili|^{\ovp} \ind_{\lbrace |\taili| \geq 1 \rbrace} g_i(y,d\taili) \big) < \infty$ for each $i \in \{0,1,2\}$ which is Assumption \ref{Cond1}\eqref{LevMeaMomCond}. 
	
	Furthermore it follows that 
	$ \frac{1+2\ovw}{1-\ovw} = \frac{6+10\tau}{3-5\tau} $ and $\frac{1+2\ovw}{1/2-\ovw}= \frac{6+10\tau}{1-5\tau},$
		and as consequence  Assumption \ref{EasierCond}\eqref{DriDiffMomCond} implies  Assumption \ref{Cond1}\eqref{DrianDiffCond}. 
	
	We are thus left with proving Assumption \ref{Cond1}\eqref{FiLevyDistCond} and \eqref{SeLevyDistCond}. Obviously, $0 < \ovv  < \ovr$ holds with the choice in \eqref{KonstDefEq}. First, we verify Assumption \ref{Cond1}\eqref{FiLevyDistCond}. To this end, we choose $\eta>0$ and a Lebesgue null set $L \in [0,1] \cap \Lc_1$ such that 
	$ h_y^{(i)}(\taili) \leq K |\taili|^{-(1+\beta)} $
	holds for all $\taili \in (- \eta, \eta) \setminus \lbrace 0 \rbrace$, $y \in [0,1] \setminus L$, $i\in\{0,1,2\}$ according to Definition \ref{rhoandgass2}\eqref{BGn0Ass} and we set $\alpha_0 \defeq \eta/2$. Then for any choice $m_1,m_2 \in \{g_0,g_1,g_2\}$ we get 
	\begin{align*}
	\int &\int \ind_{\lbrace |x-z| \leq \Delta_n^{\ovr} \rbrace} \ind_{\lbrace \Delta_n^{\ovv} /2 < |x| \leq \alpha_0 \rbrace } 
	\ind_{\lbrace \Delta_n^{\ovv} /2 < |z| \leq \alpha_0 \rbrace} m_1(y_1,dx) m_2(y_2,dz)  \\
	&\leq K \int \int \ind_{\lbrace |x-z| \leq \Delta_n^{\ovr} \rbrace} \ind_{\lbrace \Delta_n^{\ovv} /2 < |x| \leq \alpha_0 \rbrace } 
	\ind_{\lbrace \Delta_n^{\ovv} /2 < |z| \leq \alpha_0 \rbrace} |x|^{-(1+\beta)} |z|^{-(1+\beta)} dx dz  \\
	&\leq 2 K\int \limits_0^{\infty} \int_0^{\infty} \ind_{\lbrace |x-z| \leq \Delta_n^{\ovr} \rbrace} 
	\ind_{\lbrace \Delta_n^{\ovv}/2 < x \leq \alpha_0 \rbrace } \ind_{\lbrace \Delta_n^{\ovv}/2 < z \leq \alpha_0 
		\rbrace} x^{-(1+\beta)} z^{-(1+\beta)} dx dz. 
	\end{align*}
	for all $(y_1,y_2) \in ([0,1]\setminus L) \times ([0,1]\setminus L)$ and $n \in \N$ large enough. For the second inequality we have used symmetry of the integrand as well as $\Delta_n^{\ovr} < \Delta_n^{\ovv}/2$. In the following, we ignore the extra condition on $x$. Evaluation of the integral with respect to $x$ plus a Taylor expansion give the further upper bounds
	\begin{align*}
	&K \int_0^{\infty} \frac{|(z- \Delta_n^{\ovr})^{\beta} - (z+ \Delta_n^{\ovr})^{\beta}|}
	{|z^2 - \Delta_n^{2 \ovr} |^{\beta}} 
	z^{-(1+ \beta)} \ind_{\lbrace \Delta_n^{\ovv}/2 < z \leq \alpha_0 \rbrace} dz \nonumber \\
	&\leq K \Delta_n^{\ovr} \int_0^{\infty} \frac{\xi(z)^{\beta-1}}{|z^2 - \Delta_n^{2 \ovr}|^{\beta}} 
	z^{-(1+ \beta)} \ind_{\lbrace \Delta_n^{\ovv}/2 < z \leq \alpha_0 \rbrace} dz
	\end{align*}
	for some $\xi(z) \in [z- \Delta_n^{\ovr}, z + \Delta_n^{\ovr}]$. Finally, we distinguish the cases $\beta < 1$ and $\beta \geq 1$ for which the numerator has to be treated differently, depending on whether it is bounded or not. The denominator is always smallest if we plug in $\Delta_n^{\ovv}/2$ for $z$. Overall,
	\begin{align*}
	\int &\int \ind_{\lbrace |x-z| \leq \Delta_n^{\ovr} \rbrace} \ind_{\lbrace \Delta_n^{\ovv} /2 < |x| \leq \alpha_0 \rbrace } 
	\ind_{\lbrace \Delta_n^{\ovv} /2 < |z| \leq \alpha_0 \rbrace} m_1(y_1,dx) m_2(y_2,dz) \\
	&\leq \begin{cases}
	K \Delta_n^{\ovr} \Delta_n^{-(1+\beta) \ovv} \int_{\Delta_n^{\ovv}/2}^{\alpha_0} 
	z^{-(1+ \beta)} dz, \quad &\text{ if } \beta <1 \\
	K \Delta_n^{\ovr} \Delta_n^{-2\beta \ovv} \int_{\Delta_n^{\ovv}/2}^{\alpha_0} 
	z^{-(1+ \beta)} dz, \quad &\text{ if } \beta \geq 1
	\end{cases} \\
	&\leq K \Delta_n^{\ovr - (1+ 3 \beta)\ovv} = K \Delta_n^q \nonumber
	\end{align*}
	for all $m_1,m_2 \in \{0,1,2\}$ and $(y_1,y_2) \in [0,1]^2 \setminus L^2$. Finally, we consider Assumption \ref{Cond1}\eqref{SeLevyDistCond}, for which we proceed similarly with $n \in \N$ large enough, $\alpha >0$ and $(y_1,y_2)\in [0,1]^2\setminus L^2$, as well as $m_1,m_2 \in \{g_0,g_1,g_2\}$ arbitrary:
	\begin{align*}
	\int &\int \ind_{\lbrace |x-z| \leq \Delta_n^{\ovr} \rbrace} \ind_{\lbrace |x| > \alpha \rbrace } 
	\ind_{\lbrace  |z| > \alpha \rbrace} m_1(y_1,dx) m_2(y_2,dz)    \\
	&\leq O(\Delta_n^{\ovr}) + 2 K \int_{M^{\prime}}^{\infty} \int_{M^{\prime}}^{\infty} 
	\ind_{\lbrace |x-z| \leq \Delta_n^{\ovr} 
		\rbrace} \ind_{\lbrace x > \alpha \rbrace } \ind_{\lbrace z > \alpha \rbrace} 
	x^{-2} z^{-2} dx dz.
	\end{align*}
	This inequality holds with a suitable $M^{\prime} >0$ due to Definition \ref{rhoandgass2} \eqref{DiCondmiddle} and \eqref{DiCondinfty}, as we have $h_y^{\scriptscriptstyle (i)}(\taili) \leq K|\taili|^{-2}$ for $y \in [0,1] \setminus L$, $i \in \{0,1,2\}$ and large $|\taili|$. Therefore,
	\begin{align}
	\label{IntalphaAbsch}
	\int &\int \ind_{\lbrace |x-z| \leq \Delta_n^{\ovr} \rbrace} \ind_{\lbrace |x| > \alpha \rbrace } 
	\ind_{\lbrace  |z| > \alpha \rbrace} m_1(y_1,dx) m_2(y_2,dz)  \nonumber \\
	&\leq O(\Delta_n^{\ovr}) + K \Delta_n^{\ovr} \int_{M^{\prime}}^{\infty} \frac{1} {|z^2 - \Delta_n^{2 \ovr} |} z^{-2} 
	\ind_{\lbrace z > \alpha \rbrace} dz = o(\Delta_n^q) 
	\end{align}
	for $(y_1,y_2) \in [0,1]^2 \setminus L^2$ and any choice $m_1,m_2 \in \{g_0,g_1,g_2\}$.	The final bound in \eqref{IntalphaAbsch} holds since the last integral is finite.
\end{proof}

\subsection{Weak convergence of the empirical truncated L\'evy distribution function} \label{sec62}

The proofs of the  statements   in  Section \ref{sec:Infabchdf} and Section \ref{sec:gradchadf} rely on two deep  results about the weak convergence 
of empirical processes which are the basic blocks in the statistics considered there. 
We begin with  a central limit theorem for   the process
\[
G_{\rho}^{(n)}(\gseqi,\dfi) = \sqrt{n \Delta_n}\big(N_{\rho}^{(n)}(\gseqi,\dfi) - N_{\rho}(g^{(n)};\gseqi,\dfi)\big),
\]
where $N_\rho(\cdot,\cdot)$ and $N_\rho(g;\cdot,\cdot)$ are defined in \eqref{NrhoDef} and \eqref{NrhonDef}, respectively.
The following result is a generalization of  Theorem 3.1 in \cite{HofVet15} which   can be obtained
by the choice  $g_0(y,d\taili)=\nu(d\taili)$ for a L\'evy measure $\nu$ and $g_1=g_2=0$. 
The proof is given in Section \ref{ssecProTh61} of the supplement. 

\begin{theorem}
	\label{ConvThm}
	Let Assumption \ref{EasierCond} be satisfied. Then we have weak convergence
	$G_{\rho}^{(n)} \weak \Gb_{\rho}$ 
	in $\linner$, where $\Gb_{\rho}$ is a tight mean zero Gaussian process in $\linner$ with covariance function
	\begin{align}
	\label{Hrhodef}
	H_{\rho}((\gseqi_1,\dfi_1);(\gseqi_2,\dfi_2)) \defeq \int_0^{\mingi} \int_{- \infty}^{\mindfi} \rho^2(\taili) g_0(y,d\taili)dy.
	\end{align}
	Additionally, the sample paths of $\Gb_{\rho}$ are almost surely uniformly continuous with respect to the semimetric
	\begin{align}
	\label{drhodef}
	d_{\rho}((\gseqi_1,\dfi_1);(\gseqi_2,\dfi_2)) = \Big\{ \int_0^{\gseqi_1} \int_{\mindfi}^{\maxdfi} \rho^2(\taili) g_0(y,d\taili) dy + \int_{\gseqi_1}^{\gseqi_2} \int_{-\infty}^{\dfi_2} \rho^2(\taili) g_0(y,d\taili) dy \Big\}^{1/2}
	\end{align}
	for $\gseqi_1 \leq \gseqi_2$.
\end{theorem}

We also need a result regarding the weak convergence  of a   bootstrapped version of $G_\rho^{\scriptscriptstyle (n)}$. The corresponding process is defined
 by
\begin{align}
\label{BootGrnDefEq}
\hat G_\rho^{(n)}(\gseqi,\dfi) = \frac 1{\sqrt{n\Delta_n}} \sum\limits_{i=1}^{\ip{n\gseqi}} \xi_i \rho\big(\Deli X^{(n)}\big) \ind_{(-\infty,\dfi]}\big(\Deli X^{(n)}\big) \ind_{\{|\Deli X^{(n)} | > v_n\}},
\end{align}
for $(\gseqi,\dfi) \in \netir$, where the sequence of multipliers $(\xi_i)_{i \in \N}$ satisfies Assumption \ref{MultiplAss}. 
The proof is given  in Section \ref{ssecProTh68} of the supplement. 

\begin{theorem}
	\label{CondConvThm}
	If  Assumption \ref{EasierCond} holds  and  the multipliers $(\xi_i)_{i \in \N}$ satisfy Assumption \ref{MultiplAss},  we have
$	
	\hat G_\rho^{(n)} \weakP \Gb_\rho
$		in $\linner$, where  the process $\Gb_\rho$ is  defined in Theorem \ref{ConvThm}.
\end{theorem}

\subsection{Proofs of the results in Section \ref{sec:Infabchdf}} \label{sec63}  

\textbf{Proof of Theorem \ref{CUSUMTrhowc}.}
For each $(\gseqi,\dfi) \in \netir$, $n\in\N$ we have under ${\bf H}_1^{(loc)}$
\begin{align*}
\Tb_\rho^{(n)}(\gseqi,\dfi) &= h_n\big(G_\rho^{(n)}\big)(\gseqi,\dfi) + \sqrt{n\Delta_n} \big(N_\rho(g^{(n)};\gseqi,\dfi) - \frac{\ip{n\gseqi}}n N_\rho(g^{(n)};1,\dfi)\big) \\
&= h_n\big(G_\rho^{(n)}\big)(\gseqi,\dfi) + \sqrt{n\Delta_n} \Big(\gseqi - \frac{\ip{n\gseqi}}n \Big) \int_{-\infty}^{\dfi} \rho(\taili) \nu_0(d\taili) +    \\
&\hspace{5mm} + \big(N_\rho(g_1;\gseqi,\dfi) - \frac{\ip{n\gseqi}}n N_\rho(g_1;1,\dfi)\big) + \sqrt{n\Delta_n}  \big(N_\rho(\Rc_n;\gseqi,\dfi) - \frac{\ip{n\gseqi}}n N_\rho(\Rc_n;1,\dfi)\big),
\end{align*}
with the mappings $h_n : \linner \to \linner$ defined by
\begin{equation}
\label{hnFctD}
h_n(f)(\gseqi,\dfi) = f(\gseqi,\dfi) - \frac{\ip{n\gseqi}}n f(1,\dfi), \quad (n\in\N), \quad h_0(f)(\gseqi,\dfi) = f(\gseqi,\dfi) - \gseqi f(1,\dfi).
\end{equation}
Thus, by Assumption \ref{EasierCond}\eqref{rhoandgass} we obtain
$
\Tb_\rho^{(n)}(\gseqi,\dfi) = h_n\big(G_\rho^{(n)}\big)(\gseqi,\dfi) + \Tb_{\rho,g_1}(\gseqi,\dfi) + o(1),
$
where the $o$-term is deterministic. By the same reasoning as in the proof of Theorem 2.6 in \cite{BueHofVetDet15} it can be seen that $h_n\big(G_\rho^{\scriptscriptstyle (n)}\big) \weak h_0\big(\Gb_\rho\big) = \Tb_\rho$ in $\linner$. As a consequence, Slutsky's lemma (Example 1.4.7 in \cite{VanWel96}) yields the assertion, since the tight process $\Tb_\rho$ is separable (see Lemma 1.3.2 in the previously mentioned reference).
\qed

\medskip

\noindent 
\textbf{Proof of Proposition \ref{poiwconKSdi}.}
\eqref{1mipgezEq} in the proof of Proposition \ref{easier} shows that Assumption \ref{Cond1}\eqref{LevMeaMomCond} is also valid for $2p$ instead of $p$. Thus, Theorem \ref{ConvThm} also holds with the function $\rho$ replaced by $\rho^2$. As a consequence, we have $N_{\rho^2}^{\scriptscriptstyle (n)}(1,\dfi_0) - N_{\rho^2}(g^{\scriptscriptstyle (n)};1,\dfi_0) = o_\Prob(1)$. By \eqref{gon012Ass} we obtain
\begin{multline*}
N_{\rho^2}\big(g^{(n)};1,\dfi_0\big) = \int_{-\infty}^{\dfi_0} \rho^2(\taili) \nu_0(d\taili) + \frac 1{\sqrt{n\Delta_n}} \int_0^1 \int_{-\infty}^{\dfi_0} \rho^2(\taili) g_1(y,d\taili) dy +\\
+ \int_0^1 \int_{-\infty}^{\dfi_0} \rho^2(\taili) \Rc_n(y,d\taili) dy = \int_{-\infty}^{\dfi_0} \rho^2(\taili) \nu_0(d\taili) +o(1).
\end{multline*}
Finally, $(N_{\rho^2}^{\scriptscriptstyle (n)}(1,\dfi_0))^{-1/2} \ind_{\{ N_{\rho^2}^{\scriptscriptstyle (n)}(1,\dfi_0) >0 \}} = (\int_{-\infty}^{\dfi_0} \rho^2(\taili) \nu_0(d\taili))^{-1/2} + o_\Prob(1)$ follows due to $\int_{-\infty}^{\dfi_0} \rho^2(\taili)$ $ \nu_0(d\taili) >0$. Thus, Theorem \ref{CUSUMTrhowc}, the continuous mapping theorem and Slutsky's lemma (Example 1.4.7 in \cite{VanWel96}) yield
\[\Vb_{\rho,\dfi_0}^{(n)}(\gseqi) \weak \Big(\int_{-\infty}^{\dfi_0} \rho^2(\taili) \nu_0(d\taili)\Big)^{-1/2}\big( \Tb_\rho(\gseqi,\dfi_0) + \Tb_{\rho,g_1}(\gseqi,\dfi_0 \big) = \Kb(\gseqi) + \bar\Vb_{\rho,\dfi_0}^{(g_1)}(\gseqi), \]
in $\linne$, because the process $(\int_{-\infty}^{\dfi_0} \rho^2(\taili) \nu_0(d\taili))^{-1/2} \Tb_\rho(\cdot,\dfi_0)$ is a tight mean zero Gaussian process with covariance function $K(\gseqi_1,\gseqi_2) = (\gseqi_1 \wedge \gseqi_2) - \gseqi_1 \gseqi_2$.
\qed

\medskip

\noindent 
\textbf{Proof of Theorem \ref{BootTrhoThm}.}
Recall the Lipschitz continuous functions $h_n: \linner \to \linner$, $(n\in\N_0)$ defined in \eqref{hnFctD}. Then we have $\hat \Tb_\rho^{\scriptscriptstyle (n)} = h_n(\hat G_\rho^{\scriptscriptstyle (n)})$ and Proposition 10.7 in \cite{Kos08} together with Theorem \ref{CondConvThm} give $h_0(\hat G_\rho^{\scriptscriptstyle (n)}) \weakP h_0(\Gb_\rho)$ in $\linner$. Moreover, we have
\begin{multline*}
\sup_{(\gseqi,\dfi) \in \netir} \big| h_n(\hat G_\rho^{(n)})(\gseqi,\dfi) - h_0(\hat G_\rho^{(n)})(\gseqi,\dfi) \big| = \\ = \sup_{(\gseqi,\dfi) \in \netir} \big| \big(\gseqi - \frac{\ip{n\gseqi}}n \big) \hat G_\rho^{(n)}(1,\dfi) \big| = o(1) \times O_\Prob(1) = o_\Prob(1) 
\end{multline*}
and thus Lemma \ref{oP1glzcwecole} yields $\hat\Tb_\rho^{(n)} \weakP h_0(\Gb_\rho)$. The covariance structure \eqref{Tbrbcov} of $h_0(\Gb_\rho) = \Tb_\rho$ can be obtained using \eqref{Hrhodef}.
\qed

\medskip

\noindent 
\textbf{Proof of Proposition \ref{ConsuH1loc}.}
First, we show \eqref{Prop3101} with a reasoning which is similar to the proof of Proposition F.1 in the supplement to \cite{BueKoj14}. \\
Fix $\alpha \in (0,1) \setminus \Q$. According to Proposition \ref{JointConvProp} and the continuous mapping theorem we have $(T_\rho^{\scriptscriptstyle (n)},\hat T_{\scriptscriptstyle \rho,\xi^{\scriptscriptstyle (1)}}^{\scriptscriptstyle (n)}, \ldots , \hat T_{\scriptscriptstyle \rho,\xi^{\scriptscriptstyle (B)}}^{\scriptscriptstyle (n)}) \weak (T_{\rho,g_1}, T_{\rho,(1)},\ldots,T_{\rho,(B)})$ in $(\R^{B+1},\Bb^{B+1})$ for fixed $B\in\N$, where $T_{\rho,(1)},\ldots,T_{\rho,(B)}$ are independent copies of the limit $T_\rho$ in Corollary \ref{haTroCons}. Furthermore, let $L_{n,B}$ be the empirical c.d.f.\ based on the observations $\hat T_{\scriptscriptstyle \rho,\xi^{\scriptscriptstyle (1)}}^{\scriptscriptstyle (n)}, \ldots , \hat T_{\scriptscriptstyle \rho,\xi^{\scriptscriptstyle (B)}}^{\scriptscriptstyle (n)}$ and let $L_B$ be the empirical c.d.f.\ calculated from $T_{\rho,(1)},\ldots,T_{\rho,(B)}$. Due to the right continuity of $L_{n,B}$ we have
\[
\Prob\big( T_\rho^{(n)} \geq \hat q_{1-\alpha}^{(B)} \big( T_\rho^{(n)}\big) \big) = \Prob\big(L_{n,B}(T_\rho^{(n)}) \ge 1 - \alpha \big).
\]
Moreover, using Corollary 1.3 and Remark 4.1 in \cite{GaeMolRos07} as well as Assumption \ref{EasierCond}\eqref{rhoneq0} and the covariance structure \eqref{TrhoCovFkt} of the Gaussian process $\Tb_\rho$ in Theorem \ref{CUSUMTrhowc} it follows that $T_\rho$ has a continuous c.d.f.\ Thus, the function $\Psi_{(B)} : \R^{B+1} \to \R$ given by $\Psi_{(B)}(x_0,x_1,\ldots,x_B) = B^{-1} \sum_{i=1}^{\scriptscriptstyle B} \ind(x_i \le x_0)$ is almost surely continuous with respect to the image measure $(T_{\rho,g_1}, T_{\rho,(1)},$ $\ldots, T_{\rho,(B)})(\Prob)$. As a consequence, we have $L_{n,B}(T_\rho^{\scriptscriptstyle (n)}) \weak L_B(T_{\rho,g_1})$ as $n \to \infty$ and with the Portmanteau theorem we obtain
\[
\lim_{n\to\infty} \Prob\big( T_\rho^{(n)} \geq \hat q_{1-\alpha}^{(B)} \big( T_\rho^{(n)}\big) \big) = \Prob \big( L_B(T_{\rho,g_1}) \ge 1 - \alpha \big),
\]
because $1-\alpha \notin \{0, \frac 1B,\ldots,\frac{B-1}B,1\}$. By the Glivenko-Cantelli theorem for every $\eps \in (0,1-\alpha)$ we can choose $B_0(\eps) \in \N$ such that 
\begin{equation}
\label{GlivCanth}
\Prob\big( \sup_{x \in \R} | L_B(x) - L_\rho(x) | \ge \eps \big) \le \eps,
\end{equation}
for all $B \ge B_0(\eps)$, since $T_{\rho,(1)},\ldots,T_{\rho,(B)}$ are i.i.d.\ with distribution function $L_\rho$. Thus, for every such $B \in\N$ we have
\begin{align*}
\Prob\big( L_B(T_{\rho,g_1}) \ge 1-\alpha \big) &= \Prob\big( L_B(T_{\rho,g_1}) - L_\rho(T_{\rho,g_1}) + L_\rho(T_{\rho,g_1}) \ge 1-\alpha \big) \\
&\le \Prob\big( L_B(T_{\rho,g_1}) - L_\rho(T_{\rho,g_1}) \ge \eps \big) + \Prob\big( L_\rho(T_{\rho,g_1}) \ge 1-\alpha - \eps \big) \\
&\le \eps + \Prob\big( L_\rho(T_{\rho,g_1}) \ge 1-\alpha - \eps \big) \stackrel{\eps \downarrow 0}{\rightarrow} \Prob\big( L_\rho(T_{\rho,g_1}) \ge 1-\alpha \big)
\end{align*}
and we obtain
\begin{equation}
\label{limsupeq}
\limsup_{B \to \infty} \Prob\big( L_B(T_{\rho,g_1}) \ge 1-\alpha \big) \le \Prob\big( L_\rho(T_{\rho,g_1}) \ge 1-\alpha \big).
\end{equation}
The terms on both sides of inequality \eqref{limsupeq} are increasing in $\alpha$ and the right-hand side is right continuous in $\alpha$. As a consequence, \eqref{limsupeq} is also valid for each $\alpha \in (0,1) \cap \Q$. Furthermore, we have
\begin{equation}
\label{liminfeq}
\liminf_{B \to \infty} \Prob\big( L_B(T_{\rho,g_1}) \ge 1-\alpha \big) \ge \Prob\big( L_\rho(T_{\rho,g_1}) > 1-\alpha \big),
\end{equation}
because according to \eqref{GlivCanth} 
\begin{align*}
\Prob\big( L_B(T_{\rho,g_1}) \ge 1-\alpha \big) &= \Prob\big( L_B(T_{\rho,g_1}) - L_\rho(T_{\rho,g_1}) + L_\rho(T_{\rho,g_1}) \ge 1-\alpha \big) \\
&\ge \Prob\big( L_\rho(T_{\rho,g_1}) \ge 1-\alpha + \eps \big) - \eps \stackrel{\eps \downarrow 0}{\rightarrow} \Prob\big( L_\rho(T_{\rho,g_1}) > 1-\alpha \big)
\end{align*}
holds. Both sides of \eqref{liminfeq} are increasing in $\alpha$ and the right-hand side is left continuous in $\alpha$. Thus, \eqref{liminfeq} is also true for $\alpha \in (0,1) \cap \Q$. Finally, \eqref{Prop3103} can be shown by exactly the same steps as above and $\eqref{Prop3102}$ is an immediate consequence of the
Portmanteau theorem.
\qed

\medskip

\noindent 
\textbf{Proof of Corollary \ref{prop:asledf}.} Under ${\bf H}_0$ we have $\Tb_{\rho,g_1} =0$ and $T_{\rho,g_1} = T_\rho$ is distributed acccording to $L_\rho$. Due to $\nu_0 \neq 0$, Assumption \ref{EasierCond}\eqref{rhoneq0} and the covariance structure \eqref{TrhoCovFkt} of $\Tb_\rho$ the c.d.f.\ $L_\rho$ is continuous in virtue of Corollary 1.3 and Remark 4.1 in \cite{GaeMolRos07}. As a consequence, $L_\rho(T_{\rho,g_1}) = L_\rho(T_\rho)$ is uniformly distributed on $(0,1)$ and we have $\Prob\big(L_\rho(T_\rho) > 1-\alpha \big) = \Prob\big(L_\rho(T_\rho) \ge 1-\alpha \big) = \alpha$ for all $\alpha \in (0,1)$. Hence, \eqref{Cor3111} follows from \eqref{Prop3101} and the claim \eqref{Cor3112} can be obtained by a similar reasoning using \eqref{Prop3102} as well as \eqref{Prop3103}.
\qed

\medskip

\noindent 
\textbf{Proof of Proposition \ref{prop:tndfh1}.} 
As in the proof of Proposition \ref{poiwconKSdi} we obtain 
\[
\sup_{(\gseqi,\dfi) \in \netir} \big| N_\rho^{(n)} (\gseqi,\dfi) - N_\rho(g_0;\gseqi,\dfi) \big| = o_\Prob(1).
\]
$(n\Delta_n)^{\scriptscriptstyle -1/2}\Tb_\rho^{\scriptscriptstyle (n)}(\gseqi,\dfi)$ is given by $N_\rho^{\scriptscriptstyle (n)}(\gseqi,\dfi) - \frac{\ip{n\gseqi}}n N_\rho^{\scriptscriptstyle (n)}(1,\dfi)$ according to \eqref{Tbrhondef}. Consequently, a simple calculation shows
\[
(n\Delta_n)^{-1/2}\Tb_\rho^{(n)}(\gseqi,\dfi) = N_\rho(g_0;\gseqi,\dfi) - \gseqi N_\rho(g_0;1,\dfi) + o_\Prob(1) = T_{(1)}^\rho (\gseqi,\dfi) + o_\Prob(1)
\]
under ${\bf H}_1$, where the $o$-term is uniform in $(\gseqi,\dfi) \in \netir$.
\qed 

\medskip

\noindent 
\textbf{Proof of Proposition \ref{prop:conuH1}.} By the continuous mapping theorem, Theorem \ref{BootTrhoThm} and Remark \ref{rem:condweak}(ii) we have $\hat T_{\scriptscriptstyle \rho,\xi^{\scriptscriptstyle (b)}}^{\scriptscriptstyle (n)} = O_\Prob(1)$ and $\hat W_{\scriptscriptstyle \rho,\xi^{\scriptscriptstyle (b)}}^{\scriptscriptstyle (n,\dfi_0)} = O_\Prob(1)$ for all $b \in \{1,\ldots,B \}$. Therefore, it suffices to show $\Prob(V_{\scriptscriptstyle \rho,\dfi_0}^{\scriptscriptstyle (n)} \ge K) \to 1$ and $\Prob(W_{\scriptscriptstyle \rho}^{\scriptscriptstyle (n,\dfi_0)} \ge K) \to 1$ for every $K>0$ under ${\bf H}_1^{\scriptscriptstyle (\rho,\dfi_0)}$ and $\Prob(T_{\scriptscriptstyle \rho}^{\scriptscriptstyle (n)} \ge K) \to 1$ for each $K>0$ under ${\bf H}_1$. \\
According to the proof of Proposition \ref{poiwconKSdi} and Proposition \ref{prop:tndfh1} the quantities $(n\Delta_n)^{\scriptscriptstyle -1/2} V_{\scriptscriptstyle \rho,\dfi_0}^{\scriptscriptstyle (n)}$ and $(n\Delta_n)^{\scriptscriptstyle -1/2} W_{\scriptscriptstyle \rho}^{\scriptscriptstyle (n,\dfi_0)}$ converge to a constant in $(0,\infty)$ in outer probability under ${\bf H}_1^{\scriptscriptstyle (\rho,\dfi_0)}$, because $|T_{\scriptscriptstyle (1)}^{\scriptscriptstyle \rho} (\gseqi_0,\dfi_0)| >0$ in this case. Furthermore, due to Assumption \ref{EasierCond}\eqref{rhoneq0} we have $\sup_{(\gseqi,\dfi) \in \netir}$ $ |T_{\scriptscriptstyle (1)}^{\scriptscriptstyle \rho} (\gseqi,\dfi)| >0$ under ${\bf H}_1$ and $(n\Delta_n)^{\scriptscriptstyle -1/2} T_{\scriptscriptstyle \rho}^{\scriptscriptstyle (n)} = \sup_{(\gseqi,\dfi) \in \netir} |T_{\scriptscriptstyle (1)}^{\scriptscriptstyle \rho} (\gseqi,\dfi)| + o_\Prob(1)$, because of Proposition \ref{prop:tndfh1}. Thus, the assertion follows from $n\Delta_n \to \infty$.
\qed

\medskip

\noindent 
\textbf{Proof of Proposition \ref{prop:argmax}.} According to Proposition \ref{prop:tndfh1} the random functions $\gseqi \mapsto \sup_{\dfi \in \R}$ $ |(n\Delta_n)^{\scriptscriptstyle -1/2} \Tb_\rho^{\scriptscriptstyle (n)}(\gseqi,\dfi)|$ converges weakly in $\linne$ to the continuous function $\gseqi \mapsto \sup_{\dfi \in \R}$ $ |T_{\scriptscriptstyle (1)}^{\scriptscriptstyle \rho} (\gseqi,$ $\dfi)|$, which due to Assumption \ref{EasierCond}\eqref{rhoneq0} attains a unique maximum at $\gseqi_0$ under ${\bf H}_1$. Therefore, the claim for ${\bf H}_1$ follows from the argmax-continuous mapping theorem (Theorem 2.7 in \cite{KimPol90}). The assertion regarding ${\bf H}_1^{\scriptscriptstyle (\rho,\dfi_0)}$ follows with a similar reasoning.
\qed

\subsection{Proofs of the results in Section \ref{sec:gradchadf}} \label{sec64}
\label{prresec4}

\textbf{Proof of Lemma \ref{Drhg0suit}.} 
If the kernel $g_0(\cdot,d\taili)$ is Lebesgue almost everywhere constant on $[0,\gseqi]$, we have $D^{\scriptscriptstyle (g_0)}_\rho(\kseqi,\gseqi,\dfi) =0$ for all $0 \leq \kseqi \leq \gseqi$ and $\dfi \in \R$, since $\kseqi^{-1} \int_0^\kseqi \int_{- \infty}^\dfi \rho(\taili) g_0(y,d\taili) dy$ is constant on $(0,\gseqi]$.

If on the other hand $D^{\scriptscriptstyle (g_0)}_\rho(\kseqi,\gseqi,\dfi) =0$ for all $\kseqi \in [0,\gseqi]$ and $\dfi \in \R$ we have
\begin{align*}
\int \limits_0^\kseqi \int_{- \infty}^\dfi \rho(\taili) g_0(y,d\taili) dy = \kseqi \Big( \frac 1\gseqi \int_0^\gseqi \int_{-\infty}^\dfi\rho(\taili) g_0(y, d\taili)dy \Big) =: \kseqi A_\gseqi (\dfi)
\end{align*}
for each $\kseqi \in [0,\gseqi]$ and $\dfi \in \R$. Therefore, $\int_{-\infty}^\dfi \rho(\taili)g_0(y,d\taili) = A_\gseqi(\dfi)$ holds  for each fixed $\dfi \in \R$ and every $y \in [0,\gseqi]\setminus M_{(\dfi)}$ by Assumption \ref{EasierCond}\eqref{jbidenass} and the fundamental theorem of calculus. Consequently,
\begin{align}
\label{IntKernGlEq}
\int \limits_{-\infty}^\dfi \rho(\taili)g_0(y,d\taili) = A_\gseqi(\dfi)
\end{align}
holds for every $\dfi \in \mathbb Q$ and each $y \in [0, \gseqi]$ outside the Lebesgue null set $\bigcup_{\dfi \in \mathbb Q} M_{(\dfi)}$. According to Assumption \ref{EasierCond} the function $y \mapsto \int (1 \wedge |\taili|^p) g_0(y,d\taili)$ is bounded on $[0,1]$. Hence, by Lebesgue's dominated convergence theorem and the assumptions on $\rho$ the quantities on both sides of \eqref{IntKernGlEq} are right-continuous in $\dfi \in \R$. As a consequence, \eqref{IntKernGlEq} holds for every $\dfi \in \R$ and each $y \in [0,\gseqi]$ outside the Lebesgue null set $\bigcup_{\dfi \in \mathbb Q} M_{(\dfi)}$. Thus, by the uniqueness theorem for measures the kernel $\rho(\taili) g_0(y,d\taili)$ is Lebesgue almost everywhere on $[0,\gseqi]$ equal to the finite signed measure $\eta_\gseqi$ with measure generating function $\dfi \mapsto A_\gseqi(\dfi)$ of bounded variation. Now, recall that $g_0(y,d\taili)$ does not charge $\{ 0 \}$, so by Assumption \ref{EasierCond}\eqref{rhoneq0} the kernel $g_0(y,d\taili)$ is Lebesgue almost everywhere on $[0,\gseqi]$ equal to the measure with density $(1/\rho(\taili)) \ind_{\{\rho(\taili) \neq 0\}} \eta_\gseqi(d\taili)$. 
\qed

\medskip

\noindent 
\textbf{Proof of Theorem \ref{SchwKentmotv}.} We consider the functional  $\Lambda \colon \linner \rightarrow \linctr$ defined by
\begin{equation} \label{Lambdadef}
\Lambda(f)(\kseqi, \gseqi, \dfi) \defeq f(\kseqi, \dfi) - \frac\kseqi\gseqi f(\gseqi, \dfi).
\end{equation}
As $\| \Lambda(f_1) - \Lambda(f_2) \|_{\ctir} \leq 2 \| f_1 - f_2 \|_{\netir}$   the mapping $\Lambda$ is Lipschitz continuous. Thus, 
by  Theorem \ref{ConvThm} and  the  continuous mapping  theorem $\Lambda(G_\rho^{\scriptscriptstyle (n)})$ converges weakly in $\linctr$ to the tight mean zero Gaussian process $\Hb_\rho \defeq \Lambda(\Gb_\rho)$ with covariance structure \eqref{HbrhoProCov}. Furthermore, we have
$
\Hb_\rho^{(n)} = \Lambda(G_\rho^{(n)}) + D_\rho^{(g_1)} + \sqrt{n\Delta_n} D_\rho^{(\Rc_n)} = \Lambda(G_\rho^{(n)}) + D_\rho^{(g_1)} + o(1),
$
where the $o$-term is deterministic and uniform in $(\kseqi,\gseqi,\dfi) \in C \times \R$ by Assumption \ref{EasierCond}. Finally, the desired weak convergence follows using Slutsky's lemma (Example 1.4.7 in \cite{VanWel96}) and the fact that $\Hb_\rho$ is separable as it is tight (see Lemma 1.3.2 in the previously mentioned reference).
\qed

\medskip

\noindent 
\textbf{Proof of Theorem \ref{BootHnConvThm}.} We have $\hat \Hb_\rho^{\scriptscriptstyle (n)} = \Lambda(\hat G_\rho^{\scriptscriptstyle (n)})$ and $\Hb_\rho = \Lambda(\Gb_\rho)$ with the Lipschitz continuous mapping $\Lambda$ defined in \eqref{Lambdadef}. Thus, the assertion follows from Proposition 10.7 in \cite{Kos08}.
\qed

\medskip

\noindent 
\textbf{Proof of Theorem \ref{SchaezerKonvT}.  \ref{MSEdfzerlThm} and   \ref{OptthChodfThm}.} The assertions follow by a similar reasoning as  given in the proof of Theorem 4.2, 4.3
and 4.4  in \cite{HofVetDet17}, respectively.
\qed

\medskip

\noindent 
\textbf{Proof of Theorem \ref{KorThrCdfThm}.} We start with a proof of
$\varphi_{n}^* \probto 0$
which is equivalent to $\haFcothrle/$ $\sqrt{n\Delta_n} \probto 0$. Therefore, we have to show
\begin{align}
\label{empFgrlevaleq}
\Prob\big(\haFcothrle/\sqrt{n \Delta_n} \leq x\big) = \Prob \Big( \frac{1}{B_n} \sum \limits_{i=1}^{B_n} \ind_{\{\hat \Hb_{\rho,*}^{\scriptscriptstyle (n,i)}(\predfest) \leq (\sqrt{n \Delta_n} x)^{1/r} \}} \geq 1- \alpha_n \Big) \rightarrow 1,
\end{align}
for arbitrary $x>0$, by the definition of $\haFcothrle$ in \eqref{thrledfdefeq}. Since the
\begin{align*}
\ind_{\{\hat \Hb_{\rho,*}^{(n,i)}(\predfest)  \leq (\sqrt{n \Delta_n} x)^{1/r} \}} - \Prob_\xi\Big( \hat \Hb_{\rho,*}^{(n)}(\predfest) \leq (\sqrt{n \Delta_n} x)^{1/r} \Big), \quad i=1, \ldots, B_n,
\end{align*}
are pairwise uncorrelated with mean zero and bounded by $1$, we have
$$
\Prob \Big( \Big| \frac 1{B_n} \sum \limits_{i=1}^{B_n} \ind_{\{\hat \Hb_{\rho,*}^{\scriptscriptstyle (n,i)}(\predfest) \leq (\sqrt{n \Delta_n} x)^{1/r} \}} - \Prob_\xi\Big( \hat \Hb_{\rho,*}^{(n)}(\predfest) \leq (\sqrt{n \Delta_n} x)^{1/r} \Big) \Big| > \alpha_n/2 \Big) 
\leq 4 \alpha_n^{-2} B_n^{-1} \rightarrow 0.
$$
Therefore, in order to prove \eqref{empFgrlevaleq}, it suffices to verify
\begin{multline}
\label{secFProbabeq}
\Prob \Big( \Prob_\xi\Big( \hat \Hb_{\rho,*}^{(n)}(\predfest)  \leq (\sqrt{n \Delta_n} x)^{1/r} \Big)  < 1- \alpha_n/2 \Big) 
\leq \frac 2{\alpha_n} \Prob \Big( \hat \Hb_{\rho,*}^{(n)}(\predfest)  > (\sqrt{n \Delta_n} x)^{1/r} \Big) \\ \leq 2 \alpha_n^{-1} \Prob(Q_n^C) + 2 \alpha_n^{-1} \Prob \Big( \Big\{ 2 \sup \limits_{\dfi \in \R} \sup \limits_{\gseqi \in [0,1]} |\hat G_\rho^{(n)} (\gseqi, \dfi)|  > (\sqrt{n \Delta_n} x)^{1/r} \Big\} \cap Q_n  \Big)
\rightarrow 0,
\end{multline}
with $Q_n$ the set defined in \eqref{BnsetsDef}. The first inequality in the above display follows with the Markov inequality and the last inequality in \eqref{secFProbabeq} is a consequence of the fact that $\hat \Hb_{\rho,*}^{\scriptscriptstyle (n)}(\predfest) \le \hat \Hb_{\rho,*}^{\scriptscriptstyle (n)}(1) \leq 2 \sup_{\dfi \in \R} \sup_{\gseqi \in [0,1]} |\hat G_\rho^{\scriptscriptstyle (n)} (\gseqi, \dfi)|$. Due to Lemma \ref{QnConvOrdn} in Appendix \ref{appD} 
we obtain $\Prob\big(Q_n^C\big) \leq K n \Delta_n^{1+\tau}$ and consequently $\alpha_n^{-1} \Prob(Q_n^C) \to 0$. For the second summand on the right-hand side of \eqref{secFProbabeq} the definition of $\hat G_\rho^{\scriptscriptstyle (n)}$ in \eqref{BootGrnDefEq} gives
\begin{equation*}
\Eb \Big\{\sup \limits_{\dfi \in \R} \sup \limits_{\gseqi \in [0,1]} |\hat G_\rho^{(n)} (\gseqi, \dfi)| \ind_{Q_n} \Big\}\leq \frac 1{\sqrt{n \Delta_n}} \sum \limits_{i=1}^n  \Eb \big( |\xi_i| |\rho(\Deli X^{(n)})| \ind_{\{|\Deli X^{(n)}| > v_n\}}  \ind_{Q_n} \big) \leq K \sqrt{n \Delta_n}.
\end{equation*}
The final estimate above follows using Lemma \ref{MomklAbLem}  in Appendix \ref{appD}, $\Eb| \xi_i| \leq 1$ for every $i=1, \ldots, n$ and independence of the multipliers and the other involved quantities. 
Therefore,
with the Markov inequality we obtain
\begin{align*}
\alpha_n^{-1} \Prob \Big( \Big\{ 2 \sup \limits_{\dfi \in \R} \sup \limits_{\gseqi \in [0,1]} |\hat G_\rho^{(n)} (\gseqi, \dfi)|  > (\sqrt{n \Delta_n} x)^{1/r} \Big\} \cap Q_n  \Big) &\leq 
 K \Big( (n \Delta_n)^{\frac{1-r}{2r}} \alpha_n \Big)^{-1} \rightarrow 0,
\end{align*}
by the assumptions on the involved sequences. Thus, we conclude $\beta^*_n \probto 0$. 

Next we show $\haFcothrleor \probto \infty$, which is equivalent to
\begin{align*}
\Prob\big(\haFcothrle \leq x\big) = \Prob \Big( \frac{1}{B_n} \sum \limits_{i=1}^{B_n} \ind_{\{\hat \Hb_{\rho,*}^{\scriptscriptstyle (n,i)}(\predfest)  \leq x^{1/r} \}} \geq 1- \alpha_n \Big) \rightarrow 0,
\end{align*}
for each $x >0$.  By the same considerations as in the previous paragraph  it is sufficient to show
\[
\Prob \Big( \Prob_\xi\Big( \hat \Hb_{\rho,*}^{(n)}(\predfest)  > x^{1/r} \Big)  \leq 2 \alpha_n \Big) \rightarrow 0.
\]
Let $\dfi_0 \in \R$ with $N_{\rho^2}(\gseqi_0,\dfi_0) >0$. By continuity of the function $\kseqi \mapsto N_{\rho^2}(\kseqi,\dfi_0)$ we can find $0< \bar \kseqi < \bar \gseqi < \gseqi_0$ with
\begin{align}
\label{ChoiFbakseEq}
N_{\rho^2}(\bar \kseqi,\dfi_0) >0 . 
\end{align}
As $
\hat \Hb^{(n)}_\rho(\bar \kseqi, \bar \gseqi, \dfi_0) \leq \hat \Hb_{\rho,*}^{(n)}(\predfest) \Longrightarrow \Prob_\xi\big( \hat \Hb^{(n)}_\rho(\bar \kseqi, \bar \gseqi, \dfi_0) > x^{1/r} \big) \leq \Prob_\xi\big( \hat \Hb_{\rho,*}^{(n)}(\predfest) > x^{1/r} \big)
$
on the set $\{\bar \gseqi < \predfest\}$ and consistency of the preliminary estimate it further suffices to prove
\begin{align}
\label{GenFzzUngl}
\Prob \Big( \Prob_\xi\Big( \hat \Hb_{\rho,*}^{(n)}(\predfest) > x^{1/r} \Big)  \leq 2 \alpha_n ,~ \bar \gseqi < \predfest \Big) \leq \Prob \Big( \Prob_\xi\Big( \hat \Hb^{(n)}_\rho(\bar \kseqi, \bar \gseqi, \dfi_0) > x^{1/r} \Big)  \leq 2 \alpha_n \Big) \rightarrow 0.
\end{align}
For a proof \eqref{GenFzzUngl} we use a Berry-Esseen type result. Recall  the notation $
\hat \Hb^{(n)}_\rho (\bar \kseqi, \bar \gseqi, \dfi_0) = \frac 1{\sqrt{n \Delta_n}} \sum \limits_{j=1}^n \hat B_j^n \xi_j
$
from \eqref{hatHbrhondeq} with
$
\hat B_j^n =(\ind_{\{j \leq \ip{n \bar\kseqi} \}} - \frac{\bar\kseqi}{\bar\gseqi} \ind_{\{j \leq \ip{n \bar\gseqi}\}}) \hat A_j^n,
$
where
\begin{align*}
\hat A_j^n = \rho(\Delj X^{(n)}) \ind_{(-\infty, \dfi_0]}(\Delj X^{(n)}) \ind_{\{|\Delj X^{(n)}| > v_n \}}, \quad j= 1, \ldots,n.
\end{align*}
It is easy  to see that
$
\hat W_n^2 := \Eb_\xi (\hat \Hb^{(n)}_\rho (\bar \kseqi, \bar \gseqi, \dfi_0))^2 = \frac 1{n\Delta_n} \sum \limits_{j=1}^n (\hat B_j^n)^2.
$
Thus, Theorem 2.1 in \cite{CheSha01} yields
\begin{multline}
\label{FBerEssBound}
\sup \limits_{x \in \R} \Big| \Prob_\xi\Big( \hat \Hb^{(n)}_\rho(\bar \kseqi, \bar \gseqi,\dfi_0) > x \Big) - (1- \Phi(x/\hat W_n)) \Big|  \\ \leq K \Big\{ \sum \limits_{i=1}^n \Eb_\xi \hat U_{i,n}^2 \ind_{\{|\hat U_{i,n}| >1\}} + \sum \limits_{i=1}^n \Eb_\xi |\hat U_{i,n}|^3 \ind_{\{|\hat U_{i,n}| \leq 1\}} \Big\},
\end{multline}
if $\hat W_n >0$ with
$
\hat U_{i,n} = \frac{\hat B_i^n \xi_i}{\sqrt{n \Delta_n} \hat W_n}
$
and where $\Phi$ denotes the standard normal distribution function. Before we proceed further in the proof of \eqref{GenFzzUngl}, we first show
\begin{align}
\label{FOpeinszzEq}
\frac 1{\hat W_n^2} = \frac{n \Delta_n}{\sum \limits_{j=1}^n (\hat B_j^n)^2} = O_\Prob (1),
\end{align}
that is
$
\lim \limits_{M \rightarrow \infty} \limsup \limits_{n \rightarrow \infty} \Prob \big( n \Delta_n > M \sum \limits_{j=1}^n (\hat B_j^n)^2 \big) =0.
$
Let $M>0$. Then a straightforward calculation gives
\begin{align*}
\Prob \Big( n \Delta_n > M \sum \limits_{j=1}^n (\hat B_j^n)^2 \Big) &\leq 
\Prob \Big( n \Delta_n > M' \sum \limits_{j=1}^{\ip{n \bar \kseqi}} (\hat A^n_j)^2 \Big)  
= \Prob \Big( 1/M' > N_{\rho^2}^{(n)}( \bar \kseqi, \dfi_0)  \Big),
\end{align*}
with $M' = M(1- \bar \kseqi /\bar \gseqi)^2$. 
Consequently, with \eqref{ChoiFbakseEq} we obtain \eqref{FOpeinszzEq} due to
\[ 
N_{\rho^2}^{(n)}( \bar \kseqi, \dfi_0) = N_{\rho^2}(g^{(n)};\bar \kseqi,\dfi_0) + o_\Prob(1) = N_{\rho^2}(g_0;\bar \kseqi,\dfi_0) + o_\Prob(1),
\]
because Theorem \ref{ConvThm} also holds for $\rho^2$ since Assumption \ref{Cond1} is also valid for $2p$ instead of $p$ (cf. \eqref{1mipgezEq} in the proof of Proposition \ref{easier}).
Recall that our main objective is to show \eqref{GenFzzUngl} and thus we consider the Berry-Esseen bound on the right-hand side of \eqref{FBerEssBound}. For the first summand we distinguish two cases according to the assumptions on the multiplier sequence.

Let us discuss the case of bounded multipliers first. For $M>0$ we have
$
|\hat U_{i,n}| \leq \frac {\sqrt{M} K}{\sqrt{n\Delta_n}}
$
for all $i= 1, \ldots,n$ on the set $\{1/\hat W_n^2 \leq M\}$, since $|\hat B_i^n|$ is bounded. As a consequence,
\begin{align}
\label{FvanishEq}
\sum \limits_{i=1}^n \Eb_\xi \hat U_{i,n}^2 \ind_{\{|\hat U_{i,n}| >1\}} =0
\end{align}
for large $n \in \N$ on the set $\{1/\hat W_n^2 \leq M\}$. 

In the situation of normal multipliers, recall that there exist constants $K_1, K_2 >0$ such that for $\xi \sim \mathcal N(0,1)$ and $y >0$ large enough we have
\begin{align}
\label{FGaussmomAb}
\Eb_\xi \xi^2 \ind_{\{|\xi| >y\}} = \frac{2}{\sqrt{2 \pi}} \int_y^\infty z^2 e^{-z^2/2} dz \le K \Prob(\mathcal N(0,2) >y) \leq K_1 \exp(-K_2 y^2).
\end{align}
Thus, we can calculate for $n \in \N$ large enough on the set $\{1/\hat W_n^2 \leq M\}$
\begin{align*}
\sum \limits_{i=1}^n \Eb_\xi \hat U_{i,n}^2 \ind_{\{|\hat U_{i,n}| >1\}} &= \sum \limits_{i=1}^n \Big( \sum \limits_{j=1}^n (\hat B_j^n)^2 \Big)^{-1} (\hat B_i^n)^2 \Eb_\xi \xi_i^2 \ind_{\{|\xi_i| > (\sum \limits_{j=1}^n (\hat B_j^n)^2)^{1/2}/|\hat B_i^n| \}} \\
&\leq K \sum \limits_{i=1}^n \Big( \sum \limits_{j=1}^n (\hat B_j^n)^2 \Big)^{-1} \Eb_\xi \xi_i^2 \ind_{\{|\xi_i| > (\sum \limits_{j=1}^n (\hat B_j^n)^2)^{1/2} /K\}} \\
&\leq \frac{KM}{n \Delta_n} \sum \limits_{i=1}^n \Eb_\xi \xi_i^2 \ind_{\{|\xi_i| > (n\Delta_n /M)^{1/2}/K\}} \leq \frac{K_1}{\Delta_n} \exp(-K_2 n \Delta_n),
\end{align*}
where $K_1$ and $K_2$ depend on $M$. The first inequality in the display above uses boundedness of $|\hat B_i^n|$ again and the last one follows with \eqref{FGaussmomAb}. Now, according to Assumption \ref{EasierCond}\eqref{speed} let $0 < t_2 \leq 1$ and $\delta >0$ with $n^{-t_2 + \delta} = o(\Delta_n)$. Furthermore, define $\bar \delta >0$ via $1+\bar \delta = 1/(t_2 - \delta)$ and $\bar q := 1/\bar \delta$. Then we have $n \Delta_n^{1+\bar \delta} \rightarrow \infty$ and for $n \geq N(M) \in \N$ on the set $\{1/\hat W_n^2 \leq M\}$, using $\exp(-K_2 n \Delta_n) \leq (n \Delta_n)^{-\bar q}$, we conclude
\begin{align}
\label{FirstFTerkln}
\sum \limits_{i=1}^n \Eb_\xi \hat U_{i,n}^2 \ind_{\{|\hat U_{i,n}| >1\}} \leq K_1 \Delta_n^{-1} (n \Delta_n)^{-\bar q} = K_1 \big( n \Delta_n^{1+ \bar \delta} \big)^{-\bar q}.
\end{align}
We now consider the second term on the right-hand side of \eqref{FBerEssBound}, for which 
\begin{align*}
\sum \limits_{i=1}^n \Eb_\xi |\hat U_{i,n}|^3 \ind_{\{|\hat U_{i,n}| \leq 1\}} &\leq \sum \limits_{i=1}^n \Big(\sum \limits_{j=1}^n (\hat B_j^n)^2\Big)^{-3/2} |\hat B_i^n|^3 \Eb_\xi |\xi_i|^3 \leq \frac K{(n \Delta_n)^{3/2}} \sum \limits_{i=1}^n |\hat B_i^n|
\end{align*}
holds on $\{1/\hat W_n^2 \leq M\}$, using boundedness of $|\hat B_i^n|$ again. With Lemma \ref{MomklAbLem} we see that
\begin{align*}
\Eb\Big( \sum \limits_{i=1}^n |\hat B_i^n| \ind_{Q_n} \Big) \leq 2 \Eb \Big(\sum \limits_{i=1}^n |\hat A_i^n| \ind_{Q_n} \Big) \leq K n \Delta_n.
\end{align*}
Consequently,
\begin{align}
\label{SecFProbNuf}
\Prob\Big( \Big\{ 1/\hat W_n^2 \leq M  \text{ and } &K \sum \limits_{i=1}^n \Eb_\xi |\hat U_{i,n}|^3 \ind_{\{|\hat U_{i,n}| \leq 1\}} > (n\Delta_n)^{-1/4} \Big\} \cap Q_n \Big) \nonumber \\
&\leq \Prob\Big( \Big\{ \frac K{(n \Delta_n)^{3/2}} \sum \limits_{i=1}^n |\hat B_i^n|> (n\Delta_n)^{-1/4} \Big \} \cap Q_n \Big) \leq K (n\Delta_n)^{-1/4}
\end{align}
follows. Thus, from \eqref{FvanishEq}, \eqref{FirstFTerkln} and \eqref{SecFProbNuf} we see that with $K>0$ from \eqref{FBerEssBound} for each $M>0$ there exists a $K_3>0$ such that
\begin{multline}
\label{FBoundnuf}
\Prob\Big( 1/\hat W_n^2 \leq M \text{ and } K \Big\{ \sum \limits_{i=1}^n \Eb_\xi \hat U_{i,n}^2 \ind_{\{|\hat U_{i,n}| >1\}} + \sum \limits_{i=1}^n \Eb_\xi |\hat U_{i,n}|^3 \ind_{\{|\hat U_{i,n}| \leq 1\}} \Big\} \\
> K_3((n\Delta_n)^{-1/4} + (n \Delta_n^{1+ \bar \delta})^{-\bar q}) \Big) \rightarrow 0.
\end{multline}

Now we can show \eqref{GenFzzUngl}. Let $\eta >0$ and according to \eqref{FOpeinszzEq} choose an $M >0$ with
$\Prob(1/\hat W_n^2 >M) < \eta/2$
for all $n \in \N$. For this $M>0$ choose a $K_3 >0$ such that the probability in \eqref{FBoundnuf} is smaller than $\eta/2$ for large $n$. Then for $n \in \N$ large enough we have
\begin{multline*}
\Prob \Big( \Prob_\xi\Big( \hat \Hb^{(n)}_\rho(\bar \kseqi, \bar \gseqi, \dfi_0)) > x^{1/r} \Big)  \leq 2 \alpha_n \Big) < \\
\Prob\Big( (1- \Phi(x^{1/r}/\hat W_n)) \leq 2 \alpha_n + K_3((n\Delta_n)^{-1/4} + (n \Delta_n^{1+ \bar \delta})^{-\bar q}) \text{ and } 1/\hat W_n^2 \leq M \Big) + \eta = \eta,
\end{multline*}
using \eqref{FBerEssBound} and the fact, that if $1/\hat W_n^2 \leq M$ there exists a $c'>0$ with $ (1- \Phi(x^{1/r}/\hat W_n)) >c'$. 

Thus, we have shown $\haFcothrleor \probto \infty$ and it remains to prove  \eqref{ovestdfabEq}. Let 
$
K = (  (1+\eps)/c )^{1/\smooi} > (  1/c )^{1/\smooi}
$
for some $\eps >0$. Then 
\begin{align*}  
\Prob \Big(\hat \gseqi_\rho^{(n)}(\haFcothrleor) &> \gseqi_0 + K \varphi^*_n \Big) \leq \Prob \Big( \sqrt{n \Delta_n} \Db_{\rho,*}^{(n)}(\gseqi) \leq \haFcothrleor \text{ for some } \gseqi > \gseqi_0 + K \varphi^*_n \Big) \\
&\leq \Prob\Big( \sqrt{n \Delta_n} \Dc_\rho(\gseqi) - \Hb_{\rho,*}^{(n)}(1) \leq \haFcothrleor \text{ for some } \gseqi > \gseqi_0 + K \varphi^*_n \Big).
\end{align*}
By \eqref{additass} there exists a $y_0 >0$ with
\[
\inf \limits_{\gseqi \in [\gseqi_0 + K y_1,1]} \Dc_\rho(\gseqi) = \Dc_\rho(\gseqi_0 + K y_1) \geq (c/(1+\eps/2)) (K y_1)^\smooi
\]
for all $0 \leq y_1 \leq y_0$. Distinguishing the cases $\{\varphi_n^* > y_0\}$ and $\{\varphi_n^* \le y_0\}$ we get due to $\varphi_n^* \probto 0$
\begin{align*}
\Prob\Big(\hat \gseqi_\rho^{(n)}(\haFcothrleor) &> \gseqi_0 + K \varphi^*_n \Big) \\
&\leq \Prob\Big( \sqrt{n\Delta_n}  (c/(1+\eps/2)) (K \eps_n^*)^\smooi - \Hb_{\rho,*}^{(n)}(1) \leq \haFcothrleor \Big) + o(1) \\
&\leq P^{(1)}_n + P^{(2)}_n + o(1)
\end{align*}
with 
\begin{align*}
P^{(1)}_n &= \Prob\Big( \sqrt{n\Delta_n}  (c/(1+\eps/2)) (K \varphi^*_n)^\smooi - \Hb_{\rho,*}^{(n)}(1) \leq \haFcothrleor \text{ and } \Hb_{\rho,*}^{(n)}(1) \leq b_n \Big), \\
P^{(2)}_n &= \Prob\Big(\Hb_{\rho,*}^{(n)}(1) > b_n \Big),
\end{align*}
where $b_n := \sqrt{\haFcothrleor}$. Due to the choice $K = \Big(  (1+\eps)/c \Big)^{1/\smooi}$ and the definition of $\varphi^*_n$ it is clear that $P^{\scriptscriptstyle (1)}_n = o(1)$, because $\haFcothrleor \probto \infty$. 

Concerning $P^{\scriptscriptstyle (2)}_n$ let $F_n$ be the distribution function of $\Hb_{\rho,*}^{\scriptscriptstyle (n)}(1)$ and let $F$ be the distribution function of $\Hb_{\rho,*}(1)$. Then according to  Corollary 1.3 and Remark 4.1 in \cite{GaeMolRos07} the function $F$ is continuous, because $N_{\rho^2}(\gseqi_0,\dfi_0) >0$ for some $\dfi_0 \in \R$. As a consequence, by Theorem \ref{SchwKentmotv} and the continuous mapping theorem $F_n$ converges pointwise to $F$. Thus, for $\eta >0$ choose an $x>0$ with $1-F(x) < \eta/2$ and conclude
\begin{align*}
P^{(2)}_n \leq \Prob(b_n \leq x) + 1- F_n(x) \leq \Prob(b_n \leq x) + 1- F(x) + | F_n(x) - F(x) | < \eta,
\end{align*}
for $n \in \N$ large enough, because of $\haFcothrleor \probto \infty$.
\qed

\medskip

\noindent 
\textbf{Proof of Proposition \ref{localtgrapro}, Corollary \ref{H0graprop} and Proposition \ref{CorConGradf}.}
The assertions can be obtained by a similar reasoning as in the proofs of Proposition \ref{ConsuH1loc}, Corollary \ref{prop:asledf} and Proposition \ref{prop:conuH1} and we omit the details.
\qed

\medskip

\noindent 
\textbf{Proof of the results in Example \ref{Ex:SitgraCh}, Example \ref{exdf2} and Example \ref{exdf3}(\ref{exdf3Nr2}).}

\vspace{-2mm}

\begin{enumerate}[(1)]
	\item First we show that a transition kernel of the form \eqref{gformgrach} belongs to $\Gc(\hat \beta,\hat p)$ and the function $\rho_{L,\hat p}$ satisfies Assumption \ref{EasierCond}\eqref{EasrhoCond} and \eqref{rhoneq0} for $p = \hat p$. Let $\hat A$ denote a bound for $A: [0,1] \to (0, \infty)$, then for $\taili \in (-1,1) \setminus \{0\}$ we obtain
	\[
	\sup_{y \in [0,1]} A(y) h_{\beta(y),p(y)}(\taili) \leq \hat A \sup_{y \in [0,1]} |\taili|^{-(1 + \beta(y))} \leq \hat A |\taili|^{-(1 + \hat \beta)},
	\]
	so Definition \ref{rhoandgass2}\eqref{BGn0Ass} is satisfied. Furthermore, for $n \in \N$ we have
	\[
	\sup_{\taili \in C_n} \sup_{y \in [0,1]}  A(y) h_{\beta(y),p(y)}(\taili) \leq \hat A \sup_{y \in [0,1]} n^{1 + \beta(y)} \leq \hat A n^{1 + \hat \beta},
	\]
	which yields Definition \ref{rhoandgass2}\eqref{DiCondmiddle}. Definition \ref{rhoandgass2}\eqref{DiCondinfty} also holds, because for $|\taili| >2$ we obtain
	\[
	\sup_{y \in [0,1]} A(y) h_{\beta(y),p(y)}(\taili) \leq \hat A \sup_{y \in [0,1]} |\taili|^{-p(y)} \leq \hat A |\taili|^{-(2 \hat p \vee 2) - \eps},
	\]
	since $\hat p > 1$. Obviously, $\rho_{L,\hat p}: \R \to \R$ is a bounded function with $\rho_{L,\hat p}(0)=0$ and with the continuous derivative
	\[
	\rho_{L,\hat p}'(\taili) = L\text{ sign}(\taili) \times \begin{cases}
	2 \hat p   |\taili|^{\hat p -1}, \quad &\text{ for } 0 \leq |\taili| \leq 1, \\
	2 \hat p  (2 - |\taili|), \quad &\text{ for } 1 \leq |\taili| \leq 2, \\
	0, \quad &\text{ for } |\taili| > 2.
	\end{cases}
	\]
	Consequently, there exists a $K>0$ such that $|\rho_{L,\hat p}'(\taili)| \leq K |\taili|^{\hat p -1}$ holds for each $\taili \in \R$ and Assumption \ref{EasierCond}\eqref{EasrhoCond} is satisfied. Moreover, Assumption \ref{EasierCond}\eqref{rhoneq0} is valid as well, since $\rho_{L,\hat p}(1) >0$ and $\rho_{L,\hat p}'(\taili) \ge 0$ on $[1,2]$. 
	\item Now we show that if additionally \eqref{vorcpkonst} and \eqref{atcpanalytic} are satisfied, $k_0 < \infty$ holds and $N_k(\dfi)$ is a bounded function on $\R$ for each $k \in \N_0$ as stated in Example \ref{exdf3}\eqref{exdf3Nr2}. To this end, elementary calculations show that the function $\bar N$ is given by $\bar N(y,\dfi) = \Upsilon_{L,\hat p}(\bar A(y),\bar \beta(y), \bar p(y), \dfi)$ with
	\begin{align}
	\label{upsilondef}
	\Upsilon_{L, \hat p}(a,\beta,p,\dfi) = La \times \begin{cases}
	\frac{2 + \hat p}{p - 1} |\dfi|^{1-p}, \quad &\text{ for } \dfi \leq -2 \\
	\frac{2 + \hat p}{p - 1} 2^{1-p} +4 - \frac{2\hat p}3 - 2\hat p \dfi^2 - \frac{\hat p}3 \dfi^3 + (2-3\hat p)\dfi, \quad &\text{ for } - 2 \leq \dfi \leq -1 \\
	\frac{2 + \hat p}{p - 1} 2^{1-p} + 2 + \frac{2 \hat p}3 + \frac 2{\hat p - \beta}(1 + \text{ sign}(\dfi) |\dfi|^{\hat p - \beta} ), \quad &\text{ for } - 1 \leq \dfi \leq 1 \\
	\frac{2 + \hat p}{p - 1} 2^{1-p} + 2 \hat p + \frac 4{\hat p - \beta} + 2 \hat p \dfi^2 - \frac{\hat p}3 \dfi^3 + 2 \dfi - 3 \hat p \dfi, \quad &\text{ for }  1 \leq \dfi \leq 2 \\
	\frac{4 + 2\hat p}{p - 1} 2^{1-p} + \frac{4}{\hat p - \beta} + \frac{4 \hat p}3 +4+ \frac{2+\hat p}{1-p} t^{1-p}, \quad &\text{ for }  \dfi \geq 2.
	\end{cases}
	\end{align}
	Furthermore, it is well known from complex analysis that there is a domain $U \subset U^* \subset \Cb$ with holomorphic functions $A^*: U^* \to \Cb$, $\beta^*:U^* \to \Cb^{\hat p -} := \{u \in \Cb \mid \Rep(u) < \hat p \}$ and $p^*: U^* \to \Cb^{1+} := \{ u \in \Cb \mid \Rep(u) >1 \}$ such that $\bar A$, $\bar \beta$ and $\bar p$ are the restrictions of $A^*$, $\beta^*$ and $p^*$ to $U$. Moreover, it can be seen from \eqref{upsilondef} that for fixed $\dfi \in \R$ the mapping $(a,\beta,p) \mapsto \Upsilon_{L,\hat p}(a,\beta,p,\dfi)$ is partially holomorphic on $\Cb \times \Cb^{\hat p -} \times \Cb^{1+}$, that is it is holomorphic in each of the variables $a$, $\beta$ and $p$ when the remaining variables are fixed. By a deep result of complex analysis in several variables which dates back to \cite{Har1906} this implies that $(a,\beta,p) \mapsto \Upsilon_{L,\hat p}(a,\beta,p,\dfi)$ is holomorphic on $\Cb \times \Cb^{\hat p -} \times \Cb^{1+}$ for fixed $\dfi \in \R$ (see also Remark 1.2.28 in \cite{Sch05}). Additionally, by Proposition 1.2.2(5) in \cite{Sch05} the function $\Xi: U^* \to \Cb \times \Cb^{\hat p -} \times \Cb^{1+}$ with $\Xi(y) := (A^*(y),\beta^*(y),p^*(y))$ is holomorphic and thus for each fixed $\dfi \in \R$ the mapping $y \mapsto \bar N(y,\dfi)$ is real analytic, because it is the restriction of the holomorphic function $y \mapsto \Upsilon_{L, \hat p}(\Xi(y),\dfi)$ to $U$. Consequently, by shrinking the set $U$ if necessary, we have the power series expansion
	\begin{equation}
	\label{barNExpan}
	\bar N(y,\dfi) = \sum_{k = 0}^\infty \frac{N_k(\dfi)}{k!} (y - \gseqi_0)^k,
	\end{equation}
	for every $y \in U$ and $\dfi \in \R$. If $k_0 = \infty$, then for any $k \in \N$ and $\dfi \in \R$ we have $N_k(\dfi) =0$. Thus, we obtain for some constant $K>0$ 
	\begin{equation}
	\label{constinyifkin}
	\Psi_{L,\hat p}(y) + K \frac{\bar A(y)}{1-\bar p(y)} \dfi^{1- \bar p(y)} = N_0(\dfi)
	\end{equation}
	for each $\dfi \geq 2$ and $y \in U$, where
	\begin{equation}
	\label{PsiLhatpdef}
	\Psi_{L,\hat p}(y) = L \bar A(y) \Big(\frac{4 + 2\hat p}{\bar p(y) - 1} 2^{1-\bar p(y)} + \frac{4}{\hat p - \bar \beta(y)} + \frac{4 \hat p}3 +4\Big).
	\end{equation}
	Taking the derivative with respect to $y \in U$ on both sides of \eqref{constinyifkin} yields
	\begin{equation}
	\label{DerivaCons}
	\Psi_{L,\hat p}'(y) +  K\frac{\bar A'(y) (1- \bar p(y)) + \bar A(y) \bar p'(y)}{(1-\bar p(y))^2} \dfi^{1-\bar p(y)} -  \bar p'(y) \frac{K \bar A(y)}{1- \bar p(y)} \log(\dfi) \dfi^{1-\bar p (y)} =0,
	\end{equation}
	for each $y \in U$ and $\dfi \geq 2$. Hence, $\bar p'(y)$ is equal to zero for each $y \in U$, because otherwise the display above is not valid for each $\dfi \geq 2$. This fact together with \eqref{DerivaCons} gives
	\[
	\Psi_{L,\hat p}'(y) + K \frac{\bar A'(y)}{1-\bar p(y)} \dfi^{1-\bar p (y)} =0,
	\]
	for all $y \in U$ and $\dfi \geq 2$. Consequently, $\bar A'(y) = 0$ holds for every $y \in U$ and with \eqref{PsiLhatpdef} we obtain
	\[
	\Psi_{L,\hat p}'(y) = 4L\bar A(\gseqi_0) \bar \beta'(y) (\hat p - \bar \beta(y))^{-2} = 0, \quad (y \in U)
	\]
	which implies $\bar \beta'(y) = 0$ for all $y \in U$. Thus, $k_0 = \infty$ contradicts the assumption that at least one of the functions $\bar A$, $\bar \beta$ and $\bar p$ is non-constant. \\
	The following consideration will be helpful in order to show that $N_k(\dfi)$ is bounded in $\dfi \in \R$ for each $k \in \N_0$. Let $f_1,f_2: U \times \R \to \R$ be functions, which are arbitrarily often differentiable with respect to $y \in U$ for fixed $\dfi \in \R$ such that for each $\ell \in \N_0$ the $\ell$-th derivatives with respect to $y$ satisfy
	\[
	\sup_{\dfi \in \R} \{|f_1^{(\ell)}(\gseqi_0,\dfi)| \vee |f_2^{(\ell)}(\gseqi_0,\dfi)| \} \leq K(K\ell)^{\ell}
	\]
	for some constant $K>0$ which does not depend on $\ell$. (Here we set $0^0 := 1$.) Then by the product formula for higher derivatives we obtain for the $\ell$-th derivative with respect to $y$ of the product of $f_1$ and $f_2$
	$$
	\sup_{\dfi \in \R} |(f_1 f_2)^{(\ell)}(\gseqi_0,\dfi)| = \sup_{\dfi \in \R} \Big | \sum_{j=0}^\ell \binom{\ell}{j} f_1^{(j)}(\gseqi_0,\dfi) f_2^{(\ell - j)}(\gseqi_0,\dfi) \Big|
	 \leq K^2 (K \ell)^\ell \sum_{j=0}^\ell \binom{\ell}{j} \leq K(K \ell)^\ell. 
$$
 Observing \eqref{upsilondef}  now yields a  constant $K>0$ such that 
	\begin{align}
	\label{Nkbound}
	\sup_{\dfi \in \R} |N_\ell(\dfi)| \leq K(K\ell)^\ell
	\end{align}
	for each $\ell \in \N_0$ as soon as we can show that there exists a $K>0$ such that for every $\ell \in \N_0$ the following bounds for the derivatives hold
	\begin{align}
	|\bar A^{(\ell)}(\gseqi_0) | &\leq K(K\ell)^\ell, \label{ADerBound} \\
	\Big| \Big( \frac 1{\bar p(y) -1} \Big)^{(\ell)} (\gseqi_0) \Big| &\leq K(K\ell)^\ell, \label{analyfct1bou} \\
	\Big| \Big( \frac 1{\hat p - \bar \beta(y)} \Big)^{(\ell)} (\gseqi_0) \Big| &\leq K(K\ell)^\ell, \label{analyfct2bou} \\
	\sup_{t \geq 2} \Big| \Big( t^{1-\bar p(y)} \Big)^{(\ell)} (\gseqi_0) \Big|  &\leq K(K\ell)^\ell, \label{erstsuptab}\\
	\sup_{t \in [0,1]} \Big| \Big( t^{\hat p-\bar \beta(y)} \Big)^{(\ell)} (\gseqi_0) \Big| &\leq K(K\ell)^\ell. \label{zweisuptab}
	\end{align}
	Let $\bar A(y) = \sum_{\ell=0}^\infty A_\ell (y - \gseqi_0)^\ell$ be the power series expansion of the real analytic function $\bar A$ around $\gseqi_0$. By the definition of real analytic functions this power series has a positive radius of convergence and due to the Cauchy-Hadamard formula this is equivalent to the existence of a constant $K>0$ with $|A_\ell| \leq K^{\ell+1}$ for each $\ell \in \N_0$. Thus, because of $\bar A^{(\ell)}(\gseqi_0) = \ell! A_\ell$ for each $\ell \in \N_0$, \eqref{ADerBound} follows. By assumption in Example \ref{exdf2} we have $\bar \beta(y) \leq \hat \beta \le 1 \vee \hat \beta < \hat p < \bar p(y)$ for each $y \in U$. As a consequence, the functions 
$
	y \mapsto \frac 1{\bar p(y) -1}$ and $ y \mapsto \frac 1{\hat p - \bar \beta(y)}
$
	are real analytic on $U$ as compositions of real analytic functions. So the same reasoning as above yields \eqref{analyfct1bou} and \eqref{analyfct2bou}. Let the affine linear functions $\bar \beta$ and $\bar p$ be given by $\bar \beta(y) = \beta_0 + \beta_1(y-\gseqi_0)$ and $\bar p(y) = p_0 + p_1(y-\gseqi_0)$. Then for $\ell \in \N_0$, $\dfi >0$ we have 
	\begin{equation*}
	\Big( t^{1-\bar p(y)} \Big)^{(\ell)}(\gseqi_0) = t^{1-p_0}(-p_1 \log(\dfi))^\ell.
	\end{equation*}
	and for $\ell \in \N_0$ let $h_\ell^{\scriptscriptstyle (1)}:(0,\infty) \to \R$  be defined by $h_\ell^{\scriptscriptstyle (1)}(t) = t^{1-p_0} (\log(\dfi))^\ell$. $h_0^{\scriptscriptstyle (1)}$ is clearly bounded in $\dfi \geq 2$ due to $p_0 >1$ and for $\ell \in \N$ the only possible roots of the derivative of $h_\ell^{\scriptscriptstyle (1)}$ in $t \in (0,\infty)$ are $t=1$ and $t = \exp\{\ell/(p_0 -1)\}$. Thus, we obtain for the supremum in \eqref{erstsuptab}
	\[
	\sup_{t \geq 2} \Big| \Big( t^{1-p(y)} \Big)^{(\ell)} (\gseqi_0) \Big| \leq  |p_1|^\ell \max\Big\{2^{1-p_0} \log(2)^\ell, \Big(\frac{\ell}{p_0-1} \Big)^\ell e^{-\ell} \Big\} \leq K(K\ell)^\ell
	\]
	for each $\ell \in \N_0$, because $\lim_{t \to \infty} h_\ell^{\scriptscriptstyle (1)}(t) =0$. Similarly, we have for $\ell \in \N_0$, $\dfi >0$
	\[
	\Big( t^{\hat p- \bar \beta(y)} \Big)^{(\ell)}(\gseqi_0) = t^{\hat p - \beta_0}(-\beta_1 \log(\dfi))^\ell
	\]
	and for $\ell \in \N_0$ let $h_\ell^{\scriptscriptstyle (2)}:(0,1] \to \R$ be defined by $h_\ell^{\scriptscriptstyle (2)}(t) = t^{\hat p - \beta_0} (\log(\dfi))^\ell$. For $\ell \in \N$ the only possible roots in $(0,1]$ of the derivative of $h_\ell^{\scriptscriptstyle (2)}$ are $t=1$ and $t=\exp\{-\ell/(\hat p - \beta_0)\}$. As a consequence, we obtain for each $\ell \in \N_0$ for the supremum in \eqref{zweisuptab}
	\[
	\sup_{t \in [0,1]} \Big| \Big( t^{\hat p-\bar \beta(y)} \Big)^{(\ell)} (\gseqi_0) \Big| \leq |\beta_1|^\ell\Big( \frac{\ell}{\hat p -\beta_0} \Big)^\ell e^{-\ell} \leq K(K\ell)^\ell,
	\]
	because $\lim_{t \to 0} h_\ell^{\scriptscriptstyle (2)}(t) =0$. Notice that for $t=0$ the function $y \mapsto t^{\hat p - \bar \beta(y)}$ is zero constant and for $\ell =0$ the function $t \mapsto t^{\hat p - \beta_0}$ is bounded by $1$ on $[0,1]$ due to $\hat p > \beta_0$.
	\item The expansion \eqref{DcrhoLpexpan} can be deduced along the same lines as in step (3) of the proof of the results in Example 2.3 and Example 4.6(2) in \cite{HofVetDet17} by using \eqref{barNExpan} and \eqref{Nkbound} instead of their equations (6.58) and (6.61). Furthermore, due to expansion \eqref{DcrhoLpexpan} the quantity defined in \eqref{entchanpoi} is clearly given by $\gseqi_0$. \qed
\end{enumerate}

\vspace{.5cm}

{\bf Acknowledgements}
This work has been supported by the
Collaborative Research Center ``Statistical modeling of nonlinear
dynamic processes" (SFB 823, Project A1) and the Research Training Group ``High-dimensional phenomena in probability -- fluctuations and discontinuity" (RTG 2131) of the German Research Foundation (DFG).

\bibliographystyle{apalike} 
{\setlength{\bibsep}{0.1pt}
\begin{small}
\addcontentsline{toc}{chapter}{Bibliography}
\bibliography{biblio}
\end{small}
}

\newpage
\appendix{
\noindent
\LARGE \bf Supplement: Proofs and technical details}

\section{Proof of Theorem \ref{ConvThm}}
\label{ssecProTh61}
\def\theequation{A.\arabic{equation}}
\setcounter{equation}{0}

\subsection{Main steps in the proof} 
In order to prove Theorem \ref{ConvThm} we divide the process $G_\rho^{\scriptscriptstyle (n)}$ into two parts which correspond to large and small jumps of the underlying process $X^{\scriptscriptstyle (n)}$, respectively. To this end we choose an auxiliary function $\Psi \colon \R_+ \rightarrow \R$ which is $\mathcal C^{\infty}$ and satisfies $\ind_{[1, \infty)}(\taili) \leq \Psi(\taili) \leq \ind_{[1/2, \infty)}(\taili)$ for all $\taili \in \R_+$. For $\alpha >0$ define $\Psi_{\alpha} \colon \R \rightarrow \R$ via $\Psi_{\alpha}(\taili) = \Psi(|\taili|/\alpha)$ and let $\Psi_\alpha^\circ \colon \R \rightarrow \R$ be the function $\Psi^\circ_\alpha(\taili) = 1 - \Psi_{\alpha}(\taili)$. 

For the function $\rho$ we define $\rho_{\alpha}(\taili) = \rho(\taili) \Psi_{\alpha}(\taili)$ and $\rho_{\alpha}^{\circ}(\taili) = \rho(\taili) \Psi^{\circ}_{\alpha}(\taili)$. Furthermore, let 
\begin{align}
\label{ImpQuantDef}
\chi_\dfi^{(\alpha)}(\taili) = \rho(\taili) \Psi_{\alpha}(\taili) \Indit(\taili) \quad \text{ and } \quad \chi_\dfi^{\circ (\alpha)} (\taili) = 
\rho(\taili) \Psi^{\circ}_{\alpha}(\taili) \Indit(\taili),
\end{align}
for $\dfi,\taili \in \R$ and define the following empirical processes:
\begin{align*}
G_{\rho,n}^{(\alpha)}(\gseqi,\dfi) &= \sqrt{n \Delta_n} \Big\{ \frac{1}{n \Delta_n} \sum \limits_{i=1}^{\ip{n\gseqi}} \chi_\dfi^{(\alpha)}(\Deli X^{(n)}) \Truniv 
- N_{\rho_{\alpha}}(g^{(n)};\gseqi,\dfi) \Big\}, \\ 
G_{\rho,n}^{\circ (\alpha)}(\gseqi,\dfi) &= \sqrt{n \Delta_n} \Big\{ \frac{1}{n \Delta_n} \sum \limits_{i=1}^{\ip{n\gseqi}} \chi_\dfi^{\circ (\alpha)}(\Deli X^{(n)}) \Truniv - N_{\rho^{\circ}_{\alpha}}(g^{(n)};\gseqi,\dfi) \Big\}. \nonumber
\end{align*}
Then, of course, we have $G_{\rho}^{(n)}(\gseqi,\dfi) = G_{\rho,n}^{(\alpha)}(\gseqi,\dfi) + G_{\rho,n}^{\circ (\alpha)}(\gseqi,\dfi)$.
We provide several auxiliary results about the asymptotic properties of the processes 
$G_{\rho,n}^{(\alpha)}$ and $  G_{\rho,n}^{\circ (\alpha)}$
     which will be proved in Section  \ref{tedet1}.
      The first one is concerned with the behaviour of the large jumps, i.e.\ it holds for $G_{\rho,n}^{(\alpha)}$ and a fixed $\alpha > 0$.

\begin{lemma} \label{step1}
	If Assumption \ref{EasierCond} is satisfied, we have weak convergence
	\[G_{\rho,n}^{(\alpha)} \weak \Gb_{\rho_{\alpha}}\]
	in $\linner$ for each fixed $\alpha >0$, where $\Gb_{\rho_{\alpha}}$ denotes a tight centered Gaussian process with covariance function 
	\[H_{\rho_{\alpha}}((\gseqi_1,\dfi_1);(\gseqi_2,\dfi_2)) = \int_0^{\mingi} \int_{- \infty}^{\mindfi} \rho_{\alpha}^2(\taili) g_0(y,d\taili)dy.\]
	The sample paths of $\Gb_{\rho_{\alpha}}$ are almost surely uniformly continuous with respect to the semimetric
	\[d_{\rho_{\alpha}}((\gseqi_1,\dfi_1);(\gseqi_2,\dfi_2)) = \Big\{ \int_0^{\gseqi_1} \int_{\mindfi}^{\maxdfi}\rho_{\alpha}^2(\taili) g_0(y,d\taili) dy + \int_{\gseqi_1}^{\gseqi_2} \int_{-\infty}^{\dfi_2} \rho_\alpha^2(\taili) g_0(y,d\taili) dy \Big\}^{1/2}\]
	for $\gseqi_1 \leq \gseqi_2$.
\end{lemma}

The general idea behind the proof of Lemma \ref{step1} is to replace the increments of the underlying process $X^{\scriptscriptstyle (n)}$ by increments of pure jump It\=o semimartingales.  Precisely, let $\mu^{\scriptscriptstyle (n)}$ be the Poisson random measure associated with the jumps of $X^{\scriptscriptstyle (n)}$. Then we consider
\begin{align}
\label{LnDefEq}
L^{(n)} = \big(\taili \Trunx\big) \star \mu^{(n)}
\end{align}
with the truncation $v_n = \gamma \Delta_n^{\ovw}$ as above. The main advantage of the processes $L^{\scriptscriptstyle (n)}$ is that they have deterministic characteristics. Therefore, their increments are independent (see Theorem II.4.15 in \cite{JacShi02}) and we can use a central limit theorem for triangular arrays of independent stochastic processes from \cite{Kos08} to prove weak convergence of
\begin{multline}
\label{EmpPrYDefEqn}
Y_f^{(n)}(\gseqi,\dfi) =   \frac{1}{\sqrt{n \Delta_n} } \sum \limits_{i=1}^{\ip{n\gseqi}} \big \{
f(\Deli L^{(n)}) \Indit (\Deli L^{(n)})  
- \Eb (f(\Deli L^{(n)}) \Indit (\Deli L^{(n)}))\big \}
\end{multline}
to $\Gb_{\rho_{\alpha}}$, where  $(\gseqi,\dfi) \in \netir$ and where $f \colon \R \rightarrow \R$ is a bounded continuous function, for which we plug in $\rho_{\alpha}$ and $\rho_\alpha^\circ$ later. In order to prove Lemma \ref{step1} we need to ensure that the distance between $Y_{\rho_{\alpha}}^{\scriptscriptstyle (n)}$ and $G_{\rho,n}^{\scriptscriptstyle (\alpha)}$ is small. To this end, our next claim shows that the bias due to estimating $(n\Delta_n)^{-1} \sum_{i=1}^{\ip{n\gseqi}} \Eb (\rho_\alpha(\Deli L^{(n)}) \Indit ($ $\Deli L^{(n)}))$ instead of $N_{\rho_\alpha}(g^{\scriptscriptstyle (n)}; \gseqi,\dfi)$ is small compared to the rate of convergence. In order to state the result recall that for a real-valued non-negative function $f: \Xi \to \R_+$ on a measure space $(\Xi,\Bc,\vartheta)$ the essential supremum with respect to $\vartheta$ is given by
$
\vartheta-\text{ess sup}_{x \in \Xi}(f) = \inf_{B \in \Bc, \vartheta(B)=0} \sup_{x \in \Xi\setminus B} f(x).
$
Moreover, recall that $\lambda_1$ denotes the restriction of the one-dimensional Lebesgue measure to $[0,1]$.

\begin{proposition}
	\label{BiasAbschProp}
	Let $(\mu^{\scriptscriptstyle (n)})_{n\in\N}$ be a sequence of Poisson random measures with predictable compensators $\bar \mu^{\scriptscriptstyle (n)}(ds,d\taili) = \nu_s^{\scriptscriptstyle (n)}(d\taili)ds$ such that \eqref{RescAss2} is satisfied for each $n \in \N$ with a null sequence $\Delta_n >0$ and a sequence of transition kernels $g^{\scriptscriptstyle (n)}$ from $([0,1], $ $\Bb([0,1]))$ into $(\R, \Bb)$ with
	\[\lambda_1-\mathrm{ess~sup}_{y\in[0,1]} \Big(\int (1 \wedge |\taili|^{\beta})g^{(n)}(y,d\taili) \Big) \leq K\] 
	for each $n\in\N$ and some $\beta \in [0,2]$, $K>0$. Furthermore, let $f \colon \R \rightarrow \R$ be a bounded Borel measurable function satisfying $|f(\taili)| \leq K|\taili|^p$ on a neighbourhood of $0$ for some $K>0$,  $p\ge \beta$. Then if $v_n>0$ is a null sequence and $L^{\scriptscriptstyle (n)}$ is defined as in \eqref{LnDefEq} we have
	\begin{multline} \label{BiasAbschEqn}
	\sup_{i=1,\ldots,n}\sup \limits_{\dfi \in \overline \R} \left| \Eb \left\{ f\big(\Deli L^{(n)}\big) \Indit \big(\Deli L^{(n)}\big) \right\}- n\Delta_n \big( N_f(g^{(n)};i/n,\dfi) - N_f(g^{(n)};(i-1)/n,\dfi)\big) \right| \\= O(\Delta_n^2 v_n^{-2\beta} + \Delta_n v_n^{p-\beta}),
	\end{multline}
	with $\overline \R = \R \cup \lbrace - \infty, + \infty \rbrace$. 
\end{proposition}

\noindent
The following proposition establishes the desired weak convergence of the process $Y_f^{(n)}$.

\begin{proposition}
	\label{LevyCLTProp}
	Suppose Assumption \ref{EasierCond} is satisfied and let $f \colon \R \rightarrow \R$ be a continuous function with $|f(\taili)| \leq K(1 \wedge |\taili|^p)$ for all $\taili \in \R$ and some $K>0$. Then the processes $Y_f^{\scriptscriptstyle (n)}$ from \eqref{EmpPrYDefEqn} converge weakly in $\linner$ to the tight mean zero Gaussian process $\Gb_f$ from Lemma \ref{step1}, that is
$Y^{(n)}_f \weak \Gb_f.$
\end{proposition}

\noindent
In order to obtain the result from Theorem \ref{ConvThm} the following lemma ensures that the limiting process $\Gb_{\rho_{\alpha}}$ converges in a suitable sense as $\alpha \rightarrow 0$.

\begin{lemma} \label{step2}
	Under Assumption \ref{EasierCond} the weak convergence
	$\Gb_{\rho_{\alpha}} \weak \Gb_{\rho}$
	holds in $\linner$ as $\alpha \rightarrow 0$.
\end{lemma}

\noindent
Its proof is a direct consequence of the following result. 

\begin{proposition}
	\label{GalKonvLem}
	Suppose Assumption \ref{EasierCond} is satisfied and let $f_n \colon \R \rightarrow \R$ ($n \in \N_0$) be Borel measurable functions with 
	$|f_n(\taili)| \leq K(1 \wedge |\taili|^p)$ for a constant $K>0$ and all $n \in \N_0$, $\taili \in \R$. Assume further that $f_n(\taili) \rightarrow f_0(\taili)$ converges for all $\taili$ outside a set $B\in\Bb$ such that $[0,1] \times B$ is a $g_0(y,d\taili)dy$-null set. Then we have weak convergence
	$\Gb_{f_n} \weak \Gb_{f_0} $
	in $\linner$ for $n \rightarrow \infty$.
\end{proposition}

\noindent
Our final lemma shows that the contribution due to small jumps are uniformly small as $\alpha$ tends to zero.

\begin{lemma} \label{step3}
	Suppose Assumption \ref{EasierCond} is satisfied. Then for each $\eta >0$ we have:
	\[
	\lim \limits_{\alpha \to 0} \limsup \limits_{n \to \infty} \Prob \Big(\sup \limits_{(\gseqi,\dfi) \in \netir} 
	\big|G_{\rho,n}^{\circ (\alpha)}(\gseqi,\dfi)\big| > \eta \Big) = 0.
	\]
\end{lemma}

 \noindent{\bf Proof of Theorem \ref{ConvThm}} 
 In order to establish weak convergence we use Theorem 1.12.2 in \cite{VanWel96}. It is sufficient to prove
\[
\Eb^{\ast} h(G_{\rho}^{(n)}) \rightarrow \Eb h(\Gb_{\rho})
\]
for each bounded Lipschitz function $h \in \text{BL}_1(\linner)$, where $\text{BL}_1(\Db)$ for a metric space $(\Db,d)$ was introduced in Definition \ref{ConvcondDataDef}. Here, we use that the tight process $\Gb_{\rho}$ is also separable (see Lemma 1.3.2 in \cite{VanWel96}). 

Thus, let $h \in \text{BL}_1(\linner)$ and $\delta >0$. Using  Lemma \ref{step2} and Lemma \ref{step3}
we choose $\alpha >0$ such that
\begin{eqnarray}
\label{ProbAbschEq}
\limsup \limits_{n \rightarrow \infty} \Prob (\sup \limits_{(\gseqi,\dfi) \in \netir} |G_{\rho,n}^{\circ (\alpha)}(\gseqi,\dfi)| > \delta/6) < \delta/12
\\
\label{ErwwHPrAbschEq}
\left| \Eb h(\Gb_{\rho_{\alpha}}) - \Eb h(\Gb_{\rho}) \right| \leq \delta/3.
\end{eqnarray}
\eqref{ProbAbschEq} is possible using Lemma \ref{step3}, and Lemma \ref{step2} allows \eqref{ErwwHPrAbschEq}. For this $\alpha >0$ choose an $N \in \N$ with 
\begin{eqnarray*}
	\big| \Eb^{\ast} h(G_{\rho,n}^{(\alpha)}) - \Eb h(\Gb_{\rho_{\alpha}}) \big| \leq \delta/3,
\end{eqnarray*}
for $n \geq N$. This is possible due to Lemma \ref{step1}. Now, because of the previous inequalities and the Lipschitz property of $h$, we have for $n \in \N$ large enough:
\begin{multline*}
\big| \Eb^{\ast} h(G_{\rho}^{(n)}) - \Eb h(\Gb_{\rho}) \big| \leq \\
\leq \Eb^{\ast} \big| h(G_{\rho}^{(n)}) - h(G_{\rho,n}^{(\alpha)}) \big| + \big| \Eb^{\ast} h(G_{\rho,n}^{(\alpha)}) - \Eb h(\Gb_{\rho_{\alpha}}) \big| + \big| \Eb h(\Gb_{\rho_{\alpha}}) - \Eb h(\Gb_{\rho}) \big| < \delta.
\end{multline*}
\qed 
 \subsection{Proof of auxiliary results } \label{tedet1}

\textbf{Proof of Proposition \ref{GalKonvLem}.} In order to show weak convergence we want to use Theorem 1.5.4 and Theorem 1.5.7 in \cite{VanWel96}. To this end, we prove asymptotical uniform $d$-equicontinuity in probability of $\Gb_{f_n}$ for some suitable semimetric $d$  on $\netir$ with Theorem 2.2.4 in this reference. 

First, recall that $\Gb_{f_n}$ are tight centered Gaussian processes in $\linner$ with covariance function
\[ H_{f_n}((\gseqi_1,\dfi_1);(\gseqi_2,\dfi_2)) = \int_0^{\mingi} \int_{- \infty}^{\mindfi} f_n^2(\taili) g_0(y,d\taili)dy, \]
and their sample paths are almost surely uniformly continuous with respect to the semimetric
\[ d_{f_n}((\gseqi_1,\dfi_1);(\gseqi_2,\dfi_2)) = \Big\{ \int_0^{\gseqi_1} \int_{\mindfi}^{\maxdfi}f_n^2(\taili) g_0(y,d\taili) dy + \int_{\gseqi_1}^{\gseqi_2} \int_{-\infty}^{\dfi_2} f_n^2(\taili) g_0(y,d\taili) dy \Big\}^{1/2}, \]
for $\gseqi_1 \leq \gseqi_2$. Due to Lemma \ref{8MomCalLem} in Appendix \ref{appE} we obtain for the $L^8$-norm
\begin{align}
\label{8NormCal}
\|\Gb_{f_n}(\gseqi_1,\dfi_1) - \Gb_{f_n}(\gseqi_2,\dfi_2) \|_8 = 105^{\frac 18} d_{f_n}((\gseqi_1,\dfi_1);(\gseqi_2,\dfi_2)),
\end{align}
for each $n \in \N$ and $(\gseqi_1,\dfi_1),(\gseqi_2,\dfi_2) \in \netir$. Additionally, the convex, non-decreasing, non-zero function $\varphi(x) = x^8$ clearly satisfies $\varphi(0)=0$ and $\limsup_{x,y \rightarrow \infty} \varphi(x)\varphi(y)/\varphi(cxy) < \infty$ for some constant $c>0$. Furthermore, by Lemma \ref{Gfseparabel} in Appendix \ref{appE} the process $\Gb_{f_n}$ is separable for each $n \in \N_0$ in the sense of Theorem 2.2.4 in \cite{VanWel96}. Thus, this theorem can be applied and due to \eqref{8NormCal} it yields a constant $K>0$, which does not depend on $n \in \N_0$, such that for all $\zeta,\delta >0$
\begin{equation}
\label{ChainingAbsch}
\big \| \sup \limits_{d_{f_n}((\gseqi_1,\dfi_1);(\gseqi_2,\dfi_2)) \leq \delta} \big| \Gb_{f_n}(\gseqi_1,\dfi_1) - \Gb_{f_n}(\gseqi_2,\dfi_2) \big| \big\|_8  \leq K \Big\{ \int_0^\zeta \big(D(\eps,d_{f_n})\big)^{\frac 18} d\eps + \delta \big(D(\zeta,d_{f_n})\big)^{\frac 14} \Big\},
\end{equation}
where $D(\eps,d_{f_n})$ denotes the packing number of $\netir$ with respect to the semimetric $d_{f_n}$ at distance $\eps$. According to Lemma \ref{PackNumAb} 
in Appendix \ref{appE} we have $D(\eps, d_{f_n}) \leq K/ \eps^4$ for every $n \in \N_0$, where $K>0$ does not depend on $n\in\N_0$. Therefore, with \eqref{ChainingAbsch} we conclude that there exists a $K>0$ which is independent of $n$ such that 
\begin{align}
\label{ChainaIntEq}
\big \| \sup \limits_{d_{f_n}((\gseqi_1,\dfi_1);(\gseqi_2,\dfi_2)) \leq \delta} \big| \Gb_{f_n}(\gseqi_1,\dfi_1) - \Gb_{f_n}(\gseqi_2,\dfi_2) \big| \big\|_8 &\leq K \Big\{ \int_0^\zeta \eps^{-1/2} d\eps + \delta/\zeta \Big\} 
\leq K \big(\zeta^{1/2} + \delta/\zeta\big),
\end{align}
for each $\zeta, \delta>0$ and $n \in \N_0$. Now, for arbitrary $\eps, \eta >0$ and $K>0$ from \eqref{ChainaIntEq} choose a $\zeta >0$ with $2^8 K^8 \zeta^4/\eps^8 < \eta/2$ and for this $\zeta$ choose a $\delta>0$ with $(2^8 K^8 \delta^8)/(\zeta^8 \eps^8) < \eta /2$. Then, due to \eqref{ChainaIntEq} we obtain for each $n \in \N$ with the Markov inequality
\begin{align*}
\Prob \Big( \sup \limits_{d_{f_n}((\gseqi_1,\dfi_1);(\gseqi_2,\dfi_2)) < \delta} &\big| \Gb_{f_n}(\gseqi_1,\dfi_1) - \Gb_{f_n}(\gseqi_2,\dfi_2) \big| > \eps \Big) \leq 
\frac{K^8(\zeta^{1/2} + \delta/\zeta)^8}{\eps^8} 
\leq \frac{2^8 K^8}{\eps^8}  \big (  \zeta^4 + \frac{ \delta^8}{\zeta^8 } \big ) < \eta.
\end{align*}
Furthermore, $d_{f_n}$ converges uniformly to $d_{f_0}$ by Lebesgue's dominated convergence theorem. Thus, $\Gb_{f_n}$ is asymptotically uniformly $d_{f_0}$-equicontinuous in probability, because for each $\eps,\eta >0$ we have 
\begin{align*}
\limsup \limits_{n \rightarrow \infty} \Prob \Big( \sup \limits_{d_{f_0}((\gseqi_1,\dfi_1);(\gseqi_2,\dfi_2)) < \delta/2} &\big| \Gb_{f_n}(\gseqi_1,\dfi_1) - \Gb_{f_n}(\gseqi_2,\dfi_2) \big| > \eps \Big) < \eta.
\end{align*}
 Moreover, Lemma \ref{PackNumAb} in Appendix \ref{appE}  
also shows that $(\netir, d_{f_0})$ is totally bounded. Trivially, the marginals of $\Gb_{f_n}$ converge to the corresponding marginals of $\Gb_{f_0}$, because these are centered multivariate normal distributions and their covariance functions converge again by Lebesgue's dominated convergence theorem. Therefore, the desired result holds due to Theorem 1.5.4 and Theorem 1.5.7 in \cite{VanWel96}.
\qed

\medskip

\noindent 
\textbf{Proof of Proposition \ref{BiasAbschProp}.}
Let $\overline F_n = \lbrace \taili \colon |\taili| > v_n \rbrace$, $\tilde N^{(n)} = \ind_{\overline F_n}(\taili) \star \mu^{(n)}$ and let $i \in \{1,\ldots,n\}$ be fixed in the entire proof. According to Proposition II.1.14 in \cite{JacShi02} for each $n \in \N$ there exist a thin random set $D_n$ with an exhausting sequence of stopping times $(T^{\scriptscriptstyle (n)}_m)_{m \in \N}$ and an $\R$-valued optional process $\xi^{\scriptscriptstyle (n)}$ such that
\begin{equation}
\label{munthinrset}
\mu^{(n)}(\omega; ds,d\taili) = \sum_{m \in \N} \epsilon_{\big(T^{\scriptscriptstyle (n)}_m(\omega),\xi^{\scriptscriptstyle (n)}_{T^{\scriptscriptstyle (n)}_m(\omega)}(\omega)\big)}(ds,d\taili),
\end{equation}
where $\epsilon_{(s,x)}$ denotes the Dirac measure with mass in $(s,x)\in \R_+\times \R$. Furthermore, due to Lemma \ref{PRMProsevjb} $\tilde N_{t_2}^{(n)} - \tilde N_{t_1}^{(n)}$ follows a Poisson distribution for $0 \leq t_1 \leq t_2$ and the sets $\tilde A_n^{(i)} \defeq \big\{\tilde N_{i\Delta_n}^{(n)} - \tilde N_{(i-1)\Delta_n}^{(n)} \leq 1 \big\}$
satisfy
\begin{equation}
\label{ProbtiAO}
\Prob\big((\tilde A_n^{(i)})^C\big) = O(\Delta_n^2 v_n^{-2\beta}),
\end{equation}
where $M^C$ denotes the complement of a set $M$. Thus, we calculate for $n \in \N$ large enough
\begin{align*} 
\gamma^{(i,\dfi)}_n :=&\Big| \Eb \left\{ f(\Deli L^{(n)}) \Indit (\Deli L^{(n)}) \right\}- n\Delta_n \big( N_f(g^{(n)};i/n,\dfi) - N_f(g^{(n)};(i-1)/n,\dfi)\big) \Big| \\
= &\Big| \int_{\tilde A_n^{(i)}} f(\Deli L^{(n)}) \Indit (\Deli L^{(n)}) d\Prob + \int_{(\tilde A_n^{(i)})^C} f(\Deli L^{(n)}) \Indit (\Deli L^{(n)}) d\Prob -\\
&\hspace{75mm}- n\Delta_n \int_{(i-1)/n}^{i/n} \int_{-\infty}^\dfi f(\taili) g^{(n)}(y,d\taili) dy \Big| \\
\leq &K\Delta_n^2 v_n^{-2\beta} + \Big| \int_{\{\tilde N_{i\Delta_n}^{(n)} - \tilde N_{(i-1)\Delta_n}^{(n)} =1\}} f(\Deli L^{(n)}) \Indit (\Deli L^{(n)}) d\Prob -\\
&\hspace{75mm}- n\Delta_n \int_{(i-1)/n}^{i/n} \int_{-\infty}^\dfi f(\taili) g^{(n)}(y,d\taili) dy \Big|,
\end{align*}
where the inequality above follows because of two reasons, first \eqref{ProbtiAO} as well as the fact that $f$ is bounded lead to the term $K\Delta_n^2 v_n^{-2\beta}$ and secondly for each $\omega \in \big\{\tilde N_{\scriptscriptstyle i\Delta_n}^{\scriptscriptstyle (n)} - \tilde N_{\scriptscriptstyle (i-1)\Delta_n}^{\scriptscriptstyle (n)} =0 \big\}$ and $m \in \N$ we have $\big(T^{\scriptscriptstyle (n)}_m(\omega),\xi^{\scriptscriptstyle (n)}_{\scriptscriptstyle T^{\scriptscriptstyle (n)}_m(\omega)}(\omega)\big) \notin ((i-1)\Delta_n,i\Delta_n] \times \overline F_n$
such that $\Deli L^{\scriptscriptstyle (n)}(\omega) = 0$ and thus $f(\Deli L^{\scriptscriptstyle (n)}(\omega))=0$ by the assumptions on $f$. However, $\big(T^{\scriptscriptstyle (n)}_m(\omega),\xi^{\scriptscriptstyle (n)}_{\scriptscriptstyle T^{\scriptscriptstyle (n)}_m(\omega)}(\omega)\big) \in ((i-1)\Delta_n,i\Delta_n] \times \overline F_n$ holds for exactly one $m \in \N$ if $\omega \in \big\{\tilde N_{\scriptscriptstyle i\Delta_n}^{\scriptscriptstyle (n)} - \tilde N_{\scriptscriptstyle (i-1)\Delta_n}^{\scriptscriptstyle (n)} =1 \big\}$.
This observation yields the following bound
\begin{align}
\label{secgammaabeq}
&\gamma_n^{(i,\dfi)} \leq \Big| \int_{\{\tilde N_{i\Delta_n}^{(n)} - \tilde N_{(i-1)\Delta_n}^{(n)} =1\}} \int_{(i-1)\Delta_n}^{i\Delta_n} \int_{-\infty}^\dfi f(\taili)  \Trunx \mu^{(n)}(\omega;ds,d\taili) \Prob(d\omega) - \nonumber \\
&\hspace{6cm}- n\Delta_n \int_{(i-1)/n}^{i/n} \int_{-\infty}^\dfi f(\taili) g^{(n)}(y,d\taili) dy \Big| +K\Delta_n^2 v_n^{-2\beta}  \nonumber\\
&\hspace{5mm}\leq \Big| \int_\Omega \int_{(i-1)\Delta_n}^{i\Delta_n} \int_{-\infty}^\dfi f(\taili) \Trunx \mu^{(n)}(\omega;ds,d\taili) \Prob(d\omega) -\nonumber\\
&\hspace{47mm}- n\Delta_n \int_{(i-1)/n}^{i/n} \int_{-\infty}^\dfi f(\taili) g^{(n)}(y,d\taili) dy \Big| + K\Delta_n^2 v_n^{-2\beta} + \delta^{(i,\dfi)}_n,
\end{align}
with 
\begin{align}
\label{deltaitdef}
\delta^{(i,\dfi)}_n = \Big| \int_{\{\tilde N_{i\Delta_n}^{(n)} - \tilde N_{(i-1)\Delta_n}^{(n)} \geq 2\}} \int_{(i-1)\Delta_n}^{i\Delta_n} \int_{-\infty}^\dfi f(\taili) \Trunx \mu^{(n)}(\omega;ds,d\taili) \Prob(d\omega) \Big|.
\end{align}
We apply the defining relation of the predictable compensator of an optional $\Pc'$-$\sigma$-finite random measure. But notice that it cannot be guaranteed that the integrand in the stochastic integral with respect to $\mu^{\scriptscriptstyle (n)}$ in the first line of \eqref{secgammaabeq} is $\Pc'$-measurable. Therefore, we treat the leading term after the last inequality sign in \eqref{secgammaabeq} and $\delta^{\scriptscriptstyle (i,\dfi)}_n$ separately. However, the integrand $f(\taili) \Indit(\taili) \Trunx \ind_{((i-1)\Delta_n,i\Delta_n]}(s)$ on the right-hand side of \eqref{secgammaabeq} is $\Pc'$-measurable. Thus, Theorem II.1.8 in \cite{JacShi02} yields 
\begin{align*}
\gamma_n^{(i,\dfi)} &\leq K \Delta_n^2 v_n^{-2\beta} + \delta_n^{(i,\dfi)} + \Big| \int_{(i-1)\Delta_n}^{i\Delta_n} \int_{-\infty}^\dfi f(\taili)  \Trunx \nu_s^{(n)}(d\taili) ds -\\ 
&\hspace{7cm}- n\Delta_n \int_{(i-1)/n}^{i/n} \int_{-\infty}^\dfi f(\taili) g^{(n)}(y,d\taili) dy \Big| \\
&= K \Delta_n^2 v_n^{-2\beta} + \delta_n^{(i,\dfi)} + \Big| n\Delta_n \int_{(i-1)/n}^{i/n} \int_{-\infty}^\dfi f(\taili)  
\ind_{\{|\taili| \leq v_n\}} g^{(n)}(y,d\taili) dy \Big|.
\end{align*} 
Now, because of $|f(\taili)| \leq K|\taili|^p$ on a neighbourhood of $0$, the above display yields for $n \in \N$ large enough
\begin{align}
\label{gammanitAbeq}
\gamma_n^{(i,\dfi)} &\leq K \Delta_n^2 v_n^{-2\beta} + \delta_n^{(i,\dfi)} + K\Delta_n v_n^{p-\beta} n \int_{(i-1)/n}^{i/n} \int (1 \wedge |\taili|^\beta) g^{(n)}(y,d\taili)dy \nonumber \\
&\leq K \Delta_n^2 v_n^{-2\beta} + \delta_n^{(i,\dfi)} + K \Delta_n v_n^{p-\beta}.
\end{align} 
Finally, \eqref{munthinrset} and the assumption that $f$ is bounded by some constant $K>0$ gives an estimate for $\delta_n^{(i,\dfi)}$ from \eqref{deltaitdef}
\begin{align*}
\delta_n^{(i,\dfi)} &\leq  \sum \limits_{\ell=2}^\infty \int_{\big\{\tilde N_{i\Delta_n}^{(n)} - \tilde N_{(i-1)\Delta_n}^{(n)} = \ell\big\}} \sum \limits_{m \in \N} \big| f\big(\xi^{(n)}_{T^{(n)}_m}\big) \big| \Indit\big(\xi^{(n)}_{T^{(n)}_m}\big) \ind_{((i-1)\Delta_n,i\Delta_n]}\big(T^{(n)}_m\big) \ind_{\overline F_n}\big(\xi^{(n)}_{T^{(n)}_m}\big) d \Prob  \\
&\leq \sum\limits_{\ell=2}^\infty K\ell \times \Prob\Big(\tilde N_{i\Delta_n}^{(n)} - \tilde N_{(i-1)\Delta_n}^{(n)} = \ell\Big),
\end{align*}
because for $\omega \in \big\{\tilde N_{i\Delta_n}^{(n)} - \tilde N_{(i-1)\Delta_n}^{(n)} = \ell\big\}$ we have $\#\big\{m \in \N \mid \big(T^{\scriptscriptstyle (n)}_m(\omega),\xi^{\scriptscriptstyle (n)}_{\scriptscriptstyle T^{\scriptscriptstyle (n)}_m(\omega)}(\omega)\big) \in ((i-1)\Delta_n,i\Delta_n] \times \overline F_n \big\} = \ell$, where $\# M$ denotes the cardinality of a set $M$. With Lemma \ref{PRMProsevjb} 
in in Appendix \ref{appD}  and the previous inequality we obtain
\begin{align*}
\delta_n^{(i,\dfi)} \leq  \exp\big\{-\zeta_i^{(n)}\big\} \sum \limits_{\ell=2}^\infty K\ell \times \frac{\big(\zeta_i^{(n)}\big)^\ell}{\ell!} 
\leq K \big(\zeta_i^{(n)}\big)^2,
\end{align*}
with
\begin{align*}
\zeta_i^{(n)} = n \Delta_n \int_{(i-1)/n}^{i/n} \int_{\overline F_n} g^{(n)}(y,d\taili) dy 
\leq n \Delta_n v_n^{-\beta} \int_{(i-1)/n}^{i/n} \int (1 \wedge |\taili|^\beta) g^{(n)}(y,d\taili) dy \leq K \Delta_n v_n^{-\beta},
\end{align*}
for $n \in \N$ large enough. Thus, $\delta_n^{(i,\dfi)} \leq K \Delta_n^2 v_n^{-2\beta}$ holds and \eqref{gammanitAbeq} yields \eqref{BiasAbschEqn}, because neither of the bounds for $\gamma_n^{(i,\dfi)}$ or $\delta_n^{(i,\dfi)}$ depends on $i$ or $\dfi$.
\qed

\medskip

\noindent 
\textbf{Proof of Proposition \ref{LevyCLTProp}.} The processes $Y^{(n)}_f$ have the form
\[Y^{(n)}_f (\omega;(\gseqi,\dfi)) = \sum \limits_{i=1}^{m_n} \left\{ g_{ni}(\omega;(\gseqi,\dfi)) - \Eb (g_{ni}(\cdot;(\gseqi,\dfi))) \right\},\]
with $m_n = n$ and the triangular array $\lbrace g_{ni}(\omega;(\gseqi,\dfi)) \mid n \in \N; i = 1, \ldots, n ; (\gseqi,\dfi) \in \netir \rbrace$ of processes
\[g_{ni}(\omega ; (\gseqi,\dfi)) = \frac{1}{\sqrt{n \Delta_n}} f(\Deli L^{(n)}(\omega)) \Indit( \Deli L^{(n)}(\omega)) \ind_{\{i \leq \ip{n\gseqi}\}},\]
which is independent within rows, because $L^{(n)}$ has independent increments as it has deterministic characteristics (see Theorem II.4.15 in \cite{JacShi02}). Thus, by Theorem 11.16 in \cite{Kos08}, the proof is complete once we can show the following six conditions of the triangular array $\lbrace g_{ni} \rbrace$ (see for instance \cite{Kos08} for the notions of AMS and manageability):
\begin{enumerate}[(A)]
	\item $\lbrace g_{ni} \rbrace$ is almost measurable Suslin (AMS);
	\label{Liste1}
	\item $\lbrace g_{ni} \rbrace$ is manageable with envelopes $\lbrace G_{ni} \mid n \in \N; i = 1, \ldots,n \rbrace$ which are also independent within rows. Here, we set 
	$G_{ni} = \frac{K}{\sqrt{n \Delta_n}} (1 \wedge \left| \Deli L^{(n)} \right|^{p})$ with $K >0$ such that  
	$|f(x)| \leq K(1 \wedge |\taili|^{p})$;
	\label{Liste2}
	\item $H_f((\gseqi_1,\dfi_1);(\gseqi_2,\dfi_2)) = \lim \limits_{n \rightarrow \infty} \Eb \big \{ Y^{(n)}_f(\gseqi_1,\dfi_1) Y^{(n)}_f(\gseqi_2,\dfi_2) \big \}$
	for all $(\gseqi_1,\dfi_1),(\gseqi_2,\dfi_2) \in \netir$, with $H_f$ defined in \eqref{Hrhodef}; 
	\label{Liste3} 
	\item $\limsup \limits_{n \rightarrow \infty} \sum \limits_{i=1}^n \Eb G_{ni}^2 < \infty$; \label{Liste4}
	\item $\lim \limits_{n \rightarrow \infty} \sum \limits_{i=1}^n \Eb G_{ni}^2 \ind_{\lbrace G_{ni} > \epsilon \rbrace} =0$ for each $\epsilon >0$;
	\label{Liste5}
	\item For $(\gseqi_1,\dfi_1),(\gseqi_2,\dfi_2) \in \netir$  define
\[d_f^{(n)}((\gseqi_1,\dfi_1);(\gseqi_2,\dfi_2)) = \Big \{ \sum \limits_{i=1}^n \Eb \left| g_{ni}(\cdot; (\gseqi_1,\dfi_1)) - g_{ni}( \cdot;(\gseqi_2,\dfi_2)) \right|^2 \Big \}^{1/2},\]
then the limit 
	\[d_f((\gseqi_1,\dfi_1);(\gseqi_2,\dfi_2)) = \lim \limits_{n \rightarrow \infty} d_f^{(n)}((\gseqi_1,\dfi_1);(\gseqi_2,\dfi_2))\] 
	exists, where $d_f$ is defined in 
 in \eqref{drhodef}. Moreover, for all sequences $((\gseqi^{\scriptscriptstyle (1)}_n,\dfi^{\scriptscriptstyle (1)}_n))_{n \in \N},((\gseqi^{\scriptscriptstyle (2)}_n,\dfi^{\scriptscriptstyle (2)}_n))_{n \in \N} \subset \netir$ with $d_f((\gseqi^{\scriptscriptstyle (1)}_n,\dfi^{\scriptscriptstyle (1)}_n);(\gseqi^{\scriptscriptstyle (2)}_n,\dfi^{\scriptscriptstyle (2)}_n)) \rightarrow 0$ we also have $d_f^{\scriptscriptstyle (n)}((\gseqi^{\scriptscriptstyle (1)}_n,\dfi^{\scriptscriptstyle (1)}_n);(\gseqi^{\scriptscriptstyle (2)}_n,$ $\dfi^{\scriptscriptstyle (2)}_n)) \rightarrow 0$.
	\label{Liste6}
\end{enumerate}

\smallskip

\noindent 
\textit{Proof of  \eqref{Liste1}.} Using Lemma 11.15 in \cite{Kos08} the triangular array $\lbrace g_{ni} \rbrace$ is AMS if it is separable, that is for each 
$n \in \N$ there exists a countable subset $S_n \subset \netir$ such that
\[\Prob^{\ast} \bigg( \sup \limits_{(\gseqi_1,\dfi_1) \in \netir} \inf \limits_{(\gseqi_2, \dfi_2) \in S_n} \sum \limits_{i=1}^n (g_{ni}(\omega; (\gseqi_1,\dfi_1)) - 
g_{ni}(\omega;(\gseqi_2,\dfi_2)))^2 >0 \bigg) = 0.\]
But if we choose $S_n = (\netir) \cap \Q^2$ for all $n \in \N$, we obtain
\[\sup \limits_{(\gseqi_1,\dfi_1) \in \R} \inf \limits_{(\gseqi_2,\dfi_2) \in S_n} \sum \limits_{i=1}^n (g_{ni}(\omega; (\gseqi_1,\dfi_1)) - 
g_{ni}(\omega;(\gseqi_2,\dfi_2)))^2 = 0\]
for each $\omega \in \Omega$ and $n \in \N$.

\smallskip

\noindent 
\textit{Proof of  \eqref{Liste2}.}
$G_{ni}$ are independent within rows since the $L^{(n)}$ have deterministic characteristics. Thus, according to Theorem 11.17 in \cite{Kos08}, it suffices to show that the triangular arrays
\begin{align*}
\{\tilde g_{ni}(\omega;\dfi) \defeq \frac{1}{\sqrt{n \Delta_n}} f(\Deli L^{(n)}(\omega)) \Indit( \Deli L^{(n)}(\omega)) \mid n\in\N; i=1,\ldots,n;\dfi \in \R\},
\end{align*}
and
$
\{\tilde h_{ni}(\omega;\gseqi) \defeq \ind_{\{i \leq \ip{n\gseqi}\}} \mid n\in\N;i=1,\ldots,n; \gseqi\in[0,1]\}
$
are manageable with envelopes $\lbrace G_{ni} \mid n \in \N; i = 1, \ldots,n \rbrace$ and $\{\tilde H_{ni}(\omega) :\equiv 1\mid n\in\N; i=1,\ldots,n\}$, respectively. 
Concerning the first triangular array $\{\tilde g_{ni} \}$ define for $n \in \N$ and $\omega \in \Omega$ the set
\begin{multline*}
\mathcal G_{n \omega} = \bigg\{ \bigg( \frac{1}{\sqrt{n \Delta_n}} f(\Delta_1^n L^{(n)}(\omega)) \Indit(\Delta_1^n L^{(n)}(\omega)), \ldots \\ 
\ldots, \frac{1}{\sqrt{n \Delta_n}} f(\Delta_n^n L^{(n)}(\omega)) \Indit(\Delta_n^n L^{(n)}(\omega)) \bigg) \bigg{|} t \in \R \bigg\} \subset \R^n.
\end{multline*}
These sets are bounded with envelope vector
$G_n(\omega) = (G_{n1}(\omega) , \ldots, G_{nn}(\omega)) \in \R^n.$

For $i_1, i_2 \in \lbrace 1, \ldots, n \rbrace$ the projection 
\begin{multline*}
p_{i_1, i_2}(\mathcal G_{n \omega}) = \bigg \{ \bigg( \frac{1}{\sqrt{n \Delta_n}} f(\Delta_{i_1}^n L^{(n)}(\omega)) \Indit(\Delta_{i_1}^n L^{(n)}(\omega)), \\
\frac{1}{\sqrt{n \Delta_n}} f(\Delta_{i_2}^n L^{(n)}(\omega)) \Indit(\Delta_{i_2}^n L^{(n)}(\omega)) \bigg) \mid t \in \R \bigg \} \subset \R^2
\end{multline*}
onto the $i_1$-th and the $i_2$-th coordinate is an element of the set
\begin{align*}
\bigg\{ \lbrace (0,0) \rbrace, &\lbrace (0,0), (s_{i_1,n}(\omega),0) \rbrace, \lbrace (0,0), (0,s_{i_2,n}(\omega)) \rbrace, 
\lbrace (0,0), (s_{i_1,n}(\omega), s_{i_2,n}(\omega)) \rbrace, \\
\lbrace (0,0), &(s_{i_1,n}(\omega), 0), (s_{i_1,n}(\omega), s_{i_2,n}(\omega)) \rbrace,
\lbrace (0,0), (0, s_{i_2,n}(\omega)), (s_{i_1,n}(\omega), s_{i_2,n}(\omega)) \rbrace \bigg\}.
\end{align*}
with $s_{i_1,n}(\omega) = \frac{1}{\sqrt{n \Delta_n}} f(\Delta_{i_1}^n L^{(n)}(\omega))$ and 
$s_{i_2,n}(\omega) = \frac{1}{\sqrt{n \Delta_n}} f(\Delta_{i_2}^n L^{(n)}(\omega))$. Consequently, in the sense of Definition 4.2 in \cite{Pol90}, for every $s \in \R^2$ no proper coordinate projection of $\mathcal G_{n \omega}$ can surround $s$ and therefore $\mathcal G_{n \omega}$ has a pseudo dimension of at most $1$ (Definition 4.3 in \cite{Pol90}). Thus, by Corollary 4.10 in the same reference, there exist constants $A$ and $W$ which depend only on the pseudodimension such that
\[D_2\big(x \| \alpha \odot G_n(\omega) \|_2, \alpha \odot \mathcal G_{n \omega}\big) \leq A x^{-W} =: \zeta(x),\]
for all $0< x \leq 1$, $n \in \N$, $\omega \in \Omega$ and each rescaling vector $\alpha \in \R^n$ with non-negative entries, where 
$\| \cdot \|_2$ denotes the Euclidean distance on $\R^n$, $D_2$ denotes the packing number with respect to the Euclidean distance and 
$\odot$ denotes coordinate-wise multiplication. Obviously, we have 
$
\int _0^1 \sqrt{\log \zeta(x)} dx < \infty,
$
and therefore the triangular array $\lbrace \tilde g_{ni} \rbrace$ is indeed manageable with envelopes $\lbrace G_{ni} \rbrace$. 

Concerning the triangular array $\{\tilde h_{ni}\}$, we proceed similarly and consider the set
\begin{multline*}
\mathcal H_{n \omega}
\defeq
\lbrace (\tilde h_{n1}(\omega;\theta), \ldots , \tilde h_{nn}(\omega;\theta)) \mid \theta \in [0,1] \rbrace \\
=
\lbrace (0, \ldots , 0) , (1,0, \ldots ,0),(1,1,0, \ldots,0), \ldots, (1, \ldots ,1) \rbrace.
\end{multline*}
Then, for any $i_1,i_2 \in \lbrace 1, \ldots,n \rbrace$, the projection $p_{i_1,i_2}(\mathcal H_{n \omega})$ of $\mathcal H_{n \omega}$ onto the
$i_1$-th and the $i_2$-th coordinate is either $\lbrace (0,0), (1,0) , (1,1) \rbrace$ or $\lbrace (0,0), (0,1) , (1,1) \rbrace$. Therefore, the same reasoning as above shows that $\mathcal H_{n \omega}$ is a set of pseudodimension at most one, whence the triangular
array $\lbrace \tilde h_{ni} \rbrace$ is manageable with envelopes $\lbrace \tilde H_{ni} \rbrace$.

\smallskip

\noindent 
\textit{Proof of  \eqref{Liste3}.} Using independence within rows of the triangular array $\lbrace g_{ni} \rbrace$ we calculate for $(\gseqi_1,\dfi_1),(\gseqi_2,\dfi_2) \in \netir$ as follows:
\begin{multline*}
\Eb \left\{ Y^{(n)}_f(\gseqi_1,\dfi_1) Y^{(n)}_f(\gseqi_2,\dfi_2) \right\} = \frac{1}{n \Delta_n} \sum \limits_{i=1}^{\ip{n(\gseqi_1 \wedge \gseqi_2)}} \bigg\{\Eb \big [ f^2(\Deli L^{(n)}) 
\ind_{(- \infty, \dfi_1 \wedge \dfi_2]}( \Deli L^{(n)}) \big ] - \\
\bigg ( \Eb \big [ f(\Deli L^{(n)}) 
\ind_{(- \infty, \dfi_1]}(\Deli L^{(n)}) \big ] \Eb \big [ f(\Deli L^{(n)}) \ind_{(- \infty, \dfi_2]}(\Deli L^{(n)}) \big ] \bigg ) \bigg\}.
\end{multline*}
Due to Lemma \ref{MomentfgO} we have $\Eb \big [ f(\Deli L^{(n)}) \ind_{(- \infty, \dfi]}(\Deli$ $ L^{(n)}) \big ] = O(\Delta_n)$ for all $\dfi \in\R$ and $i=1,\ldots,n$. Thus, an application of Proposition \ref{BiasAbschProp} yields for some small $\delta >0$ and $n\in\N$ large enough 
\begin{align*}
\Eb \big\{ &Y^{(n)}_f(\gseqi_1,\dfi_1) Y^{(n)}_f(\gseqi_2,\dfi_2) \big\} = \int_0^{\frac{\ip{n(\gseqi_1 \wedge \gseqi_2)}}n} \int_{-\infty}^{\dfi_1 \wedge \dfi_2} f^2(\taili) g^{(n)}(y,d\taili) dy +\\
&\hspace{62mm}+ O\big(\Delta_n v_n^{-2((\beta+\delta)\wedge 2)}+ v_n^{2p-((\beta+\delta) \wedge 2)}\big) + O(\Delta_n) \\
&= \int_0^{\frac{\ip{n(\gseqi_1 \wedge \gseqi_2)}}n} \int_{-\infty}^{\dfi_1 \wedge \dfi_2} f^2(\taili) g_0(y,d\taili) dy +
\frac 1{\sqrt{n\Delta_n}} \int_0^{\frac{\ip{n(\gseqi_1 \wedge \gseqi_2)}}n} \int_{-\infty}^{\dfi_1 \wedge \dfi_2} f^2(\taili) g_1(y,d\taili) dy \\
&\hspace{72mm}+ \int_0^{\frac{\ip{n(\gseqi_1 \wedge \gseqi_2)}}n} \int_{-\infty}^{\dfi_1 \wedge \dfi_2} f^2(\taili) \Rc_n(y,d\taili) dy +o(1),       
\end{align*}
where the final equality above follows using \eqref{gon012Ass3}, as well as $p > \beta$ and $\ovw < 1/(2\beta)$. Furthermore, due to  Assumption \ref{Cond1}\eqref{BlGetCond} and $p>\beta$ the mapping $\big(y \mapsto \int_{-\infty}^{\dfi_1 \wedge \dfi_2} f^2(\taili)g_i(y,d\taili)\big)$ is Lebesgue-almost surely bounded on $[0,1]$ for each $i\in\{0,1,2\}$. Thus, we have
\begin{align*}
\Eb \big\{ Y^{(n)}_f(\gseqi_1,\dfi_1) Y^{(n)}_f(\gseqi_2,\dfi_2) \big\} &=\int \limits_0^{\frac{\ip{n(\gseqi_1 \wedge \gseqi_2)}}n} \int_{-\infty}^{\dfi_1 \wedge \dfi_2} f^2(\taili) g_0(y,d\taili) dy + O\big((n\Delta_n)^{-1/2}\big) + o(1) \\
&= H_f((\gseqi_1,\dfi_1);(\gseqi_2,\dfi_2)) + o(1).
\end{align*}

\smallskip

\noindent 
\textit{Proof of  \eqref{Liste4}.}
Using Proposition \ref{BiasAbschProp} we obtain for some small $\delta >0$
\begin{align*}
\limsup \limits_{n \rightarrow \infty} &\sum \limits_{i=1}^n \Eb G_{ni}^2 = 
\limsup \limits_{n \rightarrow \infty} \frac{K}{n\Delta_n} \sum\limits_{i=1}^n \Eb \Big\{ 1 \wedge \big| \Deli L^{(n)} \big|^{2 p} \Big\} \\
&\leq \limsup \limits_{n \rightarrow \infty} \Big\{ K \int_0^1 \int \big(1 \wedge |\taili|^{2 p}\big) g^{(n)}(y,d\taili) dy +  K\big(\Delta_n v_n^{-2((\beta+\delta)\wedge 2)} + v_n^{2p-((\beta+\delta)\wedge 2)}\big) \Big\} \\
&= K \int_0^1 \int \big(1 \wedge |\taili|^{2 p}\big) g_0(y,d\taili) dy < \infty,
\end{align*}
where the final equality above is a consequence of Assumption \ref{Cond1}\eqref{BlGetCond}, $p>\beta$ and $\ovw < 1/(2\beta)$.

\smallskip

\noindent 
\textit{Proof of  \eqref{Liste5}.} We have $n \Delta_n \rightarrow \infty$. Thus, for $\epsilon >0$, we can choose 
\[N_{\epsilon} = \min \Big\lbrace m \in \N \Big | \frac{K}{\sqrt{n \Delta_n}} \leq \epsilon \text{ for all } n \geq m \Big\rbrace < \infty.\]
So for $n \geq N_{\epsilon}$ the integrand satisfies $G_{ni}^2 \ind_{\lbrace G_{ni} > \epsilon \rbrace} =0$ for all $1 \leq i \leq n$ and this yields the assertion.

\smallskip
\noindent 
\textit{Proof of  \eqref{Liste6}.}
Due to symmetry of the semimetrics let $\gseqi_1 \leq \gseqi_2$ without loss of generality. Then an application of Proposition \ref{BiasAbschProp} gives, for $(\gseqi_1,\dfi_1),(\gseqi_2,\dfi_2) \in \netir$,
\begin{align*}
\big (d_f^{(n)}((\gseqi_1,\dfi_1)&;(\gseqi_2,\dfi_2)) \big)^2 = \sum \limits_{i=1}^n \Eb \left| g_{ni}((\gseqi_1,\dfi_1)) - g_{ni}((\gseqi_2,\dfi_2)) \right|^2 \nonumber \\
&=\frac 1{n\Delta_n} \sum\limits_{i=1}^{\ip{n\gseqi_1}} \Eb f^2(\Deli L^{(n)}) \ind_{(\dfi_1 \wedge \dfi_2, \dfi_1 \vee \dfi_2]}(\Deli L^{(n)}) +\\
&\hspace{35mm}+ \frac 1{n\Delta_n} \sum\limits_{i=\ip{n\gseqi_1}+1}^{\ip{n\gseqi_2}} \Eb f^2(\Deli L^{(n)}) \ind_{(-\infty,\dfi_2]}(\Deli L^{(n)}) \\
&= \int_0^{\frac{\ip{n\gseqi_1}}n} \int_{\dfi_1 \wedge \dfi_2}^{\dfi_1 \vee \dfi_2} f^2(\taili) g^{(n)}(y,d\taili) dy + \int_{\frac{\ip{n\gseqi_1}}n}^{\frac{\ip{n\gseqi_2}}n} \int_{-\infty}^{\dfi_2} f^2(\taili) g^{(n)}(y,d\taili)dy + O(\Delta_n^\alpha) \\
&= \big (d_f((\gseqi_1,\dfi_1);(\gseqi_2,\dfi_2)) \big)^2 + O(\Delta_n^\alpha)
\end{align*}
for an appropriately small $\alpha >0$ 
such that $1/n = o(\Delta_n^\alpha)$ and $(n\Delta_n)^{-1/2} = O(\Delta_n^\alpha)$ according to Assumption \ref{Cond1}\eqref{ObsSchCond7}. Moreover, the $O$-term is uniform in $(\gseqi_1,\dfi_1),(\gseqi_2,\dfi_2) \in \netir$. As a consequence,
\[
\big| d_f^{(n)}((\gseqi_1,\dfi_1);(\gseqi_2,\dfi_2)) - d_f((\gseqi_1,\dfi_1);(\gseqi_2,\dfi_2)) \big| = O(\Delta_n^{\alpha /2})
\]
uniformly as well, because
$
\big| \sqrt a - \sqrt b \big| \leq \sqrt{ |a-b|}
$
holds for arbitrary $a,b \geq 0$. This uniform convergence implies immediately that for deterministic sequences $((\gseqi^{(1)}_n,\dfi^{(1)}_n))_{n \in \N},((\gseqi^{(2)}_n,\dfi^{(2)}_n))_{n \in \N} \subset \netir$ with $d_f((\gseqi^{(1)}_n,$ $\dfi^{(1)}_n);(\gseqi^{(2)}_n,\dfi^{(2)}_n)) \rightarrow 0$ we also have $d_f^{(n)}((\gseqi^{(1)}_n,\dfi^{(1)}_n);(\gseqi^{(2)}_n,\dfi^{(2)}_n)) \rightarrow 0$.

Finally, $d_f$ is in fact a semimetric: Define for $(\gseqi,\dfi) \in \netir$ the random vectors $g_n(\gseqi,\dfi) = (g_{n1}(\gseqi,\dfi), \ldots, g_{nn}(\gseqi,\dfi)) \in \R^n$ and apply first the triangle inequality in $\R^n$ and afterwards the Minkowski inequality to obtain
\begin{align*}
d_f^{(n)}((\gseqi_1,\dfi_1);(&\gseqi_2,\dfi_2)) = \left\{ \Eb \| g_n(\gseqi_1,\dfi_1) - g_n(\gseqi_2,\dfi_2) \|^2_2 \right \}^{1/2} \\
&\leq \left \{ \Eb \big ( \| g_n(\gseqi_1,\dfi_1) - g_n(\gseqi_3,\dfi_3) \|_2 + \| g_n(\gseqi_3,\dfi_3) - g_n(\gseqi_2,\dfi_2) \|_2 \big )^2 \right\}^{1/2} \\
&\leq \left \{ \Eb \|g_n(\gseqi_1,\dfi_1) - g_n(\gseqi_3,\dfi_3) \|^2_2 \right \}^{1/2} + \left \{ \Eb \|g_n(\gseqi_3,\dfi_3) - g_n(\gseqi_2,\dfi_2) \|^2_2 \right \}^{1/2} \\
&= d_f^{(n)}((\gseqi_1,\dfi_1);(\gseqi_3,\dfi_3)) + d_f^{(n)}((\gseqi_3,\dfi_3);(\gseqi_2,\dfi_2)),
\end{align*}
for $(\gseqi_1,\dfi_1),(\gseqi_2,\dfi_2),(\gseqi_3,\dfi_3) \in \netir$ and $n \in \N$, where $\| \cdot \|_2$ denotes the Euclidean norm in $\R^n$. The triangle inequality for $d_f$ follows immediately. 
\qed

\medskip

\noindent
The decomposition below is similar to Step 5 in the proof of Theorem 13.1.1 in \cite{JacPro12} and it will occur frequently in the sequel. 
With the constants from Assumption \ref{Cond1} let $\ell \in \R$ satisfy
\begin{equation*}
1 < \ell < \frac{1}{2\beta \ovw} \wedge (1+ \epsilon) \quad \text{ and also } \quad \ell < \frac{2(p-1) \ovw -1}{2(\beta-1) \ovw} \text{ if } \beta >1,
\end{equation*}
with an $\epsilon >0$ for which Assumption \ref{Cond1}\eqref{ObsSchCond6} holds. Then we set
\begin{equation*}
u_n = (v_n)^{\ell} \quad \text{ and } \quad F_n = \lbrace \taili \colon |\taili| > u_n \rbrace
\end{equation*}
as well as
\begin{align}
\label{SmandlajumpDef}
\tibigj &= (\taili \ind_{F_n}(\taili)) \star \mu^{(n)}, \nonumber \\
\tibigjal &= (\taili \ind_{F_n \cap \lbrace |\taili| \leq \alpha/4 \rbrace }(\taili)) \star \mu^{(n)}, \quad \text{ for } \alpha >0 \nonumber \\
\hatbigjal &= (\taili \ind_{\lbrace |\taili| > \alpha/4 \rbrace}) \star \mu^{(n)}, \quad \text{ for } \alpha >0 \nonumber \\
N_t^n &= (\ind_{F_n} \star \mu^{(n)})_t, \nonumber \\
\tilde X_t^{\prime n} &= X^{(n)}_t - \tibigj_t  \nonumber \\
&= X^{(n)}_0 + \int_0^t b^{(n)}_s ds + \int_0^t \sigma^{(n)}_s dW^{(n)}_s+ \nonumber \\
&\hspace{2cm}+ (\taili \ind_{F_n^C}(\taili)) \star 
(\mu^{(n)} - \bar \mu^{(n)})_t - (\taili \ind_{\lbrace |\taili| \leq 1 \rbrace \cap F_n}(\taili)) \star \bar \mu^{(n)}_t, \nonumber \\
A_i^n	&= \lbrace | \Deli \tilde X^{\prime n} | \leq v_n/2 \rbrace \cap \lbrace \Deli N^n \leq 1 \rbrace.
\end{align}

In the following proofs it is necessary to ensure that with high probability at most one large jump occurs and the increments of the remaining part, that is the quantities $\Deli \tilde X^{\prime n}$, are small. To this end, we show in Lemma \ref{QnConv} in Appendix \ref{appD} for the sets 
\begin{equation}
\label{BnsetsDef}
Q_n = \bigcap \limits_{i=1}^n A_i^n.
\end{equation}
 that $\Prob(Q_n) \rightarrow 1$  as $n \to \infty$ 

\medskip

\noindent 
\textbf{Proof of Lemma \ref{step1}.} Let $\alpha > 0$ be fixed and recall the definition of the processes 
$L^{(n)} = (\taili \Trunx) \star \mu^{(n)}$ in \eqref{LnDefEq}. Due to Proposition \ref{BiasAbschProp} and Proposition \ref{LevyCLTProp} the processes
\[
\tilde Y_{\rho_{\alpha}}^{(n)} (\gseqi,\dfi) = \sqrt{n \Delta_n} \Big\{ \frac{1}{n \Delta_n} \sum \limits_{i=1}^{\ip{n\gseqi}} \chi_\dfi^{(\alpha)}(\Deli L^{(n)}) - N_{\rho_{\alpha}}(g^{(n)};\gseqi,\dfi) \Big\}
\]
converge weakly to $\Gb_{\rho_{\alpha}}$ in $\linner$, because 
\begin{align}
\label{EmpProDiffAbEq}
\sup \limits_{(\gseqi,\dfi) \in [0,1] \times \R} &\Big | \tilde Y_{\rho_{\alpha}}^{(n)} (\gseqi,\dfi) -  Y_{\rho_{\alpha}}^{(n)} (\gseqi,\dfi) \Big | \nonumber \\
&\hspace{-5mm}\leq K \sqrt{n\Delta_n} \Big( \int_{\ip{n\gseqi}/n}^\gseqi \int_{-\infty}^\dfi \rho_\alpha(\taili)g^{(n)}(y,d\taili)dy + \frac{\ip{n\gseqi}}{n\Delta_n} \big(\Delta_n^2 v_n^{-2((\beta+\delta)\wedge 2)} + \Delta_n v_n^{p-((\beta+\delta)\wedge 2)} \big) \Big) \nonumber \\
&\hspace{-5mm}= O\Big(\sqrt{\Delta_n/n} +  \sqrt{n \Delta_n^{3-4\beta\ovw-\delta}} + \sqrt{n \Delta_n^{1+2\ovw(p-\beta) - \delta}} \Big) =o(1)
\end{align}
holds for some small $\delta >0$, by Assumption \ref{Cond1}\eqref{BlGetCond}. The final equality in the display above follows using $1-2\beta\ovw >0$, $p-\beta >1$, as well as  Assumption \ref{Cond1}\eqref{ObsSchCond4} and \eqref{ObsSchCond6}. As a consequence, it suffices to show
\begin{equation}
\label{DiffKonv}
V_\alpha^{(n)} := \frac{1}{\sqrt{n \Delta_n}} \sup \limits_{(\gseqi,\dfi) \in \netir} \Big| \sum \limits_{i=1}^{\ip{n\gseqi}} \Big\{ \chi_\dfi^{(\alpha)} (\Deli X^{(n)}) \Truniv - \chi_\dfi^{(\alpha)}(\Deli L^{(n)}) \Big\} \Big| \stackrel{\Prob}{\longrightarrow} 0.
\end{equation}
According to Lemma \ref{ValnBound} in Appenidx \ref{appD} we have for $n \in \N$ large enough such that $v_n \leq \alpha/4$
\[
V_\alpha^{(n)} \leq C_n(\alpha) + D_n(\alpha),
\]
on $Q_n$ with
\begin{multline*}
C_n(\alpha) = \frac{K}{\sqrt{n \Delta_n}} \sup \limits_{t \in \R} \sum \limits_{i=1}^n \big| \Indit(\Deli \tilde X^{\prime n} + 
\Deli \hatbigjal) - \Indit(\Deli \hatbigjal) \big|  \times \\
\times \ind_{\lbrace | \Deli \hatbigjal | > \alpha/4 \rbrace} 
\ind_{\lbrace | \Deli \tilde X^{\prime n} | \leq v_n/2 \rbrace},
\end{multline*}
\begin{multline}
\label{DnalDef3}
D_n(\alpha) = \frac{1}{\sqrt{n \Delta_n}} \sum \limits_{i=1}^n \big| \rho_{\alpha} (\Deli \tilde X^{\prime n} + \Deli \hatbigjal)
\ind_{\lbrace |  \Deli \tilde X^{\prime n} + \Deli \hatbigjal| > v_n \rbrace} - \\
- \rho_{\alpha}(\Deli \hatbigjal) \ind_{\lbrace | \Deli \hatbigjal| > v_n \rbrace} \big| 
\ind_{\lbrace | \Deli \tilde X^{\prime n} | \leq v_n/2 \rbrace},
\end{multline}
where the processes in the display above are defined in \eqref{SmandlajumpDef} and where $K>0$ denotes a bound for $\rho$. Therefore, due to $\Prob(Q_n) \rightarrow 1$ it is enough to show $C_n(\alpha) =o_\Prob(1)$ and $D_n(\alpha) =o_\Prob(1)$ in order to verify \eqref{DiffKonv} and to complete the proof of Lemma \ref{step1}. 

First, we consider $D_n(\alpha)$. For later reasons, we let $f$ be either $\rho_{\alpha}$ or $\rho_{\alpha}^{\circ}$. Then there exists a constant $K >0$ which depends only on $\alpha$, such that we have for $x,z \in \R$ and $v>0$:
\begin{equation}
\label{ralralprdiabs}
\big| f(x+z) \ind_{\lbrace | x+z | >v \rbrace} - f(x) \ind_{\lbrace |x| > v \rbrace} \big| 
\ind_{\lbrace |z| \leq v/2 \rbrace} 
\leq K \big(|x|^p \ind_{\lbrace |x| \leq 2v \rbrace} + |x|^{p-1} |z| \ind_{\lbrace |z| \leq v/2 \rbrace}\big).
\end{equation}
Note that for $|x+z| >v$ and $|x| >v$ we use the mean value theorem and $|z| \leq |x|$ as well as $| \frac{df}{dx}(x) | \leq K |x|^{p-1}$ for all $x \in \R$ by the assumptions on $\rho$ and because the derivatives of $\Psi_\alpha$ and $\Psi_\alpha^\circ$ have a compact support, which is bounded away from $0$. In all other cases in which the left hand side does not vanish we have $|z| \leq |x| \leq 2v$ as well as $\left| f(x) \right| \leq K |x|^p$ for all $x \in \R$ by another application of the mean value theorem and the assumptions on $\rho$. Consequently,
\begin{equation}
\label{Dnalkonvb}
\Eb D_n(\alpha) \leq a_n(\alpha) + b_n(\alpha)
\end{equation}
holds for
\begin{align*} 
a_n(\alpha) &= \frac{1}{\sqrt{n \Delta_n}} \sum \limits_{i=1}^n \Eb \Big\{ \big| \Deli \hatbigjal \big|^p 
\ind_{\lbrace | \Deli \hatbigjal| \leq 2 v_n \rbrace} \Big\}~,~~
b_n(\alpha) &= \frac{v_n}{2 \sqrt{n \Delta_n}} \sum \limits_{i=1}^n \Eb \big| \Deli \hatbigjal \big|^{p-1},
\end{align*}
and we conclude $D_n(\alpha) =o_\Prob(1)$ because of Lemma \ref{anabnakonv} in Appendix \ref{appD}.

Finally, we show $C_n(\alpha) =o_\Prob(1)$. To this end, we define for $1 \leq i,j \leq n$ with $i \neq j$ and the constant $\ovr$ in Assumption \ref{Cond1}
\begin{align}
\label{RijnalDefEq2}
R_{i,j}^{(n)} (\alpha) = \left\{ \big| \Deli \hatbigjal - \Delj \hatbigjal \big| \leq \Delta_n^{\ovr} \right\} \cap 
\left\{ \big| \Deli \hatbigjal \big| > \alpha/4 \right\} \cap Q_n,
\end{align}
as well as the sets $J^{\scriptscriptstyle (1)}_n (\alpha)$ by:
\begin{align}
\label{J1ijnalDefEq2}
J^{(1)}_n(\alpha)^C = \bigcup \limits_{\stackrel{i,j =1}{i \neq j}}^n R_{i,j}^{(n)}(\alpha).
\end{align}
Then according to Lemma \ref{PJn1alconv}  in Appendix \ref{appD}
we have $\Prob \big(J^{\scriptscriptstyle (1)}_n(\alpha)\big) \rightarrow 1$. Moreover, Lemma \ref{Indikunglab} shows that for all $n \in\N$, $\omega \in J_n^{\scriptscriptstyle (1)}(\alpha) \cap Q_n$ and $\dfi \in \R$ the random set
\begin{align*}
\tilde A_1(\omega;\alpha,n,\dfi) &= \big\{i \in \{1,\ldots,n\} \mid | \Deli \hatbigjal (\omega) | > \alpha/4 \text{ and } \\
&\hspace{26mm}\Indit(\Deli \tilde X^{\prime n}(\omega) + \Deli \hatbigjal(\omega)) \neq \Indit(\Deli \hatbigjal(\omega)) \big\}
\end{align*}
has at most $c_n := \lceil (v_n/ \Delta_n^{\ovr}) + 1\rceil$ elements. Consequently, on $J_n^{\scriptscriptstyle (1)}(\alpha) \cap Q_n$ for each $t \in \R$ at most $c_n$ summands in the sum of the definition of $C_n(\alpha)$ can be equal to $1$ and we conclude 
\[
C_n(\alpha) \leq K/ \sqrt{n \Delta_n^{(1+ 2(\ovr - \ovw)) \vee 1}},
\]
on $J_n^{\scriptscriptstyle (1)}(\alpha) \cap Q_n$. Thus, $C_n =o_\Prob(1)$ follows using Assumption \ref{Cond1}\eqref{ObsSchCond7} and  $\Prob\big(J_n^{\scriptscriptstyle (1)}(\alpha) \cap Q_n \big) \to 1$.
\qed

\medskip 

\noindent 
\textbf{Proof of Lemma \ref{step3}.} For $\alpha > 0$ define the following processes:
\[
\tilde Y_{\rho^{\circ}_{\alpha}}^{(n)} (\gseqi,\dfi) = \sqrt{n \Delta_n} \Big\{ \frac{1}{n \Delta_n} \sum \limits_{i=1}^{\ip{n\gseqi}} \chi_\dfi^{\circ(\alpha)}(\Deli L^{(n)}) - N_{\rho^{\circ}_{\alpha}}(g^{(n)};\gseqi,\dfi) \Big\}.
\]
Similar to \eqref{EmpProDiffAbEq} we obtain with Proposition \ref{BiasAbschProp} and Proposition \ref{LevyCLTProp} that for $n \rightarrow \infty$ the processes in the display above converge weakly in $\linner$, that is
$
\tilde Y_{\rho^{\circ}_{\alpha}}^{(n)} \weak \Gb_{\rho^{\circ}_{\alpha}}.
$
On the other hand, we have weak convergence $
\Gb_{\rho^{\circ}_{\alpha}} \weak 0 $
in $\linner$ as $\alpha \rightarrow 0$, by Proposition \ref{GalKonvLem}. Therefore, by using the Portmanteau theorem (Theorem 1.3.4 in \cite{VanWel96}) twice, we obtain for arbitrary $\eta >0$:
\begin{equation*}
\limsup \limits_{\alpha \rightarrow 0} \limsup \limits_{n \rightarrow \infty} \Prob \big ( \sup \limits_{(\gseqi,\dfi) \in \netir} \big| 
\tilde Y_{\rho^{\circ}_{\alpha}}^{(n)} (\gseqi,\dfi) \big| \geq \eta \big ) \leq  
\limsup \limits_{\alpha \rightarrow 0} \Prob \big ( \sup \limits_{(\gseqi,\dfi) \in \netir} \big| \Gb_{\rho^{\circ}_{\alpha}} (\gseqi,\dfi)
\big| \geq \eta \big ) =0.
\end{equation*}
Thus, it suffices to show  $V^{\circ(n)}_{\alpha} =o_\Prob(1)$ as $n\to\infty$ for each $\alpha >0$ in a neighbourhood of $0$, where
\begin{equation}
\label{DifferDefEq}
V^{\circ(n)}_{\alpha} = \frac{1}{\sqrt{n \Delta_n}} \sup_{(\gseqi,\dfi) \in \netir} \Big| \sum \limits_{i=1}^{\ip{n\gseqi}} \Big\{ \chi_\dfi^{\circ (\alpha)} (\Deli X^{(n)}) \Truniv - 
\chi_\dfi^{\circ (\alpha)}(\Deli L^{(n)}) \Big\} \Big|.
\end{equation}
Due to Lemma \ref{ValprnBou} we have for $\alpha >0$, $\omega \in Q_n$ and $n \in \N$ large enough such that $v_n \leq \alpha$ with the processes defined in \eqref{SmandlajumpDef}
\[
V^{\circ(n)}_{\alpha} \leq C^\circ_n(\alpha) + D^\circ_n(\alpha) + E^\circ_n(\alpha),
\]
where
\begin{multline*}
C^\circ_n(\alpha) = \frac{K}{\sqrt{n \Delta_n}} \sup \limits_{t \in \R} \sum \limits_{i=1}^n \big| \Indit(\varsigma_i^n(\alpha)) - \Indit(\Deli \tibigjeial) \big|  \times \\
\times \ind_{\lbrace | \Deli \tibigjeial | > \Delta_n^{\ovv} \rbrace} 
\ind_{\lbrace | \Deli \tilde X^{\prime n} | \leq v_n/2 \rbrace},
\end{multline*}
with $K >0$ a bound for $\rho$ and
\begin{multline}
\label{Dnpraldef}
D^\circ_n(\alpha) = \frac{1}{\sqrt{n \Delta_n}} \sum \limits_{i=1}^n \big| \rho^{\circ}_{\alpha} (\varsigma_i^n(\alpha)) \ind_{\lbrace | \varsigma_i^n(\alpha)| > v_n \rbrace} - \rho^{\circ}_{\alpha}(\Deli \tibigjeial) \ind_{\lbrace | \Deli \tibigjeial| > v_n \rbrace} \big| \times \\ \times
\ind_{\lbrace | \Deli \tibigjeial | > \Delta_n^{\ovv} \rbrace} 
\ind_{\lbrace | \Deli \tilde X^{\prime n} | \leq v_n/2 \rbrace},
\end{multline}
\begin{multline*}
E^\circ_n(\alpha) = \frac{1}{\sqrt{n \Delta_n}} \sup \limits_{\dfi \in \R}  \sum \limits_{i=1}^{n}  \big|
\rho^{\circ}_{\alpha}(\varsigma_i^n(\alpha)) \ind_{\lbrace | \varsigma_i^n(\alpha)| > v_n \rbrace} \Indit(\varsigma_i^n(\alpha)) - \rho^{\circ}_{\alpha}(\Deli \tibigjeial) 
\times \\  \times \ind_{\lbrace | \Deli \tibigjeial | > v_n \rbrace} 
\Indit(\Deli \tibigjeial) \big| 
\ind_{\lbrace | \Deli \tilde X^{\prime n} | \leq v_n/2 \rbrace}  
\ind_{\lbrace | \Deli \tibigjeial | \leq \Delta_n^{\ovv} \rbrace} \ind_{Q_n},
\end{multline*}
where $\varsigma_i^n(\alpha) = \Deli \tilde X^{\prime n} + \Deli \tibigjeial$ and $\ovv >0$ is the constant from Assumption \ref{Cond1}\eqref{FiLevyDistCond}. Thus, as a consequence of $\Prob(Q_n) \to 1$ it suffices to show for each $\eta >0$ and each $\alpha >0$ in a neighbourhood of zero: 
\begin{align}
\label{a} &\lim_{n \rightarrow \infty} \Prob\big( C^\circ_n(\alpha) > \eta\big) = 0,\\
\label{b} &\lim_{n \rightarrow \infty} \Prob\big( D^\circ_n(\alpha) > \eta\big) =0, \\
\label{c} &\lim_{n \rightarrow \infty} \Prob\big( E^\circ_n(\alpha) > \eta\big) = 0.
\end{align}
Concerning \eqref{a}, similar to \eqref{RijnalDefEq2} we define for $1 \leq i,j \leq n$ with $i \neq j$ and the constants $\ovv < \ovr$ in Assumption \ref{Cond1}:
\[
S_{i,j}^{(n)}(\alpha) = \left\{ \big| \Deli \tibigjeial - \Delj \tibigjeial \big| \leq \Delta_n^{\ovr} \right\} \cap
\left\{ \big| \Deli \tibigjeial \big| > \Delta_n^{\ovv} \right\} \cap Q_n.
\]
as well as $J_n^{\scriptscriptstyle (2)} (\alpha)$ by
\begin{equation}
\label{Jn2aldef31}
\big(J_n^{(2)} (\alpha)\big)^C = \bigcup \limits_{\stackrel{i,j =1}{i \neq j}}^n S_{i,j}^{(n)} (\alpha).
\end{equation}
Then Lemma \ref{BiJovrovvAb}  in Appendix \ref{appD} shows $\Prob \big(J_n^{\scriptscriptstyle (2)}(\alpha)\big) \rightarrow 1$ for all $\alpha \in (0,\alpha_0/2)$ 
($\alpha_0$ is defined in Assumption \ref{Cond1}) and according to Lemma \ref{Indikunglab} the random set
\begin{multline*}
\tilde A_2(\omega;\alpha,n,\dfi) = \big\{i \in \{1,\ldots,n\} \mid | \Deli \tibigjeial(\omega)| > \Delta_n^{\ovv} \text{ and } \\
\Indit(\Deli \tilde X^{\prime n}(\omega) + \Deli \tibigjeial(\omega)) \neq \Indit(\Deli \tibigjeial(\omega)) \big\}
\end{multline*}
has at most $c_n = \lceil (v_n/ \Delta_n^{\ovr}) + 1\rceil$ elements for all $\omega \in J_n^{\scriptscriptstyle (2)}(\alpha) \cap Q_n$, $n\in\N$, $\dfi\in\R$ and $\alpha >0$. So for each $t \in \R$ at most $c_n$ summands in $C^\circ_n(\alpha)$ can be equal to $1$, and we have 
$
C^\circ_n(\alpha) \leq K /\sqrt{n\Delta_n^{(1+2(\ovr-\ovw)) \vee 1}}
$
on $J^{\scriptscriptstyle (2)}_n(\alpha) \cap Q_n$. Consequently, \eqref{a} follows  by Assumption \ref{Cond1}\eqref{ObsSchCond7} for all $\alpha \in (0,\alpha_0/2)$. 

Furthermore, because of \eqref{ralralprdiabs} we have 
\begin{equation}
\label{Dnpralkon}
\Eb D^\circ_n(\alpha) \leq c_n(\alpha) + d_n(\alpha),
\end{equation}
for
\begin{align*}
c_n(\alpha) &= \frac{1}{\sqrt{n \Delta_n}} \sum \limits_{i=1}^n \Eb \Big\{ \big| \Deli \tibigjeial \big|^p 
\ind_{\lbrace | \Deli \tibigjeial | \leq 2 v_n \rbrace} \Big\}, \\
d_n(\alpha) &= \frac{v_n}{2 \sqrt{n \Delta_n}} \sum \limits_{i=1}^n \Eb \big| \Deli \tibigjeial \big|^{p-1}.
\end{align*}
Thus, Lemma \ref{cnadnakonv}  in Appendix \ref{appD}. yields \eqref{b}.

Concerning \eqref{c}, let $\alpha >0$ be fixed.
Because of the triangle inequality and $|\rho_\alpha^\circ(\taili)|\leq K|\taili|^p$ for all $\taili \in \R$, an upper bound for $E^\circ_n(\alpha)$ is clearly given by 
\begin{multline*}
\frac{K}{\sqrt{n\Delta_n}} \sum \limits_{i=1}^n \Big(|\varsigma_i^n(\alpha)|^p \ind_{\{| \varsigma_i^n(\alpha)|>v_n\}} + |\Deli \tibigjeial|^p \ind_{\{|\Deli \tibigjeial|>v_n\}} \Big) \times   \\
\times \ind_{\lbrace | \Deli \tilde X^{\prime n} | \leq v_n/2 \rbrace}  \ind_{\lbrace | \Deli \tibigjeial | \leq \Delta_n^{\ovv} \rbrace} \ind_{Q_n},
\end{multline*}
with $\varsigma_i^n(\alpha) = \Deli \tilde X^{\prime n} +\Deli \tibigjeial$. As a consequence, we have $\Eb E^\circ_n(\alpha) \leq K\big( y_n^{(\alpha)} + 2 z_n^{(\alpha)} \big)$ for
\begin{align*}
y_n^{(\alpha)} &= \frac{1}{\sqrt{n \Delta_n}} \sum \limits_{i=1}^n \Eb \Big\{ \big| \big| \varsigma_i^n(\alpha) \big|^p - \big| \Deli \tibigjeial \big|^p  \big| \times \\
&\hspace{43mm}\times \ind_{\{|\Deli \tibigjeial | \leq \Delta_n^{\ovv} \}}
\ind_{\{|\varsigma_i^n(\alpha) |  >v_n \}} \ind_{\lbrace | \Deli \tilde X^{\prime n} | \leq v_n/2 \rbrace}\ind_{Q_n} \Big\}, \\
z_n^{(\alpha)} &= \frac{1}{\sqrt{n \Delta_n}} \sum \limits_{i=1}^n \Eb \Big\{ \big| \Deli \tibigjeial \big|^p \ind_{\{|\Deli \tibigjeial | \leq \Delta_n^{\ovv} \}} \ind_{Q_n}  \Big\}.
\end{align*}
Therefore, we obtain \eqref{c} by Lemma \ref{ynaznakonv} in Appendix \ref{appD}.
\qed

\section{Proof of Theorem \ref{CondConvThm}}
\label{ssecProTh68}
\def\theequation{B.\arabic{equation}}
\setcounter{equation}{0}

\subsection{Main steps in the proof}

Similar to the proof of Theorem \ref{ConvThm} we show Theorem \ref{CondConvThm} by treating small and large increments of $X^{\scriptscriptstyle (n)}$ separately. Therefore, with the quantities defined prior to \eqref{ImpQuantDef} we consider the processes
\begin{align*}
\hat G_{\rho,n}^{(\alpha)}(\gseqi,\dfi) &= \frac 1{\sqrt{n\Delta_n}} \sum\limits_{i=1}^{\ip{n\gseqi}} \xi_i \rho_\alpha(\Deli X^{(n)}) \Indit(\Deli X^{(n)}) \Truniv  \\
\hat G_{\rho,n}^{\circ (\alpha)}(\gseqi,\dfi) &= \frac 1{\sqrt{n\Delta_n}} \sum\limits_{i=1}^{\ip{n\gseqi}} \xi_i \rho^\circ_\alpha(\Deli X^{(n)}) \Indit(\Deli X^{(n)}) \Truniv.
\end{align*}

\begin{lemma}
	\label{BiJCoWeCoThm}
	If Assumption \ref{EasierCond} and Assumption \ref{MultiplAss} are satisfied, we have
$	
	\hat G_{\rho,n}^{(\alpha)} \weakP \Gb_{\rho_\alpha}
$	
	in $\linner$ for each fixed $\alpha >0$.
\end{lemma}

\begin{lemma}
	\label{NoContrfsmalL}
	Suppose Assumption \ref{EasierCond} and Assumption \ref{MultiplAss} are valid. Then for each $\alpha >0$ in a neighbourhood of $0$ we have
$
	\hat G_{\rho,n}^{\circ (\alpha)} \weakP \Gb_{\rho^\circ_\alpha}
$	
	holds in $\linner$.
\end{lemma}
\noindent
The two lemmas are the main ingredients in the proof of Theorem \ref{CondConvThm} and will be verified by approximating the truncated increments of the underlying processes by the increments
of the pure jump It\=o semimartingales from \eqref{LnDefEq}
\begin{align*}
L^{(n)} = (\taili \ind_{\{|\taili| > v_n\}}) \star \mu^{(n)},
\end{align*}
with the usual truncation $v_n= \gamma \Delta_n^{\ovw}$. The main advantage of the processes $L^{(n)}$ is the fact, that they have deterministic characteristics and therefore independent increments. As a consequence, we can use a result from \cite{Kos08} for triangular arrays of processes which are independent within rows to prove weak convergence conditional on the data in probability of the bootstrapped analogs of $Y_f^{\scriptscriptstyle (n)}$ from \eqref{EmpPrYDefEqn} which are given by
\begin{align*}
\hat Y_f^{(n)} (\gseqi,\dfi) = \frac 1{\sqrt{n\Delta_n}} \sum\limits_{i=1}^{\ip{n\gseqi}} \xi_i f(\Deli L^{(n)}) \Indit(\Deli L^{(n)}),
\end{align*}
for $(\gseqi,\dfi) \in \netir$ and where $f\colon\R \to \R$ is a bounded continuous function. More precisely, the following proposition is the main tool in order to obtain Lemma \ref{BiJCoWeCoThm} and Lemma \ref{NoContrfsmalL}.

\begin{proposition}
	\label{LnCondWeConv}
	Suppose Assumption \ref{EasierCond} and Assumption \ref{MultiplAss} are satisfied. Then for a continuous function $f\colon\R\to\R$ satisfying $|f(\taili)| \leq K(1\wedge |\taili|^p)$ for all $\taili \in \R$ and some $K>0$ we have 
$
	\hat Y_f^{(n)} \weakP \Gb_f
$
	in $\linner$, where $\Gb_f$ is the tight mean zero Gaussian process defined in Theorem \ref{ConvThm}.
\end{proposition}

\noindent{\bf Proof of Theorem \ref{CondConvThm}}
According to Definition \ref{ConvcondDataDef} we have to show
\begin{align}
\label{BouLipMetSma}
\sup \limits_{h \in \text{BL}_1(\linner)} \big| \Eb_\xi h(\hat G^{(n)}_\rho) - \Eb h(\Gb_\rho) \big| & \pto 0,
\\ 
\label{memamemidi}
\Eb_\xi h(\hat G_\rho^{(n)})^\ast - \Eb_\xi h(\hat G_\rho^{(n)})_\ast & \probto 0 \text{ for all } h \in \text{BL}_1(\linner),
\end{align}
where $h(\hat G_\rho^{\scriptscriptstyle (n)})^\ast$ and $h(\hat G_\rho^{\scriptscriptstyle (n)})_\ast$ denote a minimal measurable majorant and a maximal measurable minorant with respect to the joint data, respectively. \\
In order to show \eqref{BouLipMetSma} observe that by the properties of bounded Lipschitz functions we have for each $h \in \text{BL}_1(\linner)$
\begin{multline*}
\big| \Eb_\xi h(\hat G_\rho^{(n)}) - \Eb h(\Gb_\rho) \big| \leq \big| \Eb_\xi h(\hat G_{\rho,n}^{(\alpha)}) - \Eb h(\Gb_{\rho_\alpha}) \big| + \\
+ \big| \Eb h(\Gb_{\rho_\alpha}) - \Eb h(\Gb_\rho) \big| + |\Eb_\xi Y_n^{(\alpha)} - \Eb \Yb^{(\alpha)}| + \Eb \Yb^{(\alpha)},
\end{multline*}
for every $\alpha >0$, where
\begin{align*}
Y_n^{(\alpha)} = \sup \limits_{(\gseqi, \dfi) \in \netir} | \hat G_{\rho,n}^{\circ (\alpha)}(\gseqi,\dfi)| \wedge 2 \quad \text{ and } \quad \Yb^{(\alpha)} = \sup \limits_{(\gseqi, \dfi) \in \netir} | \Gb_{\rho_\alpha^\circ}(\gseqi,\dfi)| \wedge 2.
\end{align*} 
Thus, due to Lemma 1.2.2(i) in \cite{VanWel96} we obtain 
\begin{equation}
\label{meamajabsch}
\Big( \sup \limits_{h \in \text{BL}_1(\linner)} \big| \Eb_\xi h(\hat G_\rho^{(n)}) - \Eb h(\Gb_\rho) \big| \Big)^\ast \leq q_n^{(\alpha)} + p(\alpha) + |\Eb_\xi Y_n^{(\alpha)} - \Eb \Yb^{(\alpha)}|  + \Eb \Yb^{(\alpha)},
\end{equation}
for each $\alpha >0$, where
\begin{align*}
q_n^{(\alpha)}&=\Big( \sup \limits_{h \in \text{BL}_1(\linner)} \big| \Eb_\xi h(\hat G_{\rho,n}^{(\alpha)}) - \Eb h(\Gb_{\rho_\alpha}) \big| \Big)^\ast \\
p(\alpha) &= \sup \limits_{h \in \text{BL}_1(\linner)} \big| \Eb h(\Gb_{\rho_\alpha}) - \Eb h(\Gb_\rho) \big|.
\end{align*}
Notice that the supremum in the definition of $Y_n^{\scriptscriptstyle (\alpha)}$ is measurable, because the process $\hat G_{\rho,n}^{\scriptscriptstyle \circ (\alpha)}$ depends only via $\ip{n\gseqi}$ on $\gseqi \in [0,1]$ and is right-continuous in $\dfi \in \R$. Let $\eps >0$ be arbitrary. Then due to Proposition \ref{Gbrhoapcopr} and monotonicity of the integral we obtain
\begin{equation}
\label{EYbalsmaEq}
\Eb \Yb^{(\alpha)} \leq \eps/4,
\end{equation}
for all $\alpha >0$ in a neighbourhood of $0$. Moreover, because of  Lemma \ref{step2} and Theorem 1.12.1 in \cite{VanWel96} we have 
\begin{equation}
\label{palsmaeq}
p(\alpha) \leq \eps/4,
\end{equation}
for $\alpha >0$ small enough. Thus, choose an $\alpha >0$ such that \eqref{EYbalsmaEq}, \eqref{palsmaeq} and Lemma \ref{NoContrfsmalL} hold. Then Lemma \ref{BiJCoWeCoThm} yields $q_n^{\scriptscriptstyle (\alpha)} \probto 0$ as $n \rightarrow \infty$ and due to Lemma \ref{NoContrfsmalL} we have $|\Eb_\xi Y_n^{\scriptscriptstyle (\alpha)} - \Eb \Yb^{\scriptscriptstyle (\alpha)}| \probto 0$ as $n \rightarrow \infty$, because $Y_n^{\scriptscriptstyle (\alpha)} = h_0(\hat G_{\rho,n}^{\scriptscriptstyle \circ (\alpha)})$ and $\Yb^{\scriptscriptstyle (\alpha)} = h_0(\Gb_{\rho_\alpha^\circ})$ for the bounded Lipschitz function $h_0: \linner \rightarrow \R$ given by $h_0(f) = \sup_{(\gseqi,\dfi) \in \netir}|f(\gseqi,\dfi)| \wedge 2$. As a consequence, we obtain \eqref{BouLipMetSma} with \eqref{meamajabsch}:
\begin{multline*}
\Prob \Big( \Big( \sup \limits_{h \in \text{BL}_1(\linner)} \big| \Eb_\xi h(\hat G_\rho^{(n)}) - \Eb h(\Gb_\rho) \big| \Big)^\ast > \eps \Big) \\ 
\leq \Prob(q_n^{(\alpha)} > \eps/4) + \Prob(|\Eb_\xi Y_n^{(\alpha)} - \Eb \Yb^{(\alpha)}| > \eps/4) \rightarrow 0.
\end{multline*}
In order to show \eqref{memamemidi} we have for each $h \in \text{BL}_1(\linner)$ and each $\alpha >0$
\begin{align*}
h(\hat G_{\rho,n}^{(\alpha)}) - Y_n^{(\alpha)} \leq h(\hat G_\rho^{(n)}) \leq h(\hat G_{\rho,n}^{(\alpha)}) + Y_n^{(\alpha)}.
\end{align*}
Therefore, applying Lemma 1.2.2(i) in \cite{VanWel96} and the relation $-Z_\ast = (-Z)^\ast$ between the minimal measurable majorant and the maximal measurable minorant of a random element $Z$ several times yields
\begin{align}
\label{maminabsch}
|\Eb_\xi h(\hat G_{\rho}^{(n)})^\ast &- \Eb_\xi h(\hat G_{\rho}^{(n)})_\ast| = \Eb_\xi h(\hat G_{\rho}^{(n)})^\ast - \Eb_\xi h(\hat G_{\rho}^{(n)})_\ast \nonumber \\
&\leq |\Eb_\xi h(\hat G_{\rho,n}^{(\alpha)})^\ast - \Eb_\xi h(\hat G_{\rho,n}^{(\alpha)})_\ast| + 2 \Eb_\xi Y_n^{(\alpha)} \nonumber \\
&\leq |\Eb_\xi h(\hat G_{\rho,n}^{(\alpha)})^\ast - \Eb_\xi h(\hat G_{\rho,n}^{(\alpha)})_\ast| + 2 |\Eb_\xi Y_n^{(\alpha)} - \Eb \Yb^{(\alpha)}| + 2 \Eb \Yb^{(\alpha)},
\end{align}
for every $h \in \text{BL}_1(\linner)$ and each $\alpha >0$. For arbitrary $\eps >0$ we choose $\alpha >0$ such that Lemma \ref{NoContrfsmalL} holds and $\Eb \Yb^{\scriptscriptstyle (\alpha)} \leq \eps/8$. Then as above we see $|\Eb_\xi Y_n^{\scriptscriptstyle (\alpha)} - \Eb \Yb^{\scriptscriptstyle (\alpha)}| \probto 0$ for $n \rightarrow \infty$ and furthermore we have $|\Eb_\xi h(\hat G_{\rho,n}^{\scriptscriptstyle (\alpha)})^\ast - \Eb_\xi h(\hat G_{\rho,n}^{\scriptscriptstyle (\alpha)})_\ast| \probto 0$ by Lemma \ref{BiJCoWeCoThm} as $n \rightarrow \infty$. These facts together with \eqref{maminabsch} give \eqref{memamemidi}:
\begin{multline*}
\Prob(|\Eb_\xi h(\hat G_{\rho}^{(n)})^\ast - \Eb_\xi h(\hat G_{\rho}^{(n)})_\ast| > \eps) \leq \\
\Prob(|\Eb_\xi h(\hat G_{\rho,n}^{(\alpha)})^\ast - \Eb_\xi h(\hat G_{\rho,n}^{(\alpha)})_\ast| > \eps/2) + \Prob(|\Eb_\xi Y_n^{(\alpha)} - \Eb \Yb^{(\alpha)}| > \eps/8) \rightarrow 0.
\end{multline*}
\qed

\subsection{Proof of auxiliary results}

\textbf{Proof of Proposition \ref{LnCondWeConv}.} 
Recall the triangular array $\{g_{ni}(\gseqi,\dfi) \mid n\in\N, i=1,\ldots,n; (\gseqi,\dfi) \in \netir\}$ in the proof of Proposition \ref{LevyCLTProp} given by
\begin{align*}
g_{ni}(\omega ; (\gseqi,\dfi)) = \frac{1}{\sqrt{n \Delta_n}} f(\Deli L^{(n)}(\omega)) \Indit( \Deli L^{(n)}(\omega)) \ind_{\{i \leq \ip{n\gseqi}\}},
\end{align*}
for $(\gseqi,\dfi) \in \netir$ and let $\mu_{ni}(\gseqi,\dfi) = \Eb \big( g_{ni}(\gseqi,\dfi) \big)$. Moreover, for $n \in \N$, $i=1,\ldots,n$ and $(\gseqi,\dfi) \in \netir$ let
\begin{align*}
\hat \mu_{ni}(\gseqi,\dfi) = \frac 1{\sqrt{n\Delta_n}} \ind_{\{i \leq \ip{n\gseqi}\}} \tilde \eta_f^{(n)}(\dfi),
\end{align*}
with
\[\tilde \eta_f^{(n)}(\dfi) = \frac 1n \sum_{j=1}^n f(\Delj L^{(n)}) \ind_{(-\infty,\dfi]}(\Delj L^{(n)}),\]
be an estimator for $\mu_{ni}(\gseqi,\dfi)$. Then we proceed in two steps:
\begin{enumerate}[(a)]
	\item \label{Zza}
$	\hat Y^{(n)}_{f,0} (\gseqi,\dfi) := \sum\limits_{i=1}^n \xi_i \big( g_{ni}(\gseqi,\dfi) - \hat \mu_{ni}(\gseqi,\dfi) \big) \weakP \Gb_f,$
	\item \label{Zzb}
$	\sup_{(\gseqi,\dfi) \in \netir} \big| \hat Y^{(n)}_{f} (\gseqi,\dfi) - \hat Y^{(n)}_{f,0} (\gseqi,\dfi)\big| = \sup_{(\gseqi,\dfi) \in \netir} \big| \sum_{i=1}^n \xi_i \hat \mu_{ni}(\gseqi,\dfi) \big| = o_\Prob(1),$
\end{enumerate}
then the claim follows using Lemma \ref{oP1glzcwecole}.

\medskip

\noindent 
{\it Proof of Step \eqref{Zza}.} The sequence $\{g_{ni}\}$ satisfies conditions \eqref{Liste1}-\eqref{Liste6} in the proof of Proposition \ref{LevyCLTProp}. Thus, the conditional weak convergence \eqref{Zza} holds by Theorem 11.18 in \cite{Kos08}, if we can show the following four conditions of the triangular array $\{\hat\mu_{ni}(\gseqi,\dfi)\mid n\in\N, i=1,\ldots,n; (\gseqi,\dfi) \in \netir\}$:
\begin{enumerate}[(A)]
	\setcounter{enumi}{6}
	\item $\{\hat\mu_{ni}\}$ is AMS; \label{ListeG}
	\item  $
		\sup\limits_{(\gseqi,\dfi) \in \netir} \sum\limits_{i=1}^n \big[ \hat\mu_{ni}(\omega;(\gseqi,\dfi)) - \mu_{ni}(\gseqi,\dfi) \big]^2 = o_{\Prob} (1);
	$
 \label{ListeH}
	\item The processes $\{\hat\mu_{ni}\}$ are manageable with envelopes $	\hat F_{ni} = \frac K{n\sqrt{n\Delta_n}} \sum\limits_{j=1}^n 1 \wedge \big|\Delj L^{(n)} \big|^p,
	$
	for $n\in\N$ and $i=1,\ldots,n$, with a $K>0$ such that $|f(\taili)| \leq K(1\wedge |\taili|^p)$ for all $\taili \in \R$;\label{ListeI}
	\item There exists an $M\in\R_+$ such that 
	$
	M \vee \sum_{i=1}^n [\hat F_{ni}]^2 \probto M.
	$ \label{ListeJ}
\end{enumerate}

\noindent 
\textit{Proof of  \eqref{ListeG}.} For each $n\in\N$ define the countable set $S_n = (\netir) \cap \Q^2$ to obtain
\[
\Prob^{\ast} \bigg( \sup \limits_{(\gseqi_1,\dfi_1) \in \netir} \inf \limits_{(\gseqi_2, \dfi_2) \in S_n} \sum \limits_{i=1}^n (\hat\mu_{ni}(\omega; (\gseqi_1,\dfi_1)) - 
\hat\mu_{ni}(\omega;(\gseqi_2,\dfi_2)))^2 >0 \bigg) = 0.
\]
As a consequence, the triangular array $\{\hat\mu_{ni}\}$ is separable and therefore AMS by Lemma 11.15 in \cite{Kos08}.

\medskip

\noindent 
\textit{Proof of  \eqref{ListeH}.} Simple calculations show
\begin{align}
\label{AnUmstEqn}
A_n :&= \sup\limits_{(\gseqi,\dfi) \in \netir} \sum\limits_{i=1}^n \big[ \hat\mu_{ni}(\omega;(\gseqi,\dfi)) - \mu_{ni}(\gseqi,\dfi) \big]^2 \nonumber \\
&= \frac 1{n^3 \Delta_n} \sup\limits_{\dfi \in \R} \sum\limits_{i=1}^n \sum\limits_{j=1}^n \sum\limits_{k=1}^n \Big( f(\Delj L^{(n)}) \Indit(\Delj L^{(n)}) \nonumber\\
&\hspace{7cm} -\Eb \big( f(\Deli L^{(n)}) \Indit(\Deli L^{(n)}) \big) \Big) \times \nonumber\\
&\hspace{22mm}\times \Big( f(\Delk L^{(n)}) \Indit(\Delk L^{(n)}) - \Eb \big( f(\Deli L^{(n)}) \Indit(\Deli L^{(n)}) \big) \Big) \nonumber\\
&= \sup\limits_{\dfi \in \R} \bigg\{ \frac 1{n^2 \Delta_n} \sum\limits_{j=1}^n \sum\limits_{k=1}^n f(\Delj L^{(n)}) \Indit(\Delj L^{(n)})f(\Delk L^{(n)}) \Indit(\Delk L^{(n)}) \nonumber \\
&\hspace{8mm} -\frac 1{n^2 \Delta_n} \sum\limits_{i=1}^n \sum\limits_{k=1}^n f(\Delk L^{(n)}) \Indit(\Delk L^{(n)}) \Eb\big( f(\Deli L^{(n)}) \Indit(\Deli L^{(n)})\big)  \nonumber\\
&\hspace{8mm}- \frac 1{n^2 \Delta_n} \sum\limits_{i=1}^n \sum\limits_{j=1}^n f(\Delj L^{(n)}) \Indit(\Delj L^{(n)}) \Eb\big( f(\Deli L^{(n)}) \Indit(\Deli L^{(n)})\big)  \nonumber\\
&\hspace{8mm}+ \frac 1{n\Delta_n} \sum\limits_{i=1}^n \big(\Eb\big(f(\Deli L^{(n)}) \Indit(\Deli L^{(n)})\big)\big)^2 \bigg\}.
\end{align}
Furthermore, by Lemma \ref{MomentfgO} and the assumptions on $f$ we obtain 
\begin{align}
\label{ErwLnODeln}
&\hspace{2cm}\sup_{i \in \{1,\ldots,n\}}\Eb \big(\sup\limits_{\dfi \in \R} \big| f(\Deli L^{(n)}) \Indit(\Deli L^{(n)})\big| \big) = O(\Delta_n), \\
&\Eb \sup\limits_{\dfi \in \R} \Big| \frac 1{n^2 \Delta_n} \sum\limits_{i=1}^n \sum\limits_{k=1}^n f(\Delk L^{(n)}) \Indit(\Delk L^{(n)}) \Eb\big( f(\Deli L^{(n)}) \Indit(\Deli L^{(n)})\big) \Big| \nonumber \\
&\hspace{8cm}\leq \frac Kn \sum\limits_{k=1}^n \Eb \big( 1\wedge |\Delk L^{(n)}|^p \big)  = O(\Delta_n). \label{AnSumoProb1}
\end{align}
Thus \eqref{AnUmstEqn}, \eqref{ErwLnODeln} and \eqref{AnSumoProb1} give
\begin{align*}
0 \leq A_n &\leq \sup\limits_{\dfi \in \R} \Big\{ \frac 1{n^2 \Delta_n} \sum\limits_{j=1}^n \sum\limits_{k=1}^n |f(\Delj L^{(n)})| \Indit(\Delj L^{(n)}) \times \\
&\hspace{6cm} \times |f(\Delk L^{(n)})| \Indit(\Delk L^{(n)}) \Big\} + o_\Prob(1) \\
&= \frac 1n \sup\limits_{\dfi \in \R} \big( Y_{|f|}^{(n)}(1,\dfi) \big)^2 + o_\Prob(1) = o_\Prob (1),
\end{align*}
because  $Y_{|f|}^{(n)}$ converges weakly to the tight process $\Gb_{|f|}$ in $\linner$ by Proposition \ref{LevyCLTProp}.

\medskip
\noindent 

\noindent 
\textit{Proof of  \eqref{ListeI}.} According to Theorem 11.17 in \cite{Kos08}, it suffices to verify that the triangular arrays
\begin{align*}
\{\tilde \mu_{ni}(\omega;\dfi) \defeq \frac{1}{\sqrt{n \Delta_n}} \tilde\eta_f^{(n)}(\dfi) \mid n\in\N; i=1,\ldots,n;\dfi \in \R\},
\end{align*}
and
\begin{align*}
\{\tilde h_{ni}(\omega;\gseqi) \defeq \ind_{\{i \leq \ip{n\gseqi}\}} \mid n\in\N;i=1,\ldots,n; \gseqi\in[0,1]\}
\end{align*}
are manageable with envelopes $\lbrace \hat F_{ni} \mid n \in \N; i = 1, \ldots,n \rbrace$ and $\{\tilde H_{ni}(\omega) :\equiv 1\mid n\in\N; i=1,\ldots,n\}$, respectively. 
The manageability of the triangular array $\{\tilde h_{ni}\}$ has already been shown in the proof of Proposition \ref{LevyCLTProp}. Concerning the triangular array $\{\tilde \mu_{ni}\}$ we consider for $n \in \N$ and $\omega \in \Omega$ the sets 
\begin{align*}
\Fc_{n\omega} = \big\{ \big( \tilde \mu_{n1}(\omega;\dfi), \ldots, \tilde \mu_{nn}(\omega;\dfi) \big) \mid \dfi \in \R \big\} \subset \R^n,
\end{align*}
which are bounded with envelope vector $(\hat F_{n1}(\omega), \ldots, \hat F_{nn}(\omega))$. But $\tilde \mu_{ni}$ does not depend on $i$, such that every coordinate projection of $\Fc_{n\omega}$ onto two coordinates $i_1,i_2 \in \{1,\ldots,n\}$ is a subset of the straight line $\{(x,y) \in \R^2 \mid x=y\}$. Consequently, in the sense of Definition 4.2 in \cite{Pol90}, for every $s \in \R^2$ no proper coordinate projection of $\mathcal F_{n \omega}$ can surround $s$ and therefore $\mathcal F_{n \omega}$ has a pseudo dimension of at most $1$ (Definition 4.3 in \cite{Pol90}). Now the managebility of the triangular array $\{\tilde \mu_{ni} \}$ follows with the same reasoning as in the verification of \eqref{Liste2} in the proof of Proposition \ref{LevyCLTProp}.

\medskip

\noindent 
\textit{Proof of  \eqref{ListeJ}.} The envelopes $\{\hat F_{ni}\}$ are independent of $i$ as well. Therefore, with Lemma \ref{MomentfgO}  in Appendix \ref{appD} we obtain
\begin{align*}
\Eb \Big\{ \sum_{i=1}^n [\hat F_{ni}]^2 \Big\} &= n \Eb[\hat F_{n1}]^2 = \frac 1{n^2\Delta_n} \Eb \Big\{\sum\limits_{i=1}^n \sum \limits_{j=1}^n \big( 1 \wedge \big| \Deli L^{(n)} \big|^p \big)\big( 1 \wedge \big| \Delj L^{(n)} \big|^p \big) \Big\} \\
&= O\big(\Delta_n + 1/n \big),
\end{align*}
because the processes $L^{(n)}$ have independent increments. As a consequence, we have in fact $\sum_{i=1}^n [\hat F_{ni}]^2 = o_\Prob(1)$, which proves the claim.

\medskip

\noindent 
{\it Proof of Step \eqref{Zzb}.} With the the notation $
U_n(\dfi) = \frac{1}{n\sqrt{\Delta_n}} \sum_{j=1}^n f(\Delj L^{(n)}) \ind_{(-\infty,\dfi]}(\Delj L^{(n)})
$
we have
$
\sum_{i=1}^n \xi_i \hat\mu_{ni}(\gseqi,\dfi) = U_n(\dfi) \frac 1{\sqrt n} \sum_{i=1}^{\ip{n\gseqi}} \xi_i.
$
As an immediate consequence of Lemma \ref{MomentfgO} we obtain $\sup_{\dfi \in \R} |U_n(\dfi)| = o_\Prob(1)$. Furthermore, the $(\xi_i)_{i\in\N}$ are i.i.d.\ with mean zero and variance one, so it is well known from empirical process theory (see for instance Theorem 2.5.2 and Theorem 2.12.1 in \cite{VanWel96}) that $1/\sqrt n \times \sum_{i=1}^{\scriptscriptstyle \ip{n\gseqi}} \xi_i$ converges weakly to a Brownian motion in $\linne$. The law of a Brownian motion is tight in $\linne$ (see for example Section 8 in \cite{Bil99}) and thus $U_n(\dfi) /\sqrt n \times \sum_{i=1}^{\scriptscriptstyle \ip{n\gseqi}} \xi_i$ converges to $0$ in $\linner$ in outer probability.
\qed

\medskip

\noindent 
\textbf{Proof of Lemma \ref{BiJCoWeCoThm}.}
By Proposition \ref{LnCondWeConv} we have $\hat Y_{\rho_\alpha}^{(n)} \weakP \Gb_{\rho_\alpha}$ for each fixed $\alpha >0$ and therefore due to Lemma \ref{oP1glzcwecole} it only remains to show that the term
\begin{equation*}
\hat G_{\rho,n}^{(\alpha)}(\gseqi,\dfi) - \hat Y_{\rho_\alpha}^{(n)}(\gseqi,\dfi) = \frac 1{\sqrt{n\Delta_n}} \sum\limits_{i=1}^{\ip{n\gseqi}} \xi_i \big( \chi_\dfi^{(\alpha)}(\Deli X^{(n)}) \Truniv - \chi_\dfi^{(\alpha)}(\Deli L^{(n)}) \big)
\end{equation*}
for $(\gseqi,\dfi) \in \netir$, 
converges to $0$ in $\linner$ in outer probability. 
Consequently, it suffices to show 
\begin{align*}
\hat V_\alpha^{(n)} \probto 0,
\end{align*}
for each fixed $\alpha >0$, where
\begin{align}
\label{hatValnDef}
\hat V_\alpha^{(n)} = \frac 1{\sqrt{n\Delta_n}} \sup \limits_{(\gseqi,\dfi) \in \netir} \Big| \sum \limits_{i=1}^{\ip{n\gseqi}} \xi_i \big( \chi_\dfi^{(\alpha)}(\Deli X^{(n)}) \Truniv - \chi_\dfi^{(\alpha)}(\Deli L^{(n)}) \big) \Big|.
\end{align}
Note that Lemma \ref{haValnbou} in Appendix \ref{appD} yields the estimate
\[
\hat V_\alpha^{(n)} \leq \hat D_n(\alpha) + \hat E_n(\alpha) + \hat F_n(\alpha),
\]
on $J_n^{\scriptscriptstyle (1)}(\alpha) \cap Q_n$ for $n \in \N$ large enough such that $v_n \leq \alpha/4$, where $Q_n$ is defined in \eqref{BnsetsDef}, $J_n^{\scriptscriptstyle (1)}(\alpha)$ is defined in \eqref{J1ijnalDefEq2} and with 
\begin{multline*}
\hat D_n(\alpha) = \frac{1}{\sqrt{n \Delta_n}} \sum \limits_{i=1}^n |\xi_i| \big| \rho_{\alpha} (\Deli \tilde X^{\prime n} + \Deli \hatbigjal)
\ind_{\lbrace |  \Deli \tilde X^{\prime n} + \Deli \hatbigjal| > v_n \rbrace} - \\
- \rho_{\alpha}(\Deli \hatbigjal) \ind_{\lbrace | \Deli \hatbigjal| > v_n \rbrace} \big| 
\ind_{\lbrace | \Deli \tilde X^{\prime n} | \leq v_n/2 \rbrace},
\end{multline*}
\begin{align*}
\hat E_n(\alpha) = \sup\limits_{A \in \mathfrak S_n} \Big| \sum \limits_{i \in A} \xi_i a_i^n (\alpha) \Big|,
~~ 
\hat F_n(\alpha) = \sup\limits_{A \in \mathfrak S_n} \Big| \sum \limits_{i \in A} \xi_i b_i^n (\alpha) \Big|,
\end{align*}
and 
$
a_i^n (\alpha) = \frac 1{\sqrt{n \Delta_n}} \rho_{\alpha}(\Deli \hatbigjal) \ind_{\lbrace | \Deli \hatbigjal | > 
	\alpha/4 \rbrace}, 
$
\[
b_i^n (\alpha) = \frac 1{\sqrt{n \Delta_n}} \rho_{\alpha} (\Deli \tilde X^{\prime n} + \Deli \hatbigjal) \ind_{\lbrace |  \Deli \tilde X^{\prime n} + \Deli \hatbigjal| > v_n \rbrace} \ind_{\lbrace | \Deli \hatbigjal | > 
	\alpha/4 \rbrace},
\]
where the quantities in the displays  are introduced in \eqref{SmandlajumpDef} and $\mathfrak S_n = \{ M \subset \{1, \ldots,n\} \mid \# M \leq c_n\}$ with $c_n = \lceil (v_n/ \Delta_n^{\ovr}) + 1\rceil$. Lemma \ref{QnConv} and Lemma \ref{PJn1alconv} in Appendix \ref{appD}
show $\Prob(J_n^{\scriptscriptstyle (1)}(\alpha) \cap Q_n) \to 1$ and thus it is further enough to verify
\begin{align}
\hat D_n(\alpha) &= o_\Prob(1), \label{d12} \\
\hat E_n(\alpha) &= o_\Prob(1), \label{e12} \\
\hat F_n(\alpha) &= o_\Prob(1), \label{f12} 
\end{align}
for each $\alpha >0$ as $n \to \infty$. \\
Recall the quantity $D_n(\alpha)$ introduced in \eqref{DnalDef3}. \eqref{Dnalkonvb} and Lemma \ref{anabnakonv} yield $\Eb D_n(\alpha) \rightarrow 0$. Moreover, the bootstrap multipliers have variance $1$ and satisfy therefore $\Eb | \xi_i | \leq 1$ for all $i \in \N$. Thus because of the independence of the multipliers and the other involved processes we obtain $0 \leq \Eb \hat D_n(\alpha) \leq \Eb D_n(\alpha) \rightarrow 0$, which proves \eqref{d12}. 
Concerning \eqref{e12} we note that Lemma \ref{ainalmombou} in Appendix \ref{appD}  implies 
(for $n\in\N$ large enough)
$
\Eb |a_i^n (\alpha)|^m \leq \big( \frac{K(\alpha)}{\sqrt{n \Delta_n}} \big)^m \Delta_n,
$
for some $K(\alpha)>0$, all $m \in \N$ and all $i=1,\ldots,n$. Thus, using Assumption \ref{MultiplAss} as well as independence of $\xi_i$ and $a_i^n(\alpha)$ we obtain for every integer $m \geq 2$ and $n\in\N$ large enough
\[
\Eb |Z_i^n(\alpha)|^m \leq m! \Big( \frac{C_1}{\sqrt{n\Delta_n}}\Big)^{m-2} \frac{C_2}n,
\]
where 
$
Z_i^n(\alpha) = \xi_i a_i^n(\alpha) .
$
Furthermore, due to the definition of $\hatbigjal$ in \eqref{SmandlajumpDef} the variables $(Z_i^n(\alpha))_{i=1,\ldots,n}$ are independent with mean zero. Consequently, Lemma \ref{supsnzinlem} shows
\[
\Eb \hat E_n(\alpha) = \Eb \Big\{ \sup_{A \in \mathfrak S_n} \Big| \sum_{i \in A} Z_i^n(\alpha) \Big| \Big\} = o(1),
\]
which proves \eqref{e12}. \\
In order to show \eqref{f12} observe first that for $n \in \N$ large enough
\begin{align*}
\ind_{\lbrace |  \Deli \tilde X^{\prime n} + \Deli \hatbigjal| > v_n \rbrace} \ind_{\lbrace | \Deli \hatbigjal | > 
	\alpha/4 \rbrace} =  \ind_{\lbrace | \Deli \hatbigjal | > 
	\alpha/4 \rbrace} 
\end{align*}
holds for each $i=1,\ldots,n$ on the set $Q_n$, because by \eqref{BnsetsDef} we have $| \Deli \tilde X^{\prime n} | \leq v_n/2$ on $Q_n$. Therefore, we obtain from the mean value theorem for large $n$ on the set $Q_n$
\begin{align*}
b_i^n(\alpha) = a_i^n(\alpha) + \frac 1{\sqrt{n \Delta_n}} \ind_{\lbrace | \Deli \hatbigjal | > 
	\alpha/4 \rbrace} \ind_{\lbrace | \Deli \tilde X^{\prime n} | \leq v_n/2 \rbrace} \Deli \tilde X^{\prime n} \Big(\frac{d \rho_\alpha}{dx}\Big)(\zeta_i^n(\alpha))
\end{align*}
for some $\zeta_i^n(\alpha)$ between $\Deli \hatbigjal$ and $\Deli \tilde X^{\prime n} + \Deli \hatbigjal$. Thus, the indicators and Assumption \ref{Cond1}\eqref{RhoCond} show for large $n \in\N$
\begin{align}
\label{EnalFnaldi}
\big| \hat E_n(\alpha) - \hat F_n(\alpha) \big| \ind_{Q_n} &\leq \sup \limits_{A \in \mathfrak S_n} \frac K{\sqrt{n \Delta_n}} \sum \limits_{i \in A} |\xi_i| |\Deli \hatbigjal|^{p-1} |\Deli \tilde X^{\prime n}| \times \nonumber \\
&\hspace{45mm} \times \ind_{\lbrace | \Deli \hatbigjal | > 
	\alpha/4 \rbrace} \ind_{\lbrace | \Deli \tilde X^{\prime n} | \leq v_n/2 \rbrace} \nonumber \\
&\leq \frac {Kv_n}{\sqrt{n \Delta_n}} \sum \limits_{i=1}^n |\xi_i||\Deli \hatbigjal|^{p-1}.
\end{align}
The bootstrap multipliers are defined on a distinct probability space and satisfy $\Eb |\xi_i| \leq 1$ for all $i\in\N$. As a consequence, \eqref{EnalFnaldi} gives
\[
\Eb \big| \hat E_n(\alpha) - \hat F_n(\alpha) \big| \ind_{Q_n} \leq K b_n(\alpha),
\]
with
\[
b_n(\alpha) = \frac{v_n}{2 \sqrt{n \Delta_n}} \sum \limits_{i=1}^n \Eb \big| \Deli \hatbigjal \big|^{p-1}.
\]
Therefore, \eqref{f12} follows from \eqref{e12}, Lemma \ref{QnConv} and Lemma \ref{anabnakonv}.
\qed

\medskip

\noindent 
\textbf{Proof of Lemma \ref{NoContrfsmalL}.}
Due to Proposition \ref{LnCondWeConv} we have $\hat Y_{\rho_\alpha^\circ}^{(n)} \weakP \Gb_{\rho_\alpha^\circ}$ in $\linner$ for every $\alpha >0$. Thus, according to Lemma \ref{oP1glzcwecole}  in Apendix \ref{appF} it suffices to show
\begin{align*}
\sup\limits_{(\gseqi,\dfi) \in \netir} \big|\hat G_{\rho,n}^{\circ (\alpha)} - \hat Y_{\rho_\alpha^\circ}^{(n)} \big| = o_\Prob(1)
\end{align*}
for each $\alpha >0$ in a neighbourhood of zero. Simple manipulations give
\begin{equation*}
\hat G_{\rho,n}^{\circ (\alpha)}(\gseqi,\dfi) - \hat Y_{\rho_\alpha^\circ}^{(n)}(\gseqi,\dfi) = \frac 1{\sqrt{n\Delta_n}} \sum\limits_{i=1}^{\ip{n\gseqi}} \xi_i \big( \chi_\dfi^{\circ (\alpha)}(\Deli X^{(n)}) \Truniv  - \chi_\dfi^{\circ (\alpha)}(\Deli L^{(n)}) \big) 
\end{equation*}
for $(\gseqi,\dfi) \in \netir$. 
As a consequence, it only remains to show that 
\begin{align}
\label{genhavnalprzz}
\hat V_\alpha^{\circ (n)} \probto 0,
\end{align}
for each $\alpha >0$ in a neighbourhood of $0$ with
\begin{align*}
\hat V_\alpha^{\circ (n)} = \frac 1{\sqrt{n\Delta_n}} \sup_{(\gseqi,\dfi) \in \netir} \Big| \sum\limits_{i=1}^{\ip{n\gseqi}} \xi_i \big( \chi_\dfi^{\circ (\alpha)}(\Deli X^{(n)}) \Truniv - \chi_\dfi^{\circ (\alpha)}(\Deli L^{(n)}) \big) \Big|.
\end{align*}
 Lemma \ref{haVpralbou}  in Appendix \ref{appD} implies for each $\alpha >0$ and  sufficiently large $n \in \N$ with  $v_n \leq \alpha$  the bound 
\[
\hat V_\alpha^{\circ (n)} \leq \hat C_n^\circ (\alpha) + \hat D_n^\circ (\alpha) + \hat E_n^\circ (\alpha) + \hat F_n^\circ (\alpha)
\]
on $J_n^{\scriptscriptstyle (2)}(\alpha) \cap Q_n$, where  $Q_n$ is defined in \eqref{BnsetsDef}, $J_n^{\scriptscriptstyle (2)}(\alpha)$ is defined in \eqref{Jn2aldef31} and with
\begin{align*}
&\hat C_n^\circ (\alpha) = \frac 1{\sqrt{n\Delta_n}} \sup\limits_{\dfi \in \R} \sum \limits_{i=1}^n |\xi_i| \big| \rho_\alpha^\circ(\varsigma_i^n(\alpha)) \ind_{\{|\varsigma_i^n(\alpha)| > v_n\}} \Indit(\varsigma_i^n(\alpha)) - \\
&\hspace{2mm} - \rho_\alpha^\circ(\Deli \tibigjeial) 
\ind_{\{|\Deli \tibigjeial| > v_n\}} \Indit(\Deli \tibigjeial) \big| \ind_{\{|\Deli \tibigjeial| \leq \Delta_n^{\ovv}\}} \ind_{\{|\Deli \tilde X^{\prime n}| \leq v_n/2\}},
\end{align*}
\begin{align*}
\hat D_n^\circ (\alpha) &= \frac 1{\sqrt{n\Delta_n}} \sum \limits_{i=1}^n |\xi_i| \big| \rho^{\circ}_{\alpha}(\varsigma_i^n(\alpha)) \ind_{\lbrace | \varsigma_i^n(\alpha)| > v_n \rbrace}  -\rho^{\circ}_{\alpha}(\Deli \tibigjeial) \ind_{\lbrace | \Deli \tibigjeial | > v_n \rbrace} \big| \times \\
&\hspace{75mm} \times \ind_{\{|\Deli \tibigjeial| > \Delta_n^{\ovv}\}} \ind_{\{|\Deli \tilde X^{\prime n}| \leq v_n/2\}}, 
\end{align*}
\begin{align*}
\hat E_n^\circ (\alpha) = \sup\limits_{A \in \mathfrak S_n} \Big| \sum\limits_{i \in A} \xi_i \bar a_i^n(\alpha) \Big|,
~~
\hat F_n^\circ (\alpha) = \sup\limits_{A \in \mathfrak S_n} \Big| \sum\limits_{i \in A} \xi_i \bar b_i^n(\alpha) \Big|,
\end{align*}
where the processes involved in the display above have been introduced in \eqref{SmandlajumpDef}, $\ovv >0$ is the constant from Assumption \ref{Cond1}\eqref{FiLevyDistCond}, $\varsigma_i^n(\alpha) = \Deli \tilde X^{\prime n} + \Deli \tibigjeial$, $\mathfrak S_n = \{ M \subset \{1, \ldots,n\} \mid \# M \leq c_n\}$ for $c_n = \lceil (v_n/ \Delta_n^{\ovr}) + 1\rceil$ and with
\begin{align*}
\bar a_i^n(\alpha) &= \frac 1{\sqrt{n\Delta_n}} \rho_\alpha^\circ(\Deli \tibigjeial) \ind_{\{| \Deli \tibigjeial| > v_n \vee \Delta_n^{\ovv}\}}, \\
\bar b_i^n(\alpha) &= \frac 1{\sqrt{n\Delta_n}} \rho_\alpha^\circ(\varsigma_i^n(\alpha)) 
\ind_{\{|\varsigma_i^n(\alpha)| > v_n\}} \ind_{\{ | \Deli \tibigjeial| > \Delta_n^{\ovv}\}}.
\end{align*}
Lemma \ref{QnConv} and Lemma \ref{BiJovrovvAb} show $\Prob(J_n^{\scriptscriptstyle (2)}(\alpha) \cap Q_n) \to 1$ for each $\alpha >0$ small enough and consequently it suffices to verify
\begin{align}
\hat C_n^\circ (\alpha) &= o_\Prob(1), \label{c24} \\
\hat D_n^\circ (\alpha) &= o_\Prob(1), \label{d24} \\
\hat E_n^\circ (\alpha) &= o_\Prob(1), \label{e24} \\
\hat F_n^\circ (\alpha) &= o_\Prob(1), \label{f24}
\end{align}
for all $\alpha >0$ as $n \to \infty$. \\
Concerning \eqref{c24}, we have due to the triangle inequality and $|\rho_\alpha^\circ(\taili)|\leq K|\taili|^p$ for all $\taili \in \R$
\begin{multline*}
\hat C_n^\circ (\alpha) \leq \frac{K}{\sqrt{n\Delta_n}} \sum \limits_{i=1}^n |\xi_i| \Big(|\varsigma_i^n(\alpha)|^p \ind_{\{| \varsigma_i^n(\alpha)|>v_n\}} + |\Deli \tibigjeial|^p \ind_{\{|\Deli \tibigjeial|>v_n\}} \Big) \times \\
\times \ind_{\lbrace | \Deli \tilde X^{\prime n} | \leq v_n/2 \rbrace}  \ind_{\lbrace | \Deli \tibigjeial | \leq \Delta_n^{\ovv} \rbrace} \ind_{Q_n}
\end{multline*}
for fixed $\alpha >0$ on the set $Q_n$ and where $\varsigma_i^n(\alpha) = \Deli \tilde X^{\prime n} + \Deli \tibigjeial$. Consequently, because of $\Eb|\xi_i| \leq 1$ for each $i\in\N$ and the fact that the multipliers are defined on a distinct probability space we obtain
\begin{align*}
\Eb \big( \hat C_n^\circ (\alpha) \ind_{Q_n} \big) \leq K \big( y_n^{(\alpha)} + 2 z_n^{(\alpha)} \big),
\end{align*}
with 
\begin{multline*} 
y_n^{(\alpha)} = \frac{1}{\sqrt{n \Delta_n}} \sum \limits_{i=1}^n \Eb \Big\{ \big| \big| \Deli \tilde X^{\prime n} +\Deli \tibigjeial \big|^p - \big| \Deli \tibigjeial \big|^p  \big| \times \\
\times \ind_{\{|\Deli \tibigjeial | \leq \Delta_n^{\ovv} \}} 
\ind_{\{|\Deli \tilde X^{\prime n} +\Deli \tibigjeial |  >v_n \}} \ind_{\lbrace | \Deli \tilde X^{\prime n} | \leq v_n/2 \rbrace}\ind_{Q_n} \Big\},
\end{multline*}
and 
\begin{align*}
z_n^{(\alpha)} = \frac{1}{\sqrt{n \Delta_n}} \sum \limits_{i=1}^n \Eb \Big\{ \big| \Deli \tibigjeial \big|^p \ind_{\{|\Deli \tibigjeial | \leq \Delta_n^{\ovv} \}} \ind_{Q_n}  \Big\}.
\end{align*}
Thus, \eqref{c24} follows using $\Prob(Q_n) \to 1$ and Lemma \ref{ynaznakonv}. \\
Recall the quantity $D_n^{\circ}(\alpha)$ introduced in \eqref{Dnpraldef}. Because of \eqref{Dnpralkon}, Lemma \ref{cnadnakonv} and the fact that the multipliers $(\xi_i)_{i\in\N}$ are independent of the other involved quantities and satisfy $\Eb|\xi_i| \leq 1$ we obtain $0 \leq \Eb \hat D_n^\circ (\alpha) \leq \Eb D_n^{\circ}(\alpha) \rightarrow 0$, which proves \eqref{d24}. \\
In order to show \eqref{e24} we note that Lemma \ref{baainamomb} in Appendix \ref{appD} implies 
for $\alpha >0$ and sufficiently large $n\in\N$  the bound
\[
\Eb |\bar a_i^n (\alpha)|^m \leq \Big( \frac{K}{\sqrt{n \Delta_n}} \Big)^m \Delta_n,
\]
for some $K>0$, all $m \in \N$ and all $i=1,\ldots,n$. Thus, with Assumption \ref{MultiplAss} as well as independence of $\xi_i$ and $\bar a_i^n(\alpha)$ we obtain for every integer $m \geq 2$ and $n\in\N$ large enough
\[
\Eb |\bar Z_i^n(\alpha)|^m \leq m! \Big( \frac{C_1}{\sqrt{n\Delta_n}}\Big)^{m-2} \frac{C_2}n,
\]
 where 
$
\bar Z_i^n(\alpha) = \xi_i \bar a_i^n(\alpha)
$. Furthermore, due to the definition of $\tibigjeial$ in \eqref{SmandlajumpDef} the variables $(\bar Z_i^n(\alpha))_{i=1,\ldots,n}$ are independent with mean zero. Consequently, Lemma \ref{supsnzinlem} shows
\[
\Eb \hat E_n^\circ (\alpha) = \Eb \Big\{ \sup_{A \in \mathfrak S_n} \Big| \sum_{i \in A} \bar Z_i^n(\alpha) \Big| \Big\} = o(1),
\]
which proves \eqref{e24}. \\
Concerning \eqref{f24} notice first of all that we have $\bar a_i^n(\alpha) = \bar b_i^n(\alpha)=0$ on the set $\{|\Deli \tibigjeial| \leq v_n/2\} \cap Q_n$, because of the indicators in the definition of these terms and the fact that $|\Deli \tilde X^{\prime n} | \leq v_n/2$ holds for each $i=1, \ldots,n$ on $Q_n$ according to \eqref{BnsetsDef}. Therefore, we have for arbitrary $i \in \{1, \ldots,n\}$
\begin{multline}
\label{DifbaaibabiAb}
|\bar a_i^n(\alpha) - \bar b_i^n(\alpha)| \ind_{Q_n} = |\bar a_i^n(\alpha) - \bar b_i^n(\alpha)| \ind_{Q_n} \ind_{\{v_n/2 < |\Deli \tibigjeial | \leq 2v_n\}} + \\ +  |\bar a_i^n(\alpha) - \bar b_i^n(\alpha)| \ind_{Q_n} \ind_{\{|\Deli \tibigjeial | > 2v_n\}}.
\end{multline}
For the first summand on the right-hand side of \eqref{DifbaaibabiAb} the triangle inequality gives
\begin{multline*}
|\bar a_i^n(\alpha) - \bar b_i^n(\alpha)| \ind_{Q_n} \ind_{\{v_n/2 < |\Deli \tibigjeial | \leq 2v_n\}} \leq \\
\frac 1{\sqrt{n \Delta_n}} ( |\rho_\alpha^\circ (\Deli \tibigjeial)| + |\rho_\alpha^\circ (\Deli \tilde X^{\prime n}+ \Deli \tibigjeial)| ) \ind_{Q_n} \ind_{\{v_n/2 < |\Deli \tibigjeial | \leq 2v_n\}}.
\end{multline*}
Due to Assumption \ref{Cond1}\eqref{RhoCond} we have $|\rho(\taili)| \leq K|\taili|^p$ for all $\taili \in \R$ and some $K>0$. Thus, because $|\Deli \tilde X^{\prime n}| \leq |\Deli \tibigjeial|$ holds on $\{v_n/2 < |\Deli \tibigjeial |\} \cap Q_n$ we further obtain
\begin{align}
\label{FiSumAbsch}
|\bar a_i^n(\alpha) - \bar b_i^n(\alpha)| &\ind_{Q_n} \ind_{\{v_n/2 < |\Deli \tibigjeial | \leq 2v_n\}} \nonumber \\
&\leq \frac 1{\sqrt{n \Delta_n}} \big( K |\Deli \tibigjeial |^p + K 2^p |\Deli \tibigjeial|^p \big) \ind_{Q_n} \ind_{\{v_n/2 < |\Deli \tibigjeial | \leq 2v_n\}} \nonumber \\
&\leq \frac K{\sqrt{n \Delta_n}} |\Deli \tibigjeial |^p \ind_{Q_n} \ind_{\{ |\Deli \tibigjeial | \leq 2v_n\}} \nonumber \\
&\leq \frac {Kv_n}{\sqrt{n \Delta_n}} \big|\Deli \tibigjeial \big|^{p-1}.
\end{align}
Note   that $|\Deli \tilde X^{\prime n}| \leq v_n/2$ on $Q_n$ ($i = 1,\ldots,n$), which  yields for the second summand in \eqref{DifbaaibabiAb}
\begin{align*}
|&\bar a_i^n(\alpha) - \bar b_i^n(\alpha)| \ind_{Q_n} \ind_{\{|\Deli \tibigjeial | > 2v_n\}} \\
&= \frac 1{\sqrt{n\Delta_n}} \big| \rho_\alpha^\circ(\Deli \tibigjeial) - \rho_\alpha^\circ (\Deli \tilde X^{\prime n}+ \Deli \tibigjeial) \big| \ind_{Q_n} \ind_{\{|\Deli \tibigjeial | > 2v_n \vee \Delta_n^{\ovv}\}}.
\end{align*}
The derivative of the function $\Psi_\alpha^\circ$ from \eqref{ImpQuantDef} is supported by a compact set which is bounded away from the origin. Therefore, by Assumption \ref{Cond1}\eqref{RhoCond} there exists a constant $K>0$, which may depend on $\alpha$, such that the derivative satisfies $|\frac d{d\taili} \rho_\alpha^\circ (\taili)| \leq K |\taili|^{p-1}$. As a consequence, we have due to the mean value theorem and $|\Deli \tilde X^{\prime n}| \leq |\Deli \tibigjeial|$ on $\{|\Deli \tibigjeial| > 2v_n\} \cap Q_n$
\begin{align}
\label{SecSumAbsch}
|\bar a_i^n(\alpha) &- \bar b_i^n(\alpha)| \ind_{Q_n} \ind_{\{|\Deli \tibigjeial | > 2v_n\}}  \\
&\leq \frac K{\sqrt{n\Delta_n}} |\Deli \tilde X^{\prime n}| |\Deli \tibigjeial|^{p-1} \ind_{Q_n} \ind_{\{|\Deli \tibigjeial | > 2v_n \vee \Delta_n^{\ovv}\}} \nonumber 
\leq \frac {Kv_n}{\sqrt{n \Delta_n}} \big|\Deli \tibigjeial \big|^{p-1}.
\end{align}
Finally we conclude with \eqref{DifbaaibabiAb}, \eqref{FiSumAbsch} and \eqref{SecSumAbsch}
\begin{align}
\label{FinDiffConEq}
& \Eb \big( |\hat E_n^\circ (\alpha) - \hat F_n^\circ (\alpha)| \ind_{Q_n} \big) \leq \Eb \big( \ind_{Q_n} \sup \limits_{A \in \mathfrak S_n} \sum \limits_{i \in A} |\xi_i| |\bar a_i^n(\alpha) - \bar b_i^n(\alpha)| \big)  \\
&~~~~~~~~~~~\leq \Eb \big( \ind_{Q_n} \sum \limits_{i=1}^n |\xi_i| |\bar a_i^n(\alpha) - \bar b_i^n(\alpha)| \big) 
\leq \frac{Kv_n}{\sqrt{n\Delta_n}} \sum \limits_{i=1}^n \Eb \big| \Deli \tibigjeial \big|^{p-1} \rightarrow 0, \nonumber 
\end{align}
because the multipliers are defined on a distinct probability space and satisfy $\Eb |\xi_i| \leq 1$ for all $i \in \N$. The final convergence in the display above holds due to Lemma \ref{cnadnakonv}. \eqref{FinDiffConEq} together with $\Prob(Q_n) \rightarrow 1$ (see Lemma \ref{QnConv}) shows $\hat F_n^\circ (\alpha) \probto 0$, because in the previous part of the proof we have already verified $\hat E_n^\circ (\alpha) \probto 0$. 
\qed


\section[Technical details]{Technical details in the proofs of Theorem \ref{ConvThm} and \ref{CondConvThm}}
\label{appD}
\def\theequation{C.\arabic{equation}}
\setcounter{equation}{0}

In this appendix we give the details of the proofs of Theorem \ref{ConvThm} and Theorem \ref{CondConvThm}. Here and also in the appendices \ref{appE} and \ref{appF} $K$ or $K(\alpha)$ denote generic constants which sometimes depend on a further quantity $\alpha$ and may change from place to place. 

\subsection{Moments of functionals of integer-valued random measures}

\cite{HofVet15} used Lemma 2.1.5 and Lemma 2.1.7 of \cite{JacPro12} frequently in order to achieve their weak convergence results. However, in \cite{JacPro12} these results are only proved for Poisson random measures with a predictable compensator of the form $ds \otimes F(dz)$ with a L\'evy measure $F$. Therefore, using tools from \cite{Jac79} we prove the generalized versions stated below. First, we introduce some notations. Let $\pf$ be an integer-valued random measure on $\R_+ \times \R$ with predictable compensator $\qf(\omega;ds,d\taili) = \nu_s(\omega;d\taili) ds$ for a transition kernel $\nu_s(\omega;d\taili)$ from $(\Omega\times\R_+, \Pc)$ into $(\R,\Bb)$, where $\Pc$ is the predictable $\sigma$-algebra on $\Omega \times \R_+$ (with respect to some prespecified filtration $(\Fc_t)_{t\in\R_+}$) and where an integer-valued random measure is a random measure which satisfies the requirements of Definition II.1.3 and Definition II.1.13 in \cite{JacShi02}. Furthermore, we set $\Omega' = \Omega\times\R_+\times\R$ and $\Pc' = \Pc \otimes \Bb$ is the predictable $\sigma$-algebra on $\Omega'$. Then, for a real-valued $ \Pc'$-measurable function $\delta$ on $\Omega'$ and $p,t\in \R_+$, $u>0$ let
\begin{align*}
\hat \delta(p)_{t,u}(\omega) = \frac{1}{u} \int_t^{t+u} \int |\delta(\omega,s,\taili)|^p \nu_s(\omega;d\taili) ds.
\end{align*}

\begin{lemma}
	\label{JacProL215}
	Suppose that $\hat\delta(2)_{0,u} < \infty$ almost surely for all $u >0$. Then the process $Y= \delta \star (\pf - \qf)$ is a locally square integrable martingale, and for all finite stopping times $T$ and $u>0$ we have for $p \in [1,2]$
	\begin{align*}
	\Eb\Big( \sup\limits_{0\leq v \leq u} |Y_{T+v} - Y_T|^p \mid \Fc_T \Big) \leq K_p u \Eb\big( \hat\delta(p)_{T,u} \mid \Fc_T\big),
	\end{align*}
	and also for $p \geq 2$
	\begin{align*}
	\Eb\Big( \sup\limits_{0\leq v \leq u} |Y_{T+v} - Y_T|^p \mid \Fc_T \Big) \leq K_p \big( u \Eb\big(\hat\delta(p)_{T,u} \mid \Fc_T\big) + u^{p/2} \Eb\big(\hat\delta(2)_{T,u}^{p/2} \mid \Fc_T\big) \big).
	\end{align*}
\end{lemma}

\begin{lemma}
	\label{JacProL217}
	Suppose that $\hat \delta(1)_{0,u} < \infty$ almost surely for all $u >0$. Then the process $Y = \delta \star \pf$ is of locally integrable variation. Furthermore, for all finite stopping times $T$ and $u >0$ we have for $p \in (0,1]$
	\begin{align*}
	\Eb\Big(\sup\limits_{0\leq v\leq u} |Y_{T+v} - Y_T|^p \mid \Fc_T\Big) \leq K_p u \Eb\big(\hat\delta(p)_{T,u} \mid \Fc_T\big)
	\end{align*}
	and for $p\geq 1$
	\begin{align*}
	\Eb\Big(\sup\limits_{0\leq v\leq u} |Y_{T+v} - Y_T|^p \mid \Fc_T\Big) \leq K_p \big(u \Eb\big(\hat\delta(p)_{T,u} \mid \Fc_T\big) + u^p \Eb\big(\hat\delta(1)^p_{T,u} \mid \Fc_T\big) \big).
	\end{align*}
\end{lemma}

\noindent
\textbf{Proof of Lemma \ref{JacProL215}.}
$\delta^2 \star \qf$ is a continuous increasing process and we have $\delta^2 \star \qf_t = t\hat\delta(2)_{0,t}$ for all $t >0$. Thus, for $n \in \N$ let $M_n$ be a null set such that $\delta^2 \star \qf_n$ is finite on $M_n^C$. Such a set exists by the assumption on $\delta$. Then the increasing process $\delta^2 \star \qf_t$ is finite for all $t \in \R_+$ on $M^C$ with $M= \bigcup_{n \in \N} M_n$. Therefore, $(T_n)_{n\in\N}$ defined via $T_n = \inf\{t >0 \mid \delta^2 \star \qf_t \geq n\}$ is a localizing sequence of stopping times and the stopped continuous processes satisfy $(\delta^2 \star \qf)_t^{T_n} \leq n$ for all $t \in \R_+$. Consequently, $\delta^2 \star \qf$ is locally bounded and in particular locally integrable. Thus, by Theorem II.1.33(a) in \cite{JacShi02} the process $Y$ is well-defined and a locally square integrable martingale. \\
In order to show the claimed inequalities we want to reduce our setup to the situation of Lemma 2.1.5 in \cite{JacPro12}. To this end, let $F$ be a L\'evy measure on $(\R,\Bb)$ without atoms and $F(\R)=\infty$. Furthermore, let $x_0 \notin \R$ be an exterior point of $\R$ and let $(\R_{x_0},\Bb_{x_0})$ denote the measurable one point extension of $(\R,\Bb)$, that is $\R_{x_0} = \R \cup \{x_0\}$ and $\Bb_{x_0} = \{ B,B \cup\{x_0\} \mid B \in \Bb\}$. Then according to Theorem 14.53 in \cite{Jac79} there exist a measurable function $h \colon ( \Omega', \Pc') \to (\R_{x_0},\Bb_{x_0})$ and a $\Prob$-null set $N$ such that
\begin{align}
\label{PredCompglEq}
\qf(\omega;A) = \int \int \ind_A(s,h(\omega,s,z)) F(dz)ds
\end{align}
for each $A \in \Bb(\R_+) \otimes \Bb$ and $\omega \notin N$. Additionally, by Theorem 14.56 in \cite{Jac79} there exists a filtered measurable space $(\Omega^\circ,\Fc^\circ,(\Fc^\circ_t)_{t \in \R_+})$ and a transition probability $Q(\omega,d\omega^\circ)$ from $(\Omega,\Fc)$ into $(\Omega^\circ,\Fc^\circ)$ such that on the extended filtered probability space $(\tilde\Omega,\tilde\Fc,(\tilde\Fc_t)_{t\in\R_+},\tilde\Prob)$, which is given by $\tilde \Omega= \Omega\times\Omega^\circ$, $\tilde\Fc=\Fc\otimes\Fc^\circ$, $\tilde \Fc_t = \bigcap_{s>t} \Fc_s \otimes \Fc_s^\circ$ and $\tilde \Prob(d(\omega,\omega^\circ)) = Q(\omega,d\omega^\circ)\Prob(d\omega)$, there exists a Poisson random measure $\tilde \pf$ with predictable compensator $\tilde \qf(ds,dz) = F(dz)ds$ such that for $\tilde\Prob$-almost every $\tilde\omega=(\omega,\omega^\circ)$ we have
\begin{align}
\label{RandmeaglEq}
\pf(\omega;A) = \int \int \ind_A(s,h(\omega,s,z)) \tilde\pf((\omega,\omega^\circ),ds,dz)
\end{align}
for all $A\in \Bb(\R_+)\otimes\Bb$. 
Furthermore, we identify the filtration $(\Fc_t)_{t \in \R_+}$ on $(\Omega,\Fc)$ with the induced filtration $\Fc_t \otimes \{\emptyset,\Omega^\circ\}$ on $(\tilde\Omega,\tilde\Fc)$, which we denote by $(\Fc_t)_{t \in \R_+}$ as well. Any random variable $X$ on $(\Omega,\Fc)$ will be identified with the induced mapping $X(\omega,\omega^\circ) = X(\omega)$. Then we have for every $A \in \tilde\Fc$ and every stopping time $T$ on $(\Omega,\Fc)$
\begin{align*}
A\in \Fc_T(\Omega) \otimes \{\emptyset,\Omega^\circ\} &\Longleftrightarrow A = A_1 \times \Omega^\circ \text{ for some } A_1 \in \Fc_T(\Omega) \\
&\Longleftrightarrow A \cap \{T \leq t\} \in \Fc_t \otimes \{\emptyset, \Omega^\circ\} \text{ for every } t \in \R_+ \\
&\Longleftrightarrow A \in \Fc_T(\tilde\Omega),
\end{align*}
where for the sake of a clear notation we denote by $\Fc_T(\Omega)$ and $\Fc_T(\tilde\Omega)$, respectively, the $\sigma$-algebra of events up to time $T$  with respect to $(\Omega,\Fc,(\Fc_t)_{t\in\R_+})$ and $(\tilde\Omega,\tilde\Fc,(\Fc_t)_{t\in\R_+})$, respectively. Consequently, for $A=A_1 \times \Omega^\circ \in \Fc_T(\tilde\Omega)$ with $A_1 \in \Fc_T(\Omega)$ and a random variable $X$ on $(\Omega,\Fc)$ we have by the definition of the conditional expectation
\begin{align*}
 \int_{A_1\times\Omega^\circ} X d\tilde\Prob &= \int_{A_1} \int_{\Omega^\circ} X(\omega) Q(\omega,d\omega^\circ)\Prob(d\omega) \\
&= \int_{A_1} X d\Prob = \int_{A_1} \Eb_\Prob(X\mid\Fc_T) d\Prob = \int_{A_1\times\Omega^\circ} \Eb_\Prob(X\mid\Fc_T) d\tilde\Prob
\end{align*}
and thus 
\begin{align}
\label{CondExpGlEq}
\Eb_{\tilde\Prob}(X\mid \Fc_T) = \Eb_\Prob(X\mid\Fc_T) \quad \tilde \Prob-\text{almost surely}. 
\end{align}
Let $\Oc$ be the optional $\sigma$-algebra on $\Omega\times\R_+$ and let $\tilde \Oc$ denote the optional $\sigma$-algebra on $\tilde\Omega\times\R_+$. Then by Proposition II.1.14 there exist a thin random set $D \in \Oc$, an optional process $(\beta_s)_{s \in \R_+}$on $(\Omega, \Fc, (\Fc_t)_{t \in \R_+},\Prob)$, a thin random set $\tilde D \in \tilde \Oc$ and an optional process $(\tilde \beta_s)_{s \in \R_+}$ on $(\tilde\Omega, \tilde\Fc, (\tilde\Fc_t)_{t \in \R_+},\tilde\Prob)$ such that
\begin{align*}
\pf(\omega;ds,dz) &= \sum\limits_{t \geq 0} \ind_D(\omega,t) \epsilon_{(t,\beta_t(\omega))}(ds,dz) \\
\tilde \pf((\omega,\omega^\circ);ds,dz) &= \sum\limits_{t \geq 0} \ind_{\tilde D}((\omega,\omega^\circ),t) \epsilon_{(t,\tilde \beta_t(\omega,\omega^\circ))}(ds,dz),
\end{align*}
for every $(\omega,\omega^\circ) \in \tilde \Omega$, where $\epsilon_{(x,y)}$ is the Dirac measure on $\R_+ \times \R$ with mass in $(x,y)$. As a consequence, we obtain from \eqref{RandmeaglEq} 
\begin{align*}
\delta(\omega,t,\beta_t(\omega)) \ind_D(\omega,t) &= \int\int \delta(\omega,s,z) \ind_{\{s=t\}} \pf(\omega;ds,dz)  \\
&= \int\int \delta(\omega,s,h(\omega,s,z)) \ind_{\{s=t\}} \tilde\pf((\omega,\omega^\circ);ds,dz)  \\
&= \delta(\omega,t,h(\omega,t,\tilde\beta_t(\omega,\omega^\circ))) \ind_{\tilde D}((\omega,\omega^\circ),t),
\end{align*}
for every $t\geq 0$ and $\tilde\Prob$-almost every $(\omega,\omega^\circ)$, where we set $f(\omega,s,h(\omega,s,z))=0$ if $h(\omega,s,z) = x_0$ for a real-valued predictable function $f$ on $\Omega'$. Thus, the processes $\delta(\omega,t,\beta_t(\omega)) \ind_D(\omega,t)$ and $\delta(\omega,t,h(\omega,t,\tilde\beta_t(\omega,\omega^\circ))) \ind_{\tilde D}((\omega,\omega^\circ),t)$ are $\tilde\Prob$-indistinguishable on $(\tilde\Omega, \tilde\Fc, (\tilde\Fc_t)_{t \in \R_+},\tilde\Prob)$ and the stochastic integrals $\delta \star (\pf - \qf)$ and $(\delta \circ h) \star(\tilde \pf - \tilde \qf)$ are $\tilde \Prob$-indistinguishable as well (cf. Definition II.1.27 in \cite{JacShi02}), where for the sake of brevity $(\delta \circ h)$ denotes the predictable map $(\omega,s,z) \mapsto \delta(\omega,s,h(\omega,s,z))$ on $\Omega'$. Notice that $\delta^2 \star \qf = (\delta^2 \circ h) \star \tilde \qf$ outside a null set due to \eqref{PredCompglEq}. Thus, the same reasoning as at the beginning of the proof shows that $\tilde Y_t := (\delta \circ h) \star(\tilde \pf - \tilde \qf)_t$ is well-defined and a locally square integrable martingale. Finally, for every finite stopping time $T$ and all $u>0$, $p \geq 1$ the variables $\sup_{0\leq v \leq u} |Y_{T+v} - Y_T|^p $ and $\sup_{0\leq v \leq u} |\tilde Y_{T+v} -\tilde Y_T|^p$ coincide $\tilde\Prob$-almost surely. Consequently, using \eqref{CondExpGlEq} we obtain
\begin{align}
\label{asEFTglEq}
\Eb_\Prob\Big(\sup\limits_{0\leq v\leq u} |Y_{T+v} - Y_T|^p \mid \Fc_T\Big) &= \Eb_{\tilde\Prob}\Big(\sup\limits_{0\leq v\leq u} |Y_{T+v} - Y_T|^p \mid \Fc_T\Big) \nonumber \\
&= \Eb_{\tilde \Prob}\Big(\sup\limits_{0\leq v\leq u} |\tilde Y_{T+v} - \tilde Y_T|^p \mid \Fc_T\Big).
\end{align}
Now, Lemma 2.1.5 in \cite{JacPro12}, \eqref{PredCompglEq} and \eqref{CondExpGlEq} give for $p \in [1,2]$
\begin{align}
\label{finalineqEq}
\Eb_\Prob\Big(\sup\limits_{0\leq v\leq u} |Y_{T+v} - Y_T|^p \mid \Fc_T\Big) 
&\leq K_p u \Eb_{\tilde\Prob} \Big(\frac 1u \int_T^{T+u} \int |\delta(\omega,s,h(\omega,s,z))|^p F(dz) ds \mid \Fc_T\Big) 
\nonumber \\
&= K_p u \Eb_{\Prob} \Big(\frac 1u \int_T^{T+u} \int |\delta(\omega,s,h(\omega,s,z))|^p F(dz) ds \mid \Fc_T\Big) \nonumber \\
&= K_p u \Eb\big(\hat\delta(p)_{T,u} \mid \Fc_T\big).
\end{align}
The second asserted inequality follows with exactly the same reasoning.
\qed

\medskip

\noindent
\textbf{Proof of Lemma \ref{JacProL217}.}
In the same way as at the beginning of the proof of Lemma \ref{JacProL215} we see that the increasing, continuous and finite-valued process $|\delta| \star \qf$ is locally bounded. Hence, by the definition of the predictable compensator (Theorem II.1.8 in \cite{JacShi02}) the process $|\delta| \star \pf$ is locally integrable and thus $Y$ is of locally integrable variation. 

With the same quantities as in the proof of Lemma \ref{JacProL215} we obtain from \eqref{RandmeaglEq} 
\begin{align*}
Y_t &= \int\int \ind_{[0,t]}(s) \delta(\omega,s,z) \pf(\omega;ds,dz) \\
&= \int\int \ind_{[0,t]}(s) \delta(\omega,s,h(\omega,s,z)) \tilde\pf((\omega,\omega^\circ);ds,dz) 
= (\delta \circ h) \star \tilde \pf_t =: \tilde Y_t,
\end{align*}
for all $t \in \R_+$ $\tilde\Prob(d(\omega,\omega^\circ))$-almost surely. Thus, we have
$
\sup\limits_{0\leq v\leq u} |Y_{T+v} - Y_T|^p = \sup\limits_{0\leq v\leq u} |\tilde Y_{T+v} - \tilde Y_T|^p
$
$\tilde \Prob$-almost surely for all finite stopping times $T$ and $p,u>0$. Now, the same reasoning as in \eqref{asEFTglEq} and \eqref{finalineqEq}, but using Lemma 2.1.7 of \cite{JacPro12} instead, yields the desired inequalities.
\qed

\begin{remark}
In the proofs in this paper integral processes of the form $Y_{(1)} =  \delta \star \mu^{\scriptscriptstyle (n)}$ and $Y_{(2)} = \delta \star (\mu^{\scriptscriptstyle (n)} - \bar \mu^{\scriptscriptstyle (n)})$ occur frequently, where $\mu^{\scriptscriptstyle (n)}$ is the random measure associated with the jumps of the underlying process, $ \bar \mu^{\scriptscriptstyle (n)}$ denotes its predictable compensator and $\delta$ is some suitable $\Pc'$-measurable function on $\Omega'$. When we want to apply Lemma \ref{JacProL215} and Lemma \ref{JacProL217} to these processes the question is whether the condition $\hat \delta(2)_{0,u} < \infty$ almost surely or $\hat \delta(1)_{0,u} < \infty$ almost surely is satisfied for all $u>0$, respectively. However, due to the observation scheme $\{X^{\scriptscriptstyle (n)}_{i\Delta_n}\mid i=0,1,\ldots,n\}$ only values of the processes $Y_{(1),t}$ $,Y_{(2),t}$ for $t \leq n\Delta_n$ are relevant and we can consider the stopped processes $Y_{(1)}^{\scriptscriptstyle T_n}$, $Y_{(2)}^{\scriptscriptstyle T_n}$ instead, where $T_n \equiv n\Delta_n$ is the constant stopping time. According to Definition II.1.27 and Proposition II.1.30 in \cite{JacShi02} we have $Y_{(1)}^{\scriptscriptstyle T_n} = \delta \star \eta^{\scriptscriptstyle (n)}$ and  $Y_{(2)}^{\scriptscriptstyle T_n} = \delta \star (\eta^{\scriptscriptstyle (n)} - \bar \eta^{\scriptscriptstyle (n)})$, where $\eta^{\scriptscriptstyle (n)}$ denotes the restriction of $\mu^{\scriptscriptstyle (n)}$ to the set $[0,n\Delta_n] \times \R$. Obviously, the predictable compensator $\bar \eta^{\scriptscriptstyle (n)}$ of $\eta^{\scriptscriptstyle (n)}$ is the restriction of $\bar \mu^{\scriptscriptstyle (n)}$ to $[0,n\Delta_n] \times \R$. As a consequence, we have $n\Delta_n \hat\delta(2)_{0,n\Delta_n} = n \Delta_n \hat \delta(2,\bar\eta)_{0,n\Delta_n} =  u \hat \delta(2,\bar\eta)_{0,u}$ and $n\Delta_n \hat\delta(1)_{0,n\Delta_n} = n \Delta_n \hat \delta(1,\bar\eta)_{0,n\Delta_n} =  u \hat \delta(1,\bar\eta)_{0,u}$ for all $u \ge n \Delta_n$, where 
	$\hat \delta(2,\bar\eta)$ and $\hat\delta(1,\bar\eta)$ denote the function $\hat\delta(2)$ and $\hat\delta(1)$, respectively, calculated with respect to $\bar\eta$. Thus, Lemma \ref{JacProL215} and Lemma \ref{JacProL217} can be applied, if $T+u \le n\Delta_n$ and $\hat\delta(2)_{0,n\Delta_n} < \infty$ almost surely or $\hat\delta(1)_{0,n\Delta_n} < \infty$ almost surely, respectively. This is always satisfied when we apply these lemmas.
\end{remark}

\subsection{Results on the crucial decomposition}

Recall the quantities defined in \eqref{SmandlajumpDef} which are used frequently in the proof of Theorem \ref{ConvThm} and Theorem \ref{CondConvThm}. With the constants from Assumption \ref{Cond1} let $\ell \in \R$ have the properties
\begin{equation}
\label{elleqn2}
1 < \ell < \frac{1}{2\beta \ovw} \wedge (1+ \epsilon) \quad \text{ and also } \quad \ell < \frac{2(p-1) \ovw -1}{2(\beta-1) \ovw} \text{ if } \beta >1,
\end{equation}
with an $\epsilon >0$ for which Assumption \ref{Cond1}\eqref{ObsSchCond6} holds. Then we have
\begin{equation}
\label{unseqDefEq3}
u_n = (v_n)^{\ell} \quad \text{ and } \quad F_n = \lbrace \taili \colon |\taili| > u_n \rbrace
\end{equation}
as well as
\begin{align}
\label{SmandlajumpD2}
\tibigj &= (\taili \ind_{F_n}(\taili)) \star \mu^{(n)}, \nonumber \\
\tibigjal &= (\taili \ind_{F_n \cap \lbrace |\taili| \leq \alpha/4 \rbrace }(\taili)) \star \mu^{(n)}, \quad \text{ for } \alpha >0\nonumber \\
\hatbigjal &= (\taili \ind_{\lbrace |\taili| > \alpha/4 \rbrace}) \star \mu^{(n)}, \quad \text{ for } \alpha >0 \nonumber \\
N_t^n &= (\ind_{F_n} \star \mu^{(n)})_t, \nonumber \\
\tilde X_t^{\prime n} &= X^{(n)}_t - \tibigj_t \nonumber \\
&= X^{(n)}_0 + \int_0^t b^{(n)}_s ds + \int_0^t \sigma^{(n)}_s dW^{(n)}_s+ \nonumber \\
&\hspace{2cm}+ (\taili \ind_{F_n^C}(\taili)) \star 
(\mu^{(n)} - \bar \mu^{(n)})_t - (\taili \ind_{\lbrace |\taili| \leq 1 \rbrace \cap F_n}(\taili)) \star \bar \mu^{(n)}_t, \nonumber \\
A_i^n	&= \lbrace | \Deli \tilde X^{\prime n} | \leq v_n/2 \rbrace \cap \lbrace \Deli N^n \leq 1 \rbrace.
\end{align}

The following lemma ensures that with high probability at most one large jump occurs and the remaining part is appropriately small.

\begin{lemma}
	\label{QnConv}
	Let Assumption \ref{Cond1} be satisfied, then $\lim_{n \to \infty}\Prob(Q_n) \to 1$, where 
	\begin{equation}
	\label{Qndef2}
	Q_n = \bigcap \limits_{i=1}^n A_i^n
	\end{equation}
\end{lemma}

\begin{proof}
	Choose some $m^\prime \in \R$ with 
$
	m^\prime >  \frac{2+ \beta \ell}{\ell -1} \vee \frac{1+ 2 \ovw}{1/2 - \ovw}.
$
	Then by Lemma \ref{JacProL215} and Assumption \ref{Cond1}\eqref{BlGetCond} we obtain for $1 \leq i \leq n$ and any sufficiently small $0 < \delta < 1$
	\begin{align*}
	\Eb & \big| \Deli \big ( \taili \ind_{F_n^C}(\taili) \big ) \star \big ( \mu^{(n)} - \bar \mu^{(n)} \big ) \big|^{m^\prime}  \\
	&\leq K \bigg ( \int_{(i-1)\Delta_n}^{i \Delta_n} \int_{\lbrace |\taili| \leq u_n \rbrace} |\taili|^{m^\prime} \nu^{(n)}_s(d\taili)ds + \bigg \{ \int_{(i-1)\Delta_n}^{i \Delta_n}
	\int \limits_{\lbrace |\taili| \leq u_n \rbrace} |\taili|^2 \nu_s^{(n)}(d\taili) ds \bigg \}^{{m^\prime}/2} \bigg ) \\
	&= K \bigg ( n \Delta_n \int_{(i-1)/n}^{i/n} \int_{\lbrace |\taili| \leq u_n \rbrace} |\taili|^{m^\prime} g^{(n)}(y,d\taili) dy +
	\bigg \{ n\Delta_n \int_{(i-1)/n}^{i/n}
	\int \limits_{\lbrace |\taili| \leq u_n \rbrace} |\taili|^2 g^{(n)}(y,d\taili) dy \bigg \}^{{m^\prime}/2} \bigg ) \\
	&\leq K(\delta) \big ( \Delta_n^{1+(m^\prime-\beta-\delta) \ell \ovw} + \Delta_n^{m^\prime/2} \big ).
	\end{align*}
	Note that $m^\prime > 2$ always so the lemma quoted above can be applied. Furthermore, $\bar \mu^{(n)} (ds,d\taili) = \nu^{(n)}_s(d\taili)ds$ yields for $1 \leq i \leq n$ and arbitrary $\delta >0$ small enough
	\begin{align*}
	\left| \Deli \big ( \taili \ind_{\lbrace |\taili| \leq 1 \rbrace \cap F_n}(\taili) \star \bar \mu^{(n)} \big ) \right| &= \bigg |
	\int \limits_{(i-1)\Delta_n}^{i \Delta_n} \int_{\lbrace u_n < |\taili| \leq 1 \rbrace} \taili \nu_s^{(n)}(d\taili) ds \bigg | \\
	&\leq u_n^{-(\beta+ \delta -1)_+} n \Delta_n \int_{(i-1)/n}^{i/n} \int_{\lbrace u_n < |\taili| \leq 1 \rbrace} |\taili|^{\beta+ \delta} g^{(n)}(y,d\taili) dy \\
	& \leq K(\delta) \Delta_n^{1 - \ell \ovw (\beta+ \delta -1)_+}.
	\end{align*}
	Let $m_b, m_\sigma \in \R$ be the constants in Assumption \ref{Cond1}\eqref{DrianDiffCond}. Because of $m_b >1$ and $m_\sigma >2$ we can apply H\"older's inequality and the Burkholder-Davis-Gundy inequalities (see page 39 in \cite{JacPro12}) to obtain due to Assumption \ref{Cond1}\eqref{DrianDiffCond} for $1 \leq i \leq n$:
	\begin{align*}
	\Eb \bigg | \int_{(i-1) \Delta_n}^{i \Delta_n} b_s^{(n)} ds \bigg |^{m_b} &\leq \Delta_n^{m_b} \Eb \bigg ( \frac{1}{\Delta_n} \int_{(i-1) \Delta_n}^{i \Delta_n} | b_s^{(n)} |^{m_b} ds \bigg ) \\
	&=\Delta_n^{m_b} \bigg ( \frac{1}{\Delta_n} \int_{(i-1) \Delta_n}^{i \Delta_n} \Eb | b^{(n)}_s |^{m_b} ds \bigg) \leq K \Delta_n^{m_b} 
	\end{align*}
	and
	\begin{align*}
	\Eb \bigg | \int_{(i-1) \Delta_n}^{i \Delta_n} \sigma^{(n)}_s dW^{(n)}_s \bigg |^{m_\sigma} &\leq K \Eb \bigg ( \int_{(i-1) \Delta_n}^{i \Delta_n} | \sigma^{(n)}_s |^2 ds \bigg )^{m_\sigma/2} \\
	&\leq K \Delta_n^{m_\sigma/2} \Eb \bigg ( \frac{1}{\Delta_n} \int_{(i-1) \Delta_n}^{i \Delta_n} | \sigma^{(n)}_s |^{m_\sigma} ds \bigg ) \\
	&= K \Delta_n^{m_\sigma/2} \bigg ( \frac{1}{\Delta_n} \int_{(i-1) \Delta_n}^{i \Delta_n} \Eb | \sigma^{(n)}_s |^{m_\sigma} ds \bigg )
	\leq K \Delta_n^{m_\sigma/2},
	\end{align*}
	where the equalities in the above displays hold according to Fubini's theorem. Additionally, according to Lemma \ref{PRMProsevjb} we have
	for any $1 \leq i \leq n$ and some $K(\delta) >0$
	\[
	\Prob(\Deli N^n \geq 2) \leq K(\delta) \Delta_n^{2-2(\beta + \delta) \ell \ovw},
	\]
	for $n \in \N$ large enough. Let us now choose $\delta > 0$ in such a way that $1- \ell \ovw(\beta+ \delta - 1)_+ > \ovw$. Then, for $n$ large enough we have $\Delta_n^{1- \ell \ovw(\beta+ \delta - 1)_+} \leq K v_n$, and the Markov inequality gives 
	\begin{multline}
	\label{ProbAbscheqn}
	\sum \limits_{i=1}^n \Prob \big((A_i^n)^C\big) \leq K(\delta) n \big \{ \Delta_n^{2 - 2(\beta+\delta) \ell \ovw} + \Delta_n^{1+(m^\prime-\beta- \delta) \ell \ovw - m^\prime \ovw} +\\
	+ \Delta_n^{m^\prime/2 - m^\prime \ovw} + \Delta_n^{m_\sigma/2 - m_\sigma \ovw} + \Delta_n^{m_b - m_b \ovw}\big \}.
	\end{multline}
	From the choice of the constants we further have
	\begin{align}
	\label{ErOkl}
	2-2(\beta+ \delta) \ell \ovw \geq 2-2\beta (1+ \epsilon) \ovw
	\end{align}
	and
	\begin{equation}
	\label{ZwOkl}
	\big (1+(m^\prime-\beta- \delta) \ell \ovw - m^\prime \ovw \big ) \wedge \big ( m^\prime/2 - m^\prime \ovw \big )  \wedge \big ( m_\sigma/2 - m_\sigma \ovw \big ) \wedge \big ( m_b - m_b \ovw \big ) \geq 1+ 2 \ovw,
	\end{equation}
	again for $\delta >0$ small enough. Thus, the right hand side of \eqref{ProbAbscheqn} converges to zero for this choice of $\delta$, using Assumption \ref{Cond1}\eqref{ObsSchCond4} and \eqref{ObsSchCond6}. 
\end{proof}

If moreover Assumption \ref{EasierCond} is valid, we can even give a rate for the convergence $\Prob(Q_n) \to 1$.

\begin{lemma}
	\label{QnConvOrdn}
	If Assumption \ref{EasierCond} is satisfied for some $0 < \beta < 2$, $0 < \tau < (1/5 \wedge \frac{2-\beta}{2+5\beta})$ and $p > \beta+((\frac 12 + \frac 32 \beta) \vee \frac 2{1+5\tau})$, we have
$
	\Prob\big(Q_n^C\big) \leq K n \Delta_n^{1+ \tau},
$
	for some $K>0$.
\end{lemma}

\begin{proof}
	If Assumption \ref{EasierCond} holds, then according to \eqref{OrdnVergl} in the proof of Proposition \ref{easier} Assumption \ref{Cond1} is valid for constants satisfying
$
	1+ \tau = 2(1-  \beta \ovw (1 + \epsilon)) < (1+2\ovw).
$
	Comparing this fact with \eqref{ProbAbscheqn}, \eqref{ErOkl} and \eqref{ZwOkl} yields the assertion.
\end{proof}

\noindent
In the next auxiliary lemma we consider for $\alpha >0$, $1 \leq i,j \leq n$ with $i \neq j$ and the constant $\ovr$ in Assumption \ref{Cond1} the sets 
\begin{align}
\label{RijnalDefEq}
R_{i,j}^{(n)} (\alpha) = \left\{ \big| \Deli \hatbigjal - \Delj \hatbigjal \big| \leq \Delta_n^{\ovr} \right\} \cap 
\left\{ \big| \Deli \hatbigjal \big| > \alpha/4 \right\} \cap Q_n,
\end{align}
with the pure jump It\=o semimartingale $\hatbigjal$ from \eqref{SmandlajumpD2}. Furthermore, for $\alpha >0$ let the sets $J^{\scriptscriptstyle (1)}_n (\alpha)$ be defined by their complements:
\begin{align}
\label{J1ijnalDefEq}
J^{(1)}_n(\alpha)^C = \bigcup \limits_{\stackrel{i,j =1}{i \neq j}}^n R_{i,j}^{(n)}(\alpha).
\end{align}

\begin{lemma}
	\label{PJn1alconv}
	Grant Assumption \ref{Cond1}. Then for each $\alpha >0$ the sets $J^{(1)}_n(\alpha)$ defined in \eqref{J1ijnalDefEq} satisfy
$
	\lim_{n \to \infty} \Prob\big(J^{(1)}_n(\alpha) \big) =1.
$
\end{lemma}

\begin{proof}
	Let $x$ be arbitrary and either $z = 0$ or $|z| > \alpha/4$. Then, for $n$ large enough we have
	\[
	\ind_{\lbrace |x - z| \leq \Delta_n^{\ovr} \rbrace} \ind_{\lbrace |x| > \alpha /4 \rbrace} \leq \ind_{\lbrace |x - z| \leq \Delta_n^{\ovr} \rbrace} \ind_{\lbrace |x| > \alpha /4 \rbrace} \ind_{\lbrace |z| > \alpha /4 \rbrace}.
	\]
	Furthermore, using the fact that for large $n \in \N$ on $Q_n$ there is at most one jump of $\hatbigjal$ on an interval $((k-1) \Delta_n, k \Delta_n]$ with $1 \leq k \leq n$, we thus obtain
	\begin{multline}
	\label{PMRVorAbsch}
	\Prob\big(R_{i,j}^{(n)}(\alpha)\big) \leq \int \int \int \int \int \ind_{\lbrace |x -z| \leq \Delta_n^{\ovr} \rbrace} 
	\ind_{((j-1) \Delta_n, j \Delta_n]}(t) \ind_{\lbrace |z| > \alpha/4 \rbrace}  
	\times \\
	\times \ind_{Q_n}(\omega) \mu^{(n)}(\omega; dt,dz) \ind_{\lbrace |x| > \alpha /4 \rbrace} 
	\ind_{((i-1) \Delta_n, i \Delta_n]} (s) \mu^{(n)}(\omega; ds,dx) \Prob(d \omega).
	\end{multline}
	Now, forget about the indicator involving $Q_n$ and  assume $j<i$. If $(\mathcal F_t)_{t \in \R_+}$ denotes the underlying filtration, the inner integral in \eqref{PMRVorAbsch} with respect to $\mu^{\scriptscriptstyle (n)}(\omega; dt,dz)$ is an $(\mathcal F_{j \Delta_n} \otimes \Bb)$-measurable function in $(\omega,x)$. Accordingly, the integrand in the integral with respect to $\mu^{\scriptscriptstyle (n)}(\omega; ds,dx)$ is in fact $\Pc'$-measurable. Therefore, Fubini's theorem and the definition of the predictable compensator of an optional ${\mathcal P}'$-$\sigma$-finite random measure
	(see Theorem II.1.8 in \cite{JacShi02}) yield for $n$ large enough:

	\begin{align}
	\label{PMengenRAbsch}
	\Prob\big(R_{i,j}^{(n)}(\alpha)\big) &\leq \int_{(i-1)\Delta_n}^{i\Delta_n} \int_{(j-1)\Delta_n}^{j\Delta_n} \int
	 \int \ind_{\lbrace |x-z| \leq \Delta_n^{\ovr} \rbrace} \ind_{\lbrace |x| > \alpha /4 \rbrace} \times  \\
	&\hspace{6cm}\times\ind_{\lbrace |z| > \alpha /4 \rbrace} \nu^{(n)}_{s_1}(dz) \nu^{(n)}_{s_2}(dx) ds_1 ds_2 \nonumber \\
	&\leq n^2 \Delta_n^2 \int_{(i-1)/n}^{i/n} \int_{(j-1)/n}^{j/n} \int \int \ind_{\lbrace |x-z| \leq \Delta_n^{\ovr} \rbrace} \ind_{\lbrace |x| > \alpha /4 \rbrace} \times \nonumber \\
	&\hspace{58mm} \times\ind_{\lbrace |z| > \alpha /4 \rbrace} g^{(n)}(y_1,dz) g^{(n)}(y_2,dx) dy_1 dy_2.  \nonumber
	\end{align}
	Thus, we have $\Prob \big(J^{(1)}_n(\alpha)\big) \rightarrow 1$, because \eqref{PMengenRAbsch}, Assumption \ref{Cond1}\eqref{SeLevyDistCond} and Assumption \ref{Cond1}\eqref{ObsSchCond3} show that there is a constant $K >0$ such that 
$
	\Prob\big(J^{(1)}_n(\alpha)^C\big) \leq K n^2 \Delta_n^{2+q} \rightarrow 0.
$
\end{proof}

Similar to \eqref{RijnalDefEq} for $\alpha >0$, $1 \leq i,j \leq n$ with $i \neq j$ and the constants $\ovv < \ovr$ in Assumption \ref{Cond1} let
\[
S_{i,j}^{(n)}(\alpha) = \left\{ \big| \Deli \tibigjeial - \Delj \tibigjeial \big| \leq \Delta_n^{\ovr} \right\} \cap
\left\{ \big| \Deli \tibigjeial \big| > \Delta_n^{\ovv} \right\} \cap Q_n.
\]
and define the sets $J_n^{(2)}(\alpha)$ by
\begin{equation}
\label{Jn2aldef}
J_n^{(2)} (\alpha)^C = \bigcup \limits_{\stackrel{i,j =1}{i \neq j}}^n S_{i,j}^{(n)} (\alpha).
\end{equation}

\begin{lemma}
	\label{BiJovrovvAb}
	Grant Assumption \ref{Cond1}. Then for each $\alpha \in (0,\alpha_0/2)$, with $\alpha_0$ the constant in Assumption \ref{Cond1}\eqref{FiLevyDistCond}, the sets $J_n^{\scriptscriptstyle (2)} (\alpha)$ defined in \eqref{Jn2aldef} satisfy
$
	\lim_{n \to \infty} \Prob\big(J^{(2)}_n(\alpha) \big) =1.
$
\end{lemma}

\begin{proof} 
	The same considerations as for \eqref{PMRVorAbsch} and \eqref{PMengenRAbsch} yield for $n$ large enough
	\begin{align*}
	\Prob(S_{i,j}^{(n)}(\alpha)) &\leq \int_{(i-1)\Delta_n}^{i\Delta_n} \int_{(j-1)\Delta_n}^{j\Delta_n} \int \int \ind_{\lbrace |x-z| \leq \Delta_n^{\ovr} \rbrace} \ind_{\lbrace \Delta_n^{\ovv} /2 < |x| \leq \alpha_0 \rbrace} \times \nonumber \\
	&\hspace{5cm}\times\ind_{\lbrace \Delta_n^{\ovv} /2 < |z| \leq \alpha_0 \rbrace} \nu^{(n)}_{s_1}(dz) \nu^{(n)}_{s_2}(dx) ds_1 ds_2 \nonumber \\
	&\leq n^2 \Delta_n^2 \int_{(i-1)/n}^{i/n} \int_{(j-1)/n}^{j/n} \int \int \ind_{\lbrace |x-z| \leq \Delta_n^{\ovr} \rbrace} \ind_{\lbrace \Delta_n^{\ovv} /2 < |x| \leq \alpha_0 \rbrace} \times \nonumber \\
	&\hspace{48mm} \times\ind_{\lbrace \Delta_n^{\ovv} /2 < |z| \leq \alpha_0 \rbrace} g^{(n)}(y_1,dz) g^{(n)}(y_2,dx) dy_1 dy_2 \\
	&\leq K \Delta_n^{2+q}, 
	\end{align*}
 because of Assumption \ref{Cond1}\eqref{FiLevyDistCond} and $\ovv < \ovr$. Thus, we obtain
$
	\Prob\big(J^{(2)}_n(\alpha)^C\big) \leq K n^2 \Delta_n^{2+q} \rightarrow 0,
$	
	by Assumption \ref{Cond1}\eqref{ObsSchCond3}
\end{proof}

The next lemma yields bounds for the cardinality of the following random sets. For $\alpha >0$, $n \in \N$, $\dfi \in \R$ and $\omega \in \Omega$ let
\begin{align*}
\tilde A_1(\omega;\alpha,n,\dfi) &= \big\{i \in \{1,\ldots,n\} \mid | \Deli \hatbigjal (\omega) | > \alpha/4 \text{ and } \\
&\hspace{26mm}\Indit(\Deli \tilde X^{\prime n}(\omega) + \Deli \hatbigjal(\omega)) \neq \Indit(\Deli \hatbigjal(\omega)) \big\} \\
\tilde A_2(\omega;\alpha,n,\dfi) &= \big\{i \in \{1,\ldots,n\} \mid | \Deli \tibigjeial(\omega)| > \Delta_n^{\ovv} \text{ and } \\
&\hspace{26mm}\Indit(\Deli \tilde X^{\prime n}(\omega) + \Deli \tibigjeial(\omega)) \neq \Indit(\Deli \tibigjeial(\omega)) \big\}
\end{align*}

\begin{lemma}
	\label{Indikunglab}
	Grant Assumption \ref{Cond1} and let $c_n := \lceil (v_n/ \Delta_n^{\ovr}) + 1\rceil$. Then for all $\alpha >0$, $n \in \N$ and $\dfi \in \R$ we have 
	\begin{equation}
	\label{tiA1absch}
	\#\tilde A_1(\omega;\alpha,n,\dfi) \leq c_n
	\end{equation}
	for every $\omega \in J_n^{(1)}(\alpha) \cap Q_n$ as well as 
	\begin{equation}
	\label{tiA2absch}
	\#\tilde A_2(\omega;\alpha,n,\dfi) \leq c_n
	\end{equation}
	for all $\omega \in J_n^{(2)}(\alpha) \cap Q_n$, where $\# M$ denotes the cardinality of a set $M$.
\end{lemma}

\begin{proof}
	For $\omega \in J_n^{(1)}(\alpha) \cap Q_n$ we have 
	\[
	\big| \Deli \hatbigjal - \Delj \hatbigjal \big| > \Delta_n^{\ovr}
	\]    
	for all $i,j \in \big\{ k \in \{1,\ldots,n\} \mid | \Delk \hatbigjal | > \alpha/4 \big\} =: M_0(\omega;\alpha,n)$ with $i \neq j$ by the definition of the set $J_n^{(1)}(\alpha)$ in \eqref{J1ijnalDefEq}. Consequently, for fixed $\dfi \in \R$ 
	\[
	\Deli \hatbigjal \in [t-v_n/2,t+v_n/2]
	\]
	can only hold for at most $c_n$ indices $i \in M_0(\omega;\alpha,n)$. Thus, we conclude \eqref{tiA1absch}, because according to \eqref{Qndef2} we have $| \Deli \tilde X^{\prime n} | \leq v_n/2$ for all $i \in \{1, \ldots,n\}$ and $\omega \in Q_n$. 
	
	The assertion \eqref{tiA2absch} follows with exactly the same reasoning.
\end{proof}

In the following we gather auxiliary lemmas which have a similar proof and give bounds for crucial quantities in the proof of Theorem \ref{ConvThm} and Theorem \ref{CondConvThm}. The first one is concerned with 
\begin{equation}
\label{DiffKonv2}
V_\alpha^{(n)} := \frac{1}{\sqrt{n \Delta_n}} \sup \limits_{(\gseqi,\dfi) \in \netir} \Big| \sum \limits_{i=1}^{\ip{n\gseqi}} \left\{ \chi_\dfi^{(\alpha)} \big(\Deli X^{(n)}\big) \Truniv - \chi_\dfi^{(\alpha)}\big(\Deli L^{(n)}\big) \right\} \Big|,
\end{equation}
from \eqref{DiffKonv}, where $\alpha >0$, $\chi_\dfi^{\scriptscriptstyle (\alpha)}$ is defined in \eqref{ImpQuantDef} and $L^{\scriptscriptstyle (n)} = \big(\taili \Trunx\big) \star \mu^{\scriptscriptstyle (n)}$ is the pure jump It\=o semimartingale defined in \eqref{LnDefEq}.

\begin{lemma}
	\label{ValnBound}
	Let Assumption \ref{Cond1} be satisfied. Then for $\alpha >0$, $\omega \in Q_n$ and $n \in \N$ large enough such that $v_n \leq \alpha/4$ we have
	\[
	V_\alpha^{(n)} \leq C_n(\alpha) + D_n(\alpha),
	\]
	where
	\begin{multline*}
	C_n(\alpha) = \frac{K}{\sqrt{n \Delta_n}} \sup \limits_{t \in \R} \sum \limits_{i=1}^n \big| \Indit(\Deli \tilde X^{\prime n} + 
	\Deli \hatbigjal) - \Indit(\Deli \hatbigjal) \big|  \times \\
	\times \ind_{\lbrace | \Deli \hatbigjal | > \alpha/4 \rbrace} 
	\ind_{\lbrace | \Deli \tilde X^{\prime n} | \leq v_n/2 \rbrace},
	\end{multline*}
	for $K>0$ a bound for $\rho$ and
	\begin{multline*}
	D_n(\alpha) = \frac{1}{\sqrt{n \Delta_n}} \sum \limits_{i=1}^n \big| \rho_{\alpha} (\Deli \tilde X^{\prime n} + \Deli \hatbigjal)
	\ind_{\lbrace |  \Deli \tilde X^{\prime n} + \Deli \hatbigjal| > v_n \rbrace} - \\
	- \rho_{\alpha}(\Deli \hatbigjal) \ind_{\lbrace | \Deli \hatbigjal| > v_n \rbrace} \big| 
	\ind_{\lbrace | \Deli \tilde X^{\prime n} | \leq v_n/2 \rbrace},
	\end{multline*}
	with $\rho_\alpha$ defined prior to \eqref{ImpQuantDef} and where the particular processes are defined in \eqref{SmandlajumpD2}.
\end{lemma}

\begin{proof}
	On $Q_n$, and with $n$ large enough such that $v_n \leq \alpha/4$, one of the following mutually exclusive possibilities holds for $1 \leq i \leq n$:
	\begin{itemize}
		\item[(i)] $\Deli N^n =0$. \\
		Then we have $| \Deli X^{\scriptscriptstyle (n)} | = | \Deli \tilde X^{\prime n} | \leq v_n /2$ and there is no jump larger than 
		$u_n$ (and $v_n$) on the interval $((i-1) \Delta_n, i \Delta_n]$. Thus, $\chi_t^{\scriptscriptstyle (\alpha)} (\Deli X^{(n)}) \Truniv =0 = \chi_t^{\scriptscriptstyle (\alpha)}(\Deli L^{(n)})$ holds for all $t \in \R$ and the corresponding summand in 
		\eqref{DiffKonv2} vanishes.
		\item[(ii)] $\Deli N^n =1$ and $\Deli \tibigj = \Deli \tibigjal \neq 0$. \\
		So the only jump in $((i-1) \Delta_n, i \Delta_n]$ (of absolute size) larger than $u_n$ is in fact not larger than $\alpha /4$, and because of $v_n \leq \alpha /4$
		we have $|\Deli X^{\scriptscriptstyle (n)} | \leq \alpha /2$. Thus, as in the first case, $\chi_t^{\scriptscriptstyle (\alpha)} (\Deli X^{\scriptscriptstyle (n)}) \Truniv =0 = \chi_t^{\scriptscriptstyle (\alpha)}(\Deli L^{\scriptscriptstyle (n)})$ is true for all $t \in \R$ and the corresponding
		summand in \eqref{DiffKonv2} is equal to zero.
		\item[(iii)]$\Deli N^n =1$ and $\Deli \tibigj \neq 0$, but $\Deli \tibigjal =0$. \\
		So the only jump in $((i-1) \Delta_n, i \Delta_n]$ larger than $u_n$ is also larger than $\alpha /4$. If $\hatbigjal$ is the quantity defined in \eqref{SmandlajumpD2}, we get
		\begin{align*}
		\Deli X^{(n)} &= \Deli \tilde X^{\prime n} + \Deli \hatbigjal \quad \text{ and } \\ 
		\chi_t^{(\alpha)}(\Deli L^{(n)}) &= \rho_{\alpha}(\Deli \hatbigjal) \ind_{\lbrace | \Deli \hatbigjal | > v_n \rbrace} 
		\Indit(\Deli \hatbigjal).
		\end{align*}
	\end{itemize}
	
	Thus, we obtain an upper bound for $V_\alpha^{\scriptscriptstyle (n)}$ on $Q_n$, as soon as $v_n \leq \alpha/4$:
	\begin{align*}
	V_\alpha^{(n)}&\leq\frac{1}{\sqrt{n \Delta_n}} \sup \limits_{(\gseqi,\dfi) \in \netir} \Big| \sum \limits_{i=1}^{\ip{n\gseqi}} \Big\{ \rho_{\alpha}(\Deli X^{\scriptscriptstyle (n)}) 
	\ind_{\lbrace | \Deli X^{\scriptscriptstyle (n)}| > v_n \rbrace} \Indit(\Deli X^{\scriptscriptstyle (n)}) - \rho_{\alpha}(\Deli \hatbigjal) \times \nonumber \\
	&\hspace{18mm}\times \ind_{\lbrace | \Deli \hatbigjal | > v_n \rbrace} \Indit(\Deli \hatbigjal) \Big\} \ind_{\lbrace | \Deli \hatbigjal | > 
		\alpha/4 \rbrace} \ind_{\lbrace | \Deli \tilde X^{\prime n} | \leq v_n/2 \rbrace} \Big| \nonumber \\
	&\leq C_n(\alpha) + D_n(\alpha),
	\end{align*}
	where we can substitute $\Deli X^{\scriptscriptstyle (n)} = \Deli \tilde X^{\prime n} + \Deli \hatbigjal$ in the first line.
\end{proof}

With a similar reasoning as above we deduce a bound for
\begin{equation}
\label{DifferDefEq2}
V^{\circ(n)}_{\alpha} = \frac{1}{\sqrt{n \Delta_n}} \sup_{(\gseqi,\dfi) \in \netir} \Big| \sum \limits_{i=1}^{\ip{n\gseqi}} \Big\{ \chi_\dfi^{\circ (\alpha)} (\Deli X^{(n)}) \Truniv - 
\chi_\dfi^{\circ (\alpha)}(\Deli L^{(n)}) \Big\} \Big|,
\end{equation}
from \eqref{DifferDefEq} in the next lemma, where $\alpha >0$ and $\chi_\dfi^{\circ \scriptscriptstyle (\alpha)}$ is defined in \eqref{ImpQuantDef}.

\begin{lemma}
	\label{ValprnBou}
	Let Assumption \ref{Cond1} be satisfied. Then for $\alpha >0$, $\omega \in Q_n$ and $n \in \N$ large enough such that $v_n \leq \alpha$ we have
	\[
	V^{\circ(n)}_{\alpha} \leq C^\circ_n(\alpha) + D^\circ_n(\alpha) + E^\circ_n(\alpha),
	\]
	where
	\begin{multline*}
	C^\circ_n(\alpha) = \frac{K}{\sqrt{n \Delta_n}} \sup \limits_{t \in \R} \sum \limits_{i=1}^n \big| \Indit(\varsigma_i^n(\alpha)) - \Indit(\Deli \tibigjeial) \big|  \times \\
	\times \ind_{\lbrace | \Deli \tibigjeial | > \Delta_n^{\ovv} \rbrace} 
	\ind_{\lbrace | \Deli \tilde X^{\prime n} | \leq v_n/2 \rbrace},
	\end{multline*}
	with $K >0$ a bound for $\rho$ and
	\begin{multline*}
	D^\circ_n(\alpha) = \frac{1}{\sqrt{n \Delta_n}} \sum \limits_{i=1}^n \big| \rho^{\circ}_{\alpha} (\varsigma_i^n(\alpha)) \ind_{\lbrace | \varsigma_i^n(\alpha)| > v_n \rbrace} - \rho^{\circ}_{\alpha}(\Deli \tibigjeial) \ind_{\lbrace | \Deli \tibigjeial| > v_n \rbrace} \big| \times \\ \times
	\ind_{\lbrace | \Deli \tibigjeial | > \Delta_n^{\ovv} \rbrace} 
	\ind_{\lbrace | \Deli \tilde X^{\prime n} | \leq v_n/2 \rbrace},
	\end{multline*}
	\begin{align*}
	&E^\circ_n(\alpha) = \frac{1}{\sqrt{n \Delta_n}} \sup \limits_{\dfi \in \R}  \sum \limits_{i=1}^{n}  \big|
	\rho^{\circ}_{\alpha}(\varsigma_i^n(\alpha)) \ind_{\lbrace | \varsigma_i^n(\alpha)| > v_n \rbrace} \Indit(\varsigma_i^n(\alpha)) - \rho^{\circ}_{\alpha}(\Deli \tibigjeial) 
	\times \\ &\hspace{20mm} \times \ind_{\lbrace | \Deli \tibigjeial | > v_n \rbrace} 
	\Indit(\Deli \tibigjeial) \big| 
	\ind_{\lbrace | \Deli \tilde X^{\prime n} | \leq v_n/2 \rbrace}  
	\ind_{\lbrace | \Deli \tibigjeial | \leq \Delta_n^{\ovv} \rbrace} \ind_{Q_n},
	\end{align*}
	where $\varsigma_i^n(\alpha) = \Deli \tilde X^{\prime n} + \Deli \tibigjeial$, $\rho^{\circ}_{\alpha}$ is defined prior to \eqref{ImpQuantDef}, $\ovv >0$ is the constant from Assumption \ref{Cond1}\eqref{FiLevyDistCond} and the involved processes are defined in \eqref{SmandlajumpD2}.
\end{lemma}

\begin{proof}
	On the set $Q_n$, and if $v_n \leq \alpha$, we have three mutually exclusive possibilities for $1 \leq i \leq n$:
	\begin{enumerate}
		\item[(i)] $\Deli N^n =0$. \\
		Then we have $| \Deli X^{\scriptscriptstyle (n)} | = | \Deli \tilde X^{\prime n} | \leq v_n /2$ and there is no jump larger than 
		$u_n$ (and $v_n$) on the interval $((i-1) \Delta_n, i \Delta_n]$. Thus, $\chi_\dfi^{\circ \scriptscriptstyle (\alpha)} (\Deli X^{\scriptscriptstyle (n)}) \Truniv = 0 = \chi_\dfi^{\circ \scriptscriptstyle  (\alpha)}(\Deli 
		L^{\scriptscriptstyle (n)})$ holds for all $\dfi \in \R$ and the $i$-th summand in \eqref{DifferDefEq2} vanishes.
		\item[(ii)] $\Deli N^n =1$ and $\Deli \tibigj \neq 0$, but $\Deli \tibigjeial =0$. \\
		So the only jump in $((i-1) \Delta_n, i \Delta_n]$ larger than $u_n$ is also larger than $2 \alpha$. Because 
		$| \Deli \tilde X^{\prime n} | \leq v_n /2 \leq \alpha/2$ holds, we have 
		$| \Deli X^{\scriptscriptstyle (n)} | \geq \alpha$, and consequently $\chi_t^{\circ \scriptscriptstyle (\alpha)} (\Deli X^{\scriptscriptstyle (n)}) $ $\Truniv = 0 = 
		\chi_t^{\circ \scriptscriptstyle (\alpha)}(\Deli L^{\scriptscriptstyle (n)})$ using the definition of $\chi_t^{\circ \scriptscriptstyle (\alpha)}$.
		\item[(iii)] $\Deli N^n =1$ and $\Deli \tibigj = \Deli \tibigjeial \neq 0$. \\
		Here we can write 
		\[
		\Deli X^{(n)} = \Deli \tilde X^{\prime n} + \Deli \tibigjeial =: \varsigma_i^n(\alpha)
		\]
		and
		\[
		\chi_t^{\circ (\alpha)}(\Deli L^{(n)}) = \rho_{\alpha}^{\circ}(\Deli \tibigjeial) 
		\ind_{\lbrace | \Deli \tibigjeial | > v_n \rbrace} \Indit(\Deli \tibigjeial).
		\]
	\end{enumerate}
	
	Therefore, on $Q_n$ and as soon as $v_n \leq \alpha$, we have 
	\begin{align*}
	V^{\circ (n)}_{\alpha} &\leq \frac{1}{\sqrt{n \Delta_n}} \sup \limits_{(\gseqi,\dfi) \in \netir} \Big| 
	\sum \limits_{i=1}^{\ip{n\gseqi}} \Big\{ \rho^{\circ}_{\alpha}( \varsigma_i^n(\alpha)) \ind_{\lbrace | \varsigma_i^n(\alpha)| > v_n \rbrace} \Indit( \varsigma_i^n(\alpha)) -  \\
	&\hspace{15mm}-\rho^{\circ}_{\alpha}(\Deli \tibigjeial)	\ind_{\lbrace | \Deli \tibigjeial | > v_n \rbrace} \Indit(\Deli \tibigjeial) \Big\}
	\ind_{\lbrace | \Deli \tilde X^{\prime n} | \leq v_n/2 \rbrace} \Big|  \\
	&\leq C^\circ_n(\alpha) + D^\circ_n (\alpha) + E^\circ_n(\alpha).
	\end{align*}
\end{proof}

In the following lemma we obtain a bound for 
\begin{align}
\label{haValnDef2}
\hat V_\alpha^{(n)} = \frac 1{\sqrt{n\Delta_n}} \sup \limits_{(\gseqi,\dfi) \in \netir} \Big| \sum \limits_{i=1}^{\ip{n\gseqi}} \xi_i \big( \chi_\dfi^{(\alpha)}(\Deli X^{(n)}) \Truniv - \chi_\dfi^{(\alpha)}(\Deli L^{(n)}) \big) \Big|,
\end{align}
from \eqref{hatValnDef}, where $\alpha >0$, $\chi_\dfi^{\scriptscriptstyle (\alpha)}$ is defined in \eqref{ImpQuantDef}, $L^{\scriptscriptstyle (n)} = \big(\taili \Trunx\big) \star \mu^{\scriptscriptstyle (n)}$ is the pure jump It\=o semimartingale defined in \eqref{LnDefEq} and $(\xi_i)_{i\in\N}$ is a sequence of multipliers with mean zero and variance one defined on a distinct probability space than the remaining processes. Furthermore, for the claim of the lemma below recall the definition of the sets $Q_n$ and $J_n^{\scriptscriptstyle (1)}(\alpha)$ in \eqref{Qndef2} and \eqref{J1ijnalDefEq}, respectively, as well as the definition of $\rho_\alpha$ prior to \eqref{ImpQuantDef} and the quantities defined in \eqref{SmandlajumpD2}.

\begin{lemma}
	\label{haValnbou}
	Let Assumption \ref{Cond1} be satisfied. Then for $\alpha >0$, $\omega \in J_n^{(1)}(\alpha) \cap Q_n$ and $n \in \N$ large enough such that $v_n \leq \alpha/4$ we have
	\[
	\hat V_\alpha^{(n)} \leq \hat D_n(\alpha) + \hat E_n(\alpha) + \hat F_n(\alpha),
	\]
	with
	\begin{multline*}
	\hat D_n(\alpha) = \frac{1}{\sqrt{n \Delta_n}} \sum \limits_{i=1}^n |\xi_i| \big| \rho_{\alpha} (\Deli \tilde X^{\prime n} + \Deli \hatbigjal)
	\ind_{\lbrace |  \Deli \tilde X^{\prime n} + \Deli \hatbigjal| > v_n \rbrace} - \\
	- \rho_{\alpha}(\Deli \hatbigjal) \ind_{\lbrace | \Deli \hatbigjal| > v_n \rbrace} \big| 
	\ind_{\lbrace | \Deli \tilde X^{\prime n} | \leq v_n/2 \rbrace},
	\end{multline*}
	\begin{align*}
	\hat E_n(\alpha) = \sup\limits_{A \in \mathfrak S_n} \Big| \sum \limits_{i \in A} \xi_i a_i^n (\alpha) \Big|,
~~
	\hat F_n(\alpha) = \sup\limits_{A \in \mathfrak S_n} \Big| \sum \limits_{i \in A} \xi_i b_i^n (\alpha) \Big|,
	\end{align*}
where 
	\[
	a_i^n (\alpha) = \frac 1{\sqrt{n \Delta_n}} \rho_{\alpha}(\Deli \hatbigjal) \ind_{\lbrace | \Deli \hatbigjal | > 
		\alpha/4 \rbrace}, 
	\]
	\[
	b_i^n (\alpha) = \frac 1{\sqrt{n \Delta_n}} \rho_{\alpha} (\Deli \tilde X^{\prime n} + \Deli \hatbigjal) \ind_{\lbrace |  \Deli \tilde X^{\prime n} + \Deli \hatbigjal| > v_n \rbrace} \ind_{\lbrace | \Deli \hatbigjal | > 
		\alpha/4 \rbrace},
	\]
	and $\mathfrak S_n = \{ M \subset \{1, \ldots,n\} \mid \# M \leq c_n\}$ with $c_n = \lceil (v_n/ \Delta_n^{\ovr}) + 1\rceil$.
\end{lemma}

\begin{proof}
	Recall the cases which have been distinguished in the proof of Lemma \ref{ValnBound}: \\
	On $Q_n$, and with $n$ large enough such that $v_n \leq \alpha/4$, one of the following mutually exclusive possibilities holds for $1 \leq i \leq n$:
	\begin{itemize}
		\item[(i)] $\Deli N^n =0$. \\
		Then we have $| \Deli X^{\scriptscriptstyle (n)} | = | \Deli \tilde X^{\prime n} | \leq v_n /2$ and there is no jump larger than 
		$u_n$ (and $v_n$) on the interval $((i-1) \Delta_n, i \Delta_n]$. Thus, $\chi_t^{\scriptscriptstyle (\alpha)} (\Deli X^{(n)}) \Truniv =0 = \chi_t^{\scriptscriptstyle (\alpha)}(\Deli L^{(n)})$ holds for all $t \in \R$ and the corresponding summand in 
		\eqref{haValnDef2} vanishes.
		\item[(ii)] $\Deli N^n =1$ and $\Deli \tibigj = \Deli \tibigjal \neq 0$. \\
		So the only jump in $((i-1) \Delta_n, i \Delta_n]$ (of absolute size) larger than $u_n$ is in fact not larger than $\alpha /4$, and because of $v_n \leq \alpha /4$
		we have $|\Deli X^{\scriptscriptstyle (n)} | \leq \alpha /2$. Thus, as in the first case, $\chi_t^{\scriptscriptstyle (\alpha)} (\Deli X^{\scriptscriptstyle (n)}) \Truniv =0 = \chi_t^{\scriptscriptstyle (\alpha)}(\Deli L^{\scriptscriptstyle (n)})$ is true for all $t \in \R$ and the corresponding
		summand in \eqref{haValnDef2} is equal to zero.
		\item[(iii)]$\Deli N^n =1$ and $\Deli \tibigj \neq 0$, but $\Deli \tibigjal =0$. \\
		So the only jump in $((i-1) \Delta_n, i \Delta_n]$ larger than $u_n$ is also larger than $\alpha /4$. If $\hatbigjal$ is the quantity defined in \eqref{SmandlajumpD2}, we get
		\begin{align*}
		\Deli X^{(n)} &= \Deli \tilde X^{\prime n} + \Deli \hatbigjal \quad \text{ and } \\ 
		\chi_t^{(\alpha)}(\Deli L^{(n)}) &= \rho_{\alpha}(\Deli \hatbigjal) \ind_{\lbrace | \Deli \hatbigjal | > v_n \rbrace} 
		\Indit(\Deli \hatbigjal).
		\end{align*}
	\end{itemize}
	
	Thus, we obtain an upper bound for $\hat V_\alpha^{\scriptscriptstyle (n)}$ on $Q_n$, as soon as $v_n \leq \alpha/4$:
	\begin{multline}
	\label{EAbschonBn}
	\hat V_\alpha^{(n)} \leq \frac{1}{\sqrt{n \Delta_n}} \sup \limits_{(\gseqi,\dfi) \in \netir}\Big| \sum \limits_{i \in A_0(\omega; \alpha,n, (\gseqi, \dfi))} \xi_i \big( \rho_{\alpha}(\Deli X^{(n)}) 
	\ind_{\lbrace | \Deli X^{(n)}| > v_n \rbrace} \Indit(\Deli X^{(n)}) - \\ - \rho_{\alpha}(\Deli \hatbigjal) \ind_{\lbrace | \Deli \hatbigjal | > v_n \rbrace} \Indit(\Deli \hatbigjal) \big) \ind_{\lbrace | \Deli \hatbigjal | > 
		\alpha/4 \rbrace}  \Big|, 
	\end{multline}
	where we can substitute $\Deli X^{(n)} = \Deli \tilde X^{\prime n} + \Deli \hatbigjal$ in the first line and with the random set
	\begin{equation*}
	A_0(\omega; \alpha,n, (\gseqi, \dfi)) = \big\{i \in \{1, \ldots,n\} \mid i \leq \ip{n\gseqi} \text{ and } \Deli N^n =1,  \Deli \tibigj \neq 0, \Deli \tibigjal =0 \big\},
	\end{equation*} 
	for $(\gseqi,\dfi) \in \netir$, $\alpha >0$, $n \in \N$ and $\omega \in \Omega$. Defining the further random sets
	\begin{align*}
	A_1(\omega;\alpha,n,(\gseqi,\dfi)) &= \big\{i \in A_0(\omega;\alpha, n,(\gseqi, \dfi)) \mid  | \Deli \hatbigjal | > \alpha/4 , \\
	&\hspace{15mm} \Indit(\Deli \tilde X^{\prime n} + \Deli \hatbigjal) = \Indit(\Deli \hatbigjal) = 1 \big\}, \\
	A_2(\omega;\alpha,n,(\gseqi,\dfi)) &= \big\{i \in A_0(\omega; \alpha,n,(\gseqi, \dfi)) \mid  | \Deli \hatbigjal | > \alpha/4, \\
	&\hspace{10mm}\Indit(\Deli \tilde X^{\prime n} + \Deli \hatbigjal) = 0 \text{ and } \Indit(\Deli \hatbigjal) = 1 \big\}, \\
	A_3(\omega;\alpha,n,(\gseqi,\dfi)) &= \big\{i \in A_0(\omega;\alpha,n, (\gseqi, \dfi)) \mid  | \Deli \hatbigjal | > \alpha/4, \\
	&\hspace{10mm} \Indit(\Deli \tilde X^{\prime n} + \Deli \hatbigjal) = 1 \text{ and } \Indit(\Deli \hatbigjal) = 0 \big\},
	\end{align*}
	for $(\gseqi,\dfi) \in \netir$, $\alpha >0$, $n \in \N$ and $\omega \in \Omega$, together with \eqref{EAbschonBn} gives for $n$ large enough and $\omega \in Q_n$
	\begin{align}
	\label{RandSetsAbEq}
	\hat V_\alpha^{(n)} 
	&\leq \frac{1}{\sqrt{n \Delta_n}} \sup \limits_{(\gseqi,\dfi) \in \netir} \sum \limits_{i \in A_1(\omega;\alpha,n,(\gseqi,\dfi))} |\xi_i| \big| \rho_{\alpha} (\Deli \tilde X^{\prime n} + \Deli \hatbigjal)
	\ind_{\lbrace |  \Deli \tilde X^{\prime n} + \Deli \hatbigjal| > v_n \rbrace} - \nonumber \\
	&\hspace{67mm}- \rho_{\alpha}(\Deli \hatbigjal) \ind_{\lbrace | \Deli \hatbigjal| > v_n \rbrace} \big| 
	\ind_{\lbrace | \Deli \tilde X^{\prime n} | \leq v_n/2 \rbrace} \nonumber \\
	&\hspace{5mm}+ \sup \limits_{(\gseqi,\dfi) \in \netir} \Big| \sum \limits_{i \in A_2(\omega;\alpha,n,(\gseqi,\dfi))} \xi_i a_i^n (\alpha) \Big| + \sup \limits_{(\gseqi,\dfi) \in \netir} \Big| \sum \limits_{i \in A_3(\omega;\alpha,n,(\gseqi,\dfi))} \xi_i b_i^n (\alpha) \Big|,
	\end{align}
	because $|\Deli \tilde X^{\prime n} | \leq v_n/2$ holds for each $i=1, \ldots,n$ on $Q_n$ by \eqref{Qndef2}. Finally, according to Lemma \ref{Indikunglab} we have $\#A_2(\omega;\alpha,n,(\gseqi,\dfi)) \leq c_n$ and $\# A_3(\omega;\alpha,n,(\gseqi,\dfi)) \leq c_n$ on $J_n^{\scriptscriptstyle (1)}(\alpha) \cap Q_n$ for all $(\gseqi,\dfi) \in \netir$. By definition $A_1(\omega;\alpha,n,(\gseqi,\dfi)) \subset \{ 1, \ldots,n\}$ and thus \eqref{RandSetsAbEq} yields the assertion.
\end{proof}

In the next lemma we use a similar reasoning as above to obtain a bound for 
\begin{align}
\label{ValprnDef2}
\hat V_\alpha^{\circ (n)} = \frac 1{\sqrt{n\Delta_n}} \sup_{(\gseqi,\dfi) \in \netir} \Big| \sum\limits_{i=1}^{\ip{n\gseqi}} \xi_i \big( \chi_\dfi^{\circ (\alpha)}(\Deli X^{(n)}) \Truniv - \chi_\dfi^{\circ (\alpha)}(\Deli L^{(n)}) \big) \Big|.
\end{align}
from \eqref{genhavnalprzz}, where $\alpha >0$, $\chi_\dfi^{\circ \scriptscriptstyle (\alpha)}$ is defined in \eqref{ImpQuantDef}, $L^{\scriptscriptstyle (n)} = \big(\taili \Trunx\big) \star \mu^{\scriptscriptstyle (n)}$ is the pure jump It\=o semimartingale defined in \eqref{LnDefEq} and $(\xi_i)_{i\in\N}$ is a sequence of multipliers with mean zero and variance one defined on a distinct probability space than the remaining processes. Furthermore for the claim of the lemma below recall the definition of the sets $Q_n$ and $J_n^{\scriptscriptstyle (2)}(\alpha)$ in \eqref{Qndef2} and \eqref{Jn2aldef}, respectively, as well as the definition of $\rho^{\circ}_\alpha$ prior to \eqref{ImpQuantDef} and the quantities defined in \eqref{SmandlajumpD2}.

\begin{lemma}
	\label{haVpralbou}
	Let Assumption \ref{Cond1} be satisfied. Then for $\alpha >0$, $\omega \in J_n^{(2)}(\alpha) \cap Q_n$ and $n \in \N$ large enough such that $v_n \leq \alpha$ we have
	\[
	\hat V_\alpha^{\circ (n)} \leq \hat C_n^\circ (\alpha) + \hat D_n^\circ (\alpha) + \hat E_n^\circ (\alpha) + \hat F_n^\circ (\alpha)
	\]
	with
	\begin{align*}
	&\hat C_n^\circ (\alpha) = \frac 1{\sqrt{n\Delta_n}} \sup\limits_{\dfi \in \R} \sum \limits_{i=1}^n |\xi_i| \big| \rho_\alpha^\circ(\varsigma_i^n(\alpha)) \ind_{\{|\varsigma_i^n(\alpha)| > v_n\}} \Indit(\varsigma_i^n(\alpha)) - \\
	&\hspace{2mm} - \rho_\alpha^\circ(\Deli \tibigjeial) 
	\ind_{\{|\Deli \tibigjeial| > v_n\}} \Indit(\Deli \tibigjeial) \big| \ind_{\{|\Deli \tibigjeial| \leq \Delta_n^{\ovv}\}} \ind_{\{|\Deli \tilde X^{\prime n}| \leq v_n/2\}},
	\end{align*}
	\begin{align*}
	\hat D_n^\circ (\alpha) &= \frac 1{\sqrt{n\Delta_n}} \sum \limits_{i=1}^n |\xi_i| \big| \rho^{\circ}_{\alpha}(\varsigma_i^n(\alpha)) \ind_{\lbrace | \varsigma_i^n(\alpha)| > v_n \rbrace}  -\rho^{\circ}_{\alpha}(\Deli \tibigjeial) \ind_{\lbrace | \Deli \tibigjeial | > v_n \rbrace} \big| \times \\
	&\hspace{75mm} \times \ind_{\{|\Deli \tibigjeial| > \Delta_n^{\ovv}\}} \ind_{\{|\Deli \tilde X^{\prime n}| \leq v_n/2\}}, 
	\end{align*}
	\begin{align*}
	\hat E_n^\circ (\alpha) = \sup\limits_{A \in \mathfrak S_n} \Big| \sum\limits_{i \in A} \xi_i \bar a_i^n(\alpha) \Big|,
~~
	\hat F_n^\circ (\alpha) = \sup\limits_{A \in \mathfrak S_n} \Big| \sum\limits_{i \in A} \xi_i \bar b_i^n(\alpha) \Big|,
	\end{align*}
	where $\ovv >0$ is the constant from Assumption \ref{Cond1}\eqref{FiLevyDistCond}, $\varsigma_i^n(\alpha) = \Deli \tilde X^{\prime n} + \Deli \tibigjeial$, $\mathfrak S_n = \{ M \subset \{1, \ldots,n\} \mid \# M \leq c_n\}$ for $c_n = \lceil (v_n/ \Delta_n^{\ovr}) + 1\rceil$ and with
	\begin{align*}
	\bar a_i^n(\alpha) &= \frac 1{\sqrt{n\Delta_n}} \rho_\alpha^\circ(\Deli \tibigjeial) \ind_{\{| \Deli \tibigjeial| > v_n \vee \Delta_n^{\ovv}\}}, \\
	\bar b_i^n(\alpha) &= \frac 1{\sqrt{n\Delta_n}} \rho_\alpha^\circ(\varsigma_i^n(\alpha)) 
	\ind_{\{|\varsigma_i^n(\alpha)| > v_n\}} \ind_{\{ | \Deli \tibigjeial| > \Delta_n^{\ovv}\}}.
	\end{align*}
\end{lemma}

\begin{proof}
	Recall the cases which we have distinguished in the proof of Lemma \ref{ValprnBou}: \\
	On the set $Q_n$, and if $v_n \leq \alpha$, we have three mutually exclusive possibilities for $1 \leq i \leq n$:
	\begin{enumerate}
		\item[(i)] $\Deli N^n =0$. \\
		Then we have $| \Deli X^{\scriptscriptstyle (n)} | = | \Deli \tilde X^{\prime n} | \leq v_n /2$ and there is no jump larger than 
		$u_n$ (and $v_n$) on the interval $((i-1) \Delta_n, i \Delta_n]$. Thus, $\chi_\dfi^{\circ \scriptscriptstyle (\alpha)} (\Deli X^{\scriptscriptstyle (n)}) \Truniv = 0 = \chi_\dfi^{\circ \scriptscriptstyle  (\alpha)}(\Deli 
		L^{\scriptscriptstyle (n)})$ holds for all $\dfi \in \R$ and the $i$-th summand in \eqref{ValprnDef2} vanishes.
		\item[(ii)] $\Deli N^n =1$ and $\Deli \tibigj \neq 0$, but $\Deli \tibigjeial =0$. \\
		So the only jump in $((i-1) \Delta_n, i \Delta_n]$ larger than $u_n$ is also larger than $2 \alpha$. Because 
		$| \Deli \tilde X^{\prime n} | \leq v_n /2 \leq \alpha/2$ holds, we have 
		$| \Deli X^{\scriptscriptstyle (n)} | \geq \alpha$, and consequently $\chi_t^{\circ \scriptscriptstyle (\alpha)} (\Deli X^{\scriptscriptstyle (n)}) $ $\Truniv = 0 = 
		\chi_t^{\circ \scriptscriptstyle (\alpha)}(\Deli L^{\scriptscriptstyle (n)})$ using the definition of $\chi_t^{\circ \scriptscriptstyle (\alpha)}$.
		\item[(iii)] $\Deli N^n =1$ and $\Deli \tibigj = \Deli \tibigjeial \neq 0$. \\
		Here we can write 
		\[
		\Deli X^{(n)} = \Deli \tilde X^{\prime n} + \Deli \tibigjeial =: \varsigma_i^n(\alpha)
		\]
		and
		\[
		\chi_t^{\circ (\alpha)}(\Deli L^{(n)}) = \rho_{\alpha}^{\circ}(\Deli \tibigjeial) 
		\ind_{\lbrace | \Deli \tibigjeial | > v_n \rbrace} \Indit(\Deli \tibigjeial).
		\]
	\end{enumerate}
	Thus, we have for all $\alpha >0$, $(\gseqi, \dfi) \in \netir$, $\omega \in Q_n$ and $n \in \N$ large enough such that $v_n \leq \alpha$
	\begin{multline*}
	\hat V^{\circ (n)}_{\alpha}\leq \frac{1}{\sqrt{n \Delta_n}} \sup_{(\gseqi,\dfi) \in \netir} \Big| 
	\sum \limits_{i \in A_0^\circ(\omega;\alpha,n,(\gseqi,\dfi))} \xi_i \Big\{ \rho^{\circ}_{\alpha}(\varsigma_i^n(\alpha)) \ind_{\lbrace | \varsigma_i^n(\alpha)| > v_n \rbrace} \Indit(\varsigma_i^n(\alpha)) - \\ -\rho^{\circ}_{\alpha}(\Deli \tibigjeial) \ind_{\lbrace | \Deli \tibigjeial | > v_n \rbrace} \Indit(\Deli \tibigjeial) \Big\} \Big|,
	\end{multline*}
	with the random set 
	\begin{equation*}
	A_0^\circ(\omega;\alpha,n,(\gseqi,\dfi)) = \{ i \in \{1, \ldots, n\} \mid i \leq \ip{n \gseqi} \text{ and } \Deli N^n =1, \Deli \tibigj = \Deli \tibigjeial \neq 0 \},
	\end{equation*}
	for $(\gseqi,\dfi) \in \netir$, $\alpha >0$, $n \in \N$ and $\omega \in \Omega$. Furthermore, we define the random sets
	\begin{align*}
	A_1^\circ &(\omega;\alpha,n,(\gseqi, \dfi)) = \{ i \in A_0^\circ(\omega;\alpha,n,(\gseqi,\dfi)) \mid |\Deli \tibigjeial| > \Delta_n^{\ovv}, \\
	&\hspace{55mm} \Indit(\varsigma_i^n(\alpha)) = \Indit(\Deli \tibigjeial)=1\}, \\
	A_2^\circ &(\omega;\alpha,n,(\gseqi, \dfi)) = \{ i \in A_0^\circ(\omega;\alpha,n,(\gseqi,\dfi)) \mid |\Deli \tibigjeial| > \Delta_n^{\ovv}, \\
	&\hspace{45mm} \Indit(\varsigma_i^n(\alpha)) = 0 \text{ and } \Indit(\Deli \tibigjeial)=1\}, \\
	A_3^\circ &(\omega;\alpha,n,(\gseqi, \dfi)) = \{ i \in A_0^\circ(\omega;\alpha,n,(\gseqi,\dfi)) \mid |\Deli \tibigjeial| > \Delta_n^{\ovv}, \\
	&\hspace{45mm} \Indit(\varsigma_i^n(\alpha)) = 1 \text{ and } \Indit(\Deli \tibigjeial)=0\} ,
	\end{align*}
	for $(\gseqi,\dfi) \in \netir$, $\alpha >0$, $n \in \N$ and $\omega \in \Omega$, with $\ovv >0$ the constant in Assumption \ref{Cond1}. As a consequence, we obtain for $\omega \in Q_n$, $\alpha >0$ and $n$ large enough
	\begin{align}
	\label{haVnprAbsch}
	& \hat V^{\circ (n)}_{\alpha} \leq \nonumber \\
	&\leq \frac 1{\sqrt{n\Delta_n}} \sup_{(\gseqi,\dfi) \in \netir} \sum \limits_{i \in A_0^\circ (\omega;\alpha,n,(\gseqi, \dfi))} |\xi_i| \big| \rho_\alpha^\circ(\varsigma_i^n(\alpha)) \ind_{\{|\varsigma_i^n(\alpha)| > v_n\}} \Indit(\varsigma_i^n(\alpha)) \nonumber \\ 
	&\hspace{4mm}- \rho_\alpha^\circ(\Deli \tibigjeial)  \ind_{\{|\Deli \tibigjeial| > v_n\}} \Indit(\Deli \tibigjeial) \big| \ind_{\{|\Deli \tibigjeial| \leq \Delta_n^{\ovv}\}} \ind_{\{|\Deli \tilde X^{\prime n}| \leq v_n/2\}} \nonumber \\
	&\hspace{4mm}+ \frac 1{\sqrt{n\Delta_n}} \sup_{(\gseqi,\dfi) \in \netir} \sum \limits_{i \in A_1^\circ (\omega;\alpha,n,(\gseqi, \dfi))} |\xi_i| \big| \rho^{\circ}_{\alpha}(\varsigma_i^n(\alpha)) \ind_{\lbrace | \varsigma_i^n(\alpha)| > v_n \rbrace} - \nonumber \\ 
	&\hspace{37mm} -\rho^{\circ}_{\alpha}(\Deli \tibigjeial) \ind_{\lbrace | \Deli \tibigjeial | > v_n \rbrace} \big| \ind_{\{|\Deli \tibigjeial| > \Delta_n^{\ovv}\}} \ind_{\{|\Deli \tilde X^{\prime n}| \leq v_n/2\}} \nonumber \\
	&\hspace{4mm}+ \sup_{(\gseqi,\dfi) \in \netir} \Big| \sum\limits_{i \in A_2^\circ (\omega;\alpha,n,(\gseqi, \dfi))} \xi_i \bar a_i^n(\alpha) \Big| + \sup_{(\gseqi,\dfi) \in \netir} \Big| \sum\limits_{i \in A_3^\circ (\omega;\alpha,n,(\gseqi, \dfi))} \xi_i \bar b_i^n(\alpha) \Big|,
	\end{align}
	because  $|\Deli \tilde X^{\prime n}| \leq v_n/2$ holds for each $i=1, \ldots,n$ on $Q_n$ according to \eqref{Qndef2}.
	As a consequence, of Lemma \ref{Indikunglab} we have $\# A_2^\circ (\omega;\alpha,n,(\gseqi, \dfi)) \leq c_n$ as well as $\# A_3^\circ (\omega;\alpha,n,(\gseqi, \dfi)) \leq c_n$ with $c_n = \lceil (v_n/\Delta_n^{\ovr}) +1 \rceil$ for each $(\gseqi,\dfi) \in \netir$, $\alpha >0$ and $\omega \in J_n^{\scriptscriptstyle (2)}(\alpha) \cap Q_n$. Thus, \eqref{haVnprAbsch} yields the assertion.
\end{proof}

\subsection{Moments: Bounds and convergence results}

In the remaining part of Appendix \ref{appD} we gather results on moments of functionals of processes which occur several times in Section \ref{sec:weConv}.

\begin{lemma}
	\label{PRMProsevjb} 
	For $n \in \N$ let $\mu^{\scriptscriptstyle (n)}$ be a Poisson random measure with predictable compensator $\bar \mu^{\scriptscriptstyle (n)}(ds,d\taili) = \nu_s^{\scriptscriptstyle (n)}(d\taili)ds$ such that \eqref{RescAss3} is satisfied for all $n \in \N$ for $\Delta_n >0$ and transition kernels $g^{\scriptscriptstyle (n)}$ from $([0,1], $ $\Bb([0,1]))$ into $(\R, \Bb)$ with 
	\[\Big( \lambda_1 - \mathrm{ess~sup}_{y \in [0,1]}\int (1 \wedge |\taili|^{\beta})g^{(n)}(y,d\taili) \Big) \leq K\] 
	for each $n \in\N$ and some $\beta \in [0,2]$, $K>0$. Furthermore, let $c >0$, $F \subset \{ \taili \in \R \mid |\taili| >c\}$ and $N^{\scriptscriptstyle (n)} = \ind_{F} \star \mu^{\scriptscriptstyle (n)}$. Then for $0 \leq t_1 \leq t_2 \leq n \Delta_n$ the following equality in distribution holds
	\begin{align}
	\label{tilNtdisteq}
	N^{(n)}_{t_2} -  N^{(n)}_{t_1} =_d \ \mathrm{Poiss}\big(\zeta_{t_2}^{(n)} - \zeta_{t_1}^{(n)} \big),
	\end{align}
	with
	\begin{align}
	\label{lambdancaleq}
	\zeta_t^{(n)} = \int_0^t \int_{F} \nu_s^{(n)}(d\taili) ds = n\Delta_n \int_0^{t/(n\Delta_n)} \int_{F} g^{(n)}(y,d\taili) dy,
	\end{align}
	for $t \in [0,n \Delta_n]$. Moreover, for $i \in \{1, \ldots,n\}$ the sets $A_n^{(i)} := \big\{N^{\scriptscriptstyle (n)}_{i \Delta_n} - N^{\scriptscriptstyle (n)}_{(i-1) \Delta_n} \leq 1 \big\}$ satisfy
	\begin{equation*}
	\Prob\big((A_n^{(i)})^C\big) \leq K \Delta_n^2 (c \wedge 1)^{-2\beta}.
	\end{equation*}
\end{lemma}

\begin{proof}
	\eqref{tilNtdisteq} is a consequence of Theorem II.4.8 in \cite{JacShi02}. Furthermore, according to \eqref{tilNtdisteq} we calculate as follows
	\begin{align*}
	\Prob\big((A_n^{(i)})^C\big) &= \exp\Big\{-\Big(\zeta_{i\Delta_n}^{(n)} - \zeta_{(i-1)\Delta_n}^{(n)}\Big)\Big\}\sum \limits_{k=2}^\infty \frac{\big(\zeta_{i\Delta_n}^{(n)} - \zeta_{(i-1)\Delta_n}^{(n)}\big)^k}{k!} \nonumber \\&\leq \Big(\zeta_{i\Delta_n}^{(n)} - \zeta_{(i-1)\Delta_n}^{(n)}\Big)^2 
	=\Big(n\Delta_n \int_{(i-1)/n}^{i/n} \int_{F} g^{(n)}(y,d\taili) dy \Big)^2,
	\end{align*}
	where the final equality in the above display is a consequence of \eqref{lambdancaleq}. Now, using the assumption that the mapping $(y \mapsto \int (1\wedge |\taili|^\beta) g^{\scriptscriptstyle (n)}(y,d\taili))$ is Lebesgue almost surely bounded on $[0,1]$ we obtain
	\begin{align*}
	\Prob((A_n^{(i)})^C) \leq \Big( n\Delta_n (c \wedge 1)^{-\beta} \int_{(i-1)/n}^{i/n} \int (1 \wedge |\taili|^\beta) g^{(n)}(y,d\taili) dy \Big)^2 \leq K \Delta_n^2 (c \wedge 1)^{-2\beta},
	\end{align*}
	because $(c \wedge 1)^{-\beta} (1 \wedge |\taili|^\beta) \geq 1$ holds if $|\taili| > c$.
\end{proof}

\begin{lemma}
	\label{ainalmombou}
	Let Assumption \ref{Cond1} be satisfied and let $\alpha >0$. Then for $n \in\N$ large enough  we have
	\[
	\Eb |a_i^n (\alpha)|^m \leq \Big( \frac{K(\alpha)}{\sqrt{n \Delta_n}} \Big)^m \Delta_n,
	\]
	for all $m \in \N$ and all $i=1,\ldots,n$, where for $\alpha >0$, $n \in \N$ and $i=1,\ldots,n$
	\[
	a_i^n (\alpha) = \frac 1{\sqrt{n \Delta_n}} \rho_{\alpha}(\Deli \hatbigjal) \ind_{\lbrace | \Deli \hatbigjal | > 
		\alpha/4 \rbrace}, 
	\]
	with $\hatbigjal$ defined in \eqref{SmandlajumpD2} and where $\rho_\alpha$ is defined prior to \eqref{ImpQuantDef}.
\end{lemma}

\begin{proof}
	For $n \in \N$, $i=1, \ldots, n$ and $\alpha >0$ define $N^{\scriptscriptstyle (\alpha , n)} = \ind_{\{|\taili| > \alpha/4\}} \star \mu^{\scriptscriptstyle (n)}$ and $H_i^n(\alpha) = \{N^{\scriptscriptstyle (\alpha , n)}_{i \Delta_n} - N^{\scriptscriptstyle (\alpha , n)}_{(i-1) \Delta_n} \leq 1\}$. Then Lemma \ref{PRMProsevjb} shows that 
$	
	\Prob\big(H_i^n(\alpha)^C\big) \leq K(\alpha) \Delta_n^2.
$
	Notice that on $H_i^n(\alpha)$ the quantity $\Deli \hatbigjal$ is either zero or equal to the only jump larger than $\alpha/4$ on the interval $((i-1)\Delta_n, i\Delta_n]$. Thus, we obtain 
	\begin{align*}
	\Eb |a_i^n (\alpha)|^m &\leq \Big( \frac K{\sqrt{n \Delta_n}} \Big)^m \Prob\big(H_i^n(\alpha)^C\big) + \\
	&\hspace{1cm}+ \Big( \frac 1{\sqrt{n \Delta_n}} \Big)^m \Eb \Big\{ \int_{(i-1)\Delta_n}^{i \Delta_n} \int |\rho_{\alpha}(\taili)|^m \ind_{\lbrace | \taili | > 
		\alpha/4 \rbrace} \ind_{H_i^n(\alpha)} \mu^{(n)}(\omega;du,d\taili) \Big\} \\
	&\leq \Big( \frac K{\sqrt{n \Delta_n}} \Big)^m \Big( K(\alpha) \Delta_n^2 +  \Eb \Big\{ \int_{(i-1)\Delta_n}^{i \Delta_n} \int (1 \wedge |\taili|^p)^m \ind_{\lbrace | \taili | > 
		\alpha/4 \rbrace}  \mu^{(n)}(\omega;du,d\taili) \Big\} \Big),
	\end{align*}
	where $K>0$ is chosen such that $|\rho(\taili)| \leq K(1\wedge |\taili|^p)$ for all $\taili \in \R$. Furthermore, due to $m \geq 1$ we have $(1\wedge |\taili|^p)^m \leq (1\wedge |\taili|^p)$ and consequently the definition of the predictable compensator of an optional $\Pc'$-$\sigma$-finite random measure (see Theorem II.1.8 in \cite{JacShi02}) yields
	\begin{align*}
	\Eb |a_i^n (\alpha)|^m &\leq \Big( \frac{K(\alpha)}{\sqrt{n \Delta_n}} \Big)^m \Big\{ \Delta_n^2 + n \Delta_n  \int_{(i-1)/n}^{i/n} \int (1 \wedge |\taili|^p) \ind_{\lbrace | \taili | > 
		\alpha/4 \rbrace}  g^{(n)}(y,d\taili)dy \Big\} \\
	&\leq \Big( \frac{K(\alpha)}{\sqrt{n \Delta_n}} \Big)^m \Delta_n,
	\end{align*}
	for $n \in \N$ large enough due to Assumption \ref{Cond1}\eqref{BlGetCond} and $p>\beta$. 
\end{proof}

\begin{lemma}
	\label{baainamomb}
	Let Assumption \ref{Cond1} be satisfied and let $\alpha >0$. Then for $n \in \N$ large enough we have
	\[
	\Eb |\bar a_i^n(\alpha)|^m \leq \Big( \frac K{\sqrt{n \Delta_n}} \Big)^m \Delta_n,
	\]
	for all $m \in \N$ and all $i=1,\ldots,n$, where for $\alpha >0$, $n \in \N$ and $i=1,\ldots,n$
	\[
	\bar a_i^n(\alpha) = \frac 1{\sqrt{n\Delta_n}} \rho_\alpha^\circ(\Deli \tibigjeial) \ind_{\{| \Deli \tibigjeial| > v_n \vee \Delta_n^{\ovv}\}},
	\]
	with $\tibigjeial$ defined in \eqref{SmandlajumpD2}, $\ovv >0$ is the constant from Assumption \ref{Cond1}\eqref{FiLevyDistCond} and where $\rho^\circ_\alpha$ is defined prior to \eqref{ImpQuantDef}.
\end{lemma}

\begin{proof}
	For $n \in \N$, $i=1,\ldots,n$ and $\alpha >0$ we define the processes $\bar N^{\scriptscriptstyle (\alpha , n)} = \ind_{\{u_n < |\taili| \leq 2\alpha\}} \star \mu^{\scriptscriptstyle (n)}$, where $u_n = v_n^\ell$ with $\ell$ the constant in \eqref{elleqn2}. If we define the sets $\bar H_i^n(\alpha) = \{\bar N^{\scriptscriptstyle (\alpha , n)}_{i\Delta_n} - \bar N^{\scriptscriptstyle (\alpha , n)}_{(i-1)\Delta_n} \leq 1\}$, Lemma \ref{PRMProsevjb} yields 
	\begin{align}
	\label{barHinalAb}
	\Prob\big(\bar H_i^n(\alpha)^C \big) 
	\leq K(\delta) \Delta_n^{2-2\ell \ovw (\beta+\delta)},
	\end{align}
	for all $\delta >0$ and $n \in \N$ large enough. Consequently, using the fact that on $\bar H_i^n (\alpha)$ the increment $\Deli \tibigjeial$ is either zero or equal to the only jump of absolute size in $(u_n, 2\alpha]$ on the interval $((i-1)\Delta_n, i\Delta_n]$ we obtain
	\begin{align*}
	\Eb |\bar a_i^n(\alpha)|^m &\leq \Big( \frac{K(\delta)}{\sqrt{n \Delta_n}} \Big)^m \Prob\big( \bar H_i^n(\alpha)^C \big) + \\
	&\hspace{7mm}+ \Big(\frac 1{\sqrt{n\Delta_n}} \Big)^m \Eb \int_{(i-1)\Delta_n}^{i\Delta_n} \int |\rho_\alpha^{\circ}(\taili)|^m \ind_{\{|\taili| > v_n \vee \Delta_n^{\ovv} \}} \ind_{\bar H_i^n (\alpha)} \mu^{(n)}(\omega; du,d\taili) \\
	&\leq \Big( \frac{K(\delta)}{\sqrt{n\Delta_n}} \Big)^m \Big\{ \Delta_n^{2-2\ell \ovw (\beta+\delta)} + \Eb \int_{(i-1)\Delta_n}^{i\Delta_n} \int (1\wedge |\taili|^p)^m  \mu^{(n)}(\omega; du,d\taili) \Big\},
	\end{align*}
	for every $\delta >0$, where $K(\delta) >0$ is chosen such that \eqref{barHinalAb} holds and $|\rho(\taili)| \leq K(\delta)(1\wedge |\taili|^p)$ (see Assumption \ref{Cond1}\eqref{RhoCond}). Thus, due to $(1\wedge |\taili|^p)^m \leq (1\wedge |\taili|^p)$ and the definition of the predictable compensator of an optional $\Pc'$-$\sigma$-finite random measure (see Theorem II.1.8 in \cite{JacShi02}) we obtain for some small $\delta >0$ and $n \in \N$ large enough
	\begin{align*}
	\Eb |\bar a_i^n(\alpha)|^m &\leq \Big( \frac{K(\delta)}{\sqrt{n\Delta_n}} \Big)^m \Big\{ \Delta_n^{2-2\ell \ovw (\beta+\delta)} + n \Delta_n \int_{(i-1)/n}^{i/n} \int (1 \wedge |\taili|^p) g^{(n)}(y,d\taili) dy \Big \} \\
	&\leq \Big( \frac{K(\delta)}{\sqrt{n\Delta_n}} \Big)^m \Delta_n,
	\end{align*}
	because of $p>\beta$ and Assumption \ref{Cond1}\eqref{BlGetCond}, as well as $\ell < 1/(2\beta \ovw)$.
\end{proof}

The proof of the following lemma requires the notion of Orlicz norms. Recall from Section 2.2 in \cite{VanWel96} that for $\Lambda: \R_+ \to \R$ a non-decreasing, convex function with $\Lambda(0)=0$ and a random variable $Z$ the Orlicz norm is defined as
\[
\| Z \|_\Lambda = \inf \big\{ C>0 \mid \Eb \Lambda\big(|Z|/C\big) \leq 1 \big\},
\]
where we set $\inf \emptyset = \infty$. It is easy to check that if $\Lambda$ equals the function $x \mapsto x^p$ for some $p \ge 1$, the corresponding Orlicz norm is the well-known $L^p$-norm
$
\|Z\|_p = \big( \Eb |Z|^p \big)^{1/p}.
$
Furthermore for $\Lambda_1(x):= e^x -1$ a straight forward calculation gives
\begin{equation}
\label{ExpOrNogr}
\|Z\|_p \leq p! \|Z\|_{\Lambda_1}, \quad \text{ for all } p \in \N,
\end{equation}
because $x^p \leq p!(e^x -1)$ for all $x \in \R_+$ by the series expansion of the exponential function.

\begin{lemma}
	\label{supsnzinlem}
	Let Assumption \ref{Cond1} be satisfied and for $n\in\N$ let $(Z_i^n)_{i=1,\ldots,n}$ be independent random variables with mean zero such that there exist constants $C_1,C_2 >0$ with
	\begin{equation}
	\label{ZinMomAss}
	\Eb |Z_i^n|^m \leq m! \Big( \frac{C_1}{\sqrt{n\Delta_n}}\Big)^{m-2} \frac{C_2}n,
	\end{equation}
	for every integer $m \geq 2$. Then we have 
	\[
	\Eb \Big\{ \sup_{A \in \mathfrak S_n} \Big| \sum_{i \in A} Z_i^n \Big| \Big\} = o(1),
	\]
	as $n\to\infty$ for $\mathfrak S_n = \{ M \subset \{1, \ldots,n\} \mid \# M \leq c_n\}$ with $c_n = \lceil (v_n/ \Delta_n^{\ovr}) + 1\rceil$.
\end{lemma}

\begin{proof}
	The modified Bernstein inequality (Lemma 2.2.11 in \cite{VanWel96}) and \eqref{ZinMomAss} yields
	\begin{align*}
	\Prob\Big(\big|\sum_{i \in A} Z_i^n\big| >x\Big) \leq 2 e^{- \frac 12 \frac{x^2}{b_n+ d_nx}}
	\end{align*}
	for every $x \in \R_+$, $A \in \mathfrak S_n$ with $b_n = 2 C_2 c_n/n$ and $d_n = C_1/\sqrt{n\Delta_n}$, because each $A \in \mathfrak S_n$ consists of at most $c_n$ elements. Therefore, by Lemma 2.2.10 in the previously mentioned reference, the fact that $\# \mathfrak S_n \leq (n+1)^{c_n}$ and \eqref{ExpOrNogr} we obtain for a universal constant $C$ and $n \geq 2$
	\begin{align}
	\label{EbEnAbmitVWe}
	\Eb \Big\{ \sup_{A \in \mathfrak S_n} \Big| \sum_{i \in A} Z_i^n \Big| \Big\} &\leq \Big\| \sup_{A \in \mathfrak S_n} \Big| \sum_{i \in A} Z_i^n \Big| \Big\|_{\Lambda_1} \leq C \big( d_n \log(1+ (n+1)^{c_n}) + \sqrt{b_n \log(1+ (n+1)^{c_n})} \big) \nonumber \\
	&\leq K \frac 1{\sqrt{n\Delta_n}} \big( v_n/\Delta_n^{\ovr} +2 \big) \log(2n) + K \big( v_n/\Delta_n^{\ovr} +2 \big) \sqrt{\log(2n)/n} \nonumber \\
	&\leq K\frac{\log(2n)}{\sqrt{n\Delta_n^{(1+2\ovr - 2 \ovw) \vee 1}}} = K \frac {\Delta_n^{\delta/2} \log(2n)}{\sqrt{n\Delta_n^{((1+2\ovr - 2 \ovw) \vee 1) + \delta}}} \rightarrow 0,
	\end{align}
	with some $\delta >0$ such that Assumption \ref{Cond1}\eqref{ObsSchCond7} is satisfied. The final convergence in \eqref{EbEnAbmitVWe} holds, because by Assumption \ref{Cond1}\eqref{ObsSchCond4} we have $\Delta_n = o(n^{-u})$ for some $0<u<1$.
\end{proof}

\begin{lemma}
	\label{anabnakonv}
	Grant Assumption \ref{Cond1} we have  for all $\alpha >0$
	\begin{align*} 
	a_n(\alpha) &= \frac{1}{\sqrt{n \Delta_n}} \sum \limits_{i=1}^n \Eb \Big\{ \big| \Deli \hatbigjal \big|^p 
	\ind_{\lbrace | \Deli \hatbigjal| \leq 2 v_n \rbrace} \Big\} =o(1)  \\
	b_n(\alpha) &= \frac{v_n}{2 \sqrt{n \Delta_n}} \sum \limits_{i=1}^n \Eb \big| \Deli \hatbigjal \big|^{p-1}=o(1),
	\end{align*}
	with $\hatbigjal$ defined in \eqref{SmandlajumpD2}.
\end{lemma}

\begin{proof}
	Obviously, $|\taili|^p \ind_{\{|\taili| \leq 2v_n\}} \leq 2v_n |\taili|^{p-1}$ holds for all $\taili \in \R$. Consequently, $a_n(\alpha) \leq 4b_n(\alpha)$ and it suffices to verify $\lim_{n \to \infty} b_n(\alpha) =0$. For $\gamma \in \R_+$ set
	\[
	\widehat \delta^{(n)}_{\alpha}(\gamma) = \lambda_1 - \mathrm{ess~sup}_{y \in [0,1]} \Big(\int | \taili |^\gamma \ind_{\lbrace |\taili| > \alpha/4 \rbrace} g^{(n)}(y,dz) \Big).
	\]
	Then Assumption \ref{Cond1}\eqref{BlGetCond} and \eqref{LevMeaMomCond} yield $\{ \widehat \delta^{(n)}_{\alpha}(1) \vee \widehat \delta^{(n)}_{\alpha}(p-1) \} \leq K < \infty$ for all $\alpha >0$, $n \in \N$ and we obtain the desired result with Lemma \ref{JacProL217} and Assumption \ref{Cond1}\eqref{ObsSchCond4} as follows:
	\begin{align*}
	b_n(\alpha) 
	&\leq \frac{Kv_n}{\sqrt{n \Delta_n}} \sum \limits_{i=1}^n\Big\{ \int_{(i-1)\Delta_n}^{i\Delta_n} \int |\taili|^{p-1} \ind_{\{|\taili| > \alpha/4\}} \nu^{(n)}_s(d\taili) ds + \\
	&\hspace{35mm}+ \ind_{\{p-1 \geq 1\}} \Big( \int_{(i-1)\Delta_n}^{i\Delta_n} \int |\taili| \ind_{\{|\taili| > \alpha/4\}} \nu^{(n)}_s(d\taili) ds \Big)^{p-1} \Big\}\\
	&= \frac{Kv_n}{\sqrt{n \Delta_n}} \sum \limits_{i=1}^n\Big\{n \Delta_n \int_{(i-1)/n}^{i/n} \int |\taili|^{p-1} \ind_{\{|\taili| > \alpha/4\}} g^{(n)}(y,d\taili) dy + \\
	&\hspace{30mm}+ \ind_{\{p-1 \geq 1\}} \Big( n \Delta_n \int_{(i-1)/n}^{i/n} \int |\taili| \ind_{\{|\taili| > \alpha/4\}} g^{(n)}(y,d\taili) dy \Big)^{p-1} \Big\} \\
	&\leq \frac{K n v_n}{\sqrt{n \Delta_n}} \big\{ \Delta_n  + 
	\Delta_n^{(p-1) \vee 1} \big\} = O \bigg (\sqrt{n \Delta_n^{1+ 2 \ovw }} \bigg ) =o(1).  
	\end{align*}
\end{proof}

\begin{lemma}
	\label{cnadnakonv}
	Grant Assumption \ref{Cond1}. Then we have for all $\alpha >0$
	\begin{align*} 
	c_n(\alpha) &= \frac{1}{\sqrt{n \Delta_n}} \sum \limits_{i=1}^n \Eb \Big\{ \big| \Deli \tibigjeial \big|^p 
	\ind_{\lbrace | \Deli \tibigjeial | \leq 2 v_n \rbrace} \Big\}= o(1) \\
	d_n(\alpha) &= \frac{v_n}{2 \sqrt{n \Delta_n}} \sum \limits_{i=1}^n \Eb \big| \Deli \tibigjeial \big|^{p-1} = o(1), 
	\end{align*}
	with $\tibigjeial$ defined in \eqref{SmandlajumpD2}.
\end{lemma}

\begin{proof}
	Obviously, $|\taili|^p \ind_{\{|\taili| \leq 2v_n\}} \leq 2v_n |\taili|^{p-1}$ holds for each $\taili \in\R$. Thus, $c_n(\alpha) \leq 4 d_n(\alpha)$ and it is enough to verify $\lim_{n \to \infty} d_n(\alpha) =0$. Let $\alpha >0$ be fixed and define further for $\gamma \in \R_+$
	\[
	\widehat \delta_{n, \alpha}(\gamma) = \lambda_1 - \mathrm{ess~sup}_{y \in [0,1]} \Big(\int |\taili|^\gamma \ind_{\lbrace u_n < |\taili| \leq 2 \alpha \rbrace} g^{(n)}(y,d\taili) \Big).
	\]
	Note that due to Assumption \ref{Cond1}\eqref{BlGetCond} and $p-1 >\beta$ we have for each small $\delta >0$:
	\[
	\widehat \delta_{n,\alpha} (1) \leq K(\delta) u_n^{-(\beta+\delta-1)_+} \quad \text{ and } \quad \widehat \delta_{n,\alpha} (p-1) \leq K(\delta).
	\]
	Furthermore, Lemma \ref{JacProL217} gives
	\begin{align*}
	\Eb \big| &\Deli \tibigjeial \big|^{p-1} \leq
	\\ &\leq K \times \begin{cases}
	n \Delta_n \int_{(i-1)/n}^{i/n} \int |\taili|^{p-1} \ind_{\{u_n < |\taili| \leq 2\alpha\}} g^{(n)}(y,d\taili)dy, &\text{ if } p \leq 2, \\
	n \Delta_n \int_{(i-1)/n}^{i/n} \int |\taili|^{p-1} \ind_{\{u_n < |\taili| \leq 2\alpha\}} g^{(n)}(y,d\taili)dy + \\ \hspace{1cm}+ \Big( n \Delta_n \int_{(i-1)/n}^{i/n} \int |\taili| \ind_{\{u_n < |\taili| \leq 2\alpha\}} g^{(n)}(y,d\taili)dy \Big)^{p-1}, &\text{ if } p >2,
	\end{cases} \\
	&\leq K \times \begin{cases}
	\Delta_n \widehat \delta_{n, \alpha}(p-1), \quad &\text{ if } p \leq 2, \\
	\Delta_n \widehat \delta_{n, \alpha}(p-1) + \Delta_n^{p-1} 
	(\widehat \delta_{n, \alpha}(1))^{p-1}, \quad &\text{ if } p >2,
	\end{cases} \\
	&= K(\delta) \times \begin{cases}
	\Delta_n,  \quad &\text{ if } p \leq 2, \\
	\Delta_n + \Delta_n^{p-1} u_n^{-(p-1)(\beta+\delta -1)_+}, \quad &\text{ if } p >2,
	\end{cases}
	\end{align*}
	for each $1 \leq i \leq n$ and $\delta >0$ small enough. Thus, when $p \leq 2$, Assumption \ref{Cond1}\eqref{ObsSchCond4} yields
	\begin{align*}
	d_n(\alpha) &\leq K(\delta) \frac{1}{\sqrt{n \Delta_n}} n \Delta_n v_n = O\Big( \sqrt{n \Delta_n^{1+2\ovw}} \Big) =o(1).
	\end{align*}
	In the case $p >2$ we obtain also from Assumption \ref{Cond1}\eqref{ObsSchCond4} for $\delta$ small enough:
	\begin{align*}
	d_n(\alpha) &\leq K(\delta) \left\{ \frac{1}{\sqrt{n \Delta_n}} n \Delta_n v_n + \frac{1}{\sqrt{n \Delta_n}} n \Delta_n^{p-1} v_n 
	u_n^{-(\beta+\delta -1)_+ (p-1)} \right\} \\
	&\leq K(\delta) \left\{ \sqrt{n \Delta_n^{1+ 2 \ovw}} + \sqrt{ n \Delta_n^{2(p-1) - 2(p-1)(\beta+\delta-1)_+ \ell \ovw 
			-1 + 2 \ovw}} \right\} \\
	&\leq 
	K(\delta) \sqrt{n \Delta_n^{1+ 2 \ovw}} = o(1), 
	\end{align*}
	where the last inequality above is clear for $\beta <1$, for $\beta=1$ we have $2(p-1)  - 2(p-1)(\beta+ \delta -1) \ell \ovw -1 >1$ from $p>2$, and in the case $\beta > 1$ we calculate using $\ell < \frac 1{2\beta\ovw}$ and $p > 1+\beta$:
	\begin{align*}
	2(p-1)  - 2(p-1)(\beta+ \delta -1) \ell \ovw -1 &= 2(p-1)(1-(\beta+\delta -1) \ell \ovw) -1 \\
	&>2(p-1)\Big(1-\frac{\beta+\delta -1}{2\beta}\Big) -1 \\
	&>\beta\Big(1+ \frac 1\beta- \frac\delta \beta\Big)-1 = \beta - \delta >1.
	\end{align*}
\end{proof}

\begin{lemma}
	\label{ynaznakonv}
If Assumption \ref{Cond1} holds, then we have for all $\alpha >0$
	\begin{align*} 
	y_n^{(\alpha)} &= \frac{1}{\sqrt{n \Delta_n}} \sum \limits_{i=1}^n \Eb \Big\{ \big| \big| \Deli \tilde X^{\prime n} +\Deli \tibigjeial \big|^p - \big| \Deli \tibigjeial \big|^p  \big|  \\
	 &\times \ind_{\{|\Deli \tibigjeial | \leq \Delta_n^{\ovv} \}} 
	\ind_{\{|\Deli \tilde X^{\prime n} +\Deli \tibigjeial |  >v_n \}} \ind_{\lbrace | \Deli \tilde X^{\prime n} | \leq v_n/2 \rbrace}\ind_{Q_n} \Big\} =o(1) , \\
	z_n^{(\alpha)} &= \frac{1}{\sqrt{n \Delta_n}} \sum \limits_{i=1}^n \Eb \Big\{ \big| \Deli \tibigjeial \big|^p \ind_{\{|\Deli \tibigjeial | \leq \Delta_n^{\ovv} \}} \ind_{Q_n} 
	 \Big\}=o(1) ,
	\end{align*}
	where  $\ovv >0$ is  the constant in Assumption \ref{Cond1}\eqref{FiLevyDistCond} and the involved processes and the set $Q_n$ are defined in \eqref{SmandlajumpD2} and \eqref{Qndef2}, respectively.
\end{lemma}

\begin{proof}
	First we consider $y_n^{\scriptscriptstyle (\alpha)}$. The mean value theorem yields
	\begin{multline*}
	y_n^{(\alpha)}\leq \frac{1}{\sqrt{n \Delta_n}} \sum \limits_{i=1}^n \Eb \Big\{ \big| \Deli \tilde X^{\prime n} \big| p \xi_i^{p-1} \ind_{\{|\Deli \tibigjeial | \leq \Delta_n^{\ovv} \}} \times \\
	\times \ind_{\{|\Deli \tilde X^{\prime n} +\Deli \tibigjeial |  >v_n \}} \ind_{\lbrace | \Deli \tilde X^{\prime n} | \leq v_n/2 \rbrace} \ind_{Q_n} \Big\},
	\end{multline*}
	for some $\xi_i$ between $|\Deli \tibigjeial|$ and $|\Deli \tilde X^{\prime n} + \Deli \tibigjeial|$. Next using the fact that due to the indicators $|\Deli \tilde X^{\prime n}| \leq |\Deli \tibigjeial|$ holds, we obtain
	\begin{align*}
	y_n^{(\alpha)} \leq \frac{K}{\sqrt{n \Delta_n}} \sum \limits_{i=1}^n &\Eb \Big\{ \big| \Deli \tilde X^{\prime n} \big| \big| \Deli \tibigjeial \big|^{p-1} \ind_{\{|\Deli \tibigjeial | \leq \Delta_n^{\ovv} \}} \ind_{\lbrace | \Deli \tilde X^{\prime n} | \leq v_n/2 \rbrace} \ind_{Q_n} \Big\}.
	\end{align*}
	Note that on $Q_n$ the sum $\Deli \tibigjeial$ consists of at most one jump. Therefore, we can calculate with the definition of the predictable compensator of the random measure associated with the jumps of $X^{\scriptscriptstyle (n)}$:
	\begin{align*}
	y_n^{(\alpha)} &\leq \frac{K v_n}{\sqrt{n \Delta_n}} \sum \limits_{i=1}^n \Eb \Big \{ \big(|\taili|^{p-1} \ind_{\{u_n <|\taili| \leq \Delta_n^{\ovv}\}} \ind_{\{(i-1) \Delta_n < s \leq i \Delta_n\}}\big) \star \mu^{(n)} \Big \}  \\
	&= \frac{K v_n}{\sqrt{n \Delta_n}} \sum\limits_{i=1}^n \int_{(i-1)\Delta_n}^{i\Delta_n}\int |\taili|^{p-1} \ind_{\{u_n <|\taili| \leq \Delta_n^{\ovv}\}} \nu_s^{(n)}(d\taili) ds \\
	&= \frac{K n \Delta_n v_n}{\sqrt{n \Delta_n}} \int_{0}^{1} \int |\taili|^{p-1} \ind_{\{u_n <|\taili| \leq \Delta_n^{\ovv}\}} g^{(n)}(y,d\taili) dy= o\Big(\sqrt{n \Delta_n^{1+2\ovw}} \Big) =o(1),
	\end{align*}
	by Assumption \ref{Cond1}\eqref{BlGetCond} and \eqref{ObsSchCond4}, because $p-1 >\beta$. 
	
	Now we show the claim $z_n^{\scriptscriptstyle (\alpha)} = o(1)$. This can be seen again by the definition of the predictable compensator of the random measure associated with the jumps of $X^{\scriptscriptstyle (n)}$. By the fact that on $Q_n$ $\Deli \tibigjeial$ is either $0$ or equal to the only jump of absolute size in $(u_n,2\alpha]$ on the interval $((i-1)\Delta_n, i \Delta_n]$, we have:
	\begin{align*}
	z_n^{(\alpha)} &\leq \frac{1}{\sqrt{n \Delta_n}} \sum \limits_{i=1}^n \Eb \Big\{ \big( |\taili|^p \ind_{\{u_n < |\taili| \leq \Delta_n^{\ovv}\}} \ind_{\{(i-1)\Delta_n < s \leq i \Delta_n \}} \big) \star \mu^{(n)} \Big\} \\
	&= \sqrt{n \Delta_n} \int_0^1 \int |\taili|^p \ind_{\{u_n < |\taili| \leq \Delta_n^{\ovv}\}} g^{(n)}(y,d\taili) dy = O\Big( \sqrt{n \Delta_n^{1+2 \ovv(p-\beta-\delta)}} \Big) = o(1),
	\end{align*}
	according to Assumption \ref{Cond1}\eqref{ObsSchCond5b}, for some appropriate $\delta >0$.
\end{proof}

\begin{lemma}
	\label{MomentfgO}
	Let Assumption \ref{Cond1} be satisfied and let $f : \R \to \R$ be a Borel measurable function with $|f(\taili)| \leq K(1\wedge |\taili|^p)$ for all $\taili \in \R$ and some $K>0$. Then we have 
	\begin{equation}
	\label{SecSupMomAb}
	\sup_{i \in \{1,\ldots,n\}}\Eb \big(\big| f(\Deli L^{(n)}) \big| \big) = O(\Delta_n)
	\end{equation}
\end{lemma}

\begin{proof}
	By the assumptions on $f$ we obtain for $i=1,\ldots,n$ from Proposition \ref{BiasAbschProp}
	\begin{align}
	\label{FiSupMomAb}
	\Eb &\big( \big| f(\Deli L^{(n)}) \big| \big) \leq K \Eb \big(1\wedge |\Deli L^{(n)}|^p\big) \nonumber \\
	&\leq Kn\Delta_n \int_{(i-1)/n}^{i/n} \int \big(1\wedge |\taili|^p\big) g^{(n)}(y,d\taili) dy + O\big(\Delta_n^2 v_n^{-2((\beta+\delta) \wedge 2)} + \Delta_n v_n^{p-((\beta+\delta)\wedge 2)}\big) \nonumber \\
	&\leq K \Big(n\Delta_n \int_{(i-1)/n}^{i/n} \int \big(1\wedge |\taili|^{(\beta+\delta)\wedge 2}\big) g_0(y,d\taili) dy + \sqrt{n\Delta_n} \int_{(i-1)/n}^{i/n} \int \big(1\wedge |\taili|^{(\beta+\delta)\wedge 2}\big) g_1(y,d\taili) dy \nonumber \\
	&\hspace{1cm} + n\Delta_n a_n \int_{(i-1)/n}^{i/n} \int \big(1\wedge |\taili|^{(\beta+\delta)\wedge 2}\big) g_2(y,d\taili) dy \Big) +  O\big(\Delta_n^2 v_n^{-2((\beta+\delta) \wedge 2)} + \Delta_n v_n^{p-((\beta+\delta)\wedge 2)}\big) \nonumber \\
	&= O(\Delta_n) + O\big((\Delta_n/n)^{1/2}\big) + o\big((\Delta_n/n)^{1/2}\big) +  O\big(\Delta_n^2 v_n^{-2((\beta+\delta) \wedge 2)} + \Delta_n v_n^{p-((\beta+\delta)\wedge 2)}\big),
	\end{align}
	for each $\delta >0$, because by Assumption \ref{Cond1}\eqref{BlGetCond} we have $\int(1 \wedge |\taili|^{(\beta + \delta)\wedge 2}) g_i(y,d\taili) \leq K(\delta)$ for Lebesgue almost every $y \in [0,1]$ for all $i\in\{0,1,2\}$ and some $K(\delta)>0$. Furthermore, in the display above $a_n$ denotes a sequence of non-negative real numbers with $a_n = o((n\Delta_n)^{-1/2})$ and $\Rc_n(y,B) \leq a_n g_2(y,B)$ for all $y\in[0,1]$, $B\in\Bb$, $n \in\N$  according to  Assumption \ref{Cond1}. Now, \eqref{FiSupMomAb} yields \eqref{SecSupMomAb} because of three reasons: first $(\Delta_n/n)^{1/2} \leq \Delta_n$ for large $n \in \N$, moreover, $p > \beta$ so $v_n^{p-((\beta+\delta)\wedge 2)} \leq 1$ for large $n\in\N$ and if $\beta =2$ we have $2((\beta + \delta)\wedge 2) \ovw = 4\ovw <1$ due to $\ovw < (1/2\beta)$, while in the case $\beta <2$ we obtain $2((\beta + \delta)\wedge 2) \ovw = 2(\beta +\delta) \ovw <1$ for $\delta>0$ small enough using $\ovw < (1/2\beta)$ again.
\end{proof}

\begin{lemma}
	\label{MomklAbLem}
	Grant Assumption \ref{Cond1} and let $Q_n$ be the set defined in \eqref{Qndef2}. Then we have for sufficiently large $n \in \N$
	\begin{align}
	\Eb \big( \big| \rho(\Deli X^{(n)})\big| \ind_{\{| \Deli X^{(n)} | > v_n \}} \ind_{Q_n} \big) &\leq K \Delta_n   \label{MomKlEq} \\
	\Eb \big( \big| \rho(\Deli X^{(n)})\big| \big| \rho(\Delj X^{(n)})\big| \ind_{\{| \Deli X^{(n)} | > v_n \}} \ind_{\{| \Delj X^{(n)} | > v_n \}} \ind_{Q_n} \big) &\leq K \Delta_n^2 \label{PrMomKlEq}
	\end{align}
	for $i,j = 1, \ldots, n$ with $i \neq j$, where the constant $K>0$ is independent of $n,i$ and $j$.
\end{lemma}

\begin{proof}
	Recall the decomposition $X^{(n)}  = \tilde X^{\prime n} + \tibigj$ and the sets 
	\[
	A_i^n	= \lbrace | \Deli \tilde X^{\prime n} | \leq v_n/2 \rbrace \cap \lbrace \Deli N^n \leq 1 \rbrace
	\]
	in \eqref{SmandlajumpD2}. According to \eqref{Qndef2} we then have
	$
	Q_n = \bigcap \limits_{i=1}^n A_i^n
	$
	and in order to show \eqref{MomKlEq} we obtain
	\begin{align}
	\label{rhoDeliMomAb}
	\Eb \big( \big| \rho(\Deli X^{(n)})\big| &\ind_{\{| \Deli X^{(n)} | > v_n \}} \ind_{Q_n} \big) \leq \Eb \big( \big| \rho(\Deli \tilde X^{\prime n} + \Deli \tibigj)\big| \ind_{\{| \Deli \tibigj | > v_n/2 \}} \ind_{Q_n} \big) \nonumber \\
	&\leq \Eb \big( \big( \big| \rho(\Deli \tibigj)\big| + \big| \rho'(\xi_i^n) \Deli \tilde X^{\prime n} \big| \big)\ind_{\{| \Deli \tibigj | > v_n/2 \}} \ind_{Q_n} \big) \nonumber \\
	&\leq K \Eb \big( \big(1 \wedge \big| \Deli \tibigj \big|^p \big) \ind_{\{| \Deli \tibigj | > u_n \}} \ind_{Q_n} + v_n \big| \Deli \tibigj \big|^{p-1} \ind_{\{| \Deli \tibigj | > u_n \}} \ind_{Q_n} \big),
	\end{align}
	for some $\xi_i^n$ between $\Deli \tibigj$ and $\Deli \tilde X^{\prime n} + \Deli \tibigj$ using the mean value theorem and the fact that $| \Deli \tilde X^{\prime n} | \leq v_n/2$ on $Q_n$. Notice furthermore that due to $| \Deli \tilde X^{\prime n} | \leq v_n/2$ on $Q_n$ the condition $| \Deli X^{(n)} | > v_n$ implies $| \Deli \tibigj | > v_n/2$ and consequently $| \Deli \tibigj | > | \Deli \tilde X^{\prime n} |$. The final inequality in \eqref{rhoDeliMomAb} follows with the assumptions on $\rho$ and the definition of $u_n = (v_n)^\ell$ with $\ell >1$ in \eqref{unseqDefEq3}, such that $u_n < v_n/2$ holds for large $n \in \N$. On $Q_n$ $\Deli \tibigj$ is either zero or equal to the only jump in $((i-1)\Delta_n,i\Delta_n]$ of absolute size larger than $u_n$. Thus, the definition of the predictable compensator of an optional $\Pc'$-$\sigma$-finite random measure (Theorem II.1.8 in \cite{JacShi02}) gives
	\begin{align}
	\label{erSumAbschae}
	\Eb \big( \big(1 \wedge \big| \Deli \tibigj \big|^p \big) \ind_{\{| \Deli \tibigj | > u_n \}} \ind_{Q_n} \big) &= \Eb \big(\big(( 1 \wedge | \taili |^p) \ind_{\{|\taili| > u_n\}} \ind_{Q_n} \ind_{\{(i-1)\Delta_n < s \leq i \Delta_n\}} \big) \star \mu^{(n)} \big) \nonumber \\
	&\leq  \Eb \big( (1 \wedge | \taili |^p) \ind_{\{(i-1)\Delta_n < s \leq i \Delta_n\}} \star \mu^{(n)} \big) \nonumber \\
	&=  n \Delta_n \int_{(i-1)/n}^{i/n} \int (1 \wedge |\taili|^p) g^{(n)}(y,d\taili) dy \leq K \Delta_n,
	\end{align}
	where the last inequality above is a consequence of Assumption \ref{Cond1}\eqref{BlGetCond} and $p>\beta$. With the same reasoning we obtain for the second summand in \eqref{rhoDeliMomAb}
	\begin{align}
	\label{zwSumAbschae}
	\Eb \big(\big| \Deli \tibigj \big|^{p-1} \ind_{\{| \Deli \tibigj | > u_n \}} \ind_{Q_n} \big) &= \Eb \big( \big( |\taili|^{p-1} \ind_{\{|\taili| > u_n\}} \ind_{Q_n}\ind_{\{(i-1)\Delta_n < s \leq i \Delta_n\}} \big) \star \mu^{(n)} \big) \nonumber \\
	&\leq \Eb \big( |\taili|^{p-1} \ind_{\{(i-1)\Delta_n < s \leq i \Delta_n\}} \star \mu^{(n)} \big) \nonumber \\
	&=  n \Delta_n \int_{(i-1)/n}^{i/n} \int |\taili|^{p-1} g^{(n)}(y,d\taili)dy \leq K \Delta_n,
	\end{align}
	using Assumption \ref{Cond1}\eqref{LevMeaMomCond} for the last estimate above. \eqref{rhoDeliMomAb}, \eqref{erSumAbschae} and \eqref{zwSumAbschae} yield \eqref{MomKlEq}. In order to prove \eqref{PrMomKlEq} we use the mean value theorem, the definition of $Q_n$ and the assumptions on $\rho$ to obtain for $i \neq j$ similar to \eqref{rhoDeliMomAb}
	\begin{align}
	\label{ZwFakkgDn}
	\Eb &\big( \big| \rho(\Deli X^{(n)})\big| \big| \rho(\Delj X^{(n)})\big| \ind_{\{| \Deli X^{(n)} | > v_n \}} \ind_{\{| \Delj X^{(n)} | > v_n \}} \ind_{Q_n} \big) \nonumber \\
	&\leq \Eb \big( \big( \big| \rho(\Deli \tibigj)\big| + \big| \rho'(\xi_i^n) \Deli \tilde X^{\prime n} \big| \big)\big( \big| \rho(\Delj \tibigj)\big| + \big| \rho'(\xi_j^n) \Delj \tilde X^{\prime n} \big| \big) \times \nonumber \\
	&\hspace{85mm} \times \ind_{\{| \Deli \tibigj | > v_n/2 \}} \ind_{\{| \Delj \tibigj | > v_n/2 \}}\ind_{Q_n} \big) \nonumber\\
	&\leq K \Eb \big( \big( 1 \wedge \big|\Deli \tibigj \big|^p \big) \big( 1 \wedge \big|\Delj \tibigj \big|^p \big) \ind_{\{| \Deli \tibigj | > u_n \}} \ind_{\{| \Delj \tibigj | > u_n \}}\ind_{Q_n} \big) \nonumber\\
	&\hspace{2cm} + K v_n \Eb \big( \big( 1 \wedge \big|\Deli \tibigj \big|^p \big) \big| \Delj \tibigj \big|^{p-1} \ind_{\{| \Deli \tibigj | > u_n \}} \ind_{\{| \Delj \tibigj | > u_n \}}\ind_{Q_n} \big) \nonumber\\
	&\hspace{2cm} + K v_n \Eb \big( \big( 1 \wedge \big|\Delj \tibigj \big|^p \big) \big| \Deli \tibigj \big|^{p-1} \ind_{\{| \Deli \tibigj | > u_n \}} \ind_{\{| \Delj \tibigj | > u_n \}}\ind_{Q_n} \big) \nonumber\\
	&\hspace{2cm} + K v_n^2 \Eb \big( \big| \Deli \tibigj \big|^{p-1}  \big| \Delj \tibigj \big|^{p-1} \ind_{\{| \Deli \tibigj | > u_n \}} \ind_{\{| \Delj \tibigj | > u_n \}}\ind_{Q_n} \big),
	\end{align}
	for some $\xi_i^n$ between $\Deli \tibigj$ and $\Deli \tilde X^{\prime n} + \Deli \tibigj$ and some $\xi_j^n$ between $\Delj \tibigj$ and $\Delj \tilde X^{\prime n} + \Delj \tibigj$. $\Deli \tibigj$ is on $Q_n$ either zero or equal to the only jump of absolute size larger than $u_n$ on the interval $((i-1)\Delta_n,i\Delta_n]$. The same is true for $\Delj \tibigj$ on $((j-1)\Delta_n,j \Delta_n]$. Therefore, the definition of the predictable compensator of an optional $\Pc'$-$\sigma$-finite random measure yields for the first resulting summand in \eqref{ZwFakkgDn}
	\begin{align}
	\label{IndepAusnu}
	\Eb \big( \big( 1 &\wedge \big|\Deli \tibigj \big|^p \big) \big( 1 \wedge \big|\Delj \tibigj \big|^p \big) \ind_{\{| \Deli \tibigj | > u_n \}} \ind_{\{| \Delj \tibigj | > u_n \}}\ind_{Q_n} \big) \nonumber \\
	&= \Eb \Big( \big( (1 \wedge |\taili|^p) \ind_{\{|\taili| > u_n\}} \ind_{\{(i-1)\Delta_n < s \leq i \Delta_n \}} \star \mu^{(n)} \big) \times \nonumber\\
	&\hspace{5cm} \times \big( (1 \wedge |\taili|^p) \ind_{\{|\taili| > u_n\}} \ind_{\{(j-1)\Delta_n < s \leq j \Delta_n \}}  \star \mu^{(n)} \big) \ind_{Q_n} \Big) \nonumber\\
	&\leq \Eb \Big( \big( (1 \wedge |\taili|^p)  \ind_{\{(i-1)\Delta_n < s \leq i \Delta_n \}} \star \mu^{(n)} \big) \big( (1 \wedge |\taili|^p) \ind_{\{(j-1)\Delta_n < s \leq j \Delta_n \}}  \star \mu^{(n)} \big) \Big)\nonumber\\
	&= \Eb \big( (1 \wedge |\taili|^p)  \ind_{\{(i-1)\Delta_n < s \leq i \Delta_n \}} \star \mu^{(n)} \big) \Eb \big( (1 \wedge |\taili|^p) \ind_{\{(j-1)\Delta_n < s \leq j \Delta_n \}}  \star \mu^{(n)} \big) \leq K \Delta_n^2.
	\end{align}
	The final equality above follows using the fact that $\mu^{(n)}$ is a Poisson random measure and thus both involved factors are independent (see Theorem II.4.8 in \cite{JacShi02}). The last estimate in \eqref{IndepAusnu} is a consequence of \eqref{erSumAbschae}. The remaining summands in \eqref{ZwFakkgDn} can be treated similarly by exploiting the properties of a Poisson random measure as well as \eqref{erSumAbschae} and \eqref{zwSumAbschae}.
\end{proof}

\section{Results on the limiting process from Theorem \ref{ConvThm}}
\label{appE}
\def\theequation{D.\arabic{equation}}
\setcounter{equation}{0}

\subsection{Useful properties of the Gaussian limit and its covariance semimetric}
\label{SecappE1}

In the following we collect several lemmas which are useful to obtain bounds for the expectation of sup-functionals of the process $\Gb_f$. In particular, we apply them in the proof of Proposition \ref{GalKonvLem} in order to show asymptotical uniform $d$-equicontinuity  in probability of a sequence of processes $\Gb_{f_n}$ for some suitable semimetric $d$.

\begin{lemma}
	\label{8MomCalLem}
	Grant Assumption \ref{Cond1}, let $f: \R \to \R$ be Borel measurable with $|f(\taili)| \leq K(1\wedge |\taili|^p)$ for all $\taili \in \R$ and some $K>0$. Furthermore, let $\Gb_f$ be the tight centered Gaussian process in $\linner$ defined in Theorem \ref{ConvThm}. Then for $(\gseqi_1,\dfi_1),(\gseqi_2,\dfi_2) \in \netir$ the $L^8$-norm satisfies
	\begin{align*}
	\|\Gb_{f}(\gseqi_1,\dfi_1) - \Gb_{f}(\gseqi_2,\dfi_2) \|_8 = 105^{\frac 18} d_{f}((\gseqi_1,\dfi_1);(\gseqi_2,\dfi_2)),
	\end{align*}
	with $d_f$ the semimetric defined in \eqref{drhodef}.
\end{lemma}

\begin{proof}
	For $(\gseqi_1,\dfi_1),(\gseqi_2,\dfi_2) \in \netir$ with $\gseqi_1 \leq \gseqi_2$ we have
	\begin{align*}
	\Eb \big( &\Gb_{f}(\gseqi_1,\dfi_1) - \Gb_{f}(\gseqi_2,\dfi_2) \big)^2  \\
	&= H_{f}((\gseqi_1,\dfi_1);(\gseqi_1,\dfi_1)) - 2 H_{f}((\gseqi_1,\dfi_1);(\gseqi_2,\dfi_2)) + H_{f}((\gseqi_2,\dfi_2);(\gseqi_2,\dfi_2)) \\
	&= \int_0^{\gseqi_1} \int_{-\infty}^{\dfi_1} f^2(\taili) g_0(y,d\taili) dy - 2 \int_0^{\gseqi_1} \int_{-\infty}^{\mindfi} f^2(\taili) g_0(y,d\taili) dy + \int_0^{\gseqi_2} \int_{-\infty}^{\dfi_2} f^2(\taili) g_0(y,d\taili) dy \\
	&= d^2_{f}((\gseqi_1,\dfi_1);(\gseqi_2,\dfi_2)),
	\end{align*}
	where the last equation follows by distinguishing the cases $\dfi_1 \leq \dfi_2$ and $\dfi_2 \leq \dfi_1$. Therefore, using the properties of the normal distribution we obtain for the $L^8$-norm
	\begin{align*}
	\|\Gb_{f}(\gseqi_1,\dfi_1) - \Gb_{f}(\gseqi_2,\dfi_2) \|_8 = 105^{\frac 18} d_{f}((\gseqi_1,\dfi_1);(\gseqi_2,\dfi_2)),
	\end{align*}
	for arbitrary $(\gseqi_1,\dfi_1),(\gseqi_2,\dfi_2) \in \netir$.
\end{proof}

\begin{lemma}
	\label{Gfseparabel}
	Grant Assumption \ref{Cond1} and let $f: \R \to \R$ be Borel measurable with $|f(\taili)| \leq K(1\wedge |\taili|^p)$ for all $\taili \in \R$ and some $K>0$. Then the tight centered Gaussian process $\Gb_f$ defined in Theorem \ref{ConvThm} is separable with respect to the semimetric $d_f$ in the sense of Theorem 2.2.4 in \cite{VanWel96}. More precisely, for every $\delta >0$
	\begin{equation*}
	\sup \limits_{d_{f}((\gseqi_1,\dfi_1);(\gseqi_2,\dfi_2)) < \delta} \big| \Gb_{f}(\gseqi_1,\dfi_1) - \Gb_{f}(\gseqi_2,\dfi_2) \big| =
	\sup \limits_{\stackrel{d_{f}((\gseqi_1,\dfi_1);(\gseqi_2,\dfi_2)) < \delta}{(\gseqi_1,\dfi_1),(\gseqi_2,\dfi_2) \in (\netir) \cap \Q^2}} \big| \Gb_{f}(\gseqi_1,\dfi_1) - \Gb_{f}(\gseqi_2,\dfi_2) \big|
	\end{equation*}
	holds almost surely.
\end{lemma}

\begin{proof}
	By the assumptions on the involved quantities for each $\dfi \in \R$ the function $[0,1] \ni \gseqi \mapsto \int_0^\gseqi \int_{-\infty}^\dfi f^2(\taili) g_0(y,d\taili) dy$ is continuous and for each $\gseqi \in [0,1]$ the function $\R \ni \dfi \mapsto \int_0^\gseqi \int_{-\infty}^\dfi f^2(\taili) g_0(y,d\taili) dy$ is right-continuous. As a consequence, we can find for every  $\eps >0$ and $(\gseqi_1,\dfi_1) \in \netir$ a $(\gseqi_2,\dfi_2) \in (\netir) \cap \Q^2$ with $d_{f}((\gseqi_1,\dfi_1);(\gseqi_2,\dfi_2)) < \eps$. Thus, for every $\delta >0$
	\begin{equation*}
	\sup \limits_{d_{f}((\gseqi_1,\dfi_1);(\gseqi_2,\dfi_2)) < \delta} \big| \Gb_{f}(\gseqi_1,\dfi_1) - \Gb_{f}(\gseqi_2,\dfi_2) \big|=
	\sup \limits_{\stackrel{d_{f}((\gseqi_1,\dfi_1);(\gseqi_2,\dfi_2)) < \delta}{(\gseqi_1,\dfi_1),(\gseqi_2,\dfi_2) \in (\netir) \cap \Q^2}} \big| \Gb_{f}(\gseqi_1,\dfi_1) - \Gb_{f}(\gseqi_2,\dfi_2) \big|
	\end{equation*}
	holds almost surely, because the sample paths of $\Gb_{f}$ are almost surely uniformly $d_{f}$-continuous.
\end{proof}

\begin{lemma}
	\label{PackNumAb}
	Grant Assumption \ref{Cond1} and let $f: \R \to \R$ be Borel measurable with $|f(\taili)| \leq C(1\wedge |\taili|^p)$ for all $\taili \in \R$ and some $C>0$. Then for $d_f$ the semimetric defined in \eqref{drhodef} the semimetric space $([0,1] \times \R, d_f)$ is totally bounded and there exists a $K>0$ which depends only on $C$ such that for every $\eps >0$
	\[
	D(\eps,d_f) \leq K/\eps^4,
	\]
	where $D(\eps,d_f)$ denotes the packing number  of $[0,1] \times \R$ with respect to $d_f$ at distance $\eps >0$.
\end{lemma}

\begin{proof}
	By the well-known relation $D(\eps, d_{f}) \leq N(\eps/2,d_{f})$ of the packing number and the covering number $N(\eps/2,d_{f})$ of $([0,1] \times \R, d_f)$ at radius $\eps/2$ it suffices to show that there exists a $K>0$ with $N(\eps/2,d_{f}) \leq K/\eps^4$ for every $\eps >0$. 
	
	The measure $\Bb \ni A \mapsto C^2 \int_0^1 \int_A (1 \wedge |\taili|^{2p}) g_0(y,d\taili) dy$ is finite. Therefore, for each $\eps > 0$ we can find a finite partition $\{\dfi_0= - \infty < \dfi_1 < \ldots < \dfi_{m} < \dfi_{m+1} = \infty\}$ of $\overline \R$ with $m \leq K/\eps^2$ for some $K>0$ which depends only on $C$ such that $C^2\int_0^1 \int_{\dfi_j}^{\dfi_{j+1}} (1 \wedge |\taili|^{2p}) g_0(y,d \taili) dy < \eps^2/64$ for each $j = 0, \ldots,m$. For the same reason there is also a finite partition $\{\gseqi_0 =0 < \gseqi_1 < \ldots < \gseqi_\ell < \gseqi_{\ell+1} = 1\}$ of $[0,1]$ with $\ell \leq K/\eps^2$ and $C^2 \int_{\gseqi_i}^{\gseqi_{i+1}} \int (1 \wedge |\taili|^{2p}) g_0(y,d\taili) dy < \eps^2/64$ for all $i=0, \ldots,\ell$. Furthermore, consider the collection $M:=\{(\gseqi_i,\dfi_j) \mid i = 1, \ldots,\ell; j=1, \ldots,m\}$ which consists of at most $K/\eps^4$ points. Then for an arbitrary $(\gseqi,\dfi) \in \netir$ let $i_0 \in \text{argmin}\{|\gseqi_i - \gseqi| \mid i = 1, \ldots, \ell\}$, $j_0 \in \text{argmin}\{|\dfi_j - \dfi| \mid j = 1, \ldots, m\}$ and choose $i_1 \in \{0, \ldots, \ell +1\}$, $j_1 \in \{0, \ldots,m+1\}$ such that $\gseqi \in [\gseqi_{i_0} \wedge \gseqi_{i_1}, \gseqi_{i_0} \vee \gseqi_{i_1}]$ as well as $\dfi \in [\dfi_{j_0} \wedge \dfi_{j_1}, \dfi_{j_0} \vee \dfi_{j_1}]$ to obtain
	\begin{align*}
	d_{f}((\gseqi,\dfi);(\gseqi_{i_0},\dfi_{j_0})) &\leq  2  \max \Big\{ \Big( C^2 \int_0^1 \int_{\dfi_{j_0} \wedge \dfi_{j_1}}^{\dfi_{j_0} \vee \dfi_{j_1}} (1 \wedge |\taili|^{2p}) g_0(y,d\taili) dy \Big)^{\frac 12} , \nonumber\\ 
	&\hspace{46mm} \Big( C^2 \int_{\gseqi_{i_0} \wedge \gseqi_{i_1}}^{\gseqi_{i_0} \vee \gseqi_{i_1}} \int (1 \wedge |\taili|^{2p}) g_0(y,d\taili) dy \Big)^{\frac 12} \Big\} \nonumber \\
	&\leq \eps/4 < \eps/2.
	\end{align*}
	Thus, we have $N(\eps/2,d_{f}) \leq K/ \eps^4$, because by the inequality above the $d_{f}$-balls with radius $\eps/2$ around the points of $M$ cover $\netir$.
\end{proof}

\subsection{An auxiliary result on the supremum of the Gaussian limit}

A further application of the lemmas in Section \ref{SecappE1} is the proposition below which is necessary to prove Theorem \ref{CondConvThm}.

\begin{proposition}
	\label{Gbrhoapcopr}
	Grant Assumption \ref{Cond1} and for $\alpha >0$ let $\rho_\alpha^\circ$ be the function defined prior to \eqref{ImpQuantDef}. Then we have
	\begin{align*}
	\lim_{\alpha \to 0} \Eb \Big( \sup \limits_{(\gseqi,\dfi) \in \netir} |\Gb_{\rho_\alpha^\circ}(\gseqi,\dfi)| \Big) = 0.
	\end{align*}
\end{proposition}

\begin{proof}
	We want to use Corollary 2.2.5 in \cite{VanWel96} for the convex, non-decreasing, non-zero function $\varphi(x)= x^8$ which clearly satisfies $\varphi(0)=0$ and $\limsup_{x,y \rightarrow \infty}$ $ \varphi(x)\varphi(y)/$ $\varphi(cxy) < \infty$ for some constant $c >0$. Due to Lemma \ref{Gfseparabel} the process $\Gb_{\rho_\alpha^\circ}$ is separable in the sense of this corollary. Furthermore, Lemma \ref{8MomCalLem} shows
	\begin{align*}
	\|\Gb_{\rho_\alpha^\circ}(\gseqi_1,\dfi_1) - \Gb_{\rho_\alpha^\circ}(\gseqi_2,\dfi_2) \|_8 = 105^{\frac 18} d_{\rho_\alpha^\circ}((\gseqi_1,\dfi_1);(\gseqi_2,\dfi_2)),
	\end{align*}
	for all $(\gseqi_1,\dfi_1),(\gseqi_2,\dfi_2) \in \netir$, where 
	\begin{align*}
	d_{\rho_\alpha^\circ}((\gseqi_1,\dfi_1);(\gseqi_2,\dfi_2)) = \Big\{ \int_0^{\gseqi_1} \int_{\mindfi}^{\maxdfi}(\rho_\alpha^\circ(\taili))^2 g_0(y,d\taili) dy + \int_{\gseqi_1}^{\gseqi_2} \int_{-\infty}^{\dfi_2} (\rho_\alpha^\circ(\taili))^2 g_0(y,d\taili) dy \Big\}^{1/2}, \quad (\gseqi_1 \le \gseqi_2)
	\end{align*}
	is the semimetric for which the sample paths of $\Gb_{\rho_\alpha^\circ}$ are almost surely uniformly continuous. Thus, according to Corollary 2.2.5 in \cite{VanWel96} there exists a constant $K >0$ which does not depend on $\rho$ or $\alpha$ such that 
	\begin{align}
	\label{ChaiSupAb}
	\Big\| \sup\limits_{(\gseqi_1,\dfi_1),(\gseqi_2,\dfi_2) \in \netir} |\Gb_{\rho_\alpha^\circ}(\gseqi_1,\dfi_1) - \Gb_{\rho_\alpha^\circ}(\gseqi_2,\dfi_2) \Big\|_8 \leq K \int_0^{\bar d(\alpha)} D(\eps,d_{\rho_\alpha^\circ})^{\frac 18} d\eps,
	\end{align}
	where $D(\eps,d_{\rho_\alpha^\circ})$ denotes the packing number of $\netir$ at distance $\eps$ with respect to the semimetric $d_{\rho_\alpha^\circ}$ and where
	\begin{align}
	\label{diameter}
	\bar d(\alpha) &= \text{diam}(\netir;d_{\rho_\alpha^\circ}) = \sup\limits_{(\gseqi_1,\dfi_1),(\gseqi_2,\dfi_2) \in \netir} d_{\rho_\alpha^\circ}((\gseqi_1,\dfi_1);(\gseqi_2,\dfi_2)) \nonumber \\
	&\leq \Big\{ 2 \int_0^1 \int (\rho_\alpha^\circ(\taili))^2g_0(y,d\taili) dy \Big\}^{\frac 12} \stackrel{\alpha \to 0}\longrightarrow 0
	\end{align}
	is the diameter of $\netir$ with respect to $d_{\rho_\alpha^\circ}$. The convergence in the display above holds due to Lebesgue's dominated convergence theorem by the assumptions on $\rho_\alpha^\circ$ and on $g_0$. Moreover, Lemma \ref{PackNumAb} gives a constant $K>0$ which is independent of $\alpha$ such that $D(\eps,d_{\rho_\alpha^\circ}) \leq K/\eps^4$. Thus, with \eqref{ChaiSupAb} and \eqref{diameter} we obtain the desired result:
	\begin{align*}
	\Eb \Big( \sup \limits_{(\gseqi,\dfi) \in \netir} |\Gb_{\rho_\alpha^\circ}(\gseqi,\dfi)| \Big) &\leq \| \Gb_{\rho_\alpha^\circ}(0,0) \|_2 +  \Big\| \sup\limits_{(\gseqi_1,\dfi_1),(\gseqi_2,\dfi_2) \in \netir} |\Gb_{\rho_\alpha^\circ}(\gseqi_1,\dfi_1) - \Gb_{\rho_\alpha^\circ}(\gseqi_2,\dfi_2) | \Big\|_8 \\
	&\leq K \int_0^{\bar d(\alpha)} \eps^{-\frac 12} d\eps = K \bar d(\alpha)^{\frac 12} \stackrel{\alpha \to 0}\longrightarrow 0.
	\end{align*}
\end{proof}

\section{Further auxiliary Results}
\label{appF}
\def\theequation{E.\arabic{equation}}
\setcounter{equation}{0}

The following lemma shows that two bootstrapped random elements with values in some metric space $(\Db,d)$ which differ only by a term of order $o_\Prob(1)$ converge simultaneously weakly conditional on the data in probability. The following results can be shown by similar reasonings as given in the proof of     Lemma A.1 in  \cite{Buc11} and  Proposition A.2 in \cite{BueHofVetDet15}.

\begin{lemma}
	\label{oP1glzcwecole}
	Let $\hat G_n = \hat G_n(X_1, \ldots, X_n, \xi_1, \ldots, \xi_n)$ and $\hat H_n = \hat H_n(X_1, \ldots, X_n, \xi_1,$ $ \ldots, \xi_n)$ be two sequences of bootstrapped elements with values in some metric space $(\Db,d)$ such that $d(\hat G_n, \hat H_n) \pto 0$. Then for a tight Borel measurable process $G$ in $\Db$, we have $\hat G_n \weakP G$ if and only if $\hat H_n \weakP G$.
\end{lemma}

The next auxiliary result is useful in order to show consistency of the test procedures in this paper. In the assertion of this proposition $(\xi^{\scriptscriptstyle (b)})_{b=1, \ldots,B}$ for some $B \in \N$ denote independent sequences $\xi^{\scriptscriptstyle (b)} = (\xi^{\scriptscriptstyle (b)}_{i})_{i \in\N}$ of random variables satisfying Assumption \ref{MultiplAss}. Furthermore, $\hat \Tb_{\rho,\xi^{\scriptscriptstyle (b)}}^{\scriptscriptstyle (n)}$ and $\hat \Hb_{\rho,\xi^{\scriptscriptstyle (b)}}^{\scriptscriptstyle (n)}$ denote the processes defined in \eqref{TbrhohaDef} and \eqref{hatHbrhondeq}, respectively, calculated with respect to the $b$-th multiplier sequence $\xi^{\scriptscriptstyle (b)}$.

\begin{proposition}
	\label{JointConvProp}
	Let $B\in\N$. If ${\bf H}_1^{(loc)}$ is true and each of the independent multiplier sequences $(\xi^{\scriptscriptstyle (b)})_{b=1,\ldots,B}$ satisfies Assumption \ref{MultiplAss}, then we have
	\begin{align*}
	\big( \Tb_\rho^{(n)}, \hat \Tb_{\rho,\xi^{(1)}}^{(n)}, \ldots, \hat \Tb_{\rho,\xi^{(B)}}^{(n)} \big) \weak \big( \Tb_\rho + \Tb_{\rho,g_1}, \Tb_{\rho,(1)}, \ldots, \Tb_{\rho,(B)} \big)
	\end{align*}
	in $(\linner)^{B+1}$ and 
	\begin{align*}
	\big( \Hb_\rho^{(n)}, \hat \Hb_{\rho,\xi^{(1)}}^{(n)}, \ldots, \hat \Hb_{\rho,\xi^{(B)}}^{(n)} \big) \weak \big( \Hb_\rho + D_\rho^{(g_1)}, \Hb_{\rho,(1)}, \ldots, \Hb_{\rho,(B)} \big)
	\end{align*}
	in $(\linctr)^{B+1}$, where $\weak$ denotes (unconditional) weak convergence (with respect to the (joint) probability measure $\mathbb P$), furthermore $\Tb_{\rho,(1)}, \ldots, \Tb_{\rho,(B)}$ are independent copies of the Gaussian process $\Tb_\rho$ in Theorem \ref{CUSUMTrhowc} and $\Hb_{\rho,(1)}, \ldots, \Hb_{\rho,(B)}$ are independent copies of the stochastic process $\Hb_\rho$ defined in Theorem \ref{SchwKentmotv}.
\end{proposition}

\end{document}